\documentclass[10pt,letterpaper,oneside]{amsart}
\usepackage[T1]{fontenc}
\usepackage[latin9]{inputenc}
\usepackage{mathtools}
\usepackage{amstext}
\usepackage{amsthm}
\usepackage{amssymb}
\usepackage{graphicx}
\usepackage{esint}
\usepackage[all]{xy}
\usepackage[unicode=true,pdfusetitle,
 bookmarks=true,bookmarksnumbered=false,bookmarksopen=false,
 breaklinks=false,pdfborder={0 0 1},backref=false,colorlinks=false]
 {hyperref}

\makeatletter

\pdfpageheight\paperheight
\pdfpagewidth\paperwidth

\numberwithin{equation}{section}
\theoremstyle{plain}
\newtheorem{thm}{\protect\theoremname}[section]
\theoremstyle{plain}
\newtheorem{cor}[thm]{\protect\corollaryname}
\theoremstyle{plain}
\newtheorem{lem}[thm]{\protect\lemmaname}
\theoremstyle{remark}
\newtheorem{rem}[thm]{\protect\remarkname}
\theoremstyle{plain}
\newtheorem{prop}[thm]{\protect\propositionname}

\usepackage[initials]{amsrefs}

\makeatother

\providecommand{\corollaryname}{Corollary}
\providecommand{\lemmaname}{Lemma}
\providecommand{\propositionname}{Proposition}
\providecommand{\remarkname}{Remark}
\providecommand{\theoremname}{Theorem}

\begin{document}
\title[Blow-up solutions to the self-dual CSS]{On pseudoconformal blow-up solutions to the self-dual Chern-Simons-Schr\"odinger
equation: existence, uniqueness, and instability}
\author{\noindent Kihyun Kim}
\email{khyun1215@kaist.ac.kr}
\address{Department of Mathematical Sciences, Korea Advanced Institute of Science
and Technology, 291 Daehak-ro, Yuseong-gu, Daejeon 34141, Korea}
\author{\noindent Soonsik Kwon}
\email{soonsikk@kaist.edu}
\address{Department of Mathematical Sciences, Korea Advanced Institute of Science
and Technology, 291 Daehak-ro, Yuseong-gu, Daejeon 34141, Korea}
\keywords{Chern-Simons-Schr\"odinger equation, blow-up, pseudoconformal, self-duality,
rotational instability}
\subjclass[2010]{35B44, 35Q55}
\begin{abstract}
We consider the self-dual Chern-Simons-Schr\"odinger equation (CSS),
also known as a gauged nonlinear Schr\"odinger equation (NLS). CSS
is $L^{2}$-critical, admits solitons, and has the pseudoconformal
symmetry. These features are similar to the $L^{2}$-critical NLS.
In this work, we consider pseudoconformal blow-up solutions under
$m$-equivariance, $m\geq1$. Our result is threefold. Firstly, we
construct a pseudoconformal blow-up solution $u$ with given asymptotic
profile $z^{\ast}$: 
\[
\Big[u(t,r)-\frac{1}{|t|}Q\Big(\frac{r}{|t|}\Big)e^{-i\frac{r^{2}}{4|t|}}\Big]e^{im\theta}\to z^{\ast}\qquad\text{in }H^{1}
\]
as $t\to0^{-}$, where $Q(r)e^{im\theta}$ is a static solution. Secondly,
we show that such blow-up solutions are unique in a suitable class.
Lastly, yet most importantly, we exhibit an instability mechanism
of $u$. We construct a continuous family of solutions $u^{(\eta)}$,
$0\leq\eta\ll1$, such that $u^{(0)}=u$ and for $\eta>0$, $u^{(\eta)}$
is a global scattering solution. Moreover, we exhibit a rotational
instability as $\eta\to0^{+}$: $u^{(\eta)}$ takes an abrupt spatial
rotation by the angle 
\[
\Big(\frac{m+1}{m}\Big)\pi
\]
on the time interval $|t|\lesssim\eta$.

We are inspired by works in the $L^{2}$-critical NLS. In the seminal
work of Bourgain and Wang (1997), they constructed such pseudoconformal
blow-up solutions. Merle, Rapha\"el, and Szeftel (2013) showed an
instability of Bourgain-Wang solutions. Although CSS shares many features
with NLS, there are essential differences and obstacles over NLS.
Firstly, the soliton profile to CSS shows a slow polynomial decay
$r^{-(m+2)}$. This causes many technical issues for small $m$. Secondly,
due to the nonlocal nonlinearities, there are strong long-range interactions
even between functions in far different scales. This leads to a nontrivial
correction of our blow-up ansatz. Lastly, the instability mechanism
of CSS is completely different from that of NLS. Here, the phase rotation
is the main source of the instability. On the other hand, the self-dual
structure of CSS is our sponsor to overcome these obstacles. We exploited
the self-duality in many places such as the linearization, spectral
properties, and construction of modified profiles.
\end{abstract}

\maketitle
\global\long\def\R{\mathbb{R}}%
\global\long\def\C{\mathbb{C}}%
\global\long\def\Z{\mathbb{Z}}%
\global\long\def\N{\mathbb{N}}%
\global\long\def\D{\mathbf{D}}%
\global\long\def\Im{\mathrm{Im}}%
\global\long\def\Re{\mathrm{Re}}%
\global\long\def\rad{\mathrm{rad}}%
\global\long\def\tilde#1{\widetilde{#1}}%
\global\long\def\re{\mathrm{re}}%
\global\long\def\im{\mathrm{im}}%
\global\long\def\err{\mathrm{Err}}%

\tableofcontents{}

\section{\label{sec:intro}Introduction}

The Chern-Simons-Schr\"odinger equations can be viewed as gauged
nonlinear Schr\"odinger equations that generalize the one dimensional
cubic nonlinear Schr\"odinger equation to the planar domain, with
nonvanishing Chern-Simons gauge coupling. Interesting novel features
are the \emph{self-duality} and presence of vortex solitons. The goal
of this work is to study a dynamical feature of this \emph{gauged
nonlinear Schr\"odinger equation} under self-duality. More specifically,
we study pseudoconformal blow-up solutions. We construct a pseudoconformal
blow-up solution for each given asymptotic profile (Theorem \ref{thm:BW-sol})
and show that such a solution is unique (Theorem \ref{thm:cond-uniq}).
Next, we exhibit a rotational instability of the blow-up solution
(Theorem \ref{thm:instability}). Our work is inspired by works in
the context of $L^{2}$-critical nonlinear Schr\"odinger equation,
by Bourgain and Wang \cite{BourgainWang1997}, and Merle, Rapha\"el,
and Szeftel \cite{MerleRaphaelSzeftel2013AJM}.

The Chern-Simons theory is a three dimensional topological gauge field
theory whose action is an integral of Chern-Simons 3-form \cite{ChernSimons1974Ann.Math.}.
It has applications in condensed matter physics, knot theory, and
low dimensional invariant theory.

As the Chern-Simons theory is formulated on three dimensional domain
(space-time), it is applicable to describe the dynamics of particles
confined to a spatial plane, so called \emph{planar} physical phenomena,
e.g. quantum Hall effect and high temperature superconductivity. This
is a sharp contrast to the Yang-Mills or Maxwell theory that takes
place on four dimensional space-time. In the 90\textquoteright s,
relativistic and non relativistic Chern-Simons models were introduced
to study vortex solutions in planar quantum electromagnetic dynamics.
The Chern-Simons-Schr\"odinger equations \cite{JackiwPi1990PRL,JackiwPi1990PRD,JackiwPi1991PRD,JackiwPi1992Progr.Theoret.}
are nonrelativistic quantum models describing the dynamics of a large
number of charged particles in the plane, that interact among themselves
and the electromagnetic gauge field.

\subsection{Covariant formulation}

The Chern-Simons-Schr\"odinger model is a Lagrangian field theory
on $\R^{1+2}$ associated to the action
\begin{align*}
\mathcal{S}[\phi,A] & \coloneqq\int_{\R^{1+2}}\Big(\frac{1}{2}\Im(\overline{\phi}\D_{t}\phi)+\frac{1}{2}|\D_{x}\phi|^{2}-\frac{g}{4}|\phi|^{4}\Big)+\int_{\R^{1+2}}\frac{1}{2}A\wedge dA.
\end{align*}
Here, $\phi:\R^{1+2}\to\C$ is a scalar field, $\D_{\alpha}\coloneqq\partial_{\alpha}+iA_{\alpha}$
for $\alpha\in\{0,1,2\}$ is the covariant derivative, $A\coloneqq A_{0}dt+A_{1}dx_{1}+A_{2}dx_{2}$
is the associated 1-form, $|\D_{x}\phi|^{2}\coloneqq|\D_{1}\phi|^{2}+|\D_{2}\phi|^{2}$,
and $g>0$ is the strength of the nonlinearity.

In a more abstract fashion, one can view $\phi$ as a section of a
$U(1)$ bundle over $\R^{1+2}$. As the topology of $\R^{1+2}$ is
trivial, there is a global orthonormal frame, on which a metric connection
$\D$ is viewed as $\D=d+iA$ for some real-valued 1-form $A$.

Computing the Euler-Lagrange equation, we obtain the \emph{Chern-Simons-Schr\"odinger}
equation 
\begin{equation}
\left\{ \begin{aligned}\D_{t}\phi & =i\D_{j}\D_{j}\phi+ig|\phi|^{2}\phi,\\
F_{01} & =-\Im(\overline{\phi}\D_{2}\phi),\\
F_{02} & =\Im(\overline{\phi}\D_{1}\phi),\\
F_{12} & =-\tfrac{1}{2}|\phi|^{2},
\end{aligned}
\right.\label{eq:CSS-cov}
\end{equation}
where $F_{jk}\coloneqq\partial_{j}A_{k}-\partial_{k}A_{j}$ are components
of the curvature 2-form $F\coloneqq dA$. Repeated index $j$ means
that we sum over $j\in\{1,2\}$. We note that 
\[
\D_{\alpha}\D_{\beta}-\D_{\beta}\D_{\alpha}=iF_{\alpha\beta}.
\]

\eqref{eq:CSS-cov} is gauge-invariant. Indeed, if $(\phi,A)$ is
a solution to \eqref{eq:CSS-cov} and $\chi:\R^{1+2}\to\R$ is any
function, then 
\begin{equation}
(e^{i\chi}\phi,A-d\chi)\label{eq:gauge-inv}
\end{equation}
 is also a solution to \eqref{eq:CSS-cov}.\footnote{Gauge invariance of \eqref{eq:CSS-cov} under \eqref{eq:gauge-inv}
holds for an arbitrary $\chi$. Under the additional assumption that
$\chi$ is compactly supported in $\R^{1+2}$, the action $\mathcal{S}$
is gauge-invariant from the algebra $\int_{\R^{1+2}}d\chi\wedge dA=\int_{\R^{1+2}}d(\chi\wedge dA)=0$.
Although the Chern-Simons density $\frac{1}{2}(A\wedge dA)$ is \emph{not
}gauge-invariant itself, its integral \emph{is} gauge-invariant. Note
that the remaining densities are gauge-invariant.} When $\chi$ is a constant function on $\R^{1+2}$, we have phase
rotation symmetry.

In addition to gauge invariance, \eqref{eq:CSS-cov} enjoys many symmetries.
It is invariant under space/time translations and spatial rotations.
It also has time reversal symmetry\footnote{The time reversal symmetry is \emph{not} simply given by conjugating
the scalar field $\phi$. In this paper, we use $\phi(t,x_{1},x_{2})\mapsto\overline{\phi}(-t,x_{1},-x_{2})$,
$A_{\alpha}(t,x_{1},x_{2})\mapsto A_{\alpha}(-t,x_{1},-x_{2})$ for
$\alpha\in\{0,2\}$, and $A_{1}(t,x_{1},x_{2})\mapsto-A_{1}(-t,x_{1},-x_{2})$.} and Galilean invariance. Moreover, \eqref{eq:CSS-cov} enjoys scaling
and pseudoconformal invariance. More precisely, assuming that $(\phi,A)$
is a solution to \eqref{eq:CSS-cov}, so is $(\tilde{\phi},\tilde A)$:
\begin{itemize}
\item \emph{$L^{2}$-scaling}: for any fixed $\lambda>0$, 
\begin{equation}
\left\{ \begin{aligned}\tilde{\phi}(t,x) & \coloneqq\lambda\phi(\lambda^{2}t,\lambda x),\\
\tilde A_{0}(t,x) & \coloneqq\lambda^{2}A_{0}(\lambda^{2}t,\lambda x),\\
\tilde A_{j}(t,x) & \coloneqq\lambda A_{j}(\lambda^{2}t,\lambda x).
\end{aligned}
\right.\label{eq:scaling}
\end{equation}
\item \emph{Pseudoconformal invariance}: (discrete form)
\begin{equation}
\left\{ \begin{aligned}\tilde{\phi}(t,x) & \coloneqq\tfrac{1}{t}e^{i\frac{|x|^{2}}{4t}}\phi(-\tfrac{1}{t},\tfrac{x}{t}),\\
\tilde A_{0}(t,x) & \coloneqq\tfrac{1}{t^{2}}A_{0}(-\tfrac{1}{t},\tfrac{x}{t})-\tfrac{x_{j}}{t^{2}}A_{j}(-\tfrac{1}{t},\tfrac{x}{t}),\\
\tilde A_{j}(t,x) & \coloneqq\tfrac{1}{t}A_{j}(-\tfrac{1}{t},\tfrac{x}{t}),
\end{aligned}
\right.\label{eq:pseudo-discrete}
\end{equation}
(or, in continuous form) for any fixed $a\in\R$, 
\begin{equation}
\left\{ \begin{aligned}\tilde{\phi}(t,x) & \coloneqq\tfrac{1}{1+at}e^{ia\frac{|x|^{2}}{4(1+at)}}\phi(\tfrac{t}{1+at},\tfrac{x}{1+at}),\\
\tilde A_{0}(t,x) & \coloneqq\tfrac{1}{(1+at)^{2}}A_{0}(\tfrac{t}{1+at},\tfrac{x}{1+at})-\tfrac{x_{j}}{(1+at)^{2}}A_{j}(\tfrac{t}{1+at},\tfrac{x}{1+at}),\\
\tilde A_{j}(t,x) & \coloneqq\tfrac{1}{1+at}A_{j}(\tfrac{t}{1+at},\tfrac{x}{1+at}).
\end{aligned}
\right.\label{eq:pseudo-conti}
\end{equation}
\end{itemize}
As the scaling invariance \eqref{eq:scaling} preserves the $L^{2}$-norm
of $\phi$, we say that \eqref{eq:CSS-cov} is \emph{$L^{2}$-critical}.
The pseudoconformal invariance is a special feature of \eqref{eq:CSS-cov}.
In the context of nonlinear Schr\"odinger equations, pseudoconformal
invariance holds only in the $L^{2}$-critical setting. It is natural
to view \eqref{eq:CSS-cov} as a gauged $L^{2}$-critical nonlinear
Schr\"odinger equation.

By Noether's principle, associated with phase and space-time translation
invariance, the \emph{charge} $M$, \emph{momentum} $J_{j}$, and
\emph{energy} $E$ are conserved:
\begin{align*}
M & \coloneqq\int_{\R^{2}}|\phi|^{2},\\
J_{j} & \coloneqq\int_{\R^{2}}\Im(\overline{\phi}\D_{j}\phi),\\
E & \coloneqq\int_{\R^{2}}\Big(\frac{1}{2}|\D_{x}\phi|^{2}-\frac{g}{4}|\phi|^{4}\Big).
\end{align*}
These are gauge-invariant quantities. There are other conservation
laws, such as the angular momentum and normalized center of mass,
associated with rotation and Galilean invariance. Finally, from the
scaling and pseudoconformal symmetries, we have \emph{virial identities
\begin{equation}
\left\{ \begin{aligned}\partial_{t}\Big(\int_{\R^{2}}|x|^{2}|\phi|^{2}\Big) & =8\Phi,\\
\partial_{t}\Phi & =2E,
\end{aligned}
\right.\label{eq:virial-identities}
\end{equation}
}where 
\[
\Phi\coloneqq\frac{1}{2}\int_{\R^{2}}\Im(\overline{\phi}\cdot x_{j}\D_{j}\phi).
\]
One can alternatively write \eqref{eq:virial-identities} as 
\begin{equation}
8t^{2}E(e^{i\frac{|x|^{2}}{4t}}\phi(0))=\int_{\R^{2}}|x|^{2}|\phi(t,x)|^{2}.\label{eq:virial-alt}
\end{equation}

The aforementioned conservation laws can also be obtained from the
\emph{pseudo-stress-energy tensor}, which is defined by 
\begin{align*}
T_{00} & \coloneqq\tfrac{1}{2}|\phi|^{2},\\
T_{j0}=T_{0j} & \coloneqq\Im(\overline{\phi}\D_{j}\phi),\\
T_{jk}=T_{kj} & \coloneqq2\Re(\overline{\D_{j}\phi}\D_{k}\phi)-\delta_{jk}(\Delta T_{00}+g|\phi|^{2}T_{00}).
\end{align*}
Here $\delta_{jk}$ is the Kronecker-delta and $\Delta$ is the Laplacian
on $\R^{2}$. We have local conservation laws\footnote{There is also a local conservation law for the energy. If we define
the energy current by $e_{0}\coloneqq\frac{1}{2}|\D_{x}\phi|^{2}-\frac{g}{4}|\phi|^{4}$
and $e_{j}\coloneqq-\Re(\overline{\D_{j}\phi}\D_{t}\phi)$, then we
have $\partial_{t}e_{0}+\partial_{j}e_{j}=0$.} 
\begin{equation}
\left\{ \begin{aligned}\partial_{t}T_{00}+\partial_{k}T_{0k} & =0,\\
\partial_{t}T_{j0}+\partial_{k}T_{jk} & =0.
\end{aligned}
\right.\label{eq:local-cons}
\end{equation}
Taking space-time integrals of $T_{\alpha0}$ with various weights,
we obtain conservation laws.

\subsubsection*{Bogomol'nyi operator}

Let $\tilde{\D}_{+}$ be the \emph{Bogomol'nyi operator} defined by
\begin{equation}
\tilde{\D}_{+}\coloneqq\D_{1}+i\D_{2}.\label{eq:def-Bogomolnyi}
\end{equation}
Its formal $L^{2}$-adjoint is given by $\tilde{\D}_{+}^{\ast}=-\D_{1}+i\D_{2}$.
Observe using \eqref{eq:CSS-cov} that 
\begin{equation}
\tilde{\D}_{+}^{\ast}\tilde{\D}_{+}=-\D_{j}\D_{j}+F_{12}=-\D_{j}\D_{j}-\tfrac{1}{2}|\phi|^{2}.\label{eq:obs}
\end{equation}
Now the $\phi$-evolution of \eqref{eq:CSS-cov} is rewritten as 
\begin{equation}
i\D_{t}\phi-\tilde{\D}_{+}^{\ast}\tilde{\D}_{+}\phi+(g-\tfrac{1}{2})|\phi|^{2}\phi=0.\label{eq:CSS-Bogom-phi}
\end{equation}
Moreover, we can rewrite the energy functional as 
\begin{equation}
E=\int_{\R^{2}}\Big(\frac{1}{2}|\tilde{\D}_{+}\phi|^{2}+\frac{1-g}{4}|\phi|^{4}\Big).\label{eq:energy-Bogom-phi}
\end{equation}

If $g<1$, the energy is positive-definite. Thus it is natural to
view \eqref{eq:CSS-cov} with $g<1$ as a defocusing equation, for
which the global well-posedness and scattering are expected. If $g>1$,
then \eqref{eq:CSS-cov} is viewed as a focusing equation that admits
non-scattering (and even blow-up) solutions. We are interested in
\eqref{eq:CSS-cov} with the critical coupling $g=1$, which is referred
to as the \emph{self-dual} case. From Section \ref{subsec:The-Static-Solution},
we restrict our discussions to the self-dual case.

\subsubsection*{Polar coordinates}

Before imposing equivariant symmetry, it is convenient to formulate
\eqref{eq:CSS-cov} in polar coordinates on $\R^{2}$. Let $r\coloneqq|x|$,
and define 
\[
\partial_{r}\coloneqq\frac{1}{r}(x_{1}\partial_{1}+x_{2}\partial_{2})\quad\text{and}\quad\partial_{\theta}\coloneqq-x_{2}\partial_{1}+x_{1}\partial_{2}.
\]
The corresponding covariant derivatives, connection, and curvature
components become 
\begin{align*}
\D_{r} & =\tfrac{1}{r}(x_{1}\D_{1}+x_{2}\D_{2}); & \D_{\theta} & =-x_{2}\D_{1}+x_{1}\D_{2};\\
A_{r} & =\tfrac{1}{r}(x_{1}A_{1}+x_{2}A_{2}); & A_{\theta} & =-x_{2}A_{1}+x_{1}A_{2};\\
F_{0r} & =-\tfrac{1}{r}\Im(\overline{\phi}\D_{\theta}\phi); & F_{0\theta} & =r\Im(\overline{\phi}\D_{r}\phi);\\
F_{r\theta} & =rF_{12}=-\tfrac{r}{2}|\phi|^{2}.
\end{align*}

The Bogomol'nyi operator has a simple expression in polar coordinates:
\[
\tilde{\D}_{+}=e^{i\theta}[\D_{r}+\frac{i}{r}\D_{\theta}]\quad\text{and}\quad\tilde{\D}_{+}^{\ast}=-e^{-i\theta}[\D_{r}-\frac{i}{r}\D_{\theta}].
\]
In particular, we can rewrite the covariant Laplacian $\D_{j}\D_{j}$
as 
\[
\D_{j}\D_{j}=-\tilde{\D}_{+}^{\ast}\tilde{\D}_{+}-\frac{1}{2}|\phi|^{2}=\D_{r}\D_{r}+\frac{1}{r}\D_{r}+\frac{1}{r^{2}}\D_{\theta}\D_{\theta}.
\]
Now the $\phi$-evolution of \eqref{eq:CSS-cov} reads 
\begin{equation}
i\D_{t}\phi+\Big(\D_{r}\D_{r}+\frac{1}{r}\D_{r}+\frac{1}{r^{2}}\D_{\theta}\D_{\theta}\Big)\phi+g|\phi|^{2}\phi=0.\label{eq:CSS-rad-phi}
\end{equation}
The energy takes the form 
\begin{equation}
E=\int_{\R^{2}}\Big(\frac{1}{2}|\D_{r}\phi|^{2}+\frac{1}{2r^{2}}|\D_{\theta}\phi|^{2}-\frac{g}{4}|\phi|^{4}\Big).\label{eq:energy-radial-phi}
\end{equation}

\subsection{\label{subsec:Coulomb-gauge-and}Coulomb gauge and equivariance reduction}

Recall that \eqref{eq:CSS-cov} has gauge invariance \eqref{eq:gauge-inv}.
Physically, we cannot observe $\phi$ and $A$ by themselves, we only
observe gauge-invariant quantities. Thus the evolution of $(\phi,A)$
described by \eqref{eq:CSS-cov} should be understood modulo gauge
equivalence. In order to discuss the Cauchy problem of \eqref{eq:CSS-cov},
we need to choose one representative $(\phi,A)$ from its (gauge-)equivalence
class. This is usually established by imposing a gauge condition (or,
fixing a gauge). In general, the Cauchy problem depends on the choice
of gauge.

The Cauchy problem of \eqref{eq:CSS-cov} has been studied under the
Coulomb gauge and heat gauge. Under the Coulomb gauge, Berg\'e-de
Bouard-Saut \cite{BergeDeBouardSaut1995Nonlinearity} obtained local
well-posedness in $H^{2}$. By a regularization argument, they also
obtained global existence in $H^{1}$ for $H^{1}$ data having small
charge, without uniqueness. Huh \cite{Huh2013Abstr.Appl.Anal} showed
that \eqref{eq:CSS-cov} has a unique-in-time solution for $H^{1}$
data, without continuous dependence though. Recently, Lim \cite{Lim2018JDE}
obtained $H^{1}$ local well-posedness with local-in-time weak Lipschitz
dependence. Using the heat gauge, Liu-Smith-Tataru \cite{LiuSmithTataru2014IMRN}
established local well-posedness in $H^{\epsilon}$, $\epsilon>0$,
for small $H^{\epsilon}$ data with strong Lipschitz dependence. Under
the equivariant symmetry (with Coulomb gauge), $L^{2}$-critical local
theory is easily obtained solely using the Strichartz estimates by
Liu-Smith \cite{LiuSmith2016}.

There are also works on global-in-time behaviors. Pseudoconformal
symmetry \eqref{eq:pseudo-discrete} applied to the static solution
\eqref{eq:explicit-static} directly gives a finite time blow-up solutions
in the self-dual case $g=1$ \cite{JackiwPi1990PRD,Huh2009Nonlinearity}.
On the other hand, Berg\'e-de Bouard-Saut \cite{BergeDeBouardSaut1995Nonlinearity}
gave sufficient conditions on initial data yielding finite-time blow-up,
based on virial identities \eqref{eq:virial-identities} and a convexity
argument. Oh-Pusateri \cite{OhPusateri2015} obtained small data \emph{linear}
scattering for \eqref{eq:CSS-cov} under the Coulomb gauge, by observing
a cubic null structure. Using concentration-compactness arguments,
Liu-Smith \cite{LiuSmith2016} showed global well-posedness and scattering
for \eqref{eq:CSS-cov} under equivariance, if $g<1$ or if the data
has charge less than that of minimal standing waves when $g\geq1$.

In this work, we impose the \emph{Coulomb gauge condition}:
\begin{equation}
\partial_{1}A_{1}+\partial_{2}A_{2}=0.\label{eq:coulomb}
\end{equation}
Differentiating curvature constraints \eqref{eq:CSS-cov} and using
\eqref{eq:coulomb}, we obtain 
\begin{align*}
A_{0} & =\Delta^{-1}[\epsilon_{jk}\partial_{j}\Im(\overline{\phi}\D_{k}\phi)],\\
A_{j} & =\tfrac{1}{2}\epsilon_{jk}\Delta^{-1}\partial_{k}|\phi|^{2},
\end{align*}
where $\epsilon_{jk}$ is the anti-symmetric tensor with $\epsilon_{12}=1$.
In particular, the spatial part of $A$ is given by the \emph{Biot-Savart
law}
\begin{equation}
A_{j}(t,x)=\frac{1}{4\pi}\epsilon_{jk}\int_{\R^{2}}\frac{(x_{k}-y_{k})}{|x-y|^{2}}|\phi(t,y)|^{2}dy.\label{eq:biot-savart}
\end{equation}

As recognized in \cite{JackiwPi1990PRD}, \eqref{eq:CSS-cov} under
the Coulomb gauge has a \emph{Hamiltonian structure}. Indeed, it is
a Hamiltonian flow associated to the energy $E$ with the symplectic
form $\omega(\phi,\psi)\coloneqq\Im\int_{\R^{2}}\overline{\phi}\psi$.
It is helpful to keep this structure in mind, as our duality estimates,
self-dual form of the linearized operator, and expansion of the energy
functional will all be intuitively based on the Hamiltonian structure.

\subsubsection*{Equivariance under the Coulomb gauge}

Under the Coulomb gauge condition \eqref{eq:coulomb}, we impose an
\emph{equivariance ansatz} as 
\begin{equation}
\phi(t,x)=e^{im\theta}u(t,r)\label{eq:equiv-ansatz}
\end{equation}
for $m\in\Z$, where $x=x_{1}+ix_{2}=re^{i\theta}$. We call $m$
the \emph{equivariance index}. We will often write $u$ in place of
$\phi$ when there is no confusion. For example, we denote \eqref{eq:A_theta-form}
and \eqref{eq:A_0-form} below by $A_{\theta}[u]$ and $A_{0}[u]$.
Or, the energy $E[\phi,A]=E[u,A[u]]$ is simply denoted by $E[u]$.

Under the equivariance ansatz \eqref{eq:equiv-ansatz}, the connection
components $A_{r}$, $A_{\theta}$, and $A_{0}$ have simple expressions.
Inserting \eqref{eq:equiv-ansatz} into the Biot-Savart law \eqref{eq:biot-savart},
we obtain 
\begin{equation}
A_{r}=0.\label{eq:A_r-form}
\end{equation}
Now $A_{\theta}$ can be obtained via integrating the curvature constraint
from the origin: 
\begin{equation}
A_{\theta}=-\frac{1}{2}\int_{0}^{r}|u|^{2}r'dr'.\label{eq:A_theta-form}
\end{equation}
To get $A_{0}$, we integrate the curvature constraint from the infinity
using the boundary condition $A_{0}(r)\to0$ as $r\to\infty$:
\begin{equation}
A_{0}=-\int_{r}^{\infty}(m+A_{\theta})|u|^{2}\frac{dr'}{r'}.\label{eq:A_0-form}
\end{equation}

The Bogomol'nyi operator takes the form 
\[
\tilde{\D}_{+}\phi=e^{i\theta}[\D_{r}+\frac{i}{r}\D_{\theta}]\phi=[(\partial_{r}-\frac{m+A_{\theta}}{r})u]e^{i(m+1)\theta}.
\]
Taking the radial part, we write 
\begin{equation}
\D_{+}u\coloneqq\partial_{r}u-\frac{m+A_{\theta}}{r}u.\label{eq:def-Bogomolnyi-rad}
\end{equation}
Note that its formal $L^{2}$-adjoint is given by $\D_{+}^{\ast}u=-\partial_{r}u-\frac{m+1+A_{\theta}}{r}u$.

We sum up our discussions so far. Let $A_{r}$, $A_{\theta}$, and
$A_{0}$ be given by \eqref{eq:A_r-form}, \eqref{eq:A_theta-form},
and \eqref{eq:A_0-form}, respectively. From \eqref{eq:CSS-rad-phi},
the $\phi$-evolution is given by 
\[
i\partial_{t}\phi-A_{0}\phi+(\partial_{rr}+\frac{1}{r}\partial_{r})\phi-\Big(\frac{m+A_{\theta}}{r}\Big)^{2}\phi+g|\phi|^{2}\phi=0.
\]
Separating linear and nonlinear parts, we reorganize the above as
\begin{equation}
i\partial_{t}\phi+\Delta\phi=-g|\phi|^{2}\phi+\frac{2mA_{\theta}}{r^{2}}\phi+\frac{A_{\theta}^{2}}{r^{2}}\phi+A_{0}\phi.\label{eq:CSS-coulomb-phi}
\end{equation}
The $u$-evolution is given by
\[
i\partial_{t}u+\Delta_{m}u=-g|u|^{2}u+\frac{2mA_{\theta}}{r^{2}}u+\frac{A_{\theta}^{2}}{r^{2}}u+A_{0}u,
\]
where $\Delta_{m}$ is the Laplacian adapted to $m$-equivariant functions
\begin{equation}
\Delta_{m}\coloneqq\partial_{rr}+\frac{1}{r}\partial_{r}-\frac{m^{2}}{r^{2}}.\label{eq:def-m-Lap}
\end{equation}
From \eqref{eq:energy-Bogom-phi} and \eqref{eq:energy-radial-phi},
the energy takes either the forms 
\[
E=\left\{ \begin{aligned} & \int\Big(\frac{1}{2}|\partial_{r}u|^{2}+\frac{1}{2}\Big(\frac{m+A_{\theta}}{r}\Big)^{2}|u|^{2}-\frac{g}{4}|u|^{4}\Big),\quad\text{or}\\
 & \int\Big(\frac{1}{2}|\D_{+}u|^{2}+\frac{1-g}{4}|u|^{4}\Big).
\end{aligned}
\right.
\]
Here, $\int$ denotes the integral $2\pi\int_{0}^{\infty}rdr$, as
noted in \eqref{eq:def-integral}.

\subsection{\label{subsec:The-Static-Solution}Static solution $Q$ for the self-dual
(CSS)}
\begin{center}
\emph{From now on, we restrict to the self-dual case $g=1$.}
\par\end{center}

\ 

Substituting $g=1$ into \eqref{eq:energy-Bogom-phi}, the energy
functional reads
\[
E=\int_{\R^{2}}\frac{1}{2}|\tilde{\D}_{+}\phi|^{2},
\]
provided $F_{12}=-\frac{1}{2}|\phi|^{2}$. In particular, the energy
is always nonnegative. Here we explore the connection between zero-energy
solutions and static solutions. It turns out that static solutions
have zero energy and satisfies a first-order equation, and any zero-energy
solutions are gauge equivalent to static solutions.

Let $\phi:\R^{2}\to\C$ and $A_{x}=A_{1}dx_{1}+A_{2}dx_{2}$ satisfy
$F_{12}=-\frac{1}{2}|\phi|^{2}$. Then $(\phi,A_{x})$ has zero energy
if and only if $\tilde{\D}_{+}\phi=0$. We call the \emph{Bogomol'nyi
}(or, \emph{self-dual})\emph{ equation} as 
\begin{equation}
\left\{ \begin{aligned}\tilde{\D}_{+}\phi & =0,\\
F_{12} & =-\tfrac{1}{2}|\phi|^{2}.
\end{aligned}
\right.\label{eq:def-Bogomolnyi-eqn}
\end{equation}
The self-duality can be seen from the identity \cite{JackiwPi1990PRD}
\[
\D_{j}\phi=-i\epsilon_{jk}\D_{k}\phi.
\]
If we extend $(\phi,A)$ on $\R^{1+2}$ by requiring 
\[
\partial_{t}\phi=0,\quad A_{0}=\tfrac{1}{2}|\phi|^{2},\quad\text{and}\quad\partial_{t}A=0,
\]
then the solution $(\phi,A)$ solves \eqref{eq:CSS-cov}. As $(\phi,A)$
does not depend on time, such a solution is called \emph{static}.

Consider now a solution $(\phi,A)$ to \eqref{eq:CSS-cov} having
zero energy. It should satisfy the Bogomol'nyi equation \eqref{eq:def-Bogomolnyi-eqn}
for each time. If we make a gauge transform such that $A_{0}=\frac{1}{2}|\phi|^{2}$,
then by \eqref{eq:obs} $\phi$ becomes static: 
\[
\partial_{t}\phi=-iA_{0}\phi+i\D_{j}\D_{j}\phi+i|\phi|^{2}\phi=-i\tilde{\D}_{+}^{\ast}\tilde{\D}_{+}\phi=0.
\]
Thus zero energy solutions are gauge equivalent to static solutions.

In fact, the converse direction holds. More precisely, if $(\phi,A)$
is a static solution to \eqref{eq:CSS-cov}, then it satisfies \eqref{eq:def-Bogomolnyi-eqn},
$A_{0}=\frac{1}{2}|\phi|^{2}$, and has zero energy. To see this,
(for a rigorous proof, see Huh-Seok \cite{HuhSeok2013JMP}) we take
the inner product to the static equation 
\begin{equation}
A_{0}\phi-\D_{j}\D_{j}\phi-|\phi|^{2}\phi=0\label{eq:static-eqn}
\end{equation}
with the covariant $L^{2}$-scaling vector field $(1+x_{k}\D_{k})\phi$
to conclude that $(\phi,A)$ has zero energy.\footnote{One should take care of the boundary terms arising from integration
by parts. See \cite{HuhSeok2013JMP}.} In particular, the Bogomol'nyi equation $\tilde{\D}_{+}\phi=0$ is
satisfied. Applying \eqref{eq:obs}, we have $A_{0}=\frac{1}{2}|\phi|^{2}$.

We now impose the Coulomb gauge \eqref{eq:coulomb} and equivariance
ansatz \eqref{eq:equiv-ansatz}. We also concentrate on the \emph{physically
relevant case} $m\geq0$ \cite{Dunne1995Springer} from now on.

One can solve \eqref{eq:def-Bogomolnyi-eqn} following Jackiw-Pi \cite{JackiwPi1990PRD}.
Indeed, using \eqref{eq:def-Bogomolnyi-eqn}, we have $\partial_{r}|\phi|^{2}=\frac{2m+2A_{\theta}}{r}|\phi|^{2}$.
If $|\phi(r_{0})|=0$ for some $r_{0}\in(0,\infty)$, then $|\phi|(r)=0$
for all $r\in(0,\infty)$ by the Gronwall inequality. Henceforth,
we assume $|\phi|>0$. Differentiating the curvature constraints,
one can see that the charge density $|\phi|^{2}$ solves the Liouville
equation 
\[
\Delta\log|\phi|^{2}=-|\phi|^{2}.
\]
General forms of solutions to the Liouville equation are known. There
are explicit $m$-equivariant static solutions 
\begin{equation}
\left\{ \begin{aligned}\phi^{(m)}(t,x) & =\sqrt{8}(m+1)\frac{|x|^{m}}{1+|x|^{2(m+1)}}e^{im\theta},\\
A_{j}^{(m)}(t,x) & =2(m+1)\frac{\epsilon_{jk}x_{k}|x|^{2m}}{1+|x|^{2(m+1)}},\\
A_{0}^{(m)}(t,x) & =4\Big(\frac{(m+1)|x|^{m}}{1+|x|^{2(m+1)}}\Big)^{2}.
\end{aligned}
\right.\label{eq:explicit-static}
\end{equation}
These solutions are unique up to the symmetries of \eqref{eq:CSS};
see \cite{ChouWan1994PacificJMath,ByeonHuhSeok2012JFA,ByeonHuhSeok2016JDE}.
Taking the radial part and suppressing the equivariance index $m$,
we define\footnote{As mentioned above, we abuse notations such that $Q(x)=Q(r)e^{im\theta}$
denotes the $m$-equivariant extension of $Q(r)$.} 
\begin{equation}
Q(r)\coloneqq\sqrt{8}(m+1)\frac{r^{m}}{1+r^{2(m+1)}}.\label{eq:def-Q}
\end{equation}
We have 
\begin{align*}
M[Q] & =8\pi(m+1),\\
E[Q] & =0,\\
A_{\theta}[Q](r) & =-2(m+1)\frac{r^{2(m+1)}}{1+r^{2(m+1)}}.
\end{align*}

\subsection{Pseudoconformal blow-up solutions and main results}

We restrict ourselves \eqref{eq:CSS-cov} under the Coulomb gauge
condition \eqref{eq:coulomb}, equivariance ansatz \eqref{eq:equiv-ansatz},
and the self-dual case $g=1$. Therefore, we arrive at our main equation
\begin{equation}
i\partial_{t}u+\Delta_{m}u=-|u|^{2}u+\frac{2m}{r^{2}}A_{\theta}u+\frac{A_{\theta}^{2}}{r^{2}}u+A_{0}u,\tag{CSS}\label{eq:CSS}
\end{equation}
where $\Delta_{m}$ is defined in \eqref{eq:def-m-Lap} and the connection
components are given by \eqref{eq:A_r-form}, \eqref{eq:A_theta-form},
and \eqref{eq:A_0-form}: 
\[
\begin{cases}
A_{r}=0,\\
A_{\theta}=-\frac{1}{2}\int_{0}^{r}|u|^{2}r'dr',\\
A_{0}=-\int_{r}^{\infty}(m+A_{\theta})|u|^{2}\frac{dr'}{r'}.
\end{cases}
\]
Or, we can also rewrite \eqref{eq:CSS} as a radial form 
\begin{equation}
i\partial_{t}u=(-\partial_{rr}-\frac{1}{r}\partial_{r})u+\Big(\frac{m+A_{\theta}}{r}\Big)^{2}u+A_{0}u-|u|^{2}u.\label{eq:CSS-r}
\end{equation}
Later, we will write \eqref{eq:CSS} in a self-dual form (see \eqref{eq:CSS-L*D-form})
\[
i\partial_{t}u=L_{u}^{\ast}\D_{+}^{(u)}u.
\]
The energy functional has either of the forms 
\begin{equation}
E[u]=\left\{ \begin{aligned} & \frac{1}{2}\int|\D_{+}u|^{2},\\
 & \frac{1}{2}\int|\partial_{r}u|^{2}-\frac{1}{4}\int|u|^{4}+\frac{1}{2}\int\Big(\frac{m+A_{\theta}}{r}\Big)^{2}|u|^{2}.
\end{aligned}
\right.\label{eq:energy-Bogomolnyi}
\end{equation}

In this work, we are going to construct a family of blow-up solutions.
Applying the pseudoconformal symmetry \eqref{eq:pseudo-discrete}
to the static solution $Q$, one obtains an explicit finite-time blow-up
solution to \eqref{eq:CSS} \cite{JackiwPi1990PRD,Huh2009Nonlinearity}
\begin{equation}
S(t,r)\coloneqq\frac{1}{|t|}Q\Big(\frac{r}{|t|}\Big)e^{-i\frac{r^{2}}{4|t|}},\qquad\forall t<0,\label{eq:explicit-blow-up}
\end{equation}
where we only wrote the radial part of the solution, for the sake
of simplicity. One can apply various symmetries of \eqref{eq:CSS}
to obtain other explicit blow-up solutions.

Another blow-up solutions to \eqref{eq:CSS-cov} were considered by
Berg\'e-de Bouard-Saut \cite{BergeDeBouardSaut1995Nonlinearity}.
There, they used the virial identities and a convexity argument of
Glassey \cite{Glassey1977JMP} to give sufficient conditions for solutions
of \eqref{eq:CSS-cov} to blow up in finite time (e.g. $E<0$). Unfortunately,
such a convexity argument is not so useful in the self-dual case $g=1$,
due to $E\geq0$.\footnote{S.-J. Oh observed this fact. We are indebted to him for including
this fact in our manuscript.} Indeed, if $u(t)$ is a solution to \eqref{eq:CSS} with the initial
data $u_{0}$ and satisfies $\int|x|^{2}|u(t,x)|^{2}=0$ for some
time $t\neq0$, then $E[e^{i\frac{|x|^{2}}{4t}}u_{0}]=0$ in view
of \eqref{eq:virial-alt}. By the earlier discussion, $u$ is the
pseudoconformal transform of a static solution. This argument works
in the covariant setting \eqref{eq:CSS-cov} with $g=1$.

Let us call the blow-up rate of \eqref{eq:explicit-blow-up} as the
\emph{pseudoconformal blow-up rate}. By a \emph{pseudoconformal blow-up
solution}, we mean a finite-time blow-up solution having the pseudoconformal
blow-up rate.

In this work, we study pseudoconformal blow-up solutions with prescribed
asymptotic profiles $z^{\ast}$. We show existence, uniqueness, and
instability. To state our result, let $z^{\ast}$ be an $m$-equivariant
profile satisfying the hypothesis 
\begin{equation}
\tag{H}\begin{gathered}-\text{\ensuremath{(m+2)}-equivariant function \ensuremath{\tilde z^{\ast}\coloneqq e^{-i(2m+2)\theta}z^{\ast}} satisfies}\\
\text{ \ensuremath{\|\tilde z^{\ast}\|_{H_{-(m+2)}^{k}}}<\ensuremath{\alpha^{\ast}} for some \ensuremath{k=k(m)>m+3},}
\end{gathered}
\label{eq:H}
\end{equation}
where $H_{\ell}^{k}$ is the space of $\ell$-equivariant $H^{k}(\R^{2})$
functions (see Section \ref{subsec:equiv-Sobolev-sp}).
\begin{thm}[Construction of pseudoconformal blow-up solutions]
\label{thm:BW-sol}Let $m\geq1$. Let $z^{\ast}$ be an $m$-equivariant
profile satisfying \eqref{eq:H}. If $\alpha^{\ast}>0$ is sufficiently
small, then there is an $m$-equivariant solution $u$ to \eqref{eq:CSS}
on $(-\infty,0)$ with the property 
\begin{equation}
\Big[u(t,r)-\frac{1}{|t|}Q\Big(\frac{r}{|t|}\Big)e^{-i\frac{r^{2}}{4|t|}}\Big]e^{im\theta}\to z^{\ast}\qquad\text{in }H_{m}^{1}\label{eq:BW-sol}
\end{equation}
as $t\to0^{-}$. Indeed, $u$ satisfies the following decomposition
estimates
\begin{equation}
\left\{ \begin{aligned}\|u(t,r)-\frac{1}{|t|}Q_{|t|}\Big(\frac{r}{|t|}\Big)e^{i\gamma_{\mathrm{cor}}(t)}-z(t,r)\|_{\dot{H}_{m}^{1}} & \lesssim\alpha^{\ast}|t|^{m}\quad\text{and}\\
\|u(t,r)-\frac{1}{|t|}Q_{|t|}\Big(\frac{r}{|t|}\Big)e^{i\gamma_{\mathrm{cor}}(t)}-z(t,r)\|_{L_{m}^{2}} & \lesssim\alpha^{\ast}|t|^{m+1},
\end{aligned}
\right.\label{eq:BW-sol-temp-1}
\end{equation}
where $z(t,r)$ is a solution to \eqref{eq:fCSS} with the initial
data $z(0,r)=z^{\ast}(r)$ constructed in Section \ref{subsec:Evolution-of-z}
and $\gamma_{\mathrm{cor}}(t)$ is defined in \eqref{eq:def-gamma-eta-cor}.
\end{thm}

The solution constructed in Theorem \ref{thm:BW-sol} is unique in
the following sense.
\begin{thm}[Conditional uniqueness of the pseudoconformal blow-up solutions]
\label{thm:cond-uniq}Let $m\geq1$. Let $z^{\ast}$ be an $m$-equivariant
profile satisfying \eqref{eq:H}. Assume two $H_{m}^{1}$-solutions
$u_{1}$ and $u_{2}$ to \eqref{eq:CSS} satisfy 
\begin{equation}
\|u_{j}(t,r)-\frac{1}{|t|}Q_{|t|}\Big(\frac{r}{|t|}\Big)e^{i\gamma_{\mathrm{cor}}(t)}-z(t,r)\|_{H_{m}^{1}}\leq c|t|\label{eq:BW-sol-temp}
\end{equation}
for all $j=1,2$ and $t$ near zero, for sufficiently small $\alpha^{\ast}>0$
and $c>0$. Then $u_{1}=u_{2}$.
\end{thm}

Lastly, we show that the pseudoconformal blow-up solution is instable.
This immediately follows by constructing an one-parameter family of
solutions in the following sense.
\begin{thm}[Instability of pseudoconformal blow-up solutions]
\label{thm:instability}Let $m\geq1$. Assume the hypothesis of Theorem
\ref{thm:BW-sol}. Let $u$ be the solution constructed in Theorem
\ref{thm:BW-sol}. There exist $\eta^{\ast}>0$ and one-parameter
family of $H_{m}^{1}$-solutions $\{u^{(\eta)}\}_{\eta\in[0,\eta^{\ast}]}$
to \eqref{eq:CSS} with the following properties.
\begin{itemize}
\item $u^{(0)}=u$,
\item For $\eta>0$, $u^{(\eta)}$ scatters both forward and backward in
time,
\item The map $\eta\in[0,\eta^{\ast}]\mapsto u^{(\eta)}$ is continuous
in the $C_{(-\infty,0),\mathrm{loc}}H^{1-}$ topology.
\item The family $\{u^{(\eta)}\}_{\eta\in[0,\eta^{\ast}]}$ exhibits the
rotational instability near time $0$: we can write 
\[
u^{(\eta)}(t,x)=\frac{e^{im(\theta+\gamma^{(\eta)}(t))}}{\langle t\rangle}Q_{|t|}^{(\eta)}\Big(\frac{r}{\langle t\rangle}\Big)+O_{H_{m}^{1}}(\alpha^{\ast}),
\]
where $\langle t\rangle\coloneqq(|t|^{2}+\eta^{2})^{\frac{1}{2}}$,
$Q^{(\eta)}$ is some profile defined in Proposition \ref{prop:const-Q-eta},
and $\gamma^{(\eta)}$ satisfies 
\begin{align*}
|\gamma^{(0)}(-\tau)| & \lesssim\alpha^{\ast}\tau,\\
\limsup_{\eta\to0^{+}}\Big|\gamma^{(\eta)}(\tau)-\gamma^{(\eta)}(-\tau)-\Big(\frac{m+1}{m}\Big)\pi\Big| & \lesssim\alpha^{\ast}\tau,
\end{align*}
for all small $\tau>0$.
\end{itemize}
\end{thm}

Taking the pseudoconformal transform on the family $\{u^{(\eta)}\}_{\eta\in[0,\eta^{\ast}]}$,
we get
\begin{cor}
\label{cor:corollary}Let $m\geq1$. Let $\{u^{(\eta)}\}_{\eta\in[0,\eta^{\ast}]}$
be as in Theorem \ref{thm:instability}. Let $\mathcal{C}$ be the
pseudoconformal transformation defined in \eqref{eq:pseudo-discrete}.
Then,
\begin{itemize}
\item $\mathcal{C}u^{(0)}-Q$ scatters forward in time.
\item $\mathcal{C}u^{(\eta)}$ scatters forward in time.
\end{itemize}
\end{cor}

What can be seen at a glance is that $\mathcal{C}u^{(0)}-Q$ scatters
to a linear solution of $i\partial_{t}+\Delta_{-m-2}$, but $\mathcal{C}u^{(\eta)}$
scatters to that of $i\partial_{t}+\Delta_{m}$. However, the sense
of scattering in both cases turn out to be equivalent. See Remark
\ref{rem:A-scattering-solution}.

To our best knowledge, Theorem \ref{thm:BW-sol} provides the first
example of finite-time blow-up solutions other than $S(t)$ in the
context of \eqref{eq:CSS}. As mentioned earlier, this work is inspired
by the seminal work of Bourgain-Wang \cite{BourgainWang1997} in the
context of the nonlinear Schr\"odinger equation, in which they constructed
pseudoconformal blow-up solutions for given asymptotic profiles. As
\eqref{eq:CSS} is a $L^{2}$-critical Schr\"odinger-type equation,
it shares many features with $L^{2}$-critical NLS. It is worthwhile
to review the results in NLS and compare them with \eqref{eq:CSS}.

The mass-critical \emph{nonlinear Schr\"odinger equation} (NLS) on
two dimension is 
\begin{equation}
i\partial_{t}\psi+\Delta\psi+|\psi|^{2}\psi=0,\tag{NLS}\label{eq:NLS}
\end{equation}
where $\psi:I\times\R^{2}\to\C$. \eqref{eq:NLS} has no static solutions,
but has only standing waves of the form 
\[
(t,x)\mapsto\lambda e^{i\lambda^{2}t}R(\lambda x),\qquad\forall\lambda>0,
\]
where $R$ is the ground state satisfying 
\begin{equation}
\Delta R-R+R^{3}=0.\label{eq:NLS-ground-state}
\end{equation}
Applying the pseudoconformal symmetry \eqref{eq:pseudo-discrete},
we have an explicit pseudoconformal blow-up solution for \eqref{eq:NLS}
\begin{equation}
S_{\mathrm{NLS}}(t,x)\coloneqq\frac{1}{|t|}R\Big(\frac{x}{|t|}\Big)e^{\frac{i}{|t|}}e^{-i\frac{|x|^{2}}{4|t|}},\qquad\forall t<0.\label{eq:NLS-explicit-blow-up}
\end{equation}

Bourgain-Wang \cite{BourgainWang1997} constructed a family of pseudoconformal
blow-up solutions to \eqref{eq:NLS}:
\begin{thm}[Bourgain-Wang solutions \cite{BourgainWang1997}, stated informally]
\label{thm:BW-NLS}Let $\zeta^{\ast}:\R^{2}\to\C$ be a profile that
degenerates at the origin at \emph{large} order, i.e. $|\zeta^{\ast}(x)|\lesssim|x|^{A}$
for some large $A$, and lies in some weighted Sobolev space. Then,
there exists a (conditionally unique) solution $\psi_{\mathrm{BW}}$
to \eqref{eq:NLS} defined near time zero such that 
\[
\psi_{\mathrm{BW}}(t)-S_{\mathrm{NLS}}(t)\to\zeta^{\ast}
\]
as $t\to0$.
\end{thm}

Our Theorem \ref{thm:BW-sol} is analogous to Theorem \ref{thm:BW-NLS}.

Pseudoconformal blow-up solutions are believed to be non-generic.
In the context of \eqref{eq:NLS}, there are two supplementary works
in this direction. Krieger-Schlag \cite{KriegerSchlag2009JEMS} showed
that for the 1D critical NLS, there is a \emph{codimension 1} manifold
of initial data in the measurable category, yielding a pseudoconformal-type
blow-up. Merle-Rapha\"el-Szeftel \cite{MerleRaphaelSzeftel2013AJM}
exhibited instability of Bourgain-Wang solutions \cite{BourgainWang1997}
as follows.
\begin{thm}[Instability of Bourgain-Wang solutions \cite{MerleRaphaelSzeftel2013AJM},
stated informally]
\label{thm:BW-instability-NLS}Let $\zeta^{\ast}$ be as in Theorem
\ref{thm:BW-NLS} and small in some weighted Sobolev space; let $\psi_{\mathrm{BW}}$
be the associated Bourgain-Wang solution. Then, there exists a continuous
family of solutions $\psi_{\eta}$ to \eqref{eq:NLS} for $\eta\in[-\delta,\delta]$
with $0<\delta\ll1$ such that
\begin{itemize}
\item $\psi_{0}=\psi_{\mathrm{BW}}$ is the Bourgain-Wang solution,
\item For $\eta>0$, $\psi_{\eta}$ is global in time and scatters both
forward and backward in time.
\item For $\eta<0$, $\psi_{\eta}$ scatters backward and blows up forward
in finite time under the log-log law \eqref{eq:log-log}.
\end{itemize}
\end{thm}

They perturbed $R$ to the instability direction $\rho_{\mathrm{NLS}}$
(given in Remark \ref{rem:remark3.5}), say $R-\eta\rho_{\mathrm{NLS}}$,
and derived dynamical laws of modulation parameters depending on the
sign of $\eta$.

Although it is not present in Theorem \ref{thm:BW-instability-NLS},
it is worthnoting that they considerably relaxed the degeneracy asssumption
on $\zeta^{\ast}$ at the origin. They also improved the notion of
uniqueness of $\psi_{\mathrm{BW}}$ that is necessary for Theorem
\ref{thm:BW-instability-NLS}. Unlike \cite{BourgainWang1997}, they
rely on a robust modulation analysis and avoid the explicit use of
the pseudoconformal symmetry. However, the pseudoconformal symmetry
is implicitly used in the sense that they made approximate profiles
by conjugating the phase $e^{-ib\frac{|y|^{2}}{4}}$. Our work also
uses modulation analysis as in \cite{MerleRaphaelSzeftel2013AJM}.

Theorem \ref{thm:instability} is analogous to Theorem \ref{thm:BW-instability-NLS}.
However, the instability mechanism of \eqref{eq:CSS} turns out to
be quite different from that of \eqref{eq:NLS}; the instability of
\eqref{eq:CSS} stems from the phase rotation. This is due to the
difference of the spectral properties for each problem.

\ 

\emph{Comparison between \eqref{eq:CSS} and \eqref{eq:NLS}.}

\ 

We view \eqref{eq:CSS} as a gauged \eqref{eq:NLS}. Thus it is instructive
to compare fundamental facts and results with \eqref{eq:NLS}.

1. \emph{Symmetries and conservation laws}. All the symmetries of
\eqref{eq:CSS} (and also \eqref{eq:CSS-cov}) are valid for \eqref{eq:NLS},
including the scaling, Galilean, and pseudoconformal symmetries. Conservation
laws of \eqref{eq:CSS} are valid for \eqref{eq:NLS} as well, after
replacing the covariant derivatives by the usual derivatives. Virial
identities also hold.

2. \emph{Linearized operators around $Q$ or $R$}. As we have seen
above, \eqref{eq:CSS} admits a static solution $Q$, but \eqref{eq:NLS}
only admits a standing wave solution $e^{it}R(x)$, which is not static.
The mass term $-R$ is present in \eqref{eq:NLS-ground-state}, but
not in the static equation \eqref{eq:static-eqn}. From this, $R$
decays exponentially at spatial infinity, but $Q$ only decays as
$r^{-(m+2)}$.

The linearized dynamics of \eqref{eq:NLS} near $e^{it}R$ and that
of \eqref{eq:CSS} near $Q$ differ drastically. Denote by $\mathcal{L}_{\mathrm{NLS}}$
the linearized operator for \eqref{eq:NLS}
\begin{equation}
\mathcal{L}_{\mathrm{NLS}}f=-\Delta f+f-2R^{2}f-R^{2}\overline{f}=L_{+}\Re(f)+iL_{-}\Im(f).\label{eq:linearized-op-NLS}
\end{equation}
$\mathcal{L}_{\mathrm{NLS}}$ has continuous spectrum $[1,\infty)$
and two eigenvalues, $0$ and a negative value. On the other hand,
the linearized operator $\mathcal{L}_{Q}$ (see \eqref{eq:def-linearized-op}-\eqref{eq:L_Q-self-dual})
is non-negative definite with $\Lambda Q,iQ\in\ker\mathcal{L}_{Q}$,
where $\Lambda$ is the generator of $L^{2}$-scaling \eqref{eq:def-Lambda}.
More importantly, \eqref{eq:CSS} has the \emph{self-dual} structure.
By \eqref{eq:L_Q-self-dual}, this appears at the linearized level
as 
\[
\mathcal{L}_{Q}=L_{Q}^{\ast}L_{Q},
\]
where $L_{Q}$ is a first-order nonlocal differential operator defined
by \eqref{eq:def-L_w}. The self-duality plays a crucial role in several
places. It helps us to analyze the spectral property and coercivity
of $\mathcal{L}_{Q}$ under generic orthogonality conditions. It is
also used in the construction of the modified profiles $Q^{(\eta)}$
in Section \ref{sec:Profile}. The factorization of the linearized
operator was crucially used in blow-up problems of other evolution
equations; see for instance \cite{RodnianskiSterbenz2010Ann.Math.,RaphaelRodnianski2012Publ.Math.}
for wave maps and \cite{MerleRaphaelRodnianski2013InventMath} for
Schr\"odinger maps, and \cite{RaphaelSchweyer2013CPAM,RaphaelSchweyer2014AnalPDE}
for the energy-critical heat flow.

Let us restrict $\mathcal{L}_{\mathrm{NLS}}$ on radial functions.
It is known from \cite{Weinstein1985SIAM} that 
\[
\mathcal{L}_{\mathrm{NLS}}\Lambda R=-2R.
\]
On the other hand, the linearized operator $\mathcal{L}_{Q}$ (see
\eqref{eq:def-linearized-op} and \eqref{eq:L_Q-self-dual}) satisfies
\[
\mathcal{L}_{Q}\Lambda Q=0.
\]
This difference is intimately tied to the fact that $e^{it}R$ is
not a static solution to \eqref{eq:NLS}, but $Q$ is a static solution
to \eqref{eq:CSS}. As a consequence, the generalized null spaces
of $i\mathcal{L}_{\mathrm{NLS}}$ and $i\mathcal{L}_{Q}$ appear to
be rather different. Compare Proposition \ref{prop:gen.null.space}
with Remark \ref{rem:remark3.5}.

3. \emph{Dynamics below the threshold mass}. $R$ and $Q$ have \emph{threshold
mass}\footnote{In fact for \eqref{eq:CSS}, $M[u]$ is called charge. In this discussion,
we also call $M[u]$ mass.} in the sense that nontrivial dynamical behaviors arise from that
level of mass. Any solutions to \eqref{eq:NLS} having mass less than
that of $R$ are global-in-time \cite{Weinstein1983CMP} and scatter
\cite{Killip-Tao-Visan2009JEMS,Dodson2015AdvMath}. In the context
of \eqref{eq:CSS} (under equivariance), Liu-Smith \cite{LiuSmith2016}
proved the analogous result, where $Q$ plays the role of $R$. At
the level of the threshold mass, we have pseudoconformal blow-up solutions
$S(t)$ and $S_{\mathrm{NLS}}(t)$. There are partial results on the
threshold dynamics of \eqref{eq:NLS}; see for instance \cite{Merle1993Duke,KillipLiVisanZhang2009,LiZhang2012_d_geq_2}.
There is no related results for \eqref{eq:CSS}.

4. \emph{Stable blow-up law}. Finite-time blow-up solutions (with
negative energy) have been studied extensively in the context of \eqref{eq:NLS}.
The energy of \eqref{eq:NLS} can now attain negative values. In a
seminal series of papers by Merle-Rapha\"el \cite{MerleRaphael2005AnnMath,MerleRaphael2003GAFA,Raphael2005MathAnn,MerleRaphael2004InventMath,MerleRaphael2006JAMS,MerleRaphael2005CMP},
they obtained a sharp description of blow-up dynamics for $H^{1}$
solutions to \eqref{eq:NLS} having slightly super-critical mass and
negative energy; there exists a universal constant $c_{\ast}$ such
that $\psi(t)$ blows up in the \emph{log-log law }(almost self-similar
with the log-log correction)
\begin{equation}
\|\nabla\psi(t)\|_{L^{2}}\approx c_{\ast}\Big(\frac{\log|\log(T-t)|}{T-t}\Big)^{\frac{1}{2}}\label{eq:log-log}
\end{equation}
as $t$ approaches to the blow-up time $T$. In fact, there is a larger
open set of data yielding the log-log law. (See also a former work
by Perelman \cite{Perelman2001Ann.Henri.}) It is also known from
their results that any finite-time blow-up solutions with slightly
supercritical mass must have blow-up rate either the log-log law,
or faster than or equal to the pseudoconformal rate. For \eqref{eq:CSS-cov}
with $g>1$, when energy is negative, it is expected that a similar
log-log law holds; see Berg\'e-de Bouard-Saut \cite{BergeDeBouardSaut1995PRL}.

Note that the blow-up solution $S_{\mathrm{NLS}}$ does not follow
the log-log law. It has threshold mass and is apparently instable
with respect to shrinking its mass slightly. The solution $S$ to
\eqref{eq:CSS} also has threshold mass and is instable by the same
reason.

\ 

\emph{Comments on Theorems \ref{thm:BW-sol}-\ref{thm:instability}.}

\ 

1. \emph{Assumption \eqref{eq:H}}. Among smooth functions, (e.g.
in the space of $m$-equivariant Schwartz functions) assumption \eqref{eq:H}
is equivalent to saying that $z^{\ast}$ is small with $(\partial_{r})^{m}z^{\ast}(0)=0$.
Indeed, a smooth $m$-equivariant function $z^{\ast}$ satisfies $(\partial_{r})^{\ell}z^{\ast}(0)=0$
for either $0\leq\ell<m$ or $\ell-m$ is odd. Thus the condition
$(\partial_{r})^{m}z^{\ast}(0)=0$ can be considered as a codimension
1 condition in the space of asymptotic profiles.

Assumption \eqref{eq:H} requires $z^{\ast}$ to be degenerate at
the origin. Indeed, any smooth $m$-equivariant function $f$ exhibits
degeneracy $|f(x)|\lesssim r^{|m|}$ (see Section \ref{subsec:equiv-Sobolev-sp})
at the origin, so a $-(m+2)$-equivariant function with $m\geq1$
shows the degeneracy at least $r^{3}$. Such a degeneracy is necessary
to decouple the marginal interaction between $S$ and $z^{\ast}$
because $S(t)$ concentrates at the origin as $t\to0$. In \eqref{eq:NLS},
Bourgain-Wang \cite{BourgainWang1997} and Merle-Rapha\"el-Szeftel
\cite{MerleRaphaelSzeftel2013AJM} used the flatness of the asymptotic
profile $\zeta^{\ast}$ to construct pseudoconformal blow-up solutions.

The choice $-(m+2)$ arises naturally from the \emph{long-range interaction}
in \eqref{eq:CSS}. As opposed to the \eqref{eq:NLS} case, \eqref{eq:CSS}
has long range potential as can be easily seen from the Biot-Savart
law \eqref{eq:biot-savart} or the term 
\[
\Big(\frac{m+A_{\theta}[u]}{r}\Big)^{2}u
\]
of \eqref{eq:CSS-r}. Regardless how fast $u$ decays, the above potential
exhibits $r^{-2}$ tail in general. In our case, our blow-up ansatz
is $u\approx S+z$ (recall $S$ \eqref{eq:explicit-blow-up}) so there
is an interaction term 
\[
\Big(\frac{m+A_{\theta}[S]}{r}\Big)^{2}z.
\]
As $S$ is concentrated near the origin, we may approximate $A_{\theta}[S](r)$
by $A_{\theta}[S](+\infty)=-2(m+1)$. Now the choice $(m+2)$ of \eqref{eq:H}
arises from 
\[
(\partial_{rr}+\frac{1}{r}\partial_{r})-\Big(\frac{m-2(m+1)}{r}\Big)^{2}=(\partial_{rr}+\frac{1}{r}\partial_{r})-\Big(\frac{-(m+2)}{r}\Big)^{2}=\Delta_{-m-2}.
\]
This observation is crucial in our work and leads to a \emph{modification}
of $z$-evolution \eqref{eq:fCSS}. This will be discussed in more
detail in Section \ref{subsec:Evolution-of-z}.

2. \emph{Assumption $m\geq1$}. There are many places that the proof
does not work if $m=0$. We list a few here.

Note that the solution $S(t)$ belongs to $H^{1}$ if and only if
$m\geq1$. The assumption $m\geq1$ is best possible to construct
solutions satisfying \eqref{eq:BW-sol} without truncating the blow-up
profile. In this work, we will heavily rely on the explicit computations
coming from the pseudoconformal phase, so we hope to avoid any truncations.

There are extra simplifications coming from $m\geq1$. The equivariant
Sobolev space $\dot{H}_{m}^{1}$ (see Section \ref{subsec:equiv-Sobolev-sp})
has a nice embedding property $\|\frac{f}{r}\|_{L^{2}}+\|f\|_{L^{\infty}}\lesssim\|f\|_{\dot{H}_{m}^{1}}$,
which fails if $m=0$. Moreover, the first order operator $L_{Q}$
(given in \eqref{eq:def-L_w}) shows strong coercivity $\|L_{Q}f\|_{L^{2}}\gtrsim\|f\|_{\dot{H}_{m}^{1}}$
under suitable orthogonality conditions on $f$ for $m\geq1$. If
$m=0$, the coercivity becomes weaker; see Lemma \ref{lem:coercivity}.

Assumption $m\geq1$ reduces the bulk of the paper and make our arguments
easier to follow, both conceptually and technically. Of course, the
case $m=0$ remains as an interesting open problem.

3. \emph{Interaction of $S(t)$ and $z^{\ast}$}. A sharp contrast
to \eqref{eq:NLS} is that \eqref{eq:CSS} has the long-range interaction
even though $S(t)$ and $z^{\ast}$ live in completely different scales.
In the setup of our blow-up ansatz, we have to detect and take the
strong interaction into account. This leads to a modification of the
$z$-evolution and a correction of the phase parameter $\gamma(t)$.
This is one of the novelties of this work. Even in NLS class, if the
nonlinearity is replaced by a nonlocal nonlinearity, there could be
a long range interaction. In \cite{KriegerLenzmannRaphael2009AHP},
the authors observed a long range interaction and constructed Bourgain-Wang
type solutions for the 4D mass-critical Hartree equation.

However, it turns out that the interaction is not strong enough to
affect the blow-up rate. In general, the regularity and decay of the
asymptotic profile is related to the blow-up rate. It is an interesting
question whether one can generate exotic blow-up rates by prescribing
other types of $z^{\ast}$. There are many works on constructing exotic
blow-up solutions, for example, \cite{KriegerSchlagTataru2008Invent,KriegerSchlagTataru2009Duke,KriegerSchlag2014JMPA,DonningerHuangKriegerSchlag2014MichMathJ,MartelMerleRaphael2015AnnSci,Jendrej2017JFA}.
There are also works exhibiting strong interactions between bumps
in different scales, for instance, construction of multi-bubbles;
see \cite{MartelRaphael2018AnnSci,Jendrej2017AnalPDE,Jendrej2018AnnSc,JendrejLawrie2018Invent,Jendrej2019AJM}.

4. \emph{On uniqueness}. The time-decay condition \eqref{eq:BW-sol-temp}
of Theorem \ref{thm:cond-uniq} is tight with \eqref{eq:BW-sol-temp-1}
of Theorem \ref{thm:BW-sol} when $m=1$. \eqref{eq:BW-sol-temp}
is best from our method. When $m$ is large, another argument presented
in \cite{MerleRaphaelSzeftel2013AJM} (see also \cite{RaphaelSzeftel2011JAMS})
works similarly for \eqref{eq:CSS}. However, when $m$ is $1$ or
$2$, the argument does not work due to a slow decay of $Q$. We instead
redo the construction proof on the difference $\epsilon=\epsilon_{1}-\epsilon_{2}$
with a careful choice of modulation parameters.

5. \emph{Rotational Instability}. The source of the instability comes
from the phase rotation. This is a sharp contrast to that in \eqref{eq:NLS}.
Mathematically, the difference comes from that of spectral properties
of the linearized operators. A striking feature is that $u^{(0)}=u$
does not rotate at all, but $u^{(\eta)}$, $0<\eta\ll1$, shows an
abrupt spatial rotation in a very short time interval $|t|\lesssim\eta$,
by the angle 
\[
\Big(\frac{m+1}{m}\Big)\pi.
\]
In the energy-critical Schr\"odinger map (1-equivariant), it is believed
that there is a rotational instability of blow-up solutions. In \cite{MerleRaphaelRodnianski2013InventMath},
authors perform a slow-adiabatic ansatz both in scaling and rotation
parameters to construct smooth blow-up solutions in the codimension
1 sense. Their proof sheds light on the rotational instability of
the blow-up solutions. However, a sharp description on instability
is not fully understood.

Note that we exhibited an one-sided instability ($\eta\geq0$) of
pseudoconformal blow-up solutions. The sign condition $\eta\geq0$
is used crucially in two places of our proof: the construction of
$Q^{(\eta)}$ and Lyapunov functional method.

\subsection{\label{subsec:strategy}Strategy of the proof}

Our main theorems are to show the existence, uniqueness, and instability
of pseudoconformal blow-up solutions. We will use modulation analysis
as in \cite{MerleRaphaelSzeftel2013AJM}. Unlike \cite{BourgainWang1997},
we do not explicitly use the pseudoconformal transform, but instead
use the pseudoconformal phase $e^{-ib\frac{|y|^{2}}{4}}$.

With a prescribed asymptotic profile $z^{\ast}$, we set $t=0$ as
the blow-up time and construct $u$ on $[t_{0}^{\ast},0)$. This is
often called a \emph{backward construction}. We will show the instability
and existence at the same time by constructing a family of solutions
$\{u^{(\eta)}\}_{\eta\in(0,\eta^{\ast}]}$ as in Theorem \ref{thm:instability}.
Then, the blow-up solution $u^{(0)}$ is obtained by a limiting argument
as $\eta\to0^{+}$. The limiting argument requires two steps: obtaining
uniform estimates on $\eta$ and performing a soft compactness argument.
This kind of argument is originally due to Merle \cite{Merle1990CMP}
in the context of \eqref{eq:NLS}. Of course, there should be a separate
argument for the uniqueness of the blow-up solution to conclude the
instability (to guarantee continuity of the map $\eta\in[0,\eta^{\ast}]\mapsto u^{(\eta)}$).
From now on, we focus on the construction of $u^{(\eta)}$ (Theorem
\ref{thm:instability}).

1. \emph{Setup for the modulation analysis}. For each fixed $\eta\in(0,\eta^{\ast}]$,
we write our solution $u^{(\eta)}(t,x)$ as 
\[
u^{(\eta)}(t,x)=\frac{e^{i\gamma(t)}}{\lambda(t)}(Q_{b(t)}^{(\eta)}+\epsilon)(t,\frac{x}{\lambda(t)})+z(t,x)
\]
with the initial data 
\[
u^{(\eta)}(0,x)=\frac{1}{\eta}Q^{(\eta)}\Big(\frac{x}{\eta}\Big)+z^{\ast}(x).
\]
Here, $Q^{(\eta)}$ is some modified profile of $Q$ to be explained
soon, $f_{b}(y)\coloneqq f(y)e^{-ib\frac{|y|^{2}}{4}}$ for a function
$f$, and $z(t,x)$ is some function with $z(0,x)=z^{\ast}(x)$. We
want to construct $u^{(\eta)}$ such that 
\begin{equation}
\lambda(t)\approx\langle t\rangle=(t^{2}+\eta^{2})^{\frac{1}{2}},\quad b(t)\approx|t|,\quad\text{and}\quad\gamma(t)\approx(m+1)\tan^{-1}\Big(\frac{t}{\eta}\Big).\label{eq:strategy-choice}
\end{equation}
The motivation for the above choice of modulation parameters will
be explained later. Note that the spatial angle $(\frac{m+1}{m})\pi=\frac{1}{m}(\gamma(0+)-\gamma(0-))$
comes from the identity $e^{i\gamma}e^{im\theta}=e^{im(\theta+\frac{\gamma}{m})}$.
If we take the limit $\eta\to0^{+}$ for each fixed $t<0$, we get
the pseudoconformal regime.\footnote{In fact, since $\tan^{-1}(\frac{t}{\eta})\to-\frac{\pi}{2}$ as $\eta\to0^{+}$,
we get a phase-rotated pseudoconformal blow-up solution by the angle
$-(\frac{m+1}{2})\pi$.}

We will use the rescaled variables $(s,y)$ such that 
\[
\frac{ds}{dt}=\frac{1}{\lambda^{2}}\quad\text{and}\quad y=\frac{x}{\lambda}.
\]
Note that $Q_{b}^{(\eta)}$ and $\epsilon$ are functions of $(s,y)$,
but $u$ and $z$ are functions of $(t,x)$. We often switch between
the original variables $(t,x)$ and rescaled variables $(s,y)$ via
$\sharp$ and $\flat$ notations. For example, 
\[
\epsilon^{\sharp}(t,x)=\frac{e^{i\gamma(t)}}{\lambda(t)}\epsilon(s(t),\frac{x}{\lambda(t)})\quad\text{and}\quad z^{\flat}(s,y)=e^{-i\gamma(s)}\lambda(s)z(t(s),\lambda(s)y).
\]
For parameters $\lambda$, $\gamma$, and $b$, we abbreviated $\lambda(t(s))=\lambda(s)$
and so on. In Section \ref{subsec:Dynamic-Rescaling}, we provide
conversion formulae between the $(t,x)$-variables and $(s,y)$-variables.

2. \emph{Linearization of \eqref{eq:CSS}}. Since we use a decomposition
in the $(s,y)$-variables as 
\[
u^{\flat}=Q_{b}^{(\eta)}+\epsilon+z^{\flat},
\]
it is essential to investigate the linearized evolution around $w^{\flat}\coloneqq Q_{b}^{(\eta)}+z^{\flat}$.
As $z^{\flat}$ is assumed to be small, we may replace $w^{\flat}$
by $Q_{b}^{(\eta)}$. As $b$ and $\eta$ are small, $Q_{b}^{(\eta)}$
is a modified profile of $Q$. Thus we need to study the linearized
operator around $Q$, say $\mathcal{L}_{Q}$. In view of \emph{self-duality}
($g=1$), it turns out that the linearized operator $\mathcal{L}_{Q}$
is factorized as 
\[
\mathcal{L}_{Q}=L_{Q}^{\ast}L_{Q}
\]
for some nonlocal first-order differential operator $L_{Q}$. This
is first observed by Lawrie-Oh-Shahshahani \cite{LawrieOhShashahani_unpub}.
This factorization is crucial in analyzing the spectral properties
of $\mathcal{L}_{Q}$, including coercivity. Moreover, we are able
to invert $\mathcal{L}_{Q}$ under a suitable solvability condition,
even if $\mathcal{L}_{Q}$ is a \emph{nonlocal} second-order differential
operator. As a result, we are able to compute the generalized nullspace
of $i\mathcal{L}_{Q}$ with the relations 
\begin{align*}
i\mathcal{L}_{Q}\rho & =iQ; & i\mathcal{L}_{Q}ir^{2}Q & =4\Lambda Q;\\
i\mathcal{L}_{Q}iQ & =0; & i\mathcal{L}_{Q}\Lambda Q & =0,
\end{align*}
where $\Lambda$ is the $L^{2}$-scaling vector field defined in \eqref{eq:def-Lambda},
and $\rho$ is some profile. It is instructive to compare this with
\eqref{eq:NLS}, see Remark \ref{rem:remark3.5}.

Coming back to the nonlinear level, self-duality appears in the energy
functional as 
\[
E[u]=\frac{1}{2}\int|\D_{+}u|^{2}.
\]
In view of the Hamiltonian structure, we can write \eqref{eq:CSS}
in a more compact form 
\[
i\partial_{t}u=L_{u}^{\ast}\D_{+}^{(u)}u,
\]
where $L_{u}$ is the linearized operator of $\D_{+}u$ and $L_{u}^{\ast}$
is its dual. See \eqref{eq:CSS-L*D-form}. This compact form simplifies
the analysis technically. Moreover, it represents that we implicitly
use Hamiltonian structure throughout the analysis.

3. \emph{Profile $Q^{(\eta)}$}. Recall the expression of $S(t)$
\eqref{eq:explicit-blow-up}
\[
S(t,x)=\frac{1}{|t|}Q_{|t|}\Big(\frac{x}{|t|}\Big).
\]
Thus $S$ is an exact solution to \eqref{eq:CSS} with $\epsilon=z=0$
and $\eta=0$. Assuming $\epsilon=z=0$ for a moment, the equation
of $u^{\flat}$ and the ansatz $u^{\flat}=Q_{b(s)}$ suggest 
\[
\frac{\lambda_{s}}{\lambda}+b=0,\quad b_{s}+b^{2}=0,\quad\gamma_{s}=0.
\]
Note that $\lambda(t)=b(t)=|t|$ and $\gamma(t)=0$ indeed solve the
above ODEs.

One of novelties of this work is finding a rotational instability
mechanism by introducing the $\eta$-parameter to $\gamma_{s}$. Assume
$\frac{\lambda_{s}}{\lambda}+b=0$ and $Q_{b}^{(\eta)\sharp}$ solves
\eqref{eq:CSS}. Then, we have 
\[
L_{Q^{(\eta)}}^{\ast}\D_{+}^{(Q^{(\eta)})}Q^{(\eta)}+\gamma_{s}Q^{(\eta)}-(b_{s}+b^{2})\tfrac{|y|^{2}}{4}Q^{(\eta)}=0.
\]
It turns out that we can impose 
\[
\gamma_{s}=\eta\theta_{\eta}=\eta\Big(\frac{1}{2}\int|Q^{(\eta)}|^{2}rdr-(m+1)\Big)\quad\text{and}\quad b_{s}+b^{2}+\eta^{2}=0.
\]
See Section \ref{subsec:A-Hint-to} for the motivation. This leads
to
\[
L_{Q^{(\eta)}}^{\ast}\D_{+}^{(Q^{(\eta)})}Q^{(\eta)}+\eta\theta_{\eta}Q^{(\eta)}+\eta^{2}\tfrac{|y|^{2}}{4}Q^{(\eta)}=0.
\]
This is a \emph{nonlocal} second-order equation. Due to the presence
of $\eta\theta_{\eta}Q^{(\eta)}$, one expects $Q^{(\eta)}$ decays
exponentially. But $Q$ itself shows a polynomial decay, and hence
it is difficult to approximate $Q^{(\eta)}(r)$ by an $\eta$-expansion
of $Q(r)$ for $r$ large. Such an $\eta$-expansion is not successful
especially when $m$ is small. We thus have to search for a nonlinear
ansatz $Q^{(\eta)}$. Here, the self-duality plays a crucial role.
We essentially use the self-duality to reduce this to a first-order
equation what we call the \emph{modified Bogomol'nyi equation}
\[
\D_{+}^{(Q^{(\eta)})}P^{(\eta)}=0,\quad Q^{(\eta)}=e^{-\eta\frac{r^{2}}{4}}P^{(\eta)}.
\]
The modified profile $Q^{(\eta)}$ is found by solving this nonlocal
first-order equation. The choice \eqref{eq:strategy-choice} is obtained
by integrating the ODEs
\[
\frac{\lambda_{s}}{\lambda}+b=0,\quad\gamma_{s}=\eta\theta_{\eta},\quad b_{s}+b^{2}+\eta^{2}=0.
\]

4. \emph{Derivation of corrections}. As mentioned earlier, although
$Q_{b}^{(\eta)\sharp}$ and $z$ live on far different scales, there
are strong interactions between $Q_{b}^{(\eta)\sharp}$ and $z$ from
the nonlocal nonlinearities. Before performing a perturbative analysis,
we need to capture and incorporate them into the blow-up ansatz. Recall
\eqref{eq:CSS-r}. It turns out that there are two types of corrections:
one is from $Q_{b}^{(\eta)\sharp}$ to $z$ and the other is from
$z$ to $Q_{b}^{(\eta)\sharp}$. The former is absorbed by modifying
the $z$-evolution and the latter is absorbed by correcting the phase
parameter $\gamma_{s}$ by the amount $\theta_{z^{\flat}\to Q_{b}^{(\eta)}}$.
It is also interesting that the modified $z$-evolution becomes $-(m+2)$-equivariant
\eqref{eq:CSS}. This is a sharp contrast to \eqref{eq:NLS} since
it is from the nonlocal nonlinearities. See more details in Section
\ref{subsec:Evolution-of-z}.

5. \emph{Setup for bootstrapping}. Combining all the above information,
we arrive at the $\epsilon$-equation 
\begin{align*}
 & i\partial_{s}\epsilon-\mathcal{L}_{w^{\flat}}\epsilon+ib\Lambda\epsilon-\eta\theta_{\eta}\epsilon\\
 & =i\Big(\frac{\lambda_{s}}{\lambda}+b\Big)\Lambda(Q_{b}^{(\eta)}+\epsilon)+(\tilde{\gamma}_{s}-\eta\theta_{\eta})Q_{b}^{(\eta)}+(\gamma_{s}-\eta\theta_{\eta})\epsilon\\
 & \quad-(b_{s}+b^{2}+\eta^{2})\tfrac{|y|^{2}}{4}Q_{b}^{(\eta)}+\tilde R_{Q_{b}^{(\eta)},z^{\flat}}+V_{Q_{b}^{(\eta)}-Q_{b}}z^{\flat}+R_{u^{\flat}-w^{\flat}}.
\end{align*}
Here, $\tilde{\gamma}_{s}=\gamma_{s}+\theta_{z^{\flat}\to Q_{b}^{(\eta)}}$
and $\tilde R_{Q_{b}^{(\eta)},z^{\flat}}$ represents the interaction
between $Q_{b}^{(\eta)}$ and $z^{\flat}$ after deleting the aforementioned
strong interactions. The term $V_{Q_{b}^{(\eta)}-Q_{b}}z^{\flat}$
is originated from the difference of $Q^{(\eta)}$ and $Q$. Now $R_{u^{\flat}-w^{\flat}}$
collects the quadratic and higher terms in $\epsilon$. These are
all treated perturbatively. Among these, $\tilde R_{Q_{b}^{(\eta)},z^{\flat}}$
is estimated by exploiting the decoupling in scales of $Q_{b}^{(\eta)}$
and $z^{\flat}$.

So far, we have not specified the choice of $(\lambda,\gamma,b)$.
To fix them, we will require two orthogonality conditions and one
dynamical law governing $\lambda$ and $b$\@. We impose two orthogonality
conditions to guarantee the coercivity 
\[
(\epsilon,\mathcal{L}_{Q}\epsilon)_{r}\gtrsim\|\epsilon\|_{\dot{H}_{m}^{1}}^{2}.
\]
Because of the factorization $\mathcal{L}_{Q}=L_{Q}^{\ast}L_{Q}$,
a generic choice of orthogonality conditions suffices. For the last
one, we choose to impose a \emph{dynamical law} of $\lambda$ and
$b$:
\begin{equation}
2\Big(\frac{\lambda_{s}}{\lambda}+b\Big)b-(b_{s}+b^{2}+\eta^{2})=0.\label{eq:dyna-law-1}
\end{equation}
We are motivated to make this choice to cancel out $|y|^{2}Q_{b}^{(\eta)}$
terms of the $\epsilon$-equation since $|y|^{2}Q_{b}^{(\eta)}$ has
insufficient decay (uniformly in $\eta$) for small $m$. More precisely,
in the $\epsilon$-equation we have 
\begin{align*}
 & i\Big(\frac{\lambda_{s}}{\lambda}+b\Big)\Lambda Q_{b}^{(\eta)}-(b_{s}+b^{2}+\eta^{2})\tfrac{|y|^{2}}{4}Q_{b}^{(\eta)}\\
 & =i\Big(\frac{\lambda_{s}}{\lambda}+b\Big)[\Lambda Q^{(\eta)}]_{b}+\Big[2\Big(\frac{\lambda_{s}}{\lambda}+b\Big)b-(b_{s}+b^{2}+\eta^{2})\Big]\tfrac{|y|^{2}}{4}Q_{b}^{(\eta)}.
\end{align*}
Note that the first term with $[\Lambda Q^{(\eta)}]_{b}$ has a sufficient
decay.

Having fixed the dynamical laws of $\lambda,\gamma,b$, and $\epsilon$,
we setup the bootstrap procedure for $\lambda$, $b$, $\gamma$,
$\|\epsilon\|_{\dot{H}_{m}^{1}}$, and $\|\epsilon\|_{L^{2}}$. By
standard modulation estimates, we can control the variation of $\lambda$,
$b$, and $\gamma$ by $\epsilon$. We also need a separate argument
to control $\|\epsilon\|_{L^{2}}$ by $\|\epsilon\|_{\dot{H}_{m}^{1}}$.
The usual energy inequality does not suffice, we use the Strichartz
estimates to control $L^{2}$ norm of $\epsilon$. On the way, we
obtain estimates on $\|\epsilon\|_{L^{2}}$ by integrating in time.
For this purpose, we introduce a weighted time maximal function $\mathcal{T}_{\dot{H}_{m}^{1}}^{(s,\eta)}[\epsilon]$.

6. \emph{Lyapunov/virial functional method}. To close the bootstrap
argument, we need to control $\dot{H}_{m}^{1}$ norm of $\epsilon$.
From the choice of our initial data, we have $\epsilon(0,x)=0$. We
will propagate smallness from $t=0$ to $t=t_{0}^{\ast}$ via the
Lyapunov method. Such an energy method was first used by Martel \cite{Martel2005AJM}
in a backward construction. We will construct a functional $\mathcal{I}$
such that $\mathcal{I}$ controls the $\dot{H}_{m}^{1}$ norm of $\epsilon$,
and $\partial_{t}\mathcal{I}$ is almost nonnegative (as we evolve
backward in time). To exploit the coercivity, a natural candidate
for $\mathcal{I}$ would be the quadratic (and higher) parts of the
energy functional 
\[
E_{w^{\flat}}^{\mathrm{(qd)}}[\epsilon]=E[w^{\flat}+\epsilon]-E[w^{\flat}]-(\frac{\delta E}{\delta u}\Big|_{u=w^{\flat}},\epsilon)_{r}.
\]
Solely using $E_{w^{\flat}}^{\mathrm{(qd)}}[\epsilon]$, we cannot
have sufficient monotonicity. This is due to the scaling and mass
terms $ib\Lambda\epsilon-\eta\theta_{\eta}\epsilon$, i.e. $\epsilon$
essentially evolves under 
\[
i\partial_{s}\epsilon-\mathcal{L}_{w^{\flat}}\epsilon+ib\Lambda\epsilon-\eta\theta_{\eta}\epsilon\approx0.
\]
To incorporate $ib\Lambda\epsilon$, we add $b\Phi_{A}[\epsilon]$
to $E_{w^{\flat}}^{(\mathrm{qd})}[\epsilon]$, where $\Phi_{A}$ is
a deformation of $\Phi$ in terms of truncation weight function $\phi_{A}$
(see \eqref{eq:def-phiA}). $b\Phi_{A}$ is called a \emph{virial
correction,} which was first introduced by Rapha\"el-Szeftel \cite{RaphaelSzeftel2011JAMS}.
To incorporate $-\eta\theta_{\eta}\epsilon$, we further add $\frac{\eta\theta_{\eta}}{2}M[\epsilon]$.
As a result, we will consider 
\[
\mathcal{I}_{A}\coloneqq\lambda^{-2}(E_{w^{\flat}}^{\mathrm{(qd)}}[\epsilon]+b\Phi_{A}[\epsilon]+\tfrac{\eta\theta_{\eta}}{2}M[\epsilon]).
\]

With this functional $\mathcal{I}_{A}$, we are able to prove that
$\lambda^{2}\mathcal{I}_{A}-\frac{\eta\theta_{\eta}}{2}M[\epsilon]\sim\|\epsilon\|_{\dot{H}_{m}^{1}}^{2}$
from the coercivity of $\mathcal{L}_{Q}$. For $\lambda^{2}\partial_{s}\mathcal{I}_{A}$,
we roughly have 
\[
\lambda^{2}\partial_{s}\mathcal{I}_{A}\approx2b(E_{Q}^{(A),\mathrm{(qd)}}[\epsilon]+\tfrac{\eta\theta_{\eta}}{2}M[\epsilon]),
\]
where $E^{(A)}$ is a deformation of $E$ in terms of truncation weight
function $\phi_{A}$. Note that $b$ has positive sign. We have 
\begin{align*}
E_{Q}^{(A),\mathrm{(qd)}}[\epsilon] & \approx\frac{1}{2}\int\mathbf{1}_{r\leq A}|L_{Q}\epsilon|^{2}\\
 & \quad+\frac{1}{2}\int\mathbf{1}_{r>A}\phi_{A}''|\partial_{r}\epsilon|^{2}+\frac{1}{2}\int\mathbf{1}_{r>A}\frac{\phi_{A}'}{r}\Big(\frac{m+A_{\theta}[Q]}{r}\Big)^{2}|\epsilon|^{2}\\
 & \quad-\frac{1}{8}\int(\Delta^{2}\phi_{A})|\epsilon|^{2}+\mathrm{(bdry)},
\end{align*}
where $\mathrm{(bdry)}$ is a boundary term evaluated at $r=A$. The
first and second lines are apparently nonnegative. For the third line,
a crude estimate 
\[
\Big|\frac{1}{8}\int(\Delta^{2}\phi_{A})|\epsilon|^{2}\Big|\lesssim\|r^{-1}\epsilon\|_{L^{2}}^{2}
\]
using $|\Delta^{2}\phi_{A}|\lesssim\mathbf{1}_{r\geq A}\frac{A}{r^{3}}$
is on the borderline of acceptable error size. So is $(\mathrm{bdry})$.
More precisely, if $\epsilon$ is localized on the region $\{r\sim A\}$,
we would not have an improvement. One of the reasons causing this
difficulty is that the coercivity of $\lambda^{2}\mathcal{I}_{A}$
\emph{does not} control $L^{2}$ norm of $\epsilon$ (uniformly in
$\eta$). This is a difference from \eqref{eq:NLS}. To resolve this,
we average $\lambda^{2}\mathcal{I}_{A}$ as 
\[
\lambda^{2}\mathcal{I}=\frac{2}{\log A}\int_{A^{1/2}}^{A}\lambda^{2}\mathcal{I}_{A'}\frac{dA'}{A'}.
\]
Then, we can prove that $\lambda^{2}\partial_{s}\mathcal{I}$ is almost
nonnegative. With this $\mathcal{I}$, we can close our bootstrap
procedure. See Figure 1.

\begin{figure}
\label{fig:1}\includegraphics{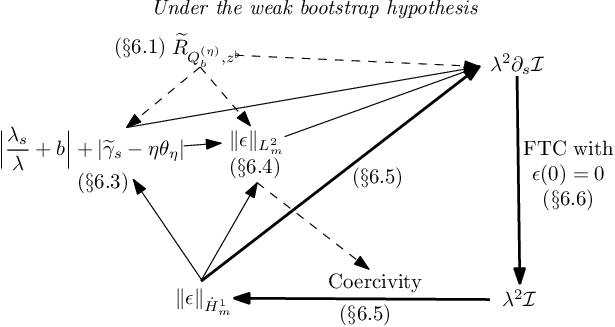}\caption{Scheme of the bootstrap}
\end{figure}

7. \emph{Conditional uniqueness}. Let $u_{1}$ and $u_{2}$ be as
in Theorem \ref{thm:cond-uniq}. We decompose each $u_{j}$ with its
modulation parameters $(b_{j},\lambda_{j},\gamma_{j})$ and $\epsilon_{j}$,
and redo the modulation analysis on $\epsilon\coloneqq\epsilon_{1}-\epsilon_{2}$.
A similar argument is used in \cite{MartelMerleRaphael2015JEMS}.
We conclude $\epsilon=0$ and $(b_{1},\lambda_{1},\gamma_{1})=(b_{2},\lambda_{2},\gamma_{2})$
again by Lyapunov/virial method. Here the crucial idea is to write
for $j\in\{1,2\}$ 
\[
u_{j}(t,x)=\frac{e^{i\gamma_{j}}}{\lambda_{j}}Q_{b_{j}}\Big(\frac{1}{\lambda_{j}}\Big)+z(t,x)+\frac{e^{i\gamma_{1}}}{\lambda_{1}}\epsilon_{j}(t,\frac{x}{\lambda_{1}}).
\]
We take the $\flat$-operation with the common parameter set $(\lambda_{1},\gamma_{1})$.
Here, the point is that $u_{1}$ and $u_{2}$ should be rescaled with
the common parameters since the hypothesis is $\|u_{1}(t)-u_{2}(t)\|_{H^{1}}\leq c|t|$.
From this, the parameter difference only appears in the difference
of $Q_{b_{1},\lambda_{1},\gamma_{2}}$ and $Q_{b_{2},\lambda_{2},\gamma_{2}}$.
This leads that estimating $\epsilon$ boils down to estimating the
parameter differences. Furthermore, we observe that \eqref{eq:dyna-law-1}
with $\eta=0$ is explicitly solvable: 
\[
\frac{b_{1}}{(\lambda_{1})^{2}}=|t|^{-1}=\frac{b_{2}}{(\lambda_{2})^{2}}.
\]
This observation improves the estimate of $Q_{b_{1},\lambda_{1},\gamma_{2}}$
and $Q_{b_{2},\lambda_{2},\gamma_{2}}$ by a logarithmic factor, which
is crucial to conclude the proof.

\subsection*{Organization of the paper}

Throughout the paper, we present analysis with the $\eta$-parameter,
$u^{(\eta)}$. The reader who is interested only in the construction
(or, anyone at first reading) may read all assuming $\eta=0$ with
a minor correction. In fact, the authors conducted this work with
$\eta=0$ first.

In Section \ref{sec:notations}, we collect notations, some properties
of equivariant Sobolev spaces, and preliminary multilinear estimates.
We also introduce time maximal functions on $[t_{0}^{\ast},0)$ and
present equivariant Cauchy theory on $L_{m}^{2}$ and $H_{m}^{s}$.
In Section \ref{sec:linearization}, we discuss the linearization
of \eqref{eq:CSS}. In Section \ref{sec:Profile}, we discuss an instability
mechanism and construct the profile $Q^{(\eta)}$. In Section \ref{sec:modulation},
we set up the modulation analysis, fix modulation parameters, and
reduce Theorems \ref{thm:BW-sol} and \ref{thm:instability} to the
main bootstrap lemma. In Section \ref{sec:Bootstrap}, we close our
bootstrap argument by employing the Lyapunov/virial functional method.
In Section \ref{sec:cond-uniq}, we prove Theorem \ref{thm:cond-uniq}.

There are two appendices. In Appendix \ref{sec:equiv-sob-sp}, we
discuss equivariant functions and Sobolev spaces for self-containedness.
In Appendix \ref{sec:Local-Theory}, we sketch the proof of equivariant
$H_{m}^{s}$-Cauchy theory (Proposition \ref{prop:Hs-Cauchy}).

\subsection*{Acknowledgement}

We appreciate A. Lawrie, S.-J. Oh, and S. Shahshahani for sharing
their manuscript \cite{LawrieOhShashahani_unpub} on the linearization
of \eqref{eq:CSS}. We are also grateful to S.-J. Oh for helpful discussions
and encouragement to this work. The authors were partially supported
by Samsung Science \& Technology Foundation BA1701-01 and National
Research Foundation of Korea NRF-2019R1A5A1028324.

\section{\label{sec:notations}Notations and preliminaries}

\subsection{Basic notations}

We mainly work with equivariant (including radial) functions on $\R^{2}$,
say $\phi:\R^{2}\to\C$. We also work with their radial part $u:(0,\infty)\to\C$.
Thus we use the integral symbol $\int$ to mean 
\begin{equation}
\int=\int_{\R^{2}}dx=2\pi\int_{0}^{\infty}rdr.\label{eq:def-integral}
\end{equation}
We also need the generator $\Lambda$ for the $L^{2}$-scaling, which
is defined by 
\begin{equation}
\Lambda f\coloneqq\frac{d}{d\lambda}\Big|_{\lambda=1}\lambda f(\lambda\cdot)=[1+r\partial_{r}]f.\label{eq:def-Lambda}
\end{equation}

The linearized operator at the static solution is only $\R$-linear.
We thus work with $L^{2}(\R^{2};\C)$ viewed as a \emph{real} Hilbert
space with the inner product 
\begin{equation}
(f,g)_{r}\coloneqq\int\Re(f\overline{g}).\label{eq:real-inner-prod}
\end{equation}
All the functional derivatives will be computed with respect to this
real inner product.

For $\eta\geq0$, we will use the $\eta$\emph{-dependent} \emph{Japanese
bracket}:
\begin{equation}
\langle t\rangle\coloneqq(t^{2}+\eta^{2})^{\frac{1}{2}}.\label{eq:def-jap-brac}
\end{equation}

In addition, we will use notations such as $L_{m}^{2}$, $\dot{H}_{m}^{1}$,
$f^{\sharp}$, $g^{\flat}$, $f_{b}$, and so on. These will be defined
on the way.

We also use the conventional notations. We say $A\lesssim B$ when
there is some implicit constant $C$ that does not depend on $A$
and $B$ satisfying $A\leq CB$. For some parameter $r$, we write
$A\lesssim_{r}B$ if the implicit constant $C$ depends on $r$. In
this paper, dependence on the equivariance index $m$ is suppressed;
we simply write $A\lesssim B$ for $A\lesssim_{m}B$.

We use (mixed) Lebesgue norms. For $1\leq p\leq\infty$, we define
$\ell^{p}$ and $L^{p}$ norms by 
\begin{align*}
\|a\|_{\ell^{p}} & \coloneqq\Big(\sum_{k\geq0}|a_{k}|^{p}\Big)^{\frac{1}{p}},\\
\|f\|_{L^{p}} & \coloneqq\Big(\int_{\R^{d}}|f(x)|^{p}dx\Big)^{\frac{1}{p}},
\end{align*}
for any sequence $a=(a_{0},a_{1},\dots)$ and function $f:\R^{d}\to\C$.
For $1\leq q\leq\infty$, an interval $I\subseteq\R$, a Banach space
$X$, and $f:I\to X$, we denote 
\begin{align*}
\|f\|_{L_{I}^{q}X} & \coloneqq\big\|\|f(t)\|_{X}\big\|_{L_{t}^{q}(I)}.
\end{align*}
In particular, we simply write $L_{I}^{q}\coloneqq L_{t}^{q}(I)$.
In case of $I=\R$, we write $L_{t}^{q}\coloneqq L_{t}^{q}(\R)$.
If $t\in I\mapsto G(t)\in X$ is continuous and bounded, we write
$G\in C_{I}X$. We also use $f\in L_{I,\mathrm{loc}}^{q}X$ or $f\in C_{I,\mathrm{loc}}X$
if the restriction of $f$ on any subintervals $J$ of $I$ belong
to $L_{J}^{q}X$ or $C_{J}X$, respectively. Finally, the mixed Lebesgue
norm $L_{I}^{q}L_{x}^{q}$ is simply written as $L_{I,x}^{q}$ in
view of Fubini's theorem.

We also need the following Sobolev norms
\[
\|f\|_{\dot{H}^{1}}\coloneqq\|\nabla f\|_{L^{2}}\quad\text{and}\quad\|f\|_{H^{1}}\coloneqq(\|f\|_{L^{2}}^{2}+\|f\|_{\dot{H}^{1}}^{2})^{\frac{1}{2}}.
\]
In some places (but not often), we need fractional Sobolev norms ($s\in\R$)
\[
\|f\|_{H^{s}}\coloneqq\|\langle D\rangle^{s}f\|_{L^{2}},
\]
where $\langle D\rangle^{s}$ is the Fourier multiplier operator with
the symbol $\langle\xi\rangle^{s}$.

\subsection{\label{subsec:equiv-Sobolev-sp}Equivariant Sobolev spaces}

Let $m\in\Z\setminus\{0\}$. In this subsection, we record some properties
of $m$-equivariant functions and associated Sobolev spaces. We confine
to $m\neq0$ for simplicity of the exposition. We also assume $m>0$
since $m$-equivariant functions with $m<0$ are merely conjugates
of $(-m)$-equivariant function. An elementary detailed exposition
is given in Appendix \ref{sec:equiv-sob-sp}.

A function $f:\R^{2}\to\C$ is said to be \emph{$m$-equivariant}
if 
\[
f(x)=g(r)e^{im\theta},\qquad\forall x\in\R^{2}\setminus\{0\}
\]
for some $g:(0,\infty)\to\C$, where we expressed $x_{1}+ix_{2}=re^{i\theta}$
with $x=(x_{1},x_{2})$.

Here is a brief discussion of smooth $m$-equivariant functions $f$
for $m\neq0$. If we require $f$ to be continuous at the origin,
we should have $f(0)=0$ because the factor $e^{im\theta}$ is not
continuous at the origin. Moreover, regularity of $f$ forces $f$
to be degenerate at the origin; in fact, Lemma \ref{lem:smooth-equiv-criterion}
says that smooth $m$-equivariant $f$ should have the expression\footnote{If $m<0$, then $f$ has expression $f(x)=h(x)\cdot(x_{1}-ix_{2})^{|m|}$.}
\[
f(x)=h(x)\cdot(x_{1}+ix_{2})^{m}
\]
for some \emph{smooth} radial function $h(x)=h(|x|)$. This explicitly
shows the degeneracy of $f$ at the origin from its regularity and
$m$.

We turn our attention to equivariant Sobolev spaces. For $s\geq0$,
we define $H_{m}^{s}$ by the set of $m$-equivariant functions lying
in the usual Sobolev space $H^{s}(\R^{2})$ with denoting $L_{m}^{2}\coloneqq H_{m}^{0}$.
We equip $H_{m}^{s}$ with the usual $H^{s}(\R^{2})$-norm. We denote
by $C_{c,m}^{\infty}(\R^{2}\setminus\{0\})$ the set of smooth $m$-equivariant
functions having compact support in $\R^{2}\setminus\{0\}$. Here
is a useful density theorem for equivariant functions (Lemma \ref{lem:density-equiv}):
the space $C_{c,m}^{\infty}(\R^{2}\setminus\{0\})$ is dense in $H_{m}^{s}$
if and only if $s\leq m+1$.

For functions in $C_{c,m}^{\infty}(\R^{2}\setminus\{0\})$, we have
\begin{align*}
\Delta f & =\partial_{rr}f+\frac{1}{r}\partial_{r}f-\frac{m^{2}}{r^{2}}f,\\
\|f\|_{\dot{H}^{1}}^{2} & =\|\nabla f\|_{L^{2}}^{2}=\|\partial_{r}f\|_{L^{2}}^{2}+m^{2}\|r^{-1}f\|_{L^{2}}^{2},\\
\|f\|_{H^{1}}^{2} & =\|f\|_{L^{2}}^{2}+\|f\|_{\dot{H}^{1}}^{2}.
\end{align*}
In particular, we have the Hardy-Sobolev inequality (Lemma \ref{lem:hardy-equiv}):
for $m\neq0$, 
\begin{equation}
\|r^{-1}f\|_{L^{2}}+\|f\|_{L^{\infty}}\lesssim\|f\|_{\dot{H}^{1}},\qquad\forall f\in C_{c,m}^{\infty}(\R^{2}\setminus\{0\}).\label{eq:hardy-Sobolev}
\end{equation}
There is also a generalization of Hardy's inequality (Lemma \ref{lem:gen-hardy-equiv}):
if $0\leq k\leq m$ is an integer, then we have 
\begin{equation}
\|r^{-k}f\|_{L^{2}}+\|r^{-(k-1)}\partial_{r}f\|_{L^{2}}+\cdots+\|\partial_{r}^{k}f\|_{L^{2}}\sim\|f\|_{\dot{H}^{k}},\qquad\forall f\in C_{c,m}^{\infty}(\R^{2}\setminus\{0\}).\label{eq:gen-hardy}
\end{equation}
Moreover, one can also have $L^{\infty}$-estimate for derivatives
of $f$; for example, we have 
\begin{equation}
\|\partial_{r}f\|_{L^{\infty}}\lesssim\|f\|_{\dot{H}^{2}},\qquad\forall m\geq2,\ f\in C_{c,m}^{\infty}(\R^{2}\setminus\{0\})\label{eq:gen-hardy-Linfty}
\end{equation}
by mimicking the proof of \eqref{eq:hardy-Sobolev}.

For an integer $0\leq k\leq m$, we define the space $\dot{H}_{m}^{k}$
by the completion of $C_{c,m}^{\infty}(\R^{2}\setminus\{0\})$ in
$L_{\mathrm{loc}}^{2}(\R^{2})$ with respect to $\dot{H}_{m}^{k}$
norm. We have the following embeddings of spaces: 
\[
C_{c,m}^{\infty}(\R^{2}\setminus\{0\})\hookrightarrow H_{m}^{k}\hookrightarrow\dot{H}_{m}^{k}\hookrightarrow L_{\mathrm{loc}}^{2}(\R^{2}).
\]
By density, it suffices to check estimates regarding $H_{m}^{k}$
or $\dot{H}_{m}^{k}$ norms for $C_{c,m}^{\infty}(\R^{2}\setminus\{0\})$
functions. In particular, \eqref{eq:hardy-Sobolev}, \eqref{eq:gen-hardy},
and \eqref{eq:gen-hardy-Linfty} hold for all functions in $\dot{H}_{m}^{1}$,
$\dot{H}_{m}^{k}$, and $\dot{H}_{m}^{2}$, respectively. In this
paper, the space $\dot{H}_{m}^{1}$ is of particular importance.

\subsection{\label{subsec:Envelopes}Time maximal functions}

In this subsection, we introduce time maximal functions with the time
variable $t$. In the proof of Theorems \ref{thm:BW-sol}-\ref{thm:instability},
we will use these for functions defined on $[t_{0}^{\ast},0)$. Several
estimates (say at time $t<0$) will be obtained by a time integral
on $[t,0)$. For example, $\|\epsilon(t)\|$ will be estimated by
a time integral of $\epsilon(t')$ on $[t,0)$ with $\epsilon(0)=0$.
Thus it is convenient to use a maximal function of time instead of
fixed-time quantities.

To motivate this, consider the following situation. If $f:[t_{0}^{\ast},0)\to\R$
is given by $f(t)\sim|t|^{p}$, $p\geq0$, then we have $\int_{t}^{0}|f(t')|dt'\sim\frac{|t|^{p+1}}{p+1}\sim_{p}|t||f(t)|$.
However, if we only know $f(t)\lesssim|t|^{p}$, then we cannot guarantee
that $\int_{t}^{0}|f(t')|dt'\lesssim_{p}|t||f(t)|$. Instead, we want
to introduce some maximal function $\mathcal{T}f$ such that $\int_{t}^{0}|f(t')|dt'\lesssim_{p}|t|[\mathcal{T}f](t)$.

Fix $\eta\geq0$. We consider a function $f:[t_{0}^{\ast},0)\to X$,
where $X$ is a Banach space and $t_{0}^{\ast}\in(-\infty,0)$. For
$t\in[t_{0}^{\ast},0)$, let us introduce 
\[
n(t)\coloneqq\begin{cases}
\lceil\max\{0,\log_{2}\tfrac{|t|}{\eta}\}\rceil, & \text{if }\eta>0,\\
\infty & \text{if }\eta=0,
\end{cases}
\]
where $\lceil\alpha\rceil$ is the smallest integer greater than or
equal to $\alpha$. For each nonnegative integer $j$, we define 
\[
a[f;t,j]\coloneqq\begin{cases}
{\displaystyle \sup_{t'\in[2^{-j}t,2^{-(j+1)}t)}\|f(t')\|_{X}} & \text{if }0\leq j<n(t),\\
{\displaystyle \sup_{t'\in[2^{-j}t,0)}\|f(t')\|_{X}} & \text{if }j\geq n(t).
\end{cases}
\]
For each $s\in\R$, we define the \emph{time maximal function }$\mathcal{T}_{X}^{(\eta,s)}[f]:[t_{0}^{\ast},0)\to[0,+\infty]$
by 
\begin{equation}
\mathcal{T}_{X}^{(\eta,s)}[f](t)\coloneqq\sum_{j=0}^{n(t)}2^{js}(j^{2}+1)^{-1}a[f;t,j].\label{eq:def-env}
\end{equation}
When $\eta=0$, we denote $\mathcal{T}_{X}^{(0,s)}[f]=\mathcal{T}_{X}^{(s)}[f]$.
When $f$ is scalar-valued, we simply write $\mathcal{T}_{X}^{(\eta,s)}[f]=\mathcal{T}^{(\eta,s)}[f]$.
The decay rate $(j^{2}+1)^{-1}$ in \eqref{eq:def-env} is a technical
factor that is required in the proof of Lemma \ref{lem:env-prop}
(see \eqref{eq:env-temp1} and \eqref{eq:env-temp2}).

It is clear that $\mathcal{T}_{X}^{(\eta,s)}$ is sub-additive. Indeed,
for any scalars $c_{n}$ and functions $f_{n}$, we have 
\[
\mathcal{T}_{X}^{(\eta,s)}\Big[\sum_{n}c_{n}f_{n}\Big]\leq\sum_{n}|c_{n}|\mathcal{T}_{X}^{(\eta,s)}[f_{n}].
\]
$\mathcal{T}_{X}^{(\eta,s)}$ preserves power-type bounds; if $\|f(t)\|_{X}\sim\langle t\rangle^{q}$
and $s\leq q$ (recall that we defined $\langle t\rangle\coloneqq(|t|^{2}+\eta^{2})^{\frac{1}{2}}$),
then we have 
\[
\mathcal{T}_{X}^{(\eta,s)}[f](t)\sim_{q}\langle t\rangle^{q}.
\]
$\mathcal{T}_{X}^{(\eta,s)}$ is idempotent in the sense of \eqref{eq:env-prop2}.
Moreover, $\mathcal{T}_{X}^{(\eta,s)}$ is designed to behave nicely
with integration in $t$; see \eqref{eq:env-prop4} below.

One can easily verify qualitative properties of $\mathcal{T}_{X}^{(\eta,s)}$.
If $\eta>0$ or $f(t)\lesssim|t|^{s}$ for $t$ near zero, then $\mathcal{T}_{X}^{(\eta,s)}[f]$
has finite values. If $\mathcal{T}_{X}^{(\eta,s)}[f](t_{0})<\infty$
for some $t_{0}\in[t_{0}^{\ast},0)$, then $\mathcal{T}_{X}^{(\eta,s)}[f](t)<\infty$
for all $t\in[t_{0},0)$. Finally, if $f$ is continuous, then so
is $\mathcal{T}_{X}^{(\eta,s)}[f]$.

We conclude this subsection with quantitative properties.
\begin{lem}[Quantitative properties of $\mathcal{T}_{X}^{(\eta,s)}$]
\label{lem:env-prop}Fix $\eta\geq0$. Let $q\in\R$, $s\geq0$,
and $1\leq p\leq\infty$. Then,
\begin{align}
\|f(t)\|_{X} & \leq\mathcal{T}_{X}^{(\eta,s)}[f](t),\label{eq:env-prop1}\\
\mathcal{T}^{(\eta,s)}[\mathcal{T}_{X}^{(\eta,s)}[f]](t) & \sim_{s}\mathcal{T}_{X}^{(\eta,s)}[f](t),\label{eq:env-prop2}\\
\mathcal{T}_{X}^{(\eta,s)}[\langle\cdot\rangle^{q}f](t) & \sim_{q}\langle t\rangle^{q}\mathcal{T}_{X}^{(\eta,s-q)}[f](t).\label{eq:env-prop3}
\end{align}
If $\frac{1}{p}+q+s>0$, we have 
\begin{equation}
\|\langle\cdot\rangle^{q}\mathcal{T}_{X}^{(\eta,s)}[f]\|_{L_{[t,0)}^{p}}\lesssim_{p,q,s}\langle t\rangle^{q+\frac{1}{p}}\mathcal{T}_{X}^{(\eta,s)}[f](t).\label{eq:env-prop4}
\end{equation}
We remark that all of the above estimates are uniform in $\eta\geq0$.
\end{lem}

\begin{proof}
\eqref{eq:env-prop1} is immediate from the definition \eqref{eq:def-env}.

We turn to \eqref{eq:env-prop2}. We have the $(\gtrsim)$ direction
by \eqref{eq:env-prop1}. To show the $(\lesssim)$ direction, we
use the following properties of $a[f;t,j]$: 
\begin{align}
a[a[f;\cdot,k];t,j] & \leq a[f;t,j+k]+a[f;t,j+k+1]\text{ for all }j,k\geq0,\label{eq:env-temp3}\\
j+k & \leq n(t)\text{ if }0\leq j\leq n(t)\text{ and }0\leq k\leq n(2^{-j}t).\label{eq:env-temp4}
\end{align}
By \eqref{eq:env-temp3} and \eqref{eq:env-temp4}, we have 
\begin{align*}
 & \mathcal{T}^{(\eta,s)}[\mathcal{T}_{X}^{(\eta,s)}[f]](t)\\
 & \leq\sum_{j=0}^{n(t)}\sum_{k=0}^{n(2^{-j}t)}2^{(j+k)s}(j^{2}+1)^{-1}(k^{2}+1)^{-1}a[a[f;\cdot,k];t,j]\\
 & \leq\sum_{\ell=0}^{n(t)}\Big(\sum_{j=0}^{\ell}\frac{\ell^{2}+1}{(j^{2}+1)((\ell-j)^{2}+1)}\Big)2^{\ell s}(\ell^{2}+1)^{-1}\Big(a[f;t,\ell]+a[f;t,\ell+1]\Big).
\end{align*}
Using 
\begin{equation}
\sum_{j=0}^{\ell}\frac{\ell^{2}+1}{(j^{2}+1)((\ell-j)^{2}+1)}\lesssim\sum_{j=0}^{\infty}\frac{1}{j^{2}+1}\lesssim1,\label{eq:env-temp1}
\end{equation}
the estimate \eqref{eq:env-prop2} follows.

The estimate \eqref{eq:env-prop3} easily follows from the observation
\begin{equation}
\langle t'\rangle\sim2^{-j}\langle t\rangle\quad\text{whenever }\begin{cases}
t'\in[2^{-j}t,2^{-(j+1)}t) & \text{if }j<n(t)\text{, or}\\
t'\in[2^{-j}t,0) & \text{if }j=n(t).
\end{cases}\label{eq:env-temp5}
\end{equation}

Finally, we assume $\frac{1}{p}+q+s>0$ and show \eqref{eq:env-prop4}.
We only provide the proof for $1\leq p<\infty$ as the case $p=\infty$
can be proved with an obvious modification of the argument. We estimate
the LHS of \eqref{eq:env-prop4} as 
\begin{align*}
 & \|\langle\cdot\rangle^{q}\mathcal{T}_{X}^{(\eta,s)}[f]\|_{L_{[t,0)}^{p}}\\
 & \lesssim\||2^{-j}t|^{\frac{1}{p}}\langle2^{-j}t\rangle^{q}\mathcal{T}_{X}^{(\eta,s)}[f](2^{-j}t)\|_{\ell_{0\leq j\leq n(t)}^{p}}\\
 & \lesssim_{q}|t|^{\frac{1}{p}}\langle t\rangle^{q}\big\|2^{-j(\frac{1}{p}+q)}\cdot2^{ks}(k^{2}+1)^{-1}a[f;2^{-j}t,k]\big\|_{\ell_{0\leq j\leq n(t)}^{p}\ell_{0\leq k\leq n(2^{-j}t)}^{1}},
\end{align*}
where we used \eqref{eq:env-temp5} in the last inequality. We then
estimate the above $\ell_{j}^{p}\ell_{k}^{1}$ term using H\"older's
inequality with 
\begin{align}
\big\|2^{-j(\frac{1}{p}+q+s)}((j+k)^{2}+1)(k^{2}+1)^{-2}\big\|_{\ell_{j\geq0}^{p}\ell_{k\geq0}^{\infty}} & \lesssim_{p,q,s}1,\label{eq:env-temp2}\\
\big\|2^{(j+k)s}((j+k)^{2}+1)^{-1}a[f;t,j+k]\big\|_{\ell_{0\leq j\leq n(t)}^{\infty}\ell_{0\leq k\leq n(2^{-j}t)}^{1}} & \leq\mathcal{T}_{X}^{(\eta,s)}[f](t),\nonumber 
\end{align}
where  we used $a[f;2^{-j}t,k]=a[f;t,j+k]$ and \eqref{eq:env-temp4}.
This completes the proof of \eqref{eq:env-prop4}.
\end{proof}

\subsection{Decomposition of nonlinearity and duality estimates}

\subsubsection*{Decomposition of nonlinearity}

Denote by $\mathcal{N}(u)$ the nonlinearity of \eqref{eq:CSS}: 
\begin{equation}
\mathcal{N}(u)\coloneqq-|u|^{2}u+\frac{2m}{r^{2}}A_{\theta}[u]u+\frac{A_{\theta}^{2}[u]}{r^{2}}u+A_{0}[u]u.\label{eq:nonlinearity}
\end{equation}
We introduce trilinear terms $\mathcal{N}_{3,0}$, $\mathcal{N}_{3,1}$,
$\mathcal{N}_{3,2}$ and quintilinear terms $\mathcal{N}_{5,1}$,
$\mathcal{N}_{5,2}$. For $\psi_{j}:(0,\infty)\to\C$, with abbreviations
\begin{align*}
\mathcal{N}_{3,k} & =\mathcal{N}_{3,k}(\psi_{1},\psi_{2},\psi_{3}),\quad\forall k\in\{0,1,2\}\\
\mathcal{N}_{5,k} & =\mathcal{N}_{5,k}(\psi_{1},\psi_{2},\psi_{3},\psi_{4},\psi_{5}),\quad\forall k\in\{1,2\},
\end{align*}
we define 
\begin{equation}
\left\{ \begin{aligned}\mathcal{N}_{3,0} & \coloneqq-\psi_{1}\overline{\psi_{2}}\psi_{3},\\
\mathcal{N}_{3,1} & \coloneqq-\frac{m}{r^{2}}\Big(\int_{0}^{r}\Re(\psi_{1}\overline{\psi_{2}})r'dr'\Big)\psi_{3},\\
\mathcal{N}_{3,2} & \coloneqq-\Big(\int_{r}^{\infty}\frac{m}{r'}\Re(\psi_{1}\overline{\psi_{2}})dr'\Big)\psi_{3},\\
\mathcal{N}_{5,1} & \coloneqq\frac{1}{4r^{2}}\Big(\int_{0}^{r}\Re(\psi_{1}\overline{\psi_{2}})r'dr'\Big)\Big(\int_{0}^{r}\Re(\psi_{3}\overline{\psi_{4}})r'dr'\Big)\psi_{5},\\
\mathcal{N}_{5,2} & \coloneqq\frac{1}{2}\bigg(\int_{r}^{\infty}\Big(\int_{0}^{r'}\Re(\psi_{1}\overline{\psi_{2}})r''dr''\Big)\Re(\psi_{3}\overline{\psi_{4}})\frac{dr'}{r'}\bigg)\psi_{5}.
\end{aligned}
\right.\label{eq:def-N}
\end{equation}
In case of $\psi_{j}=u$ for all $j$, we write 
\[
\mathcal{N}_{\ast}(u)\coloneqq\mathcal{N}_{\ast}(u,\dots,u)
\]
for any possible choice of $\ast$. Then we can write the nonlinearity
$\mathcal{N}$ as 
\[
\mathcal{N}(u)=[\mathcal{N}_{3,0}+\mathcal{N}_{3,1}+\mathcal{N}_{3,2}+\mathcal{N}_{5,1}+\mathcal{N}_{5,2}](u).
\]

\subsubsection*{Introduction of $\mathcal{M}_{4,0}^{(A)}$, $\mathcal{M}_{4,1}^{(A)}$,
and $\mathcal{M}_{6}^{(A)}$}

We introduce a smooth radial weight $\phi$ such that $\partial_{r}\phi$
is increasing and 
\begin{align*}
\partial_{r}\phi(r) & =\begin{cases}
r & \text{if }r\leq1,\\
3-e^{-r} & \text{if }r\geq2.
\end{cases}
\end{align*}
We write for $A\in[1,+\infty]$: 
\begin{equation}
\phi_{A}(r)\coloneqq\begin{cases}
A^{2}\phi(\frac{r}{A}), & \text{if }1\leq A<\infty,\\
\frac{1}{2}r^{2} & \text{if }A=\infty.
\end{cases}\label{eq:def-phiA}
\end{equation}
We introduce quartic forms $\mathcal{M}_{4,0}^{(A)}$, $\mathcal{M}_{4,1}^{(A)}$,
and sextic form $\mathcal{M}_{6}^{(A)}$. For $\psi_{j}:(0,\infty)\to\C$,
with abbreviations 
\begin{align*}
\mathcal{M}_{4,k}^{(A)} & =\mathcal{M}_{4,k}^{(A)}(\psi_{1},\psi_{2},\psi_{3},\psi_{4}),\quad\forall k\in\{0,1\}\\
\mathcal{M}_{6}^{(A)} & =\mathcal{M}_{6}^{(A)}(\psi_{1},\psi_{2},\psi_{3},\psi_{4},\psi_{5},\psi_{6}),
\end{align*}
we define 
\begin{equation}
\left\{ \begin{aligned}\mathcal{M}_{4,0}^{(A)} & \coloneqq\int\frac{\Delta\phi_{A}}{2}\cdot\psi_{1}\overline{\psi_{2}}\psi_{3}\overline{\psi_{4}},\\
\mathcal{M}_{4,1}^{(A)} & \coloneqq\int\frac{\phi_{A}'}{r}\cdot\frac{1}{r^{2}}\Big(\int_{0}^{r}\Re(\psi_{1}\overline{\psi_{2}})r'dr'\Big)\Re(\psi_{3}\overline{\psi_{4}}),\\
\mathcal{M}_{6}^{(A)} & \coloneqq\int\frac{\phi_{A}'}{r}\cdot\frac{1}{r^{2}}\Big(\int_{0}^{r}\Re(\psi_{1}\overline{\psi_{2}})r'dr'\Big)\Big(\int_{0}^{r}\Re(\psi_{3}\overline{\psi_{4}})r'dr'\Big)\Re(\psi_{5}\overline{\psi_{6}}).
\end{aligned}
\right.\label{eq:def-M}
\end{equation}
In case of $\psi_{k}=u$ for all $k$, we simply write as 
\[
\mathcal{M}_{\ast}^{(A)}(u)\coloneqq\mathcal{M}_{\ast}^{(A)}(u,\dots,u)
\]
for any choice of $\ast$. In case of $A=\infty$, we write as 
\[
\mathcal{M}_{\ast}\coloneqq\mathcal{M}_{\ast}^{(\infty)}.
\]
We will mostly use the case $A=\infty$; the case $1\leq A<\infty$
will only appear when we do analysis with the Lyapunov/virial functional
in Section \ref{subsec:virial-lyapunov}.

The above forms $\mathcal{M}_{\ast}$ arise naturally in the energy
functional. In particular, we can rewrite the energy functional $E[u]$
as 
\begin{equation}
E[u]=\frac{1}{2}\int\Big(|\partial_{r}u|^{2}+\frac{m^{2}}{r^{2}}|u|^{2}\Big)-\frac{1}{4}\mathcal{M}_{4,0}(u)-\frac{m}{2}\mathcal{M}_{4,1}(u)+\frac{1}{8}\mathcal{M}_{6}(u).\label{eq:energy-M-expn}
\end{equation}
If we recall that \eqref{eq:CSS} has a Hamiltonian formulation with
the Hamiltonian $E$, the nonlinearity $\mathcal{N}$ arises in the
functional derivative of the energy. We in fact have 
\begin{equation}
\left\{ \begin{aligned}(\mathcal{N}_{3,0}(\psi_{1},\psi_{2},\psi_{3}),\psi_{4})_{r} & =\mathcal{M}_{4,0}(\psi_{1},\psi_{2},\psi_{3},\psi_{4}),\\
(\mathcal{N}_{3,1}(\psi_{1},\psi_{2},\psi_{3}),\psi_{4})_{r} & =-m\mathcal{M}_{4,1}(\psi_{1},\psi_{2},\psi_{3},\psi_{4}),\\
(\mathcal{N}_{3,2}(\psi_{1},\psi_{2},\psi_{3}),\psi_{4})_{r} & =-m\mathcal{M}_{4,1}(\psi_{3},\psi_{4},\psi_{1},\psi_{2}),\\
(\mathcal{N}_{5,1}(\psi_{1},\psi_{2},\psi_{3},\psi_{4},\psi_{5}),\psi_{6})_{r} & =\tfrac{1}{4}\mathcal{M}_{6}(\psi_{1},\psi_{2},\psi_{3},\psi_{4},\psi_{5},\psi_{6}),\\
(\mathcal{N}_{5,2}(\psi_{1},\psi_{2},\psi_{3},\psi_{4},\psi_{5}),\psi_{6})_{r} & =\tfrac{1}{2}\mathcal{M}_{6}(\psi_{1},\psi_{2},\psi_{5},\psi_{6},\psi_{3},\psi_{4}).
\end{aligned}
\right.\label{eq:duality-rel}
\end{equation}
In view of duality, this says that estimates of multilinear forms
$\mathcal{M}_{\ast}$ transfer to estimates of $\mathcal{N}_{\ast}$.
This formulation simplifies multilinear estimates of $\mathcal{N}_{\ast}$.

We conclude this subsection by recording some duality estimates. The
first one is well-adapted to the fixed time energy estimate.
\begin{lem}[Duality estimates, I]
\label{lem:duality-1}Let $A\in[1,\infty]$. The following estimates
hold uniformly in $A$.
\begin{enumerate}
\item For any distinct indices $i,i'\in\{1,2,3,4\}$, we have 
\begin{equation}
|\mathcal{M}_{4,0}^{(A)}|\lesssim\|\psi_{i}\|_{L^{\infty}}\|\psi_{i'}\|_{L^{\infty}}\prod_{j\neq i,i'}\|\psi_{j}\|_{L^{2}}.\label{eq:multi-1-1}
\end{equation}
\item For any distinct indices $i,i'\in\{1,2,3,4\}$, we have 
\begin{equation}
|\mathcal{M}_{4,1}^{(A)}|\lesssim\Big\|\frac{\psi_{i}}{r}\Big\|_{L^{2}}\Big\|\frac{\psi_{i'}}{r}\Big\|_{L^{2}}\prod_{j\neq i,i'}\|\psi_{j}\|_{L^{2}}.\label{eq:multi-1-2}
\end{equation}
\item For any distinct indices $i,i'\in\{1,\dots,6\}$, we have 
\begin{equation}
|\mathcal{M}_{6}^{(A)}|\lesssim\Big\|\frac{\psi_{i}}{r}\Big\|_{L^{2}}\Big\|\frac{\psi_{i'}}{r}\Big\|_{L^{2}}\prod_{j\neq i,i'}\|\psi_{j}\|_{L^{2}}.\label{eq:multi-1-3}
\end{equation}
\end{enumerate}
\end{lem}

\begin{proof}
In each case, it suffices to consider the case with $A=\infty$ since
we have $|\Delta\phi_{A}|+|\frac{\phi_{A}'}{r}|\lesssim1$ uniformly
in $A$. The estimate \eqref{eq:multi-1-1} easily follows by H\"older's
inequality. To show \eqref{eq:multi-1-2} and \eqref{eq:multi-1-3},
decompose $\frac{1}{r^{2}}=\frac{1}{r}\cdot\frac{1}{r}$ and move
each $\frac{1}{r}$ to $\psi_{i}$ and $\psi_{i'}$. We note that
one can put $\frac{1}{r}$ inside the integral as 
\[
\Big|\frac{1}{r}\int_{0}^{r}\Re(\overline{\psi}\varphi)r'dr'\Big|\lesssim\int_{0}^{r}\Big|\frac{\psi}{r'}\cdot\varphi\Big|r'dr'.
\]
Now the estimates \eqref{eq:multi-1-2} and \eqref{eq:multi-1-3}
follow by H\"older's inequality with 
\[
\Big\|\int_{0}^{r}\Re(\overline{\psi}\varphi)r'dr'\Big\|_{L^{\infty}}\lesssim\|\psi\|_{L^{2}}\|\varphi\|_{L^{2}}.
\]
This completes the proof.
\end{proof}
\begin{lem}[$H_{m}^{1}$-estimate of $\mathcal{N}$]
\label{lem:L2-est-N}Let $m\geq1$. For any possible choice of $\ast$
in \eqref{eq:def-N} and any two distinct indices $i,i'$, we have
\begin{align}
\|\mathcal{N}_{\ast}\|_{L_{m}^{2}} & \lesssim\|\psi_{i}\|_{\dot{H}_{m}^{1}}\|\psi_{i'}\|_{\dot{H}_{m}^{1}}\prod_{j\neq i,i'}\|\psi_{j}\|_{L_{m}^{2}},\label{eq:L2-est-N}\\
\|\mathcal{N}_{\ast}\|_{H_{m}^{1}} & \lesssim\|\psi_{i}\|_{\dot{H}_{m}^{1}}\|\psi_{i'}\|_{\dot{H}_{m}^{1}}\prod_{j\neq i,i'}\|\psi_{j}\|_{H_{m}^{1}}.\label{eq:H1-est-N}
\end{align}
\end{lem}

\begin{proof}
The estimate \eqref{eq:L2-est-N} easily follows from the duality
relations \eqref{eq:duality-rel}, duality estimates (Lemma \ref{lem:duality-1}),
and embeddings \eqref{eq:hardy-Sobolev}.

To estimate \eqref{eq:H1-est-N}, we act $\partial_{r}$ and $\frac{1}{r}$
on $\mathcal{N}_{3,0},\mathcal{N}_{3,1},\mathcal{N}_{3,2}$. These
consist of linear combinations of the following expressions:
\begin{align*}
 & (\partial_{r}\psi_{1})\overline{\psi_{2}}\psi_{3};\quad\psi_{1}(\partial_{r}\overline{\psi_{2}})\psi_{3};\quad\psi_{1}\overline{\psi_{2}}(\partial_{r}\psi_{3});\quad\psi_{1}\overline{\psi_{2}}\tfrac{\psi_{3}}{r};\\
 & {\textstyle (\int_{r}^{\infty}\psi_{1}\overline{\psi_{2}}\frac{dr'}{r'})\frac{\psi_{3}}{r};\quad\frac{1}{r^{2}}(\int_{0}^{r}\psi_{1}\overline{\psi_{2}}r'dr')\frac{\psi_{3}}{r}.}
\end{align*}
These expressions are of types $\mathcal{N}_{3,0}$, $\mathcal{N}_{3,1}$,
and $\mathcal{N}_{3,2}$, where we replace exactly one $\psi_{j}$
by $\partial_{r}\psi_{j}$ or $r^{-1}\psi_{j}$. Thus we can apply
\eqref{eq:L2-est-N} and \eqref{eq:hardy-Sobolev} to estimate all
the above terms by 
\[
\|\psi_{1}\|_{\dot{H}_{m}^{1}}\|\psi_{2}\|_{\dot{H}_{m}^{1}}\|\psi_{3}\|_{\dot{H}_{m}^{1}}.
\]
We turn to $\mathcal{N}_{5,1}$ and $\mathcal{N}_{5,2}$. As before,
we will investigate what expressions arise when $\partial_{r}$ and
$\frac{1}{r}$ act on $\mathcal{N}_{5,1}$ and $\mathcal{N}_{5,2}$.
We first note the following useful tricks: 
\begin{align*}
\partial_{r}\Big(\int_{0}^{r}\psi_{1}\overline{\psi_{2}}r'dr'\Big) & =\int_{0}^{r}\Big((\partial_{r}\psi_{1})\overline{\psi_{2}}+\psi_{1}(\partial_{r}\overline{\psi_{2}})+\frac{1}{r'}\psi_{1}\overline{\psi_{2}}\Big)r'dr',\\
\partial_{r}\Big(\int_{r}^{\infty}\psi_{1}\overline{\psi_{2}}\frac{dr'}{r'}\Big) & =\int_{r}^{\infty}\Big((\partial_{r}\psi_{1})\overline{\psi_{2}}+\psi_{1}(\partial_{r}\overline{\psi_{2}})-\frac{1}{r'}\psi_{1}\overline{\psi_{2}}\Big)\frac{dr'}{r'}.
\end{align*}
Therefore, $\partial_{r}$ and $\frac{1}{r}$ acted on $\mathcal{N}_{5,1}$
and $\mathcal{N}_{5,2}$ consist of expressions of types $\mathcal{N}_{5,1}$
and $\mathcal{N}_{5,2}$, where we replace exactly one $\psi_{j}$
by $\partial_{r}\psi_{j}$ or $r^{-1}\psi_{j}$. Applying \eqref{eq:L2-est-N}
completes the proof.
\end{proof}
We also need the following duality estimate adapted to the Strichartz
estimates. It will be used mostly in Section \ref{subsec:Transfer-H1-to-L2}
in a crucial way.
\begin{lem}[Duality estimates, II]
\label{lem:duality-2}Let $A\in[1,\infty]$. The following estimates
hold uniformly in $A$.
\begin{enumerate}
\item For any distinct indices $i,i'\in\{1,\dots,4\}$, we have 
\begin{equation}
|\mathcal{M}_{4,0}^{(A)}|\lesssim\|\psi_{i}\|_{L^{\infty}}\|\psi_{i'}\|_{L^{4}}\prod_{j\neq i,i'}\|\psi_{j}\|_{L^{p_{j}}}\label{eq:multi-2-1}
\end{equation}
for any $p_{j}\in\{2,4\}$ ($j\neq i,i')$ satisfying \eqref{eq:multi-2-scal-rel}.
\item For any distinct indices $i,i'\in\{1,\dots,4\}$, there exist $p_{j}\in\{2,4\}$
($j\neq i,i'$) such that 
\begin{equation}
|\mathcal{M}_{4,1}^{(A)}|\lesssim\Big\|\frac{\psi_{i}}{r}\Big\|_{L^{2}}\|\psi_{i'}\|_{L^{4}}\prod_{j\neq i,i'}\|\psi_{j}\|_{L^{p_{j}}}.\label{eq:multi-2-2}
\end{equation}
\item For any distinct indices $i,i'\in\{1,\dots,6\}$, there exist $p_{j}\in\{2,4\}$
($j\neq i,i'$) such that 
\begin{equation}
|\mathcal{M}_{6}^{(A)}|\lesssim\Big\|\frac{\psi_{i}}{r}\Big\|_{L^{2}}\|\psi_{i'}\|_{L^{4}}\prod_{j\neq i,i'}\|\psi_{j}\|_{L^{p_{j}}}.\label{eq:multi-2-3}
\end{equation}
\end{enumerate}
In the above estimates, $p_{j}$'s satisfy the scaling conditions
\begin{equation}
\sum_{j\neq i,i'}\frac{1}{p_{j}}=\begin{cases}
\frac{3}{4} & \text{for \eqref{eq:multi-2-1} and \eqref{eq:multi-2-2},}\\
\frac{7}{4} & \text{for \eqref{eq:multi-2-3}.}
\end{cases}\label{eq:multi-2-scal-rel}
\end{equation}
\end{lem}

\begin{rem}
Due to expressions of $\mathcal{M}_{4,1}$ and $\mathcal{M}_{6}$,
we cannot choose arbitrary $p_{j}$'s satisfying the scaling condition
\eqref{eq:multi-2-scal-rel}. The choice of $p_{j}$'s \emph{depends}
on the pair $(i,i')$.
\end{rem}

\begin{proof}[Proof of Lemma \ref{lem:duality-2}]
In each case, it suffices to consider the case with $A=\infty$ since
we have $|\Delta\phi_{A}|+|\frac{\phi_{A}'}{r}|\lesssim1$ uniformly
in $A$. The estimate \eqref{eq:multi-2-1} clearly follows by H\"older's
inequality. We omit the proof of \eqref{eq:multi-2-2} and only show
\eqref{eq:multi-2-3}, as the former is easier than the latter. Henceforth,
we focus on \eqref{eq:multi-2-3} with $A=\infty$. Due to symmetry,
it suffices to consider five cases: $(i,i')\in\{(1,2),(1,3),(1,5),(5,1),(5,6)\}$.
We will rely on 
\[
\Big\|\frac{1}{r}\int_{0}^{r}|f|r'dr'\Big\|_{L^{4}}\lesssim\|f\|_{L^{\frac{4}{3}}}\quad\text{and}\quad\Big\|\int_{0}^{r}|f|r'dr'\Big\|_{L^{\infty}}\lesssim\|f\|_{L^{1}},
\]
where the first inequality is obtained by interpolating the second
one and \eqref{eq:B.2}.

In case of $(i,i')=(1,2)$, we estimate as 
\begin{align*}
|\mathcal{M}_{6}| & \lesssim\Big\|\frac{1}{r}\int_{0}^{r}\Big|\frac{\psi_{1}}{r'}\cdot\psi_{2}\Big|r'dr'\Big\|_{L^{4}}\Big\|\int_{0}^{r}|\psi_{3}\psi_{4}|r'dr'\Big\|_{L^{\infty}}\|\psi_{5}\psi_{6}\|_{L^{\frac{4}{3}}}\\
 & \lesssim\|r^{-1}\psi_{1}\|_{L^{2}}\|\psi_{2}\|_{L^{4}}\|\psi_{3}\|_{L^{2}}\|\psi_{4}\|_{L^{2}}\|\psi_{5}\|_{L^{4}}\|\psi_{6}\|_{L^{2}}.
\end{align*}
In case of $(i,i')\in\{(1,3),(1,5)\}$, we estimate as 
\begin{align*}
|\mathcal{M}_{6}| & \lesssim\Big\|\int_{0}^{r}\Big|\frac{\psi_{1}}{r'}\cdot\psi_{2}\Big|r'dr'\Big\|_{L^{\infty}}\Big\|\frac{1}{r}\int_{0}^{r}|\psi_{3}\psi_{4}|r'dr'\Big\|_{L^{4}}\|\psi_{5}\psi_{6}\|_{L^{\frac{4}{3}}}\\
 & \lesssim\|r^{-1}\psi_{1}\|_{L^{2}}\|\psi_{2}\|_{L^{2}}\|\psi_{3}\|_{L^{4}}\|\psi_{4}\|_{L^{2}}\|\psi_{5}\|_{L^{4}}\|\psi_{6}\|_{L^{2}}.
\end{align*}
In case of $(i,i')\in\{(5,1),(5,6)\}$, we estimate as 
\begin{align*}
|\mathcal{M}_{6}| & \lesssim\Big\|\frac{1}{r}\int_{0}^{r}|\psi_{1}\psi_{2}|r'dr'\Big\|_{L^{4}}\Big\|\int_{0}^{r}|\psi_{3}\psi_{4}|r'dr'\Big\|_{L^{\infty}}\Big\|\frac{\psi_{5}}{r}\cdot\psi_{6}\Big\|_{L^{\frac{4}{3}}}\\
 & \lesssim\|\psi_{1}\|_{L^{4}}\|\psi_{2}\|_{L^{2}}\|\psi_{3}\|_{L^{2}}\|\psi_{4}\|_{L^{2}}\|r^{-1}\psi_{5}\|_{L^{2}}\|\psi_{6}\|_{L^{4}}.
\end{align*}
This completes the proof.
\end{proof}
\begin{lem}[$L^{\frac{4}{3}}$-estimate of $\mathcal{N}$]
\label{lem:4/3-est-N}Let $m\geq1$. For any possible choice of $\ast$
in \eqref{eq:def-N} and any index $i$, there exist $p_{j}\in\{2,4\}$
($j\neq i$) depending on $\ast$ and $i$ satisfying 
\begin{align}
\|\mathcal{N}_{3,\ast}\|_{L^{\frac{4}{3}}} & \lesssim\|\psi_{i}\|_{\dot{H}_{m}^{1}}\prod_{j\neq i}\|\psi_{j}\|_{L^{p_{j}}},\label{eq:N-est-3}\\
\|\mathcal{N}_{5,\ast}\|_{L^{\frac{4}{3}}} & \lesssim\|\psi_{i}\|_{\dot{H}_{m}^{1}}\prod_{j\neq i}\|\psi_{j}\|_{L^{p_{j}}},\label{eq:N-est-4}
\end{align}
and the scaling conditions 
\begin{equation}
\sum_{j\neq i}\frac{1}{p_{j}}=\begin{cases}
\frac{3}{4} & \text{for \eqref{eq:N-est-3}},\\
\frac{7}{4} & \text{for }\eqref{eq:N-est-4}.
\end{cases}\label{eq:multi-2-scal-rel-1}
\end{equation}
\end{lem}

\begin{proof}
The proof follows from Lemma \ref{lem:duality-2} and duality relations
\eqref{eq:duality-rel}. We omit the details.
\end{proof}

\subsection{\label{subsec:Dynamic-Rescaling}Dynamic rescaling}

In works of modulation analysis, one introduces modulation parameters
$\lambda(t)$, $\gamma(t)$, and rescale the spacetime variables $(t,x)$
to $(s,y)$. We use $\lambda(t)$ as a scaling parameter function
and define the rescaled variables $(s,y)$ by\footnote{Note that $s=s(t)$ is defined up to addition of constants. Only the
difference $s(t_{1})-s(t_{2})$ is of importance.} 
\[
\frac{ds}{dt}=\frac{1}{\lambda^{2}(t)};\qquad y\coloneqq\frac{x}{\lambda(t)}.
\]

We work with solutions $u(t,x)$ such that 
\[
u(t,x)\approx\frac{1}{\lambda(t)}Q\Big(\frac{x}{\lambda(t)}\Big)e^{i\gamma(t)}
\]
for some fixed profile $Q$, and time-dependent scaling and phase
parameters $\lambda(t)$ and $\gamma(t)$. Note that $u$ is a function
of $(t,x)$ but $Q$ is a function of $y$. It is thus convenient
to introduce notations switching between the $(t,x)$ and $(s,y)$-variables.

\subsubsection*{Introduction of $\sharp$ and $\flat$}

Let $\lambda$ and $\gamma$ be given. When $f$ is a function of
$y$, we use the \emph{raising} operation $\sharp$ to convert $f$
to a function of $x$. Similarly, for a function $g(x)$, we use the
\emph{lowering} operation $\flat$ to convert $g$ to a function of
$y$. We define $\sharp$ and $\flat$ via the formulae\footnote{Musical notations $\sharp$ and $\flat$ are standard in tensor calculus.
In our setting, we use $\sharp$ and $\flat$ with completely different
meaning; we use them to indicate on which scales we view our functions.
But we think that $\sharp$ and $\flat$ in our setting still shares
same spirit with those used in tensor calculus.} 
\begin{align*}
f^{\sharp}(x) & \coloneqq\frac{1}{\lambda}f\Big(\frac{x}{\lambda}\Big)e^{i\gamma},\\
g^{\flat}(y) & \coloneqq\lambda g(\lambda y)e^{-i\gamma}.
\end{align*}
The rasing/lowering operations respect $L^{2}$-scalings in the following
sense: 
\begin{align*}
\Lambda f^{\sharp} & =[\Lambda f]^{\sharp}; & if^{\sharp} & =[if]^{\sharp}; & \|f^{\sharp}\|_{L^{2}} & =\|f\|_{L^{2}};\\
\Lambda g^{\flat} & =[\Lambda g]^{\flat}; & ig^{\flat} & =[ig]^{\flat}; & \|g^{\flat}\|_{L^{2}} & =\|g\|_{L^{2}}.
\end{align*}
Moreover, we have 
\begin{align*}
\Delta f^{\sharp} & =\frac{1}{\lambda^{2}}[\Delta f]^{\sharp}; & \mathcal{N}(f^{\sharp}) & =\frac{1}{\lambda^{2}}[\mathcal{N}(f)]^{\sharp}; & \mathcal{M}(f^{\sharp}) & =\frac{1}{\lambda^{2}}\mathcal{M}(f);\\
\Delta g^{\flat} & =\lambda^{2}[\Delta g]^{\flat}; & \mathcal{N}(g^{\flat}) & =\lambda^{2}[\mathcal{N}(g)]^{\flat}; & \mathcal{M}(g^{\flat}) & =\lambda^{2}\mathcal{M}(g).
\end{align*}

In Sections \ref{sec:modulation} and \ref{sec:Bootstrap}, we also
use the notation 
\[
f_{b}(y)\coloneqq f(y)e^{-ib\frac{|y|^{2}}{4}}.
\]
Note that $[f_{b}]^{\sharp}\neq[f^{\sharp}]_{b}$ in general. As we
will not use $[f^{\sharp}]_{b}$ in any places, we will use the notation
\[
f_{b}^{\sharp}\coloneqq[f_{b}]^{\sharp}.
\]

\subsubsection*{Introduction of the rescaled variables $(s,y)$}

Define the rescaled variables $(s,y)$ by 
\[
\frac{ds}{dt}=\frac{1}{\lambda^{2}(t)};\qquad y\coloneqq\frac{x}{\lambda(t)}.
\]
We will work both on the $(t,x)$ and $(s,y)$-variables. When we
apply $\sharp$ or $\flat$ operations on time-dependent functions,
we presume that the time variables $t$ or $s$ are changed in a suitable
fashion. For example, $\sharp$ acts on $f(s,y)$ as 
\[
f^{\sharp}(t,x)=\frac{1}{\lambda(t)}f(s(t),\frac{x}{\lambda(t)})e^{i\gamma(t)}.
\]
The following formulae are useful for converting evolutions on $(t,x)$-variables
to those on $(s,y)$-variables, and vice versa.
\begin{align*}
\partial_{t}f^{\sharp} & =\frac{1}{\lambda^{2}}\Big[\partial_{s}f-\frac{\lambda_{s}}{\lambda}\Lambda f+i\gamma_{s}f\Big]^{\sharp},\\
\partial_{s}g^{\flat} & =\lambda^{2}\Big[\partial_{t}g+\frac{\lambda_{t}}{\lambda}\Lambda g-i\gamma_{t}g\Big]^{\flat}.
\end{align*}
For example, if $u(t,x)$ solves \eqref{eq:CSS}, then we have 
\begin{align*}
i\partial_{t}u+\Delta_{m}u & =\mathcal{N}(u),\\
i\partial_{s}u^{\flat}+\Delta_{m}u^{\flat} & =\mathcal{N}(u^{\flat})+i\frac{\lambda_{s}}{\lambda}\Lambda u^{\flat}+\gamma_{s}u^{\flat}.
\end{align*}

\subsection{Local theory of \eqref{eq:CSS} under equivariance}

The discussion in this subsection holds for any $g\in\R$ and $m\in\Z$.
Under the equivariance, $L^{2}$-critical local well-posedness of
\eqref{eq:CSS-coulomb-phi} is shown in Liu-Smith \cite[Section 2]{LiuSmith2016}
using a standard application of Strichartz estimates. Note that without
equivariance ansatz, local theory of \eqref{eq:CSS-cov} was developed
by varioius authors \cite{BergeDeBouardSaut1995Nonlinearity,Huh2013Abstr.Appl.Anal,Lim2018JDE}
(under Coulomb gauge) and \cite{LiuSmithTataru2014IMRN} (under heat
gauge). The aim of this subsection is to record the Strichartz estimates
and Cauchy theory of \eqref{eq:CSS-coulomb-phi} under the equivariance.
Specializing to the self-dual case $g=1$, the local theory of \eqref{eq:CSS}
follows.

A pair $(q,r)$ is said to be \emph{admissible} if $2\leq q,r\leq\infty$,
$\frac{1}{q}+\frac{1}{r}=\frac{1}{2}$, and $(q,r)\neq(2,\infty)$.
We have the well-known Strichartz estimates:
\begin{lem}[Strichartz estimates]
\label{lem:Strichartz}Let $\phi:I\times\R^{2}\to\C$ solve 
\[
\left\{ \begin{aligned}i\partial_{t}\phi+\Delta\phi & =F,\\
\phi(0) & =\phi_{0},
\end{aligned}
\right.
\]
in the sense of Duhamel. For any admissible $(q,r)$ and $(\tilde q,\tilde r)$,
we have 
\begin{equation}
\|\phi\|_{L_{I}^{q}L_{x}^{r}}\lesssim_{q,r,\tilde q,\tilde r}\|\phi_{0}\|_{L^{2}}+\|F\|_{L_{I}^{\tilde q'}L_{x}^{\tilde r'}},\label{eq:Strichartz}
\end{equation}
where $\tilde q'$ and $\tilde r'$ are conjugate Lebesgue exponents
for $\tilde q$ and $\tilde r$, respectively.
\end{lem}

\begin{rem}[Endpoint Strichartz estimate]
If we further assume that $\phi$ is equivariant, then the Strichartz
estimates \eqref{eq:Strichartz} hold for the endpoint pairs $(q,r)=(2,\infty)$
or $(\tilde q,\tilde r)=(2,\infty)$. See the discussion \cite[Section 2]{LiuSmith2016},
which is based on \cite{Stefanov2001Proc.AMS,Tao2000CommPDE}.
\end{rem}

Until the end of this subsection, $\phi_{0}^{(n)}$ (with $n\in\N$)
and $\phi_{0}$ will always denote initial data. Corresponding maximal
lifespan evolution guaranteed by Propositions \ref{prop:L2-Cauchy}
and \ref{prop:Hs-Cauchy} with the initial data $\phi_{0}^{(n)}$
and $\phi_{0}$ will be denoted by $\phi^{(n)}:I_{n}\times\R^{2}\to\C$
and $\phi:I\times\R^{2}\to\C$, respectively.
\begin{prop}[{$L^{2}$-critical Cauchy theory \cite[Section 2]{LiuSmith2016}\footnote{Compared to the statement given in \cite{LiuSmith2016}, we give a
qualitative version of continuous dependence. We also include stability
of scattering solutions.}}]
\label{prop:L2-Cauchy}Let $m\in\Z$ and $t_{0}\in\R$.
\begin{enumerate}
\item (Local existence and uniqueness) For any $\phi_{0}\in L_{m}^{2}$,
there exists a unique maximal lifespan solution $\phi:I\times\R^{2}\to\C$
in $L_{I}^{\infty}L_{m}^{2}\cap L_{I,\mathrm{loc}}^{4}L_{x}^{4}$
to \eqref{eq:CSS-coulomb-phi} with the initial data $\phi(t_{0})=\phi_{0}$.
Here, $I$ is an open interval containing $t_{0}$. Moreover, $\phi$
indeed lies in $C_{I,\mathrm{loc}}L_{m}^{2}$ and $L_{I,\mathrm{loc}}^{q}L_{x}^{r}$
for all admissible pairs $(q,r)$.
\item (Scattering criterion) Solution $u$ scatters forward in time if and
only if $\|\phi\|_{L_{[t_{0},\sup I)}^{4}L_{x}^{4}}<\infty$. If this
is the case, then we have $\sup I=+\infty$. Analogous statement holds
for scattering backward in time.
\item (Small data scattering) There exists $\epsilon>0$ such that if $\|\phi_{0}\|_{L_{x}^{2}}\leq\epsilon$,
then the solution $\phi$ satisfies $\|\phi\|_{L_{t,x}^{4}}<\infty$.
In particular, $u$ is global and scatters in both time directions.
\item (Continuous dependence and stability of scattering solutions) Assume
that $\{\phi_{0}^{(n)}\}_{n\in\N}\subset L_{m}^{2}$ converges to
$\phi_{0}$ strongly in $L_{m}^{2}$. Then, $\liminf_{n\to\infty}I_{n}\supseteq I$.
Moreover, $\phi^{(n)}\to\phi$ in $C_{I,\mathrm{loc}}L_{m}^{2}$ and
$L_{I,\mathrm{loc}}^{q}L_{x}^{r}$ topology for all admissible pairs
$(q,r)$. If furthermore $\phi$ scatters forward in time in $L_{m}^{2}$,
then so does $\phi^{(n)}$ for all large $n$. Analogous statements
holds for scattering backward in time.
\end{enumerate}
\end{prop}

We will also need $H_{m}^{s}$-subcritical local theory for $s>0$.
As the proof is not explicitly given in \cite{LiuSmith2016}, we provide
a brief sketch of the proof in Appendix \ref{sec:Local-Theory}.
\begin{prop}[$H_{m}^{s}$-subcritical Cauchy theory]
\label{prop:Hs-Cauchy}Let $m\in\Z$, $t_{0}\in\R$, and $s>0$.
\begin{enumerate}
\item (Local existence, uniqueness, and persistence of regularity) For any
$\phi_{0}\in H_{m}^{s}$, there exists a unique maximal lifespan solution
$\phi:I\times\R^{2}\to\C$ in $L_{I,\mathrm{loc}}^{\infty}H_{m}^{s}\cap L_{I,\mathrm{loc}}^{4}B_{4,2}^{s}$
to \eqref{eq:CSS-coulomb-phi} with the initial data $\phi(t_{0})=\phi_{0}$.
Here, $I$ is an open interval containing $t_{0}$ and equal to that
given in Proposition \ref{prop:L2-Cauchy}. Moreover, $\phi$ indeed
lies in $C_{I,\mathrm{loc}}H_{m}^{s}$ and $L_{I,\mathrm{loc}}^{q}B_{r,2}^{s}$
for any admissible pairs $(q,r)$.
\item (Local existence in subcritical sense) For any $R>0$, there exists
$\delta>0$ such that we can guarantee $I\supseteq[t_{0}-\delta,t_{0}+\delta]$
whenever the initial data satisfies $\|\phi_{0}\|_{H_{m}^{s}}\leq R$.
\item (Continuous dependence and stability of scattering solutions) The
corresponding statement in Proposition \ref{prop:L2-Cauchy} holds
with $L_{m}^{2}$, $C_{I,\mathrm{loc}}L_{m}^{2}$, and $L_{I,\mathrm{loc}}^{q}L_{x}^{r}$
replaced by $H_{m}^{s}$, $C_{I,\mathrm{loc}}H_{m}^{s}$, and $L_{I,\mathrm{loc}}^{q}B_{r,2}^{s}$,
respectively.
\item (Finite-time blowup criterion) If $\sup I<+\infty$, then $\|\phi(t)\|_{H_{m}^{s}}\to\infty$
as $t\to\sup I$. One can replace $\sup I$ by $\inf I$ to obtain
the analogous statement.
\end{enumerate}
\end{prop}

\begin{rem}
The space $B_{r,2}^{s}$ is the Besov space, whose definition is given
in Appendix \ref{sec:Local-Theory}. Such a choice of the solution
norm is not important in this paper.
\end{rem}

\section{\label{sec:linearization}Linearization of \eqref{eq:CSS} under
equivariance}

In this section, we collect information on the linearization of \eqref{eq:CSS}.
For $g=1$, the self-dual case, it turns out that the linearized operator
at static solution $Q$ is also written as a self-dual form 
\begin{equation}
\mathcal{L}_{Q}=L_{Q}^{\ast}L_{Q}.\label{eq:self-dual}
\end{equation}
This is naturally expected in view of the Hamiltonian formulation.

Linearization under the equivariance symmetry was conducted by Lawrie-Oh-Shahshahani
\cite{LawrieOhShashahani_unpub} and they observed that linearized
operator has the form \eqref{eq:self-dual}. More precisely, they
linearized \eqref{eq:CSS} in the form 
\[
i\partial_{t}\epsilon-L_{Q}^{\ast}L_{Q}\epsilon=\text{(h.o.t.)},
\]
where $u=Q+\epsilon$ solves \eqref{eq:CSS}.

In the following, we start with Bogomol'nyi operator and its linearization,
and use them to write \eqref{eq:CSS} in a self-dual form. Then, we
can linearize \eqref{eq:CSS} and observe that the linearized operator
at $Q$ is of the self-dual form. We study the generalized null space
of $i\mathcal{L}_{Q}$ and coercivity of $\mathcal{L}_{Q}$.

\subsection{Linearization of Bogomol'nyi operator}

Because of the Hamiltonian and self-dual structures, it is natural
to linearize the Bogomol'nyi operator first. For later use, however,
we will linearize at arbitrary profile $w$. In order to avoid confusion,
let us write 
\begin{align*}
\D_{+}^{(u)} & \coloneqq\partial_{r}-\tfrac{1}{r}(m+A_{\theta}[u]),\\
\D_{+}^{(u)\ast} & \coloneqq-\partial_{r}-\tfrac{1}{r}(1+m+A_{\theta}[u]).
\end{align*}

We assume the decomposition 
\[
u=w+\epsilon.
\]
We first observe the operator identities 
\begin{align*}
\D_{+}^{(u)} & =\D_{+}^{(w)}+(B_{w}\epsilon)+\tfrac{1}{2}(B_{\epsilon}\epsilon),\\
\D_{+}^{(u)\ast} & =\D_{+}^{(w)\ast}+(B_{w}\epsilon)+\tfrac{1}{2}(B_{\epsilon}\epsilon),
\end{align*}
where $(B_{w}\epsilon)$ (and so is $\frac{1}{2}(B_{\epsilon}\epsilon)$)
is interpreted as multiplication operators by real functions, whose
explicit formulae are given by
\begin{align*}
B_{f}g & \coloneqq\frac{1}{r}\int_{0}^{r}\Re(\overline{f}g)r'dr'.
\end{align*}
We note that 
\[
A_{\theta}[u]=-\tfrac{1}{2}rB_{u}u.
\]
We thus have 
\begin{equation}
\D_{+}^{(u)}u=\D_{+}^{(w)}w+L_{w}\epsilon+N_{w}(\epsilon),\label{eq:lin-Bogom}
\end{equation}
where
\begin{align}
L_{w} & \coloneqq\D_{+}^{(w)}+wB_{w},\label{eq:def-L_w}\\
N_{w}(\epsilon) & \coloneqq\epsilon B_{w}\epsilon+\tfrac{1}{2}wB_{\epsilon}\epsilon+\tfrac{1}{2}\epsilon B_{\epsilon}\epsilon.\nonumber 
\end{align}
Here $N_{w}(\epsilon)$ contains the quadratic and cubic parts of
$\epsilon$ that are all nonlocal due to integration in $r$. We remark
that the operators $B_{w}$ and $L_{w}$ are only $\R$-linear, but
\emph{not} $\C$-linear. We also record the formal real adjoints:\footnote{By real adjoint, we view $L^{2}(\R^{2};\C)$ as $\R$-Hilbert space
equipped with the inner product $(u,v)_{r}=\int\Re(u\overline{v})$.}
\begin{align}
B_{w}^{\ast}f & =w\int_{r}^{\infty}(\Re f)dr',\nonumber \\
L_{w}^{\ast}f & =\D_{+}^{(w)\ast}f+B_{w}^{\ast}(\overline{w}f).\label{eq:def-L^ast_w}
\end{align}
Note that 
\[
A_{0}[u]u=-B_{u}^{\ast}[\tfrac{m}{r}|u|^{2}-\tfrac{1}{2}|u|^{2}B_{u}u].
\]

In particular when $w=Q$, we use $\D_{+}^{(Q)}Q=0$ to have 
\begin{align*}
E[Q+\epsilon] & =\frac{1}{2}\int|\D_{+}^{(Q+\epsilon)}(Q+\epsilon)|^{2}=\frac{1}{2}\int|L_{Q}\epsilon+N_{Q}(\epsilon)|^{2}.
\end{align*}
Therefore, we can extract the quadratic part of the energy 
\begin{equation}
E[Q+\epsilon]=\frac{1}{2}\int|L_{Q}\epsilon|^{2}+\text{(h.o.t.)}.\label{eq:energy-Q-expn}
\end{equation}

\subsection{Linearization of \eqref{eq:CSS}}

In the self-dual case, it is possible to write \eqref{eq:CSS} using
the (linearized) Bogomol'nyi operator. This is due to the Hamiltonian
formulation and a special form of energy \eqref{eq:energy-Bogomolnyi}.
\begin{prop}[\eqref{eq:CSS} in the self-dual form]
\label{prop:prop3.1}\eqref{eq:CSS} is equivalent to 
\begin{equation}
i\partial_{t}u=L_{u}^{\ast}\D_{+}^{(u)}u.\label{eq:CSS-L*D-form}
\end{equation}
\end{prop}

\begin{rem}
As mentioned in Section \ref{subsec:Coulomb-gauge-and}, \eqref{eq:CSS-cov}
under the Coulomb gauge enjoys a Hamiltonian formulation. The equation
\eqref{eq:CSS} is naturally written as a Hamiltonian equation 
\[
\partial_{t}u=-i\frac{\delta E}{\delta u}=-i\frac{\delta}{\delta u}\Big(\frac{1}{2}\int|\D_{+}^{(u)}u|^{2}\Big).
\]
Using the linearization \eqref{eq:lin-Bogom}, one can formally verify
\eqref{eq:CSS-L*D-form}.
\end{rem}

\begin{proof}[Proof of Proposition \ref{prop:prop3.1}]
Recall that 
\begin{align*}
\D_{+}^{(u)}u & =(\partial_{r}-\tfrac{m}{r})u+\tfrac{1}{2}(B_{u}u)u,\\
L_{u}^{\ast} & =(-\partial_{r}-\tfrac{m+1}{r})+\tfrac{1}{2}(B_{u}u)+B_{u}^{\ast}[\overline{u}\,\cdot\,].
\end{align*}
Thus 
\begin{align*}
L_{u}^{\ast}\D_{+}^{(u)}u & =(-\partial_{r}-\tfrac{m+1}{r})(\partial_{r}-\tfrac{m}{r})u\\
 & \quad+(-\partial_{r}-\tfrac{m+1}{r})[\tfrac{1}{2}(B_{u}u)u]+\tfrac{1}{2}(B_{u}u)(\partial_{r}-\tfrac{m}{r})u+B_{u}^{\ast}[\overline{u}(\partial_{r}-\tfrac{m}{r})u]\\
 & \quad+\tfrac{1}{4}[B_{u}u]^{2}u+\tfrac{1}{2}B_{u}^{\ast}[|u|^{2}B_{u}u].
\end{align*}
We can rewrite the terms in the second line of the above display as
\begin{align*}
 & (-\partial_{r}-\tfrac{m+1}{r})[\tfrac{1}{2}(B_{u}u)u]+\tfrac{1}{2}(B_{u}u)(\partial_{r}-\tfrac{m}{r})u\\
 & =\tfrac{1}{2r}(-\partial_{r}-\tfrac{m}{r})[r(B_{u}u)u]+\tfrac{1}{2r}(rB_{u}u)(\partial_{r}-\tfrac{m}{r})u\\
 & =\tfrac{1}{2r}[-\partial_{r},r(B_{u}u)]u-\tfrac{m}{r}(B_{u}u)u\\
 & =-\tfrac{1}{2}|u|^{2}u-\tfrac{m}{r}(B_{u}u)u,
\end{align*}
and 
\[
B_{u}^{\ast}[\overline{u}(\partial_{r}-\tfrac{m}{r})u]=-\tfrac{1}{2}|u|^{2}u-B_{u}^{\ast}[\tfrac{m}{r}|u|^{2}].
\]
Therefore, we recover the nonlinearity expression as in \eqref{eq:CSS},
but expressed in terms of $B_{u},B_{u}^{\ast}$: 
\begin{align*}
L_{u}^{\ast}\D_{+}^{(u)}u & =(-\partial_{r}-\tfrac{m+1}{r})(\partial_{r}-\tfrac{m}{r})u\\
 & \quad-|u|^{2}u-\tfrac{m}{r}(B_{u}u)u-B_{u}^{\ast}[\tfrac{m}{r}|u|^{2}]\\
 & \quad+\tfrac{1}{4}(B_{u}u)^{2}u+\tfrac{1}{2}B_{u}^{\ast}[|u|^{2}B_{u}u]\\
 & =-\Delta_{m}u-|u|^{2}u+\frac{2m}{r^{2}}A_{\theta}[u]u+\frac{A_{\theta}^{2}[u]}{r^{2}}u+A_{0}[u]u.
\end{align*}
This completes the proof.
\end{proof}
Using \eqref{eq:CSS-L*D-form}, we can easily linearize \eqref{eq:CSS}.
To linearize $L_{u}^{\ast}\D_{+}^{(u)}u$ at $w$, we write 
\begin{align*}
u & =w+\epsilon,\\
\D_{+}^{(u)}u & =\D_{+}^{(w)}w+L_{w}\epsilon+N_{w}(\epsilon),\\
L_{u}^{\ast} & =L_{w}^{\ast}+[(B_{w}\epsilon)+B_{w}^{\ast}[\overline{\epsilon}\,\cdot\,]+B_{\epsilon}^{\ast}[\overline{w}\,\cdot\,]]+[\tfrac{1}{2}(B_{\epsilon}\epsilon)+B_{\epsilon}^{\ast}[\overline{\epsilon}\,\cdot\,]].
\end{align*}
Then the linear part $\mathcal{L}_{w}\epsilon$ of $L_{u}^{\ast}\D_{+}^{(u)}u$
is given by 
\begin{equation}
\mathcal{L}_{w}\epsilon=L_{w}^{\ast}L_{w}\epsilon+[(B_{w}\epsilon)+B_{w}^{\ast}[\overline{\epsilon}\,\cdot\,]+B_{\epsilon}^{\ast}[\overline{w}\,\cdot\,]]\D_{+}^{(w)}w.\label{eq:def-linearized-op}
\end{equation}
In particular, when $w=Q$, we observe the \emph{self-duality} using
$\D_{+}^{(Q)}Q=0$: 
\begin{equation}
\mathcal{L}_{Q}=L_{Q}^{\ast}L_{Q}.\label{eq:L_Q-self-dual}
\end{equation}
With this formulation, we can recover a full linearized equation
\begin{align*}
i\partial_{t}\epsilon-\mathcal{L}_{Q}\epsilon & =L_{Q}^{\ast}N_{Q}(\epsilon)+[(B_{Q}\epsilon)+B_{Q}^{\ast}[\overline{\epsilon}\,\cdot\,]+B_{\epsilon}^{\ast}[Q\,\cdot\,]][L_{Q}\epsilon+N_{Q}(\epsilon)]\\
 & \quad+[\tfrac{1}{2}(B_{\epsilon}\epsilon)+B_{\epsilon}^{\ast}[\overline{\epsilon}\,\cdot\,]][L_{Q}\epsilon+N_{Q}(\epsilon)],
\end{align*}
as is done in Lawrie-Oh-Shahshahani \cite{LawrieOhShashahani_unpub}.
They also observed that there is no derivative falling on $\epsilon$
in the nonlinearity. This is naturally expected as the nonlinearity
of \eqref{eq:CSS} contains no derivatives.

\subsection{\label{subsec:solvability}Algebraic relations, solvability, and
coercivity of $\mathcal{L}_{Q}$}

The linearization of \eqref{eq:CSS} at $Q$ is given by 
\[
\partial_{t}\epsilon=-i\mathcal{L}_{Q}\epsilon.
\]
Our main goals are to compute the generalized nullspace of $i\mathcal{L}_{Q}$
and to prove the coercivity of $\mathcal{L}_{Q}$ under a transversality
condition. Such spectral information stems from the symmetries of
\eqref{eq:CSS}, i.e. phase, scaling, and pseudoconformal symmetries.
It turns out that the generalized null space for the \emph{self-dual}
\eqref{eq:CSS} is different from that of \eqref{eq:NLS} studied
in \cite{Weinstein1985SIAM}.

Differentiating symmetries of \eqref{eq:CSS} at the static solution
$Q$, we obtain explicit algebraic identities satisfied by $\mathcal{L}_{Q}$.
Assume we have a continuous family of solutions $u^{(a)}(t,r)$ to
\eqref{eq:CSS} with $u^{(0)}(t,r)=Q(r)$. Differentiating 
\[
L_{u^{(a)}}^{\ast}\D_{+}^{(u^{(a)})}u^{(a)}=i\partial_{t}u^{(a)}
\]
in $a$ at $a=0$, we obtain 
\[
\mathcal{L}_{Q}(\partial_{a}u^{(a)})|_{a=0}=i\partial_{t}(\partial_{a}u^{(a)})|_{a=0}.
\]
If we substitute phase, scaling, and pseudoconformal symmetry 
\[
u^{(a)}(t,r)=\begin{cases}
e^{ia}Q(r) & \text{(phase)}\\
aQ(ar) & \text{(scaling)}\\
\frac{1}{1+at}e^{ia\frac{r^{2}}{4(1+at)}}Q(\frac{r}{1+at}) & \text{(pseudoconformal)}
\end{cases}
\]
then we obtain 
\begin{align*}
\mathcal{L}_{Q}(iQ) & =0, &  & \text{(phase)}\\
\mathcal{L}_{Q}(\Lambda Q) & =0, &  & \text{(scaling)}\\
\mathcal{L}_{Q}(ir^{2}Q) & =-4i\Lambda Q. &  & \text{(pseudoconformal)}
\end{align*}

Here is a rough discussion on spectral properties of $\mathcal{L}_{Q}$.
Rigorous results are proved below. Noticing that $\mathcal{L}_{Q}f=0$
if and only if $L_{Q}f=0$, we have $L_{Q}(iQ)=L_{Q}(\Lambda Q)=0$.
As $L_{Q}$ is a first-order (nonlocal) differential operator, one
can indeed see that $\ker L_{Q}=\mathrm{span}_{\R}\{iQ,\Lambda Q\}$.
Moreover, one can invert the operator $\mathcal{L}_{Q}$ on the orthogonal
complement of $\ker L_{Q}$. As a consequence, $\mathcal{L}_{Q}$
is coercive on a subspace having trivial intersection with $\ker L_{Q}$.

As an application, we can characterize the generalized null space
of $i\mathcal{L}_{Q}$. We first show that the kernel of $L_{Q}$
(and hence of $\mathcal{L}_{Q}$) is $\mathrm{span}_{\R}\{\Lambda Q,iQ\}$.
\begin{lem}[Kernel of $L_{Q}$]
\label{lem:kernel-LQ}If $f:\R^{2}\to\C$ is a smooth $m$-equivariant
function such that $L_{Q}f=0$,\footnote{Actually $L_{Q}$ acts on functions from $(0,\infty)$ to $\C$. By
an abuse of notation, we denote by $f(r)$ the radial component of
the $m$-equivariant function $f(x)$.} then $f$ is a $\R$-linear combination of $\Lambda Q$ and $iQ$.
\end{lem}

\begin{proof}
As $f$ is smooth $m$-equivariant, Lemma \ref{lem:smooth-equiv-criterion}
says that $f^{(k)}(0)=0$ for all $k\leq m-1$. As $\varphi\in\{\Lambda Q,iQ\}$
satisfies $\varphi^{(k)}(0)=0$ for $k\leq m-1$ and $\varphi^{(m)}(0)\neq0$,
we may subtract an $\R$-linear combination of $\Lambda Q$ and $iQ$
from $f$ to assume $f^{(m)}(0)=0$. It now suffices to show that
$f=0$. Since $f$ is smooth such that $f^{(k)}(0)=0$ for all $k\leq m$,
by Lemma \ref{lem:remov-sing}, the function $r^{-m}f(r)$ is smooth,
$r^{-m}f(r)|_{r=0}=0$, and 
\[
\partial_{r}(r^{-m}f)=r^{-m}(\partial_{r}f-\tfrac{m}{r}f)=-r^{-m}(QB_{Q}+\tfrac{1}{2}(B_{Q}Q))f.
\]
We integrate from the origin to get 
\begin{align*}
\frac{|f(r)|}{r^{m}} & =\Big|\int_{0}^{r}\frac{Q}{(r')^{m+1}}\Big(\int_{0}^{r'}Q\Re fdr''\Big)dr'+\frac{1}{2}\int_{0}^{r}(B_{Q}Q)\frac{f(r')}{(r')^{m}}dr'\Big|\\
 & \leq\int_{0}^{r}\bigg\{(r')^{m+1}Q\Big(\int_{r'}^{r}\frac{Q}{(r'')^{m+1}}dr''\Big)+\frac{1}{2}B_{Q}Q\bigg\}\frac{|f(r')|}{(r')^{m}}dr'
\end{align*}
As the factor in the curly bracket is bounded, Gronwall's inequality
concludes that $f=0$.
\end{proof}
To find the generalized null space of $i\mathcal{L}_{Q}$, it is natural
to ask whether we can solve the equations $i\mathcal{L}_{Q}\psi=\Lambda Q$
and $i\mathcal{L}_{Q}\varphi=iQ$. More generally, we may consider
the equation $i\mathcal{L}_{Q}f=h$, where $h$ is given. We will
check solvability in the following sense. If there is a nice function\footnote{We will not be specific here, a nice function is required to make
sense of \eqref{eq:LQ-solvability}. One may require smoothness, integrability,
or boundary conditions (at $r=0$ or $r=\infty$) for such functions.} $f$ with $i\mathcal{L}_{Q}f=h$, then 
\begin{equation}
(h,ig)_{r}=(i\mathcal{L}_{Q}f,ig)_{r}=(f,\mathcal{L}_{Q}g)_{r}\label{eq:LQ-solvability}
\end{equation}
must vanish for any $g\in\{\Lambda Q,iQ\}$. Thus if $(h,i\Lambda Q)_{r}$
or $(h,Q)_{r}$ does not vanish, then we say that $i\mathcal{L}_{Q}f=h$
is \emph{not solvable}. On the other hand, if $(h,i\Lambda Q)_{r}=(h,Q)_{r}=0$,
then we are able to find $f$ (thus \emph{solvable}) for sufficiently
nice $h$ as in the proof of Lemma \ref{lem:generalized-mode-rho}
below.

It is natural to expect that we can solve $i\mathcal{L}_{Q}\psi=\Lambda Q$
and $i\mathcal{L}_{Q}\psi=iQ$ due to $(\Lambda Q,Q)_{r}=0$. For
the former, we already know a solution, which is $i\frac{1}{4}r^{2}Q$
given by the pseudoconformal symmetry. But there is no $\psi$ satisfying
$i\mathcal{L}_{Q}\psi=i\frac{1}{4}r^{2}Q$ because $(r^{2}Q,\Lambda Q)_{r}=-\|rQ\|_{L^{2}}^{2}\neq0$
(when $m\geq1$). For the latter, we can in fact construct a solution
$\rho$ to $i\mathcal{L}_{Q}\rho=iQ$ in Lemma \ref{lem:generalized-mode-rho}
below. But again there is no $\psi$ satisfying $i\mathcal{L}_{Q}\psi=\rho$
because $(-i\rho,iQ)_{r}=-\|L_{Q}\rho\|_{L^{2}}^{2}\neq0$ (when $m\geq1$).
Therefore, we have found the basis of the generalized null space.
In summary,
\begin{prop}[The generalized null space of $i\mathcal{L}_{Q}$]
\label{prop:gen.null.space}The generalized null space\footnote{In fact, when $m=0$ or $1$, $ir^{2}Q$ and $\rho$ do not belong
to $L^{2}$ . We will still call them generalized eigenmodes due to
the algebraic identities in this proposition.} of $i\mathcal{L}_{Q}$ has $\R$-basis $\{iQ,\Lambda Q,ir^{2}Q,\rho\}$
with relations
\begin{align*}
i\mathcal{L}_{Q}\rho & =iQ; & i\mathcal{L}_{Q}ir^{2}Q & =4\Lambda Q;\\
i\mathcal{L}_{Q}iQ & =0; & i\mathcal{L}_{Q}\Lambda Q & =0;
\end{align*}
where $\rho$ is given in Lemma \ref{lem:generalized-mode-rho} below.
\end{prop}

\begin{rem}
\label{rem:remark3.5}It is instructive to compare with \eqref{eq:NLS}.
Recall the ground state $R$ \eqref{eq:NLS-ground-state} and associated
linearized operator $\mathcal{L}_{\mathrm{NLS}}$ \eqref{eq:linearized-op-NLS}
\[
\mathcal{L}_{\mathrm{NLS}}f=-\Delta f+f-2R^{2}f-R^{2}\overline{f}.
\]
It is known from \cite{Weinstein1985SIAM} that under radial symmetry,
the generalized null space of $i\mathcal{L}_{\mathrm{NLS}}$ is characterized
as
\begin{align*}
i\mathcal{L}_{\mathrm{NLS}}\rho_{\mathrm{NLS}} & =ir^{2}R,\\
i\mathcal{L}_{\mathrm{NLS}}i|y|^{2}R & =4\Lambda R,\\
i\mathcal{L}_{\mathrm{NLS}}\Lambda R & =-2iR,\\
i\mathcal{L}_{\mathrm{NLS}}iR & =0.
\end{align*}
Notice that $\Lambda R$ does not belong to the kernel of $i\mathcal{L}_{\mathrm{NLS}}$.
This is because $e^{it}R(x)$ is not a static solution to \eqref{eq:NLS}.
Moreover, $\mathcal{L}_{\mathrm{NLS}}$ restricted on real-valued
radial functions is invertible. Thus there exists real-valued $\rho_{\mathrm{NLS}}$
solving $\mathcal{L}_{\mathrm{NLS}}\rho_{\mathrm{NLS}}=r^{2}R$. This
is in strong contrast to the case of \eqref{eq:CSS}.
\end{rem}

\begin{lem}[The generalized eigenmode $\rho$]
\label{lem:generalized-mode-rho}There exists a smooth real-valued
solution $\rho:(0,\infty)\to\R$ to 
\[
L_{Q}^{\ast}L_{Q}\rho=Q
\]
satisfying 
\[
|\rho(r)|\lesssim r^{2}Q\sim\begin{cases}
r^{m+2} & \text{if }r\leq1,\\
r^{-m} & \text{if }r>1.
\end{cases}
\]
Moreover, if $m\geq1$, we have the nondegeneracy\footnote{If $m=0$, then $\rho$ has tail $\sim1$ as $r\to\infty$, so $(\rho,Q)_{r}$
is not defined.} 
\[
(\rho,Q)_{r}=\|L_{Q}\rho\|_{L^{2}}^{2}\neq0.
\]
\end{lem}

\begin{proof}
In the proof, we present how we can solve $L_{Q}^{\ast}L_{Q}\rho=\varphi$
whenever $\varphi$ is a generic real-valued function with $(\varphi,\Lambda Q)_{r}=0$.
For sake of simplicity, however, we only present the special case
when $\varphi=Q$.

At first, we construct a real-valued function $\psi$ satisfying 
\begin{equation}
L_{Q}^{\ast}\psi=Q;\qquad|\psi(r)|\lesssim rQ\sim\begin{cases}
r^{m+1} & \text{if }r\leq1,\\
r^{-m-1} & \text{if }r>1.
\end{cases}\label{eq:inversion-Q-to-psi}
\end{equation}
Rewrite $L_{Q}^{\ast}\psi=Q$ as 
\[
-\partial_{r}\psi-\frac{1}{r}(1+m+A_{\theta}[Q])\psi+Q\int_{r}^{\infty}Q\psi dr'=Q.
\]
Using the relation $\partial_{r}Q-\frac{1}{r}(m+A_{\theta}[Q])Q=\D_{+}^{(Q)}Q=0$,
we renormalize the equation using $\tilde{\psi}\coloneqq(rQ)\psi$
as 
\[
-\partial_{r}\tilde{\psi}+rQ^{2}\int_{r}^{\infty}\tilde{\psi}\frac{dr'}{r'}=rQ^{2}.
\]
We integrate from the spatial infinity with $\tilde{\psi}(\infty)=0$
to get the integral equation 
\begin{equation}
\tilde{\psi}(r)+\int_{r}^{\infty}Q^{2}\Big(\int_{r'}^{\infty}\tilde{\psi}\frac{dr''}{r''}\Big)r'dr'=\int_{r}^{\infty}Q^{2}r'dr'.\label{eq:iteration}
\end{equation}
As $Q^{2}r\lesssim r^{-2m-3}$ for $r\gg1$, we seek for $\tilde{\psi}(r)$
having decay $r^{-2m-2}$. Note that 
\begin{align*}
\int_{r}^{\infty}Q^{2}\Big(\int_{r'}^{\infty}r^{-2m-2}\frac{dr''}{r''}\Big)r'dr' & \leq\frac{1}{2m+2}\int_{r}^{\infty}Q^{2}(r')^{-2m-1}dr'\\
 & \leq\frac{\|rQ\|_{L^{\infty}}^{2}}{(2m+2)^{2}}r^{-2m-2}=\frac{1}{2}r^{-2m-2}.
\end{align*}
As the coefficient of $\frac{1}{2}r^{-2m-2}$ is less than $1$, a
standard contraction principle allows us to construct $\tilde{\psi}$
satisfying 
\[
\tilde{\psi}(r)\lesssim r^{-2m-2},\qquad\forall r>0.
\]
In other words, we have 
\[
\psi(r)\lesssim(rQ)^{-1}r^{-2m-2},\qquad\forall r>0,
\]
which shows $r^{-m-1}$ decay at infinity.

In order to show the decay of $\psi$ at the origin, we use a recursive
argument. We crucially exploit the facts $(Q,\Lambda Q)_{r}=0$ and
$L_{Q}\Lambda Q=0$. For any $0<\delta\leq1$, we observe that 
\begin{align*}
|(L_{Q}^{\ast}\psi,\mathbf{1}_{r\geq\delta}\Lambda Q)|=|(Q,\mathbf{1}_{r\geq\delta}\Lambda Q)| & =|(Q,\mathbf{1}_{r<\delta}\Lambda Q)|\lesssim\delta^{2m+2}.
\end{align*}
On the other hand, we have 
\begin{align*}
 & (L_{Q}^{\ast}\psi,\mathbf{1}_{r\geq\delta}\Lambda Q)\\
 & =\psi(\delta)\Lambda Q(\delta)\delta+\int_{\delta}^{\infty}\psi\Big((\partial_{r}-\frac{m+A_{\theta}[Q]}{r})\Lambda Q+\frac{Q}{r}\int_{\delta}^{r}Q\Lambda Qr'dr'\Big)rdr\\
 & =\psi(\delta)\Lambda Q(\delta)\delta-\int_{\delta}^{\infty}\psi Q\Big(\int_{0}^{\delta}Q\Lambda Qr'dr'\Big)dr.
\end{align*}
Note that 
\begin{align*}
\Lambda Q(\delta)\delta & \sim\delta^{m+1},\\
\Big|\int_{\delta}^{\infty}\psi Q\Big(\int_{0}^{\delta}Q\Lambda Qr'dr'\Big)dr\Big| & \lesssim\delta^{2m+2}\int_{\delta}^{\infty}|\tilde{\psi}(r)|\frac{dr}{r}.
\end{align*}
Recalling $\tilde{\psi}=rQ\psi$, we arrange the above estimate to
obtain the recursive estimate 
\begin{equation}
|\psi(\delta)|\lesssim\delta^{m+1}\Big(1+\int_{\delta}^{\infty}rQ|\psi|\frac{dr}{r}\Big).\label{eq:psi-recursive-est}
\end{equation}
Starting from our rough bound $\psi(r)\lesssim(rQ)^{-1}r^{-2m-2}$,
iterating \eqref{eq:psi-recursive-est} three times yields $|\psi(\delta)|\lesssim\delta^{m+1}$,
which is the desired behavior near the origin. This completes the
proof of \eqref{eq:inversion-Q-to-psi}.

Next, we construct a real-valued function $\rho$ satisfying 
\begin{equation}
L_{Q}\rho=\psi,\qquad|\rho(r)|\lesssim r^{2}Q\sim\begin{cases}
r^{m+2} & \text{if }r\leq1,\\
r^{-m} & \text{if }r\geq1.
\end{cases}\label{eq:inversion-L_Q}
\end{equation}
We solve this by integrating from the origin. Rewrite $L_{Q}\rho=\psi$
as 
\[
\partial_{r}\rho-\frac{1}{r}(m+A_{\theta}[Q])\rho+\frac{Q}{r}\int_{0}^{r}Q\rho r'dr'=\psi.
\]
We conjugate $Q$ to $\rho$, i.e. we consider $\rho=Q\tilde{\rho}$
and rewrite the equation as 
\[
\partial_{r}\tilde{\rho}+\frac{1}{r}\int_{0}^{r}Q^{2}\tilde{\rho}r'dr'=Q^{-1}\psi.
\]
We integrate from the origin with $\tilde{\rho}(0)=0$ to write 
\begin{equation}
\tilde{\rho}(r)+\int_{0}^{r}\Big(\int_{0}^{r'}Q^{2}\tilde{\rho}r''dr''\Big)\frac{dr'}{r'}=\int_{0}^{r}Q^{-1}\psi dr'.\label{eq:iteration-2}
\end{equation}
As $Q^{-1}\psi\lesssim r$ near the origin, we expect $\tilde{\rho}(r)\lesssim r^{2}$
near the origin. Since 
\[
\int_{0}^{r}\Big(\int_{0}^{r'}Q^{2}(r'')^{2}r''dr''\Big)\frac{dr'}{r'}\lesssim r^{2m+2}\ll r^{2},\qquad\forall r\ll1,
\]
 a standard contraction principle allows us to construct $\tilde{\rho}$
on $[0,r_{0})$ for some $r_{0}\ll1$ with expected pointwise bound
$\tilde{\rho}(r)\lesssim r^{2}$.

To construct $\tilde{\rho}$ beyond $[0,r_{0})$, choose an increasing
sequence $\{r_{n}\}_{n\geq0}\subset\R$ such that $r_{n}\to\infty$,
$r_{n+1}\leq2r_{n}$, and $\int_{r_{n}}^{r_{n+1}}Q^{2}r''dr''\ll1$.
We will construct $\tilde{\rho}$ on $[r_{n},r_{n+1})$ inductively
on $n$. Suppose we have constructed $\tilde{\rho}$ on the interval
$[0,r_{n})$ for some $n\geq0$. We then rewrite the integral equation
of $\tilde{\rho}$ as 
\[
\tilde{\rho}(r)+\int_{r_{n}}^{r}\Big(\int_{r_{n}}^{r'}Q^{2}\tilde{\rho}r''dr''\Big)\frac{dr'}{r'}=\int_{0}^{r}\Big(Q^{-1}\psi-\frac{1}{r'}\int_{0}^{\min(r_{n},r')}Q^{2}\tilde{\rho}r''dr''\Big)dr'
\]
for $r\in[r_{n},r_{n+1})$. Since (by Fubini and $r_{n+1}\leq2r_{n}$)
\[
\int_{r_{n}}^{r}\Big(\int_{r_{n}}^{r'}Q^{2}r''dr''\Big)\frac{dr'}{r'}\leq2\int_{r_{n}}^{r_{n+1}}Q^{2}r''dr''\ll1,
\]
we can use contraction principle to construct $\tilde{\rho}(r)$ on
$[r_{n},r_{n+1})$.

From the above recursive construction, we have defined $\tilde{\rho}$
globally on $[0,+\infty)$ but the growth of $\tilde{\rho}$ can be
wild. In order to obtain the desired growth $\tilde{\rho}(r)\lesssim r^{2}$,
choose a positive increasing sequence $\{C_{n}\}_{n\in\N}\subset\R$
such that $|\tilde{\rho}(r)|\leq C_{n}r^{2}$ for all $r\in[0,r_{n}]$
and $n\in\N$. We claim that there exist large $N\in\N$ and $C>C_{N}$
such that $|\tilde{\rho}(r)|\leq Cr^{2}$ for all $r\geq r_{N}$.
To see this, substitute the bootstrap hypothesis $|\tilde{\rho}(r)|\leq Cr^{2}$
for $r\geq r_{N}$ into \eqref{eq:iteration-2} to get 
\[
|\tilde{\rho}(r)|\lesssim r^{2}+C\log^{2}r
\]
for $r\geq r_{N}$. If we choose $N$ large such that $\log^{2}r\ll r^{2}$
for $r\geq r_{N}$ and $C>C_{N}$ large, we get a strong conclusion
$|\tilde{\rho}(r)|\leq\frac{1}{2}Cr^{2}$. Therefore, we obtain $\tilde{\rho}(r)\leq Cr^{2}$
by the continuity argument. This translates to the desired bound 
\[
\rho(r)\lesssim r^{2}Q.
\]
This completes the proof of \eqref{eq:inversion-L_Q}.

In case of $m\geq1$, note that $Q\rho$ is integrable. Thus we get
the nondegeneracy
\[
(Q,\rho)_{r}=(L_{Q}^{\ast}L_{Q}\rho,\rho)_{r}=\|L_{Q}\rho\|_{L^{2}}^{2}\neq0.
\]
This completes the proof.
\end{proof}
\begin{rem}[Explicit construction of $\psi$]
\label{rem:explicit-form-psi}One can find the explicit formula of
$\psi$ to \eqref{eq:inversion-Q-to-psi} as 
\begin{equation}
L_{Q}\rho=\psi=\frac{1}{2(m+1)}rQ=\sqrt{2}\frac{r^{m+1}}{1+r^{2m+2}}.\label{eq:explicit-psi}
\end{equation}
One can easily verify that this $\psi$ satisfies $L_{Q}^{\ast}\psi=Q$,
though we found this formula by writing an ansatz (recall $\tilde{\psi}=rQ\psi$)
\[
\tilde{\psi}=\sum_{k=1}^{\infty}\frac{c_{k}}{(1+r^{2m+2})^{k}}
\]
with $c_{1}=4(m+1)$. Note that this ansatz is closed under the iteration
scheme \eqref{eq:iteration}; we have 
\begin{align*}
 & \int_{r}^{\infty}Q^{2}\Big(\int_{r'}^{\infty}\frac{1}{(1+(r'')^{2m+2})^{k}}\frac{dr''}{r''}\Big)r'dr'\\
 & =\int_{r}^{\infty}Q^{2}\Big(\sum_{\ell\geq k}\frac{1}{(2m+2)\ell}\cdot\frac{1}{(1+(r')^{2m+2})^{\ell}}\Big)r'dr'\\
 & =\sum_{\ell\geq k}\frac{2}{\ell(\ell+1)}\cdot\frac{1}{(1+r^{2m+2})^{\ell+1}}.
\end{align*}
Starting from $c_{1}=4(m+1)$, the equation \eqref{eq:iteration}
determines all the other coefficients $c_{2}=-4(m+1)$ and $c_{3}=c_{4}=\cdots=0$.
This shows 
\[
\tilde{\psi}=4(m+1)\frac{r^{2m+2}}{(1+r^{2m+2})^{2}}
\]
and hence \eqref{eq:explicit-psi}.
\end{rem}

\begin{rem}[Asymptotics of $\rho$]
For $m\geq1$, we claim that 
\begin{equation}
\rho(r)=\Big(\frac{1}{4(m+1)}r^{2}-\frac{\|\psi\|_{L^{2}}^{2}}{2\pi}\log r+O(1)\Big)Q(r)\label{eq:rho-asymp}
\end{equation}
as $r\to\infty$. Indeed, the relation $L_{Q}^{\ast}L_{Q}\rho=L_{Q}^{\ast}\psi=Q$
and estimate $\tilde{\rho}\lesssim r^{2}$ (recall $\tilde{\rho}=Q^{-1}\rho$)
say that 
\[
\int_{0}^{r'}Q^{2}\tilde{\rho}r''dr''=\int_{0}^{\infty}Q\rho r''dr''-\int_{r'}^{\infty}Q^{2}\tilde{\rho}r''dr''=\frac{\|\psi\|_{L^{2}}^{2}}{2\pi}-O((r')^{-2m})
\]
as $r\to\infty$. Substituting this and \eqref{eq:explicit-psi} into
the integral equation \eqref{eq:iteration-2} shows 
\[
\tilde{\rho}(r)=\frac{1}{4(m+1)}r^{2}-\frac{\|\psi\|_{L^{2}}^{2}}{2\pi}\log r+O(1)
\]
as $r\to\infty$. This completes the proof of \eqref{eq:rho-asymp}.
\end{rem}

We conclude this section with the coercivity estimate. Thanks to the
factorization $\mathcal{L}_{Q}=L_{Q}^{\ast}L_{Q}$, coercivity of
$\mathcal{L}_{Q}$ easily follows from the characterization of the
kernel of $L_{Q}$.
\begin{lem}[Coercivity of $L_{Q}$]
\label{lem:coercivity}Let $m\geq1$. We have 
\begin{align}
\|L_{Q}f\|_{L^{2}} & \lesssim\|f\|_{\dot{H}_{m}^{1}}, &  & \forall f\in\dot{H}_{m}^{1}\label{eq:bdd-LQ}\\
\|L_{Q}f\|_{L^{2}}^{2} & \geq c_{\psi_{1},\psi_{2}}\|f\|_{\dot{H}_{m}^{1}}^{2}, &  & \forall f\in\dot{H}_{m}^{1},\ f\perp\{\psi_{1},\psi_{2}\},\label{eq:coercivity-LQ}
\end{align}
where $\{\psi_{1},\psi_{2}\}\subset(\dot{H}_{m}^{1})^{\ast}$ are
such that the following matrix has nonzero determinant: 
\[
\begin{bmatrix}(iQ,\psi_{1})_{r} & (\Lambda Q,\psi_{1})_{r}\\
(iQ,\psi_{2})_{r} & (\Lambda Q,\psi_{2})_{r}
\end{bmatrix}
\]
In case of $m=0$, \eqref{eq:bdd-LQ} and \eqref{eq:coercivity-LQ}
hold true if one replaces $\dot{H}_{m}^{1}$ by 
\[
\|f\|_{\dot{\mathcal{H}}_{0}}^{2}\coloneqq\|\partial_{r}f\|_{L^{2}}^{2}+\|(1+r)^{-1}f\|_{L^{2}}^{2}.
\]
\end{lem}

\begin{proof}
For simplicity of exposition, we hope to introduce a unified notation
\[
\|f\|_{\dot{\mathcal{H}}_{m}}^{2}\coloneqq\begin{cases}
\|\partial_{r}f\|_{L^{2}}^{2}+\|(1+r)^{-1}f\|_{L^{2}}^{2} & \text{if }m=0,\\
\|f\|_{\dot{H}_{m}^{1}}^{2} & \text{if }m\geq1,
\end{cases}
\]

We first show \eqref{eq:bdd-LQ}. Since 
\[
\|L_{Q}f\|_{L^{2}}\lesssim\|f\|_{\dot{\mathcal{H}}_{m}}+\|QB_{Q}f\|_{L^{2}},
\]
it suffices to control the term $QB_{Q}f$. For this purpose, we express
\begin{align*}
[QB_{Q}f](r) & =\int_{0}^{\infty}K(r,r')\frac{\Re f(r')}{1+r'}r'dr',\\
K(r,r') & \coloneqq r^{-1}Q(r)Q(r')(1+r')\mathbf{1}_{r'\leq r}.
\end{align*}
Since 
\begin{align*}
\sup_{r>0}\int_{0}^{\infty}|K(r,r')|r'dr' & \lesssim\|Q\|_{L^{\infty}}\sup_{r>0}\Big(\frac{1}{r}\int_{0}^{r}(1+r')Q(r')r'dr'\Big)\lesssim1,\\
\sup_{r'>0}\int_{0}^{\infty}|K(r,r')|rdr & \lesssim\|Q\|_{L^{\infty}}\sup_{r'>0}\Big((1+r')\int_{r'}^{\infty}Q(r)dr\Big)\lesssim1,
\end{align*}
Schur's test yields 
\[
\|QB_{Q}f\|_{L^{2}}\lesssim\|(1+r)^{-1}f\|_{L^{2}}\lesssim\|f\|_{\dot{\mathcal{H}}_{m}}.
\]
This completes the proof of \eqref{eq:bdd-LQ}.

We turn to \eqref{eq:coercivity-LQ}. We assert the following weak
coercivity: 
\begin{equation}
\|L_{Q}f\|_{L^{2}}\geq c_{\psi_{1},\psi_{2}}\begin{cases}
\|(1+r)^{-1}f\|_{L^{2}} & \text{if }m=0,\\
\|r^{-1}f\|_{L^{2}} & \text{if }m\geq1,
\end{cases}\qquad\forall f\perp\{\psi_{1},\psi_{2}\}.\label{eq:weak-coer}
\end{equation}
We follow the standard strategy used in the proof of Poincar\'e inequality.
If \eqref{eq:weak-coer} fails, then there exists a sequence $\{f_{n}\}_{n\in\N}\perp\{\psi_{1},\psi_{2}\}$
such that 
\[
\|L_{Q}f_{n}\|_{L^{2}}\to0\quad\text{and}\quad\begin{cases}
\|(1+r)^{-1}f_{n}\|_{L^{2}}=1 & \text{if }m=0,\\
\|r^{-1}f_{n}\|_{L^{2}}=1 & \text{if }m\geq1.
\end{cases}
\]
We claim for such a sequence that there exist $0<r_{1}\ll1$ and $r_{2}\gg1$
with 
\begin{equation}
\begin{cases}
\liminf_{n\to\infty}\|f_{n}\|_{L^{2}(|x|<r_{2})}>0, & \text{if }m=0,\\
\liminf_{n\to\infty}\|f_{n}\|_{L^{2}(r_{1}<|x|<r_{2})}>0, & \text{if }m\geq1.
\end{cases}\label{eq:claim}
\end{equation}
Let us assume this claim and finish the proof of \eqref{eq:weak-coer}.
As all the terms of $L_{Q}f$ except $\partial_{r}f$ can be controlled
by $(1+r)^{-1}f$ if $m=0$ and $r^{-1}f$ if $m\geq1$, we see that
$\|\partial_{r}f_{n}\|_{L^{2}}$ is bounded. Passing to a subsequence,
we can assume $f_{n}\rightharpoonup f$ in $\dot{\mathcal{H}}_{m}$.
Moreover, we can apply the Rellich-Kondrachov compactness lemma (on
the region $\{|x|<r_{2}\}$ if $m=0$ and $\{r_{1}<|x|<r_{2}\}$ if
$m\geq1$) to guarantee that $f\neq0$. From the weak convergence,
we have $L_{Q}f=0$ (weakly) and $(f,\psi_{1})_{r}=(f,\psi_{2})_{r}=0$.
We understand $L_{Q}f=0$ on the ambient space $\R^{2}$ as 
\[
(\partial_{1}+i\partial_{2})[f(r)e^{im\theta}]=\Big(\frac{A_{\theta}[Q]}{r}f-QB_{Q}f\Big)(r)e^{i(m+1)\theta}.
\]
By elliptic regularity, $f$ is a smooth $m$-equivariant solution.
By Lemma \ref{lem:kernel-LQ}, $f$ is an $\R$-linear combination
of $\Lambda Q$ and $iQ$. Because of orthogonality conditions, we
should have $f=0$, yielding a contradiction. This completes the proof
of the weak coercivity \eqref{eq:weak-coer} assuming the claim \eqref{eq:claim}.

We turn to show the claim \eqref{eq:claim}. Suppose not; we can choose
a subsequence (still denoted by $\{f_{n}\}_{n\in\N}$) such that
\[
\begin{cases}
\lim_{n\to\infty}\|f_{n}\|_{L^{2}(|x|<r_{2})}=0, & \text{if }m=0,\\
\lim_{n\to\infty}\|f_{n}\|_{L^{2}(r_{1}<|x|<r_{2})}=0, & \text{if }m\geq1,
\end{cases}
\]
for any $0<r_{1}<r_{2}<\infty$. We first observe the following computation:
for any $0<y_{1}<y_{2}<\infty$, 
\begin{align*}
\int_{y_{1}}^{y_{2}}\Big|L_{Q}f_{n}-\frac{f_{n}}{r}\Big|^{2}rdr & =\int_{y_{1}}^{y_{2}}|L_{Q}f_{n}|^{2}rdr-2\int_{y_{1}}^{y_{2}}\Re(\overline{f}_{n}L_{Q}f_{n})dr+\int_{y_{1}}^{y_{2}}\Big|\frac{f_{n}}{r}\Big|^{2}rdr\\
 & =\int_{y_{1}}^{y_{2}}|L_{Q}f_{n}|^{2}rdr-|f_{n}|^{2}\Big|_{y_{1}}^{y_{2}}-2\int_{y_{1}}^{y_{2}}\Re\Big(\frac{f_{n}}{r}\Big)QB_{Q}(f_{n})rdr\\
 & \quad+\int_{y_{1}}^{y_{2}}\big(1+2(m+A_{\theta}[Q])\big)\Big|\frac{f_{n}}{r}\Big|^{2}rdr.
\end{align*}
Since $A_{\theta}[Q](r)\to-2(m+1)$ as $r\to\infty$, we take $y_{1}=r_{2}$
for sufficiently large $r_{2}$ and $y_{2}\to\infty$ to obtain
\[
\int_{r_{2}}^{\infty}\Big|\frac{f_{n}}{r}\Big|^{2}rdr\leq|f_{n}|^{2}(r_{2})+\int_{r_{2}}^{\infty}|L_{Q}f_{n}|^{2}rdr+2\int_{r_{2}}^{\infty}\Big|\frac{f_{n}}{r}\Big|QB_{Q}(|f_{n}|)rdr.
\]
The last term of the RHS can be absorbed into the LHS as $r_{2}$
is large. We then average in the $r_{2}$-variable and take limit
$n\to\infty$ to get 
\begin{equation}
\liminf_{n\to\infty}\int_{2r_{2}}^{\infty}\Big|\frac{f_{n}}{r}\Big|^{2}rdr\lesssim\liminf_{n\to\infty}\int_{r_{2}}^{2r_{2}}\Big|\frac{f_{n}}{r}\Big|^{2}rdr.\label{eq:claim-temp1}
\end{equation}
If $m=0$, then we get a contradiction. Hence \eqref{eq:claim} is
true for $m=0$.

When $m\geq1$, we should prevent $f_{n}$ from concentrating at the
origin. For this, we need one more computation: for any $0<y_{1}<y_{2}<\infty$,
\begin{align*}
\int_{y_{1}}^{y_{2}}\Big|L_{Q}f_{n}+\frac{f_{n}}{r}\Big|^{2}rdr & =\int_{y_{1}}^{y_{2}}|L_{Q}f_{n}|^{2}rdr+|f_{n}|^{2}\Big|_{y_{1}}^{y_{2}}+2\int_{y_{1}}^{y_{2}}\Re\Big(\frac{f_{n}}{r}\Big)QB_{Q}(f_{n})rdr\\
 & \quad+\int_{y_{1}}^{y_{2}}\big(1-2(m+A_{\theta}[Q])\big)\Big|\frac{f_{n}}{r}\Big|^{2}rdr.
\end{align*}
Since $A_{\theta}[Q](r)\to0$ as $r\to0$, we take $y_{1}\to0$ and
$y_{2}=r_{1}$ sufficiently small to obtain 
\[
\int_{0}^{r_{1}}\Big|\frac{f_{n}}{r}\Big|^{2}rdr\leq|f_{n}|^{2}(r_{1})+\int_{0}^{r_{1}}|L_{Q}f_{n}|^{2}rdr+2\int_{0}^{r_{1}}\Big|\frac{f_{n}}{r}\Big|QB_{Q}(|f_{n}|)rdr.
\]
The last term of the RHS can be absorbed into the LHS as $r_{1}$
is small. We average in the $r_{1}$-variable and take $n\to\infty$
to get 
\begin{equation}
\liminf_{n\to\infty}\int_{0}^{r_{1}}\Big|\frac{f_{n}}{r}\Big|^{2}rdr\lesssim\liminf_{n\to\infty}\int_{r_{1}}^{2r_{1}}\Big|\frac{f_{n}}{r}\Big|^{2}rdr.\label{eq:claim-temp2}
\end{equation}
From \eqref{eq:claim-temp1} and \eqref{eq:claim-temp2}, we get a
contradiction. This proves the claim \eqref{eq:claim}.

We are now in position to prove the strong coercivity \eqref{eq:coercivity-LQ}.
If $m\geq1$, one can proceed as 
\begin{align*}
\|L_{Q}f\|_{L^{2}} & =\delta\|L_{Q}f\|_{L^{2}}+(1-\delta)\|L_{Q}f\|_{L^{2}}\\
 & \geq\delta(\|\partial_{r}f\|_{L^{2}}-C\|r^{-1}f\|_{L^{2}})+c(1-\delta)\|r^{-1}f\|_{L^{2}}\geq\delta\|f\|_{\dot{\mathcal{H}}_{m}},
\end{align*}
provided $\delta$ is sufficiently small. If $m=0$, one uses $|r^{-1}A_{\theta}[Q]|\lesssim(1+r)^{-1}$
to obtain the desired coercivity. This ends the proof of Lemma \ref{lem:coercivity}.
\end{proof}

\section{\label{sec:Profile}Profile $Q^{(\eta)}$}

\emph{From now on, we assume $m\geq1$.}

\subsection{Role of pseudoconformal phase}

In view of pseudoconformal symmetry \eqref{eq:pseudo-conti}, we use
the modified profile 
\[
Q_{b}(y)\coloneqq Q(y)e^{-ib\frac{|y|^{2}}{4}}.
\]
We note that for all small $b\geq0$ 
\begin{equation}
\|Q_{b}-Q\|_{L^{2}}\lesssim\begin{cases}
b|\log b|^{\frac{1}{2}} & \text{if }m=1,\\
b & \text{if }m\geq2.
\end{cases}\label{eq:Qb-Q}
\end{equation}
If $m=1$, the logarithmic factor comes from $\||y|^{2}Q\|_{L^{2}(|y|\lesssim b^{-\frac{1}{2}})}\sim|\log b|^{\frac{1}{2}}$.
As we handle the case $m\geq1$ in the sequel, we only use the upper
bound $b|\log b|^{\frac{1}{2}}$. Moreover, as we will be in the regime
$b\lesssim\lambda$, we use the upper bound $\lambda|\log\lambda|^{\frac{1}{2}}$.
See \eqref{eq:Lqb-Lq}, \eqref{eq:quad-energy-coer}, and \eqref{eq:unav-Lyapunov-coer}
for instance.

For a function $f$ and $b\in\R$, recall the notation 
\[
f_{b}(y)\coloneqq f(y)e^{-ib\frac{|y|^{2}}{4}}.
\]
We discuss how the pseudoconformal phase $e^{-ib\frac{|y|^{2}}{4}}$
acts on the (linearized) evolution of \eqref{eq:CSS}. We record explicit
algebras for later use.
\begin{lem}[Conjugation by pseudoconformal phase]
\label{lem:conj-pseudo-phase}For any $b\in\R$, we have 
\begin{align}
L_{f_{b}}^{\ast}\D_{+}^{(f_{b})}f_{b} & =[L_{f}^{\ast}\D_{+}^{(f)}f]_{b}+ib\Lambda f_{b}-ib^{2}\partial_{b}(f_{b}),\label{eq:conj-b-phase1}\\
\mathcal{L}_{w_{b}}\epsilon_{b} & =[\mathcal{L}_{w}\epsilon]_{b}+ib\Lambda\epsilon_{b}-ib^{2}\partial_{b}(\epsilon_{b}),\label{eq:conj-b-phase2}
\end{align}
for any functions $f,w,\epsilon$.
\end{lem}

\begin{proof}
By a direct computation, we get 
\begin{align*}
\D_{+}^{(f_{b})}g_{b} & =(\D_{+}^{(f)}g-ib\tfrac{r}{2}g)_{b},\\
L_{f_{b}}g_{b} & =(L_{f}g-ib\tfrac{r}{2}g)_{b},\\
L_{f_{b}}^{\ast}g_{b} & =(L_{f}^{\ast}g+ib\tfrac{r}{2}g)_{b},
\end{align*}
for any functions $f$ and $g$. Thus 
\[
L_{f_{b}}^{\ast}\D_{+}^{(f_{b})}f_{b}=[(L_{f}^{\ast}+ib\tfrac{r}{2})(\D_{+}^{(f)}-ib\tfrac{r}{2})f]_{b}.
\]
Using the identities 
\begin{align*}
ib\tfrac{r}{2}\D_{+}^{(f)}f-L_{f}^{\ast}(ib\tfrac{r}{2}f) & =ib\Lambda f,\\{}
[\Lambda f]_{b}-\Lambda f_{b} & =ib\tfrac{r^{2}}{2}f_{b},\\
b^{2}\tfrac{r^{2}}{4}f_{b} & =ib^{2}\partial_{b}(f_{b}),
\end{align*}
the first identity \eqref{eq:conj-b-phase1} follows.

To show \eqref{eq:conj-b-phase2}, we recall that $\mathcal{L}_{w}\epsilon$
is the linear part in $\epsilon$ of $L_{w+\epsilon}^{\ast}\D_{+}^{(w+\epsilon)}(w+\epsilon)$.
Therefore, the identity \eqref{eq:conj-b-phase2} follows by substituting
$f=w+\epsilon$ into \eqref{eq:conj-b-phase1} and taking the linear
part in $\epsilon$.
\end{proof}
We transfer algebraic identities for $Q$ to $Q_{b}$:
\begin{lem}[Algebraic identities from $Q_{b}$]
\label{lem:algebraic-relation-Qb}For each $b\in\R$, we have 
\begin{equation}
-L_{Q_{b}}^{\ast}\D_{+}^{(Q_{b})}Q_{b}+ib\Lambda Q_{b}=\tfrac{b^{2}}{4}|y|^{2}Q_{b}.\label{eq:alg-idty}
\end{equation}
Moreover, the linearized operator $\mathcal{L}_{Q_{b}}$ satisfies
the following algebraic identities 
\begin{align*}
\mathcal{L}_{Q_{b}}iQ_{b}+b\Lambda Q_{b} & =b^{2}\partial_{b}Q_{b},\\
\mathcal{L}_{Q_{b}}[\Lambda Q]_{b}-ib\Lambda[\Lambda Q]_{b} & =-ib^{2}\partial_{b}[\Lambda Q]_{b},\\
\mathcal{L}_{Q_{b}}i|y|^{2}Q_{b}+b\Lambda(|y|^{2}Q_{b}) & =-4i[\Lambda Q]_{b}+b^{2}\partial_{b}(|y|^{2}Q_{b})\\
\mathcal{L}_{Q_{b}}\rho_{b}-ib\Lambda\rho_{b} & =Q_{b}-ib^{2}\partial_{b}\rho_{b}.
\end{align*}
\end{lem}

\begin{proof}
Substitute $f=Q$ into \eqref{eq:conj-b-phase1} and use $\D_{+}^{(Q)}Q=0$
to get the first identity. To obtain the algebraic identities satisfied
by $\mathcal{L}_{Q_{b}}$, substitute $w=Q$ into \eqref{eq:conj-b-phase2}
with help of Proposition \ref{prop:gen.null.space}.
\end{proof}
We discuss how the pseudoconformal phase $e^{-ib\frac{|y|^{2}}{4}}$
represents the pseudoconformal blow-up. Assume that 
\[
Q_{b}^{\sharp}(t,r)=\frac{1}{\lambda(t)}Q_{b(t)}\Big(\frac{r}{\lambda(t)}\Big)e^{i\gamma(t)}
\]
 solves \eqref{eq:CSS}. By the dynamical rescaling, $Q_{b}$ solves
\[
i\partial_{s}Q_{b}-L_{Q_{b}}^{\ast}\D_{+}^{(Q_{b})}Q_{b}=i\frac{\lambda_{s}}{\lambda}\Lambda Q_{b}+\gamma_{s}Q_{b}.
\]
Using $i\partial_{s}Q_{b}=b_{s}(\frac{|y|^{2}}{4}Q_{b})$ and \eqref{eq:alg-idty},
$Q_{b}$ satisfies 
\begin{equation}
i\Big(\frac{\lambda_{s}}{\lambda}+b\Big)\Lambda Q_{b}+\gamma_{s}Q_{b}-(b_{s}+b^{2})\tfrac{|y|^{2}}{4}Q_{b}=0.\label{eq:alg-idty-temp}
\end{equation}
Therefore, we obtain the formal parameter equations 
\[
\frac{\lambda_{s}}{\lambda}+b=0,\quad\gamma_{s}=0,\quad b_{s}+b^{2}=0,
\]
which are verified by the pseudoconformal blow-up regime 
\[
\lambda(t)=b(t)=|t|\quad\text{and}\quad\gamma(t)=0,\qquad\forall t<0.
\]

\subsection{\label{subsec:A-Hint-to}A hint to rotational instability}

We now discuss the instability mechanism of pseudoconformal blow-up
solutions. We first start formal computations that lead to the instability
of $S(t)$ or pseudoconformal blow-up solutions. This section assumes
the existence of the modified profile $Q^{(\eta)}$ below. In fact,
construction of $Q^{(\eta)}$ is an \emph{issue} and will be discussed
in the next subsection.

Fix a small $\eta\geq0$. We consider a real-valued profile $Q^{(\eta)}(r)$
that modifies $Q=Q^{(0)}$. We let $Q_{b}^{(\eta)}\coloneqq e^{-ib\frac{|y|^{2}}{4}}Q^{(\eta)}$
and 
\[
Q_{b}^{(\eta)\sharp}(t,r)\coloneqq\frac{1}{\lambda(t)}Q_{b(t)}^{(\eta)}\Big(\frac{r}{\lambda(t)}\Big)e^{i\gamma(t)}.
\]
Assume that $Q_{b}^{(\eta)\sharp}$ solves \eqref{eq:CSS}. Taking
the imaginary part of \eqref{eq:alg-idty-temp}, we further require
\[
\frac{\lambda_{s}}{\lambda}+b=0.
\]
Now the profile $Q^{(\eta)}$ should satisfy 
\begin{equation}
L_{Q^{(\eta)}}^{\ast}\D_{+}^{(Q^{(\eta)})}Q^{(\eta)}+\gamma_{s}Q^{(\eta)}-(b_{s}+b^{2})\tfrac{|y|^{2}}{4}Q^{(\eta)}=0.\label{eq:Q-eta-prelim}
\end{equation}

In case of \eqref{eq:NLS}, the authors of \cite{MerleRaphaelSzeftel2013AJM}
were able to choose\footnote{There the sign condition for $\eta$ is not necessary.}
$b_{s}+b^{2}=-\eta$ and $\gamma_{s}=1$. The instability direction
is $\rho_{\mathrm{NLS}}$, which satisfies $\mathcal{L}_{\mathrm{NLS}}\rho_{\mathrm{NLS}}=|y|^{2}R$.
For \eqref{eq:CSS}, however, 
\[
\mathcal{L}_{Q}(\partial_{\eta}|_{\eta=0}Q^{(\eta)})=\tfrac{|y|^{2}}{4}Q
\]
is \emph{not solvable}.\footnote{in the sense discussed in Section \ref{subsec:solvability}.}
In other words, the law $b_{s}+b^{2}=-\eta$ is forbidden. Thus we
have to search for a different instability mechanism.

It turns out that there is a \emph{rotational instability} for \eqref{eq:CSS}.
Let us require 
\[
\gamma_{s}=\eta\quad\text{and}\quad b_{s}+b^{2}=o(\eta).
\]
Assume that there exists a well-decaying solution $Q^{(\eta)}$ to
\eqref{eq:Q-eta-prelim}. Differentiating \eqref{eq:Q-eta-prelim}
in the $\eta$-variable, we have 
\[
\mathcal{L}_{Q}(\partial_{\eta}|_{\eta=0}Q^{(\eta)})+\eta Q=0.
\]
Thus we may take 
\[
\partial_{\eta}|_{\eta=0}Q^{(\eta)}=-\rho,
\]
where $\rho$ is constructed in Lemma \ref{lem:generalized-mode-rho}.

We now claim that $b_{s}+b^{2}$ has nontrivial $\eta^{2}$-order
terms as 
\[
b_{s}+b^{2}=-c\eta^{2},\quad c\coloneqq\frac{1}{(m+1)^{2}}+o_{\eta\to0^{+}}(1).
\]
To see this, we take the inner product of \eqref{eq:Q-eta-prelim}
with $\Lambda Q^{(\eta)}$ to obtain the Pohozaev identity 
\[
2E[Q^{(\eta)}]+(b_{s}+b^{2})\cdot\tfrac{1}{4}\|yQ^{(\eta)}\|_{L^{2}}^{2}=0.
\]
We can compute the leading order term of $2E[Q^{(\eta)}]$ as 
\[
2E[Q^{(\eta)}]=\int|\D_{+}^{(Q^{(\eta)})}Q^{(\eta)}|^{2}=(\eta^{2}+o(\eta^{2}))\|L_{Q}\rho\|_{L^{2}}^{2}.
\]
Therefore, 
\[
b_{s}+b^{2}=-\frac{4\|L_{Q}\rho\|_{L^{2}}^{2}}{\|rQ\|_{L^{2}}^{2}}(\eta^{2}+o(\eta^{2})).
\]
As we know by Remark \ref{rem:explicit-form-psi} that 
\[
L_{Q}\rho=\psi=\frac{1}{2(m+1)}rQ,
\]
the claim follows.

Summing up the above discussions, we are led to the ODEs 
\begin{equation}
\frac{\lambda_{s}}{\lambda}+b=0;\quad\gamma_{s}=\eta;\quad b_{s}+b^{2}+c\eta^{2}=0.\label{exact-mod-law-temp}
\end{equation}
There is a global-in-time exact solution 
\[
\begin{cases}
\lambda(t)=\sqrt{t^{2}+c\eta^{2}},\\
\gamma(t)=\frac{1}{\sqrt{c}}\tan^{-1}(\frac{t}{\eta\sqrt{c}}), & \forall t\in\R\\
b(t)=-t.
\end{cases}
\]
Thus for any fixed $t>0$, we observe a phase rotation 
\[
\lim_{\eta\to0^{+}}\big(\gamma(t)-\gamma(-t)\big)=\frac{\pi}{\sqrt{c}}=(m+1)\pi.
\]
Finally noting that $e^{i\gamma}\cdot f(r)e^{im\theta}=f(r)e^{im(\theta+\frac{\gamma}{m})}$,
this corresponds to a spatial rotation with the angle 
\[
\Big(\frac{m+1}{m}\Big)\pi
\]
near time zero.

When $\eta>0$ is small, $Q_{b}^{(\eta)\sharp}$ is an exact scattering
solution to \eqref{eq:CSS}, which is close to $S(t)$. At the blow-up
time $t=0$ of $S(t)$, the solution $Q_{b}^{(\eta)\sharp}$ abruptly
takes a spatial rotation by the angle $(\frac{m+1}{m})\pi$. This
is a sharp contrast to $S(t)$, which does not rotate at all. We will
exhibit the same instability mechanism for other pseudoconformal blow-up
solutions.

\subsection{Construction of $Q^{(\eta)}$}

In this subsection, we rigorously construct the modified profiles
$Q^{(\eta)}$ for all small $\eta>0$. From \eqref{eq:Q-eta-prelim}
with \eqref{exact-mod-law-temp}, we expect $Q^{(\eta)}$ to show
an exponential decay. If we recall that $Q$ has a polynomial decay,
then approximating $Q^{(\eta)}(r)$ by an $\eta$-expansion of $Q(r)$
does not make sense for large $r$. For $m$ small, performing an
$\eta$-expansion leads to a harmful spatial decay. To resolve this
issue, we have to search for a \emph{nonlinear deformation} of $Q^{(\eta)}$.
We crucially exploit the \emph{self-duality}. The main novelty of
the construction of $Q^{(\eta)}$ is that we can reduce \eqref{eq:Q-eta-prelim}
to a first-order differential equation. Although we exhibit the same
mechanism, we \emph{warn} that $Q^{(\eta)}$ in this subsection is
slightly different from that in Section \ref{subsec:A-Hint-to}. This
is merely due to a technical reason.

When $\eta=0$, we could impose the laws $\gamma_{s}=0$ and $b_{s}+b^{2}=0$,
and \eqref{eq:Q-eta-prelim} becomes 
\[
L_{Q}^{\ast}\D_{+}^{(Q)}Q=0.
\]
Because of the self-duality, this reduces to the Bogomol'nyi equation
\[
\D_{+}^{(Q)}Q=0.
\]
Inspired by this reduction, we twist the Bogomol'nyi equation by the
$\eta$-parameter that introduces a nontrivial $\gamma_{s}$$Q^{(\eta)}$
term in \eqref{eq:Q-eta-prelim}. This is one of the crucial points
that we make use of the self-duality.
\begin{lem}[Modified Bogomol'nyi equation]
Let $\eta\geq0$. Assume we have a profile $Q^{(\eta)}$ satisfying
\[
\D_{+}^{(Q^{(\eta)})}P^{(\eta)}=0,\quad Q^{(\eta)}=e^{-\eta\frac{r^{2}}{4}}P^{(\eta)},
\]
or equivalently, 
\begin{equation}
\D_{+}^{(Q^{(\eta)})}Q^{(\eta)}=-\tfrac{\eta}{2}rQ^{(\eta)}.\label{eq:first-order}
\end{equation}
Then $Q^{(\eta)}$ satisfies\footnote{Here the profile $Q^{(\eta)}$ is different from that in Section \ref{subsec:A-Hint-to}.
Compare \eqref{eq:exact-mod-law} and \eqref{exact-mod-law-temp}.
This is nothing but transforming $\eta$ to $\eta_{1}=\eta_{1}(\eta)$.} 
\begin{equation}
L_{Q^{(\eta)}}^{\ast}\D_{+}^{(Q^{(\eta)})}Q^{(\eta)}+\eta\theta_{\eta}Q^{(\eta)}+\eta^{2}\tfrac{r^{2}}{4}Q^{(\eta)}=0\label{eq:second-order}
\end{equation}
with 
\begin{equation}
\theta_{\eta}\coloneqq\frac{1}{4\pi}\int|Q^{(\eta)}|^{2}-(m+1).\label{eq:def-theta^eta}
\end{equation}
\end{lem}

\begin{proof}
We compute 
\begin{align*}
L_{Q^{(\eta)}}^{\ast} & \D_{+}^{(Q^{(\eta)})}Q^{(\eta)}=-\frac{\eta}{2}L_{Q^{(\eta)}}^{\ast}[rQ^{(\eta)}]\\
 & =\frac{\eta}{2}\Big[\Big(\partial_{r}+\frac{m+1+A_{\theta}[Q^{(\eta)}]}{r}\Big)(rQ^{(\eta)})-\Big(\int_{r}^{\infty}|Q^{(\eta)}|^{2}r'dr'\Big)Q^{(\eta)}\Big]\\
 & =\frac{\eta}{2}\Big[r\Big(\D_{+}^{(Q^{(\eta)})}+\frac{2m+2+2A_{\theta}[Q^{(\eta)}]}{r}\Big)Q^{(\eta)}-\Big(\int_{r}^{\infty}|Q^{(\eta)}|^{2}r'dr'\Big)Q^{(\eta)}\Big].
\end{align*}
Using the algebra
\[
2A_{\theta}[Q^{(\eta)}]Q^{(\eta)}-\Big(\int_{r}^{\infty}|Q^{(\eta)}|^{2}r'dr'\Big)Q^{(\eta)}=-\frac{1}{2\pi}\Big(\int|Q^{(\eta)}|^{2}\Big)Q^{(\eta)},
\]
we have 
\[
L_{Q^{(\eta)}}^{\ast}\D_{+}^{(Q^{(\eta)})}Q^{(\eta)}=-\eta^{2}\frac{r^{2}Q^{(\eta)}}{4}+\eta\Big(m+1-\frac{1}{4\pi}\int|Q^{(\eta)}|^{2}\Big)Q^{(\eta)}.
\]
This completes the proof of \eqref{eq:second-order}.
\end{proof}
In particular, 
\[
Q_{b}^{(\eta)\sharp}(t,x)=\frac{1}{\lambda(t)}Q_{b(t)}^{(\eta)}\Big(\frac{x}{\lambda(t)}\Big)e^{i\gamma(t)}
\]
is an exact solution to \eqref{eq:CSS} if $\lambda,\gamma,b$ satisfy
\begin{equation}
\left\{ \begin{aligned}\frac{\lambda_{s}}{\lambda}+b & =0,\\
b_{s}+b^{2} & =-\eta^{2},\\
\gamma_{s} & =\eta\theta_{\eta}.
\end{aligned}
\right.\label{eq:exact-mod-law}
\end{equation}
For each fixed $\eta>0$, the following $(\lambda_{\eta},\gamma_{\eta},b_{\eta})$
solves \eqref{eq:exact-mod-law}: 
\begin{equation}
\left\{ \begin{aligned}\lambda_{\eta}(t) & \coloneqq\langle t\rangle,\\
\gamma_{\eta}(t) & \coloneqq\theta_{\eta}\tan^{-1}(\tfrac{t}{\eta}),\\
b_{\eta}(t) & \coloneqq-t,
\end{aligned}
\right.\qquad\forall t\in\R.\label{eq:exact-mod-law-1}
\end{equation}
Since we will construct $Q^{(\eta)}$ as a deformation of $Q$, we
have $\theta_{\eta}=(m+1)+O(\eta)$.

Now, we solve the first-order equation \eqref{eq:first-order} and
derive estimates on $Q^{(\eta)}$. We solve \eqref{eq:first-order}
from $r=0$ up to $r=R_{\eta}$ using a perturbative argument around
$Q$. This is possible on the regime where the linearization $Q^{(\eta)}\approx Q+\eta\partial_{\eta}|_{\eta=0}Q^{(\eta)}$
is valid. The presence of $\eta^{2}\tfrac{|y|^{2}Q^{(\eta)}}{4}$
in \eqref{eq:second-order} suggests that $Q^{(\eta)}$ has additional
exponential decay $e^{-\eta\frac{r^{2}}{4}}$. Since $Q$ decays polynomially,
we expect that the linearization is valid on the region $r\ll\eta^{-\frac{1}{2}}$.
Beyond $r=R_{\eta}$, we cannot view $Q^{(\eta)}$ as a perturbation
of $Q$. Instead, we directly look at \eqref{eq:first-order} and
observe that $A_{\theta}[Q^{(\eta)}](R_{\eta})\approx-2(m+1)$. This
says that $m+A_{\theta}[Q^{(\eta)}]$ takes negative values. Thus
$P^{(\eta)}$ should decay polynomially and can be solved globally
on $(0,\infty)$.

For $B>1$ to be chosen large, let 
\[
R_{\eta}\coloneqq(B\eta)^{-\frac{1}{2}}.
\]

\begin{prop}[Construction of $Q^{(\eta)}$]
\label{prop:const-Q-eta}There exist $B>1$, $\eta^{\ast\ast}>0$,
and one-parameter families $\{Q^{(\eta)}\}_{\eta\in[0,\eta^{\ast\ast}]}$
and $\{P^{(\eta)}\}_{\eta\in\{0,\eta^{\ast\ast}]}$ of smooth real-valued
functions on $(0,\infty)$ satisfying the following properties.
\begin{enumerate}
\item (Equations of $Q^{(\eta)}$ and $P^{(\eta)}$) 
\begin{equation}
\begin{cases}
\D_{+}^{(Q^{(\eta)})}P^{(\eta)}=0,\\
Q^{(\eta)}=e^{-\eta\frac{r^{2}}{4}}P^{(\eta)},\\
Q^{(0)}=P^{(0)}=Q.
\end{cases}\label{eq:eqn-Q-P-eta}
\end{equation}
\item (Uniform bounds on $Q^{(\eta)}$) For any $\ell\in\{0,1\}$, we have
\begin{equation}
|\partial_{r}^{(\ell)}Q^{(\eta)}(r)|\lesssim\mathbf{1}_{r\leq R_{\eta}}r^{-\ell}Q+(R_{\eta})^{-\ell}Q(R_{\eta})\cdot\mathbf{1}_{r\geq R_{\eta}}e^{-\eta\frac{r^{2}}{4}}.\label{eq:ref-unif-bound}
\end{equation}
In particular, we have 
\begin{equation}
\sup_{\eta\in[0,\eta^{\ast\ast}]}|\partial_{r}^{(\ell)}Q^{(\eta)}(r)|\lesssim r^{-\ell}Q.\label{eq:unif-bound}
\end{equation}
\item (Differentiability in the $\eta$-variable) We have 
\begin{equation}
|Q^{(\eta)}-Q+\eta\cdot(m+1)\rho|\lesssim\eta^{2}r^{4}Q,\qquad\forall r\leq R_{\eta}.\label{eq:differentiability}
\end{equation}
\item (Subcritical mass of $Q^{(\eta)}$) We have 
\begin{equation}
\int|Q^{(\eta)}|^{2}=\int|Q|^{2}-\frac{\eta}{2(m+1)}\int|rQ|^{2}+o_{\eta\to0}(\eta).\label{eq:subcritic-mass}
\end{equation}
In particular, 
\begin{equation}
\theta_{\eta}=(m+1)-\frac{\eta}{8\pi(m+1)}\int|rQ|^{2}+o_{\eta\to0}(\eta).\label{eq:theta-asymp}
\end{equation}
\item (Decay of $P^{(\eta)}$) We have 
\begin{equation}
\lim_{r\to\infty}\frac{\log P^{(\eta)}}{\log r}=-(m+2)+\frac{\eta}{8\pi(m+1)}\int|rQ|^{2}+o_{\eta\to0}(\eta).\label{eq:decay-of-P}
\end{equation}
In particular, the decay of $P^{(\eta)}$ is slower than $Q$.
\item ($L^{2}$-difference of $Q^{(\eta)}$ and $Q$) We have 
\begin{equation}
\|Q^{(\eta)}-Q\|_{L^{2}}\lesssim\begin{cases}
\eta|\log\eta|^{\frac{1}{2}} & \text{if }m=1,\\
\eta & \text{if }m\geq2.
\end{cases}\label{eq:est-Q-eta-Q}
\end{equation}
\end{enumerate}
\end{prop}

\begin{proof}
The small parameter $\eta^{\ast\ast}>0$ will be chosen later. Let
us assume $\eta^{\ast\ast}\leq1$ for now. We will introduce a renormalized
unknown $\tilde v$ and write the equation in $\tilde v$. To get
some motivation, we first note a formal computation. If $Q^{(\eta)}$
exists, then we have 
\[
\D_{+}^{(Q^{(\eta)})}Q^{(\eta)}=-\eta\tfrac{r}{2}Q^{(\eta)}.
\]
Differentiating in the $\eta$-variable at $0$, we get the relation
\[
L_{Q}[\partial_{\eta}Q^{(\eta)}|_{\eta=0}]=-\tfrac{1}{2}rQ.
\]
Recalling $L_{Q}\rho=\psi=\frac{1}{2(m+1)}rQ$ given in Remark \ref{rem:explicit-form-psi},
we may set 
\[
\partial_{\eta}Q^{(\eta)}|_{\eta=0}=-(m+1)\rho.
\]

We now introduce the unknown 
\[
v\coloneqq\eta^{-2}(Q^{(\eta)}-Q+\eta\cdot(m+1)\rho).
\]
In terms of $v$, the equation \eqref{eq:eqn-Q-P-eta} reads
\begin{align*}
L_{Q}v & =\big(\tfrac{m+1}{2}r\rho-(m+1)^{2}\rho B_{Q}\rho\big)\\
 & \quad+\eta\big(-\tfrac{1}{2}r+(m+1)B_{Q}\rho\big)v+\eta^{2}\big(\tfrac{1}{2}B_{Q}v\big)v.
\end{align*}
As in the proof of Lemma \ref{lem:generalized-mode-rho}, we introduce
the renormalized unknown 
\[
\tilde v\coloneqq Q^{-1}v
\]
 and write the integral equation 
\begin{align}
\tilde v(r) & =\int_{0}^{r}\big(\tfrac{m+1}{2}\cdot\frac{r'\rho}{Q}-(m+1)^{2}\frac{\rho}{Q}B_{Q}\rho\big)dr'+\int_{0}^{r}\Big(\int_{0}^{r'}Q^{2}\tilde vr''dr''\Big)\frac{dr'}{r'}\label{eq:const-Q-eta-1}\\
 & \quad+\eta\int_{0}^{r}\big(-\tfrac{1}{2}r'+(m+1)B_{Q}\rho\big)\tilde vdr'+\eta^{2}\int_{0}^{r}\big(\tfrac{1}{2}B_{Q}v\big)\tilde vdr'.\nonumber 
\end{align}
It is now natural to expect that $\tilde v(r)\lesssim r^{4}$. By
a simple contraction argument, we can choose small $r_{0}>0$ and
large $C_{0}>0$ independent of $\eta$ such that $\tilde v(r)\leq C_{0}r^{4}$
for all $r\in[0,r_{0}]$ and $\eta\in[0,1]$.

We now construct $\tilde v$ beyond $[0,r_{0})$. For $\delta>0$
to be chosen small later independent of $\eta$, choose an increasing
sequence $\{r_{n}\}_{n\geq0}\subset\R$ such that $r_{n}\to\infty$,
$r_{n+1}\leq2r_{n}$, and $\int_{r_{n}}^{r_{n+1}}Q^{2}r''dr''\leq\delta$.
Set $\eta_{0}=1$. We follow the proof of Lemma \ref{lem:generalized-mode-rho}
with an extra care on the smallness of $\eta$. Indeed, if $\delta>0$
is chosen sufficiently small, then we can inductively construct sequences
of $\{\eta_{n}\}_{n\geq0}$, $\{C_{n}\}_{n\geq0}$, and $\tilde v(r)$
on $[0,r_{n}]$: assume we have constructed $\tilde v(r)$ on $[0,r_{n}]$
for all $\eta\in[0,\eta_{n}]$ such that $|\tilde v(r)|\leq C_{n}r^{4}$
on $[0,r_{n}]$ for all $\eta$. Then, there exist $0<\eta_{n+1}\leq\eta_{n}$
and $C_{n+1}\geq C_{n}$ such that we can construct $\tilde v(r)$
on $[r_{n},r_{n+1}]$ for all $\eta\in[0,\eta_{n+1}]$ such that $|\tilde v(r)|\leq C_{n+1}r^{4}$
on $[0,r_{n+1}]$ for all $\eta$. It is important that $C_{n}$ is
independent of $\eta$, once $\eta$ is smaller than $\eta_{n}$.

Now we claim that after finite steps we can further construct $\tilde v$
on $[0,R_{\eta}]$. In what follows, we only show a bound on $\tilde v$
via a bootstrap argument. The existence of $\tilde v$ automatically
follows by a standard arguemnt. We claim that there exist large $N\in\N$,
$B>1$, $C>C_{N}$, and small $\eta^{\ast\ast}\in(0,\eta_{N}]$ such
that we can construct $\tilde v$ with the bound $|\tilde v(r)|\leq Cr^{4}$
on $[0,R_{\eta}]$ for all $\eta\in[0,\eta^{\ast\ast}]$. To show
this claim, it suffices to show the bound $|\tilde v(r)|\leq Cr^{4}$
for $r\in[r_{N},R_{\eta}]$. We simply write the equation \eqref{eq:const-Q-eta-1}
as 
\[
\tilde v(r)=O(r^{4})+\int_{0}^{r}\Big(\int_{0}^{r'}Q^{2}\tilde vr''dr''\Big)\frac{dr'}{r'}+\eta\int_{0}^{r}O(r')\tilde vdr'+\eta^{2}\int_{0}^{r}\big(\tfrac{1}{2}B_{Q}v\big)\tilde vdr'.
\]
We consider the bootstrap hypothesis $|\tilde v(r)|\leq Cr^{4}$.
Using the estimate $\int_{1}^{r}Q^{2}(r')^{5}dr'\lesssim\log r$ for
$r\geq r_{N}$, we get 
\begin{align*}
\tilde v(r) & \lesssim r^{4}+C\log^{2}r+C\eta r^{6}+C^{2}\eta^{2}r^{4}\log r
\end{align*}
for $r\geq r_{N}$. If we choose $N$ large such that $\log^{2}r\ll r^{4}$,
$C>C_{N}$ large, and $\eta^{\ast\ast}\in(0,\eta_{N}]$ small such
that $C\eta^{\ast\ast}\ll1$, then we have 
\[
\tilde v(r)\leq(\tfrac{1}{2}+\eta\log r+\eta r^{2})\cdot Cr^{4}
\]
for $r\geq r_{N}$. Therefore, if we choose $B>1$ sufficiently large,
then we have 
\[
\tilde v(r)\leq\tfrac{3}{4}Cr^{4}
\]
for all $r\in[r_{N},R_{\eta}]$. By a continuity argument, the claim
follows. Tracking back the definitions of $\tilde v$ and $v$, the
proof of \eqref{eq:differentiability} is completed.

We have constructed $Q^{(\eta)}$ on $(0,R_{\eta}]$ for all $\eta\in[0,\eta^{\ast\ast}]$.
By \eqref{eq:differentiability}, we have 
\begin{align*}
\big(A_{\theta}[Q^{(\eta)}]-A_{\theta}[Q]\big)(R_{\eta}) & =\frac{1}{4\pi}\int\mathbf{1}_{r\leq R_{\eta}}(|Q|^{2}-|Q^{(\eta)}|^{2})\\
 & =\eta\frac{(m+1)}{2\pi}\Big(\int\mathbf{1}_{r\leq R_{\eta}}Q\rho\Big)+o_{\eta\to0}(\eta)\\
 & =\eta\frac{(m+1)}{2\pi}(Q,\rho)_{r}+o_{\eta\to0}(\eta).
\end{align*}
We then apply $(Q,\rho)_{r}=(L_{Q}^{\ast}L_{Q}\rho,\rho)_{r}=\|L_{Q}\rho\|_{L^{2}}^{2}$,
Remark \ref{lem:generalized-mode-rho}, and $A_{\theta}[Q](+\infty)=-2(m+1)$
to obtain 
\begin{equation}
A_{\theta}[Q^{(\eta)}](R_{\eta})=-2(m+1)+\frac{\eta}{8\pi(m+1)}\|rQ\|_{L^{2}}^{2}+o_{\eta\to0}(\eta).\label{eq:const-Q-eta-2}
\end{equation}

We now construct $Q^{(\eta)}$ beyond $r=R_{\eta}$. On the region
$r\gtrsim R_{\eta}$, the linear approximation of $Q^{(\eta)}$ is
not accurate. We consider the equation of $P^{(\eta)}$ instead: 
\[
\partial_{r}P^{(\eta)}=\frac{m+A_{\theta}[Q^{(\eta)}]}{r}P^{(\eta)}.
\]
By \eqref{eq:const-Q-eta-2}, we already have $A_{\theta}[Q^{(\eta)}](R_{\eta})\leq-2m-\frac{3}{2}$
for all $\eta\in[0,\eta^{\ast\ast}]$. By the comparison principle,
we get 
\begin{equation}
P^{(\eta)}(r)\leq P^{(\eta)}(R_{\eta})\Big(\frac{r}{R_{\eta}}\Big)^{-m-\frac{3}{2}}\label{eq:const-Q-eta-3}
\end{equation}
for all $r\geq R_{\eta}$. This shows that $P^{(\eta)}$ (and $Q^{(\eta)}$)
exists globally on $(0,\infty)$.

We now show \eqref{eq:ref-unif-bound} and \eqref{eq:unif-bound}.
Since \eqref{eq:unif-bound} easily follows from \eqref{eq:ref-unif-bound},
we focus on \eqref{eq:ref-unif-bound}. Note that \eqref{eq:differentiability}
says that $Q^{(\eta)}(r)\sim Q(r)$ for $r\leq R_{\eta}$. This verifies
\eqref{eq:ref-unif-bound} with $\ell=0$ on the region $r\leq R_{\eta}$.
On the region $r\geq R_{\eta}$, we use the exponential decay $e^{-\eta\frac{r^{2}}{4}}$
of $Q^{(\eta)}$ and \eqref{eq:const-Q-eta-3} to conclude \eqref{eq:ref-unif-bound}
with $\ell=0$ on the region $r\geq R_{\eta}$. Thus \eqref{eq:ref-unif-bound}
hold when $\ell=0$. We note that \eqref{eq:ref-unif-bound} with
$\ell=1$ follows by applying the equation 
\[
\partial_{r}Q^{(\eta)}=-\frac{m+A_{\theta}[Q^{(\eta)}]}{r}Q^{(\eta)}-\frac{1}{2}\eta r^{2}Q^{(\eta)}.
\]

We turn to show \eqref{eq:subcritic-mass}, \eqref{eq:theta-asymp},
and \eqref{eq:decay-of-P}. Note that \eqref{eq:unif-bound} yields
\[
\|\mathbf{1}_{r>R_{\eta}}Q^{(\eta)}\|_{L^{2}}=o_{\eta\to0}(\eta).
\]
Substituting this into \eqref{eq:const-Q-eta-2} gives \eqref{eq:subcritic-mass}.
To prove \eqref{eq:theta-asymp}, substitute \eqref{eq:subcritic-mass}
into \eqref{eq:def-theta^eta}. To prove \eqref{eq:decay-of-P}, observe
that 
\[
\partial_{r}\log P^{(\eta)}=\frac{m+A_{\theta}[Q^{(\eta)}]}{r}.
\]
Therefore, the decay rate of $P^{(\eta)}$ is obtained from the quantity
$m+A_{\theta}[Q^{(\eta)}](+\infty)$, which is $-(m+2)+\frac{\eta}{8\pi(m+1)}\|rQ\|_{L^{2}}^{2}$.
This completes the proof of \eqref{eq:decay-of-P}.

Finally, we show \eqref{eq:est-Q-eta-Q}. Note that 
\begin{align*}
\|Q^{(\eta)}-Q\|_{L^{2}} & =\|\mathbf{1}_{r\leq R_{\eta}}(Q^{(\eta)}-Q)\|_{L^{2}}+o_{\eta\to0}(\eta)\\
 & =\eta(m+1)\|\mathbf{1}_{r\leq R_{\eta}}\rho\|_{L^{2}}+o_{\eta\to0}(\eta).
\end{align*}
From the asymptotics \eqref{eq:rho-asymp} of $\rho$, \eqref{eq:est-Q-eta-Q}
follows.
\end{proof}
We record identities induced by the phase and scaling symmetries on
\eqref{eq:second-order}.
\begin{lem}[Algebraic Identities]
We have 
\begin{align}
\mathcal{L}_{Q^{(\eta)}}iQ^{(\eta)}+\eta\theta_{\eta}iQ^{(\eta)}+\eta^{2}\tfrac{|y|^{2}}{4}iQ^{(\eta)} & =0,\label{eq:phase-inv}\\
\mathcal{L}_{Q^{(\eta)}}\Lambda Q^{(\eta)}+\eta\theta_{\eta}\Lambda Q^{(\eta)}+\eta^{2}\tfrac{|y|^{2}}{4}\Lambda Q^{(\eta)} & =-2\eta\theta^{(\eta)}Q^{(\eta)}-4\eta^{2}\tfrac{|y|^{2}}{4}Q^{(\eta)}.\label{eq:scale-inv}
\end{align}
\end{lem}

\begin{proof}
One can obtain these identities by differentiating the phase and scaling
symmetries at $Q^{(\eta)}$. Differentiating 
\[
L_{e^{ia}Q^{(\eta)}}^{\ast}\D_{+}^{(e^{ia}Q^{(\eta)})}e^{ia}Q^{(\eta)}+\eta\theta_{\eta}e^{ia}Q^{(\eta)}+\eta^{2}\tfrac{|y|^{2}}{4}e^{ia}Q^{(\eta)}=0
\]
in the $a$-variable at $a=0$, we obtain \eqref{eq:phase-inv}. To
obtain \eqref{eq:scale-inv}, we temporarily write $f_{\lambda}(y)\coloneqq\frac{1}{\lambda}f(\frac{y}{\lambda})$
and consider 
\begin{align*}
0 & =[L_{Q^{(\eta)}}^{\ast}\D_{+}^{(Q^{(\eta)})}Q^{(\eta)}+\eta\theta_{\eta}Q^{(\eta)}+\eta^{2}\tfrac{|y|^{2}}{4}Q^{(\eta)}]_{\lambda}\\
 & =\lambda^{2}L_{Q_{\lambda}^{(\eta)}}^{\ast}\D_{+}^{(Q_{\lambda}^{(\eta)})}Q_{\lambda}^{(\eta)}+\eta\theta_{\eta}Q_{\lambda}^{(\eta)}+\eta^{2}\lambda^{-2}\tfrac{|y|^{2}}{4}Q_{\lambda}^{(\eta)}
\end{align*}
for $\lambda\in(0,\infty)$. Differentiating this in the $\lambda$-variable
at $\lambda=1$, we get 
\begin{align*}
 & \mathcal{L}_{Q^{(\eta)}}\Lambda Q^{(\eta)}+\eta\theta_{\eta}\Lambda Q^{(\eta)}+\eta^{2}\tfrac{|y|^{2}}{4}\Lambda Q^{(\eta)}\\
 & \quad=2L_{Q^{(\eta)}}^{\ast}\D_{+}^{(Q^{(\eta)})}Q^{(\eta)}-2\eta^{2}\tfrac{|y|^{2}}{4}Q^{(\eta)}=-2\eta\theta_{\eta}Q^{(\eta)}-4\eta^{2}\tfrac{|y|^{2}}{4}Q^{(\eta)}.
\end{align*}
This completes the proof of \eqref{eq:scale-inv}\@.
\end{proof}

\section{\label{sec:modulation}Setup for modulation analysis}

In this section, we will reduce Theorems \ref{thm:BW-sol} and \ref{thm:instability}
to the main bootstrap Lemma \ref{lem:bootstrap}. We will introduce
time-dependent modulation parameters $(\lambda,\gamma,b)$. Here $\lambda$,
$\gamma$, and $b$ correspond to the scaling, phase, and pseudoconformal
phase parameters, respectively. We write the blow-up ansatz in terms
of $(\lambda,\gamma,b)$ and the decomposition into the modified profile
$Q^{(\eta)}$, some small asymptotic profile $z^{\ast}$, and error
$\epsilon$. We first determine the evolution equation of $z$. We
then fix dynamical laws of $(\lambda,\gamma,b)$ to grant the desired
pseudoconformal blow-up rate. The main novelty of this step is to
detect the strong interactions between $Q_{b}^{\sharp}$ and $z$,
and incorporate them into our blow-up ansatz.

\subsection{\label{subsec:Evolution-of-z}Corrections from the interaction between
$Q_{b}^{\sharp}$ and $z$}

Our goal is to construct pseudoconformal blow-up solution with given
$m$-equivariant asymptotic profile $z^{\ast}$. We will require that
$z^{\ast}$ is small in some Sobolev space and is degenerate at the
origin. Given $z^{\ast}$, we fix the backward-in-time evolution of
$z(t)$ with $z(0)=z^{\ast}$. Near the blow-up time, $Q_{b}^{\sharp}$
is concentrated at the origin but $z$ is a regular solution; they
are decoupled in scales. Thus when one presumes that $Q_{b}^{\sharp}+z$
solves \eqref{eq:CSS}, one might expect that the mixed term of the
nonlinearity becomes negligible. This is indeed true for local nonlinearities,
for example $|u|^{2}u$. In the context of \eqref{eq:NLS}, Bourgain-Wang
\cite{BourgainWang1997} and Merle-Rapha\"el-Szeftel \cite{MerleRaphaelSzeftel2013AJM}
make use of this observation, and hence there is no strong influence
between $Q_{b}^{\sharp}$ and $z$. Thus it is enough to evolve $z(t)$
by \eqref{eq:NLS} itself. In \eqref{eq:CSS}, it turns out that there
are strong interactions between $Q_{b}^{\sharp}$ and $z$ due to
the long-range interactions in $A_{\theta}$ and $A_{0}$. We have
to capture strong interactions and incorporate the effects into the
decomposition of our blow-up ansatz. More precisely, we will modify
the $z$-evolution and add a correction to the $\gamma$-evolution.

We first explain the heuristics of our argument. Assume that $Q_{b}^{\sharp}+z$
solves \eqref{eq:CSS} (equivalently, \eqref{eq:CSS-r}) in the sense
that 
\[
0\approx\eqref{eq:heuristic-1},
\]
where
\begin{align}
 & i\partial_{t}(Q_{b}^{\sharp}+z)-L_{Q_{b}^{\sharp}+z}^{\ast}\D_{+}^{(Q_{b}^{\sharp}+z)}(Q_{b}^{\sharp}+z)\nonumber \\
 & =\left\{ \begin{aligned} & (i\partial_{t}+\partial_{rr}+\frac{1}{r}\partial_{r})(Q_{b}^{\sharp}+z)-\Big(\frac{m+A_{\theta}[Q_{b}^{\sharp}+z]}{r}\Big)^{2}(Q_{b}^{\sharp}+z)\\
 & -A_{0}[Q_{b}^{\sharp}+z](Q_{b}^{\sharp}+z)+|Q_{b}^{\sharp}+z|^{2}(Q_{b}^{\sharp}+z)
\end{aligned}
\right.\label{eq:heuristic-1}
\end{align}
We assume the decoupling in scales, for example
\[
A_{\theta}[Q_{b}^{\sharp}+z]\approx A_{\theta}[Q_{b}^{\sharp}]+A_{\theta}[z]\quad\text{and}\quad|Q_{b}^{\sharp}+z|^{2}\approx|Q_{b}^{\sharp}|^{2}+|z|^{2},
\]
but $A_{\theta}[Q_{b}^{\sharp}]z$ is not negligible since $A_{\theta}[Q_{b}^{\sharp}](r)\to-2(m+1)$
as $r\to\infty$. We then obtain 
\[
0\approx\eqref{eq:heuristic-2},
\]
where
\begin{equation}
\left\{ \begin{aligned} & (i\partial_{t}+\partial_{rr}+\frac{1}{r}\partial_{r})(Q_{b}^{\sharp}+z)-\Big(\frac{m+A_{\theta}[Q_{b}^{\sharp}]+A_{\theta}[z]}{r}\Big)^{2}(Q_{b}^{\sharp}+z)\\
 & \quad+\Big(\int_{r}^{\infty}(m+A_{\theta}[Q_{b}^{\sharp}]+A_{\theta}[z])(|Q_{b}^{\sharp}|^{2}+|z|^{2})\frac{dr'}{r'}\Big)(Q_{b}^{\sharp}+z)+|Q_{b}^{\sharp}|^{2}Q_{b}^{\sharp}+|z|^{2}z.
\end{aligned}
\right.\label{eq:heuristic-2}
\end{equation}
Separately collecting the evolutions for $Q_{b}^{\sharp}$ and $z$,
and further applying the decoupling in scales, we have 
\[
0\approx\eqref{eq:Q-heuristic}+\eqref{eq:z-heuristic},
\]
where
\begin{align}
 & \left\{ \begin{aligned} & (i\partial_{t}+\partial_{rr}+\frac{1}{r}\partial_{r})Q_{b}^{\sharp}-\Big(\frac{m+A_{\theta}[Q_{b}^{\sharp}]}{r}\Big)^{2}Q_{b}^{\sharp}\\
 & \quad-A_{0}[Q_{b}^{\sharp}]Q_{b}^{\sharp}+\Big(\int_{r}^{\infty}(m+A_{\theta}[Q_{b}^{\sharp}]+A_{\theta}[z])|z|^{2}\frac{dr'}{r'}\Big)Q_{b}^{\sharp}+|Q_{b}^{\sharp}|^{2}Q_{b}^{\sharp}
\end{aligned}
\right.\label{eq:Q-heuristic}\\
 & \left\{ \begin{aligned} & (i\partial_{t}+\partial_{rr}+\frac{1}{r}\partial_{r})z-\Big(\frac{m+A_{\theta}[Q_{b}^{\sharp}]+A_{\theta}[z]}{r}\Big)^{2}z\\
 & \quad+\Big(\int_{r}^{\infty}(m+A_{\theta}[Q_{b}^{\sharp}]+A_{\theta}[z])|z|^{2}\frac{dr'}{r'}\Big)z+|z|^{2}z.
\end{aligned}
\right.\label{eq:z-heuristic}
\end{align}
The above suggests us to evolve $z$ by setting \eqref{eq:z-heuristic}
equal to zero. Observing that $Q_{b}^{\sharp}$ is concentrated at
the origin, $z$ feels $Q_{b}^{\sharp}$ as a point charge. Thus we
further approximate $A_{\theta}[Q_{b}^{\sharp}](r)$ by its value
at spatial infinity $A_{\theta}[Q_{b}^{\sharp}](\infty)=-2(m+1)$.
Then one can view \eqref{eq:z-heuristic} as \eqref{eq:CSS} under
$-(m+2)$-equivariance. For \eqref{eq:Q-heuristic}, in view of decoupling
in scales, we again replace $A_{\theta}[Q_{b}^{\sharp}]$ by its value
at spatial infinity and $\int_{r}^{\infty}$ by $\int_{0}^{\infty}$.
As a result, we have 
\[
0\approx\eqref{eq:Q-heuristic-1}+\eqref{eq:z-heuristic-1},
\]
where 
\begin{align}
 & \left\{ \begin{aligned} & (i\partial_{t}+\partial_{rr}+\frac{1}{r}\partial_{r})Q_{b}^{\sharp}-\Big(\frac{m+A_{\theta}[Q_{b}^{\sharp}]}{r}\Big)^{2}Q_{b}^{\sharp}\\
 & \quad-A_{0}[Q_{b}^{\sharp}]Q_{b}^{\sharp}-\theta_{z\to Q_{b}^{\sharp}}Q_{b}^{\sharp}+|Q_{b}^{\sharp}|^{2}Q_{b}^{\sharp},
\end{aligned}
\right.\label{eq:Q-heuristic-1}\\
 & \left\{ \begin{aligned} & (i\partial_{t}+\partial_{rr}+\frac{1}{r}\partial_{r})z-\Big(\frac{-m-2+A_{\theta}[z]}{r}\Big)^{2}z\\
 & \quad+\Big(\int_{r}^{\infty}(-m-2+A_{\theta}[z])|z|^{2}\frac{dr'}{r'}\Big)z+|z|^{2}z,
\end{aligned}
\right.\label{eq:z-heuristic-1}
\end{align}
and 
\begin{equation}
\theta_{z\to Q_{b}^{\sharp}}\coloneqq-\int_{0}^{\infty}(-m-2+A_{\theta}[z])|z|^{2}\frac{dr'}{r'}.\label{eq:def-theta-cor}
\end{equation}

So far, we have discussed assuming $\eta=0$. In case of $\eta>0$,
we perform the same argument with a suitable modification, replacing
$Q_{b}^{\sharp}$ by $Q_{b}^{(\eta)\sharp}$. However, in the $z$-evolution,
one cannot view \eqref{eq:z-heuristic} as \eqref{eq:CSS} under $-(m+2)$-equivariance
because of $A_{\theta}[Q^{(\eta)\sharp}](\infty)\neq-2(m+1)$. Nevertheless,
as we know $A_{\theta}[Q^{(\eta)\sharp}](\infty)=-2(m+1)+O(\eta)$
by \eqref{eq:subcritic-mass}, we evolve $z$ the same as in the case
$\eta=0$. Later, we have to handle the error from $A_{\theta}[Q^{(\eta)\sharp}]-A_{\theta}[Q^{\sharp}]$.
On the other hand, the analogue of $\theta_{z\to Q_{b}^{\sharp}}$
in the case of $\eta>0$ is 
\begin{equation}
\theta_{z\to Q_{b}^{(\eta)\sharp}}\coloneqq-\int_{0}^{\infty}(m+A_{\theta}[Q^{(\eta)\sharp}](\infty)+A_{\theta}[z])|z|^{2}\frac{dr'}{r'}.\label{eq:def-theta-cor-eta}
\end{equation}

\begin{rem}
In the above $z$-evolution, it is important that $|m+A_{\theta}[Q](\infty)|-m=2$
is a nonnegative even integer. Indeed, assume that $z$ evolves under
\[
i\partial_{t}z+\Big(\partial_{rr}+\frac{1}{r}\partial_{r}-\frac{k^{2}}{r^{2}}\Big)z\approx0
\]
for some $k\in\R$. Then a generic solution $z$ would have $z(t,r)\sim r^{|k|}$
at the origin. In view of Lemma \ref{lem:smooth-equiv-criterion},
if $|k|-m$ is not a nonnegative even integer, then it is absurd to
see that $z(t,x)=z(t,r)e^{im\theta}$ is a smooth $m$-equivariant
solution. In other words, if we evolve $z$ using $m+A_{\theta}[Q^{(\eta)\sharp}](\infty)$,
we cannot have sufficient smoothness of $z$, especially at the origin.
Note that from \eqref{eq:def-Bogomolnyi-rad} and \eqref{eq:def-Q},
$m+A_{\theta}[Q](\infty)$ corresponds to the decay rate of $Q\sim r^{-(m+2)}\sim r^{(m+A_{\theta}[Q](\infty))}$
as $r\to\infty$. In other words, it is important that the decay rate
of the $m$-equivariant static solution $Q$ is $m+2\ell$ for some
nonnegative integer $\ell$.
\end{rem}

Motivated from \eqref{eq:z-heuristic-1}, we define $\tilde z(t,x)\coloneqq z(t,x)e^{-i(2m+2)\theta}$
and evolve $\tilde z$ under the \emph{$-(m+2)$-equivariant} \eqref{eq:CSS}
\begin{equation}
i\partial_{t}\tilde z-L_{\tilde z}^{\ast}\D_{+}^{(\tilde z)}\tilde z=0.\label{eq:z-tilde-equation}
\end{equation}
Then, \eqref{eq:z-tilde-equation} is equivalent to that the $m$-equivariant
solution $z$ evolves under the \emph{(CSS) with the external potential
}
\begin{equation}
\begin{cases}
i\partial_{t}z-L_{z}^{\ast}\D_{+}^{(z)}z=V_{Q_{b}^{\sharp}\to z}z,\\
z(0,x)=z^{\ast}(x),
\end{cases}\tag{zCSS}\label{eq:fCSS}
\end{equation}
where $z^{\ast}$ is a small profile and $V_{Q_{b}^{\sharp}\to z}$
is the external potential defined by 
\begin{equation}
V_{Q_{b}^{\sharp}\to z}\coloneqq\frac{4(m+1)}{r^{2}}-\frac{4(m+1)}{r^{2}}A_{\theta}[z]+\Big(\int_{r}^{\infty}2(m+1)|z|^{2}\frac{dr'}{r'}\Big).\label{eq:forcing-term}
\end{equation}

By Proposition \ref{prop:L2-Cauchy}, have the standard $L^{2}$-critical
local well-posedness, and small data global well-posedness and scattering
for \eqref{eq:z-tilde-equation} under the $-(m+2)$-equivariant symmetry.

We prepare small $-(m+2)$-equivariant data $\tilde z^{\ast}$ and
denote by $\tilde z$ the corresponding global $-(m+2)$-equivariant
solution $\tilde z$. More precisely, we assume $\|\tilde z\|_{H_{-(m+2)}^{k}}<\alpha^{\ast}$
for some $k=k(m)$. In view of Lemma \ref{lem:gen-hardy-Linfty},
we may choose $k=k(m)>m+3$ to have a degeneracy bound of $\tilde z$
near the origin: 
\[
\sup_{t\in[-1,0]}\big(|\partial_{r}^{(\ell)}\tilde z(t,r)|+\tfrac{1}{r^{\ell}}|\tilde z(t,r)|\big)\lesssim\alpha^{\ast}\begin{cases}
r^{m+2-\ell} & \text{if }r\leq1,\\
1 & \text{if }r\geq1,
\end{cases}\qquad\forall\ell\in\{0,1\}.
\]
This bound easily transfers to $z\coloneqq e^{i(2m+2)\theta}\tilde z$.
Note that if $\tilde z$ is smooth $-(m+2)$-equivariant, then $z$
is smooth $m$-equivariant.

\emph{Therefore, we let $z^{\ast}$ satisfy the assumption \eqref{eq:H}
with sufficiently small $\alpha^{\ast}$, and $z:\R\times\R^{2}\to\C$
be an $m$-equivariant solution to \eqref{eq:fCSS} with the initial
data $z^{\ast}$. Then $z$ satisfies the following properties:}
\begin{enumerate}
\item $z$ is global and scatters under the evolution $i\partial_{t}+\Delta_{-m-2}$.\footnote{$z$ also scatters under the usual evolution $i\partial_{t}+\Delta_{m}$.
See Remark \ref{rem:A-scattering-solution}.}
\item (Strichartz bound) $\|z\|_{L_{t}^{p}L_{x}^{q}}\lesssim\alpha^{\ast}$
for any admissible pairs $(p,q)$.
\item (Energy bound) $\|z\|_{L_{t}^{\infty}H_{x}^{k}}\lesssim\alpha^{\ast}$.
\item (Vanishing at origin) 
\[
\sup_{t\in[-1,0]}\big(|\partial_{r}^{(\ell)}z(t,r)|+\tfrac{1}{r^{\ell}}|z(t,r)|\big)\lesssim\alpha^{\ast}\cdot\begin{cases}
r^{m+2-\ell} & \text{if }r\leq1,\\
1 & \text{if }r\geq1,
\end{cases}\qquad\forall\ell\in\{0,1\}.
\]
\end{enumerate}

\subsection{Evolution of $\epsilon$}

We begin the modulation analysis. For a small fixed $\eta\geq0$,
we write a solution $u^{(\eta)}$ of the form 
\begin{equation}
u^{(\eta)}(t,x)=\frac{e^{i\gamma(t)}}{\lambda(t)}(Q_{b(t)}^{(\eta)}+\epsilon)(t,\frac{x}{\lambda(t)})+z(t,x).\label{eq:e-form}
\end{equation}
Here $z$ is a small global solution to \eqref{eq:fCSS} with $z(0)=z^{\ast}$
as in Section \ref{subsec:Evolution-of-z}. We want to construct a
family of solutions $u^{(\eta)}$ such that $u^{(0)}$ blows up at
$t=0$ with the pseudoconformal rate and for $\eta>0$, $u^{(\eta)}$
is a global scattering solution with $u^{(\eta)}\to u^{(0)}$ as $\eta\to0$.
More precisely, we expect our modulation parameters to satisfy 
\[
(b(t),\lambda(t),\gamma(t))\approx\begin{cases}
(-t,|t|,-\tfrac{m+1}{2}\pi) & \text{if }\eta=0,\\
(-t,\langle t\rangle,\theta_{\eta}\tan^{-1}(\tfrac{t}{\eta})) & \text{if }\eta>0,
\end{cases}
\]
for $t$ near zero ($t<0$ when $\eta=0$).

If one is only interested in the construction of a blow-up solution,
one may disregard $\eta$ and only work with the case $\eta=0$. See
Remark \ref{rem:alt-proof}. In our presentation, we first derive
an equation of $\epsilon$ in the case $\eta=0$ for readability.
And then, we will point out what are changed when $\eta>0$.

Note that we haven't specified the choice of the modulation parameters
$b(t),\lambda(t),\gamma(t)$ at each time $t$. We will fix them by
imposing two orthogonality conditions and one relation between $b$
and $\lambda$ in Section \ref{subsec:choice-of-geometric}. In this
subsection, we focus on deriving the evolution equation for $\epsilon$
without fixing the laws of $b$, $\lambda$, and $\gamma$.

Recall Section \ref{subsec:Dynamic-Rescaling} on $\sharp$ and $\flat$
operations. In \eqref{eq:e-form}, the functions $u$ and $z$ are
defined with the $(t,x)$-variables, but $Q_{b}$ and $\epsilon$
are functions of $(s,y)$. Thus by $\flat$ operation one can change
$u$ and $z$ to functions of $(s,y)$, i.e. $u^{\flat}(s,y)$ and
$z^{\flat}(s,y)$. Similarly by $\sharp$ operation, one can change
$Q_{b}$ and $\epsilon$ to functions of $(t,x)$, i.e. $Q_{b(t)}^{\sharp}(x)$
and $\epsilon^{\sharp}(t,x)$. In the following, we take $\flat$
or $\sharp$ operations to convert equations in the $(t,x)$-variables
and in the $(s,y)$-variables.

In view that $Q_{b}$ and $z$ live on different scales, it is convenient
to introduce general notations to describe their interactions. Consider
two functions $f$ and $g$ such that $f$ has shorter length scale
than $g$. $f$ and $g$ can be functions of either $(s,y)$ or $(t,x)$.
For example, we may consider $(f,g)=(Q_{b}^{\sharp},z)$ or $(Q_{b},z^{\flat})$.
Firstly, we use 
\[
R_{f,g}\coloneqq L_{f+g}^{\ast}\D_{+}^{(f+g)}(f+g)-L_{f}^{\ast}\D_{+}^{(f)}f-L_{g}^{\ast}\D_{+}^{(g)}g,
\]
which is the interaction term between $f$ and $g$. We will use this
for $R_{Q_{b},z^{\flat}}$, $R_{Q_{b}^{\sharp},z}$, and so on. We
will use notations 
\begin{align*}
V_{f\to g} & \coloneqq\frac{1}{r^{2}}\Big((m+A_{\theta}[f](+\infty)+A_{\theta}[g])^{2}-(m+A_{\theta}[g])^{2}\Big)\\
 & \quad-\int_{r}^{\infty}A_{\theta}[f](+\infty)|g|^{2}\frac{dr'}{r'},\\
\theta_{g\to f} & \coloneqq-\int_{0}^{\infty}(m+A_{\theta}[f](\infty)+A_{\theta}[g])|g|^{2}\frac{dr'}{r'}.
\end{align*}
The strong interaction from $f$ to $g$ is given by $V_{f\to g}g$,
and that from $g$ to $f$ is given by $\theta_{g\to f}f$. We denote
the marginal interaction by 
\[
\tilde R_{f,g}\coloneqq R_{f,g}-V_{f\to g}g-\theta_{g\to f}f.
\]
Note that 
\[
[R_{f,g}]^{\sharp}=\lambda^{2}[R_{f^{\sharp},g^{\sharp}}],\quad[V_{f\to g}g]^{\sharp}=\lambda^{2}V_{f^{\sharp}\to g^{\sharp}}g^{\sharp},\quad\theta_{g\to f}=\lambda^{2}\theta_{g^{\sharp}\to f^{\sharp}}.
\]
In case of $(f,g)=(Q_{b}^{(\eta)},z^{\flat})$, we will denote 
\begin{equation}
\tilde{\gamma}_{s}\coloneqq\gamma_{s}+\theta_{z^{\flat}\to Q_{b}^{(\eta)}}=\gamma_{s}+\lambda^{2}\theta_{z\to Q_{b}^{(\eta)\sharp}}.\label{eq:def-tilde-gamma}
\end{equation}

Next, for functions $w$ and $\epsilon$, we denote by 
\[
R_{(w+\epsilon)-w}\coloneqq L_{w+\epsilon}^{\ast}\D_{+}^{(w+\epsilon)}(w+\epsilon)-L_{w}^{\ast}\D_{+}^{(w)}w-\mathcal{L}_{w}\epsilon
\]
the quadratic and higher order terms in $\epsilon$.

\subsubsection*{The case $\eta=0$}

Here, we assume $\eta=0$ and derive the equation of $\epsilon$.
If we take the $\flat$ operation to \eqref{eq:CSS}, then by Section
\ref{subsec:Dynamic-Rescaling}, 
\begin{align*}
i\partial_{t}u-L_{u}^{\ast}\D_{+}^{(u)}u & =0,\\
i\partial_{s}u^{\flat}-L_{u^{\flat}}^{\ast}\D_{+}^{(u^{\flat})}u^{\flat} & =i\frac{\lambda_{s}}{\lambda}\Lambda u^{\flat}+\gamma_{s}u^{\flat}.
\end{align*}

The evolution of $Q_{b}$ is given by Lemma \ref{lem:algebraic-relation-Qb}
\begin{align*}
i\partial_{s}Q_{b}-L_{Q_{b}}^{\ast}\D_{+}^{(Q_{b})}Q_{b}+ib\Lambda Q_{b} & =(b_{s}+b^{2})\tfrac{|y|^{2}}{4}Q_{b},\\
i\partial_{t}Q_{b}^{\sharp}-L_{Q_{b}^{\sharp}}^{\ast}\D_{+}^{(Q_{b}^{\sharp})}Q_{b}^{\sharp} & =-i\frac{1}{\lambda^{2}}\Big(\frac{\lambda_{s}}{\lambda}+b\Big)\Lambda Q_{b}^{\sharp}-\frac{\gamma_{s}}{\lambda^{2}}Q_{b}^{\sharp}+\frac{1}{\lambda^{2}}(b_{s}+b^{2})[\tfrac{|y|^{2}}{4}Q]_{b}^{\sharp}.
\end{align*}
Here, we remark that $\Lambda Q_{b}=[\Lambda Q]_{b}-ib\frac{|y|^{2}}{2}Q_{b}$
and $\Lambda Q_{b}^{\sharp}=[\Lambda Q]_{b}^{\sharp}-ib[\frac{|y|^{2}}{2}Q]_{b}^{\sharp}$.

Define 
\[
w(t,x)\coloneqq Q_{b}^{\sharp}(t,x)+z(t,x)=u(t,x)-\epsilon^{\sharp}(t,x).
\]
Viewing $w=Q_{b}^{\sharp}+z$, by \eqref{eq:fCSS}, $w$ solves 
\begin{align*}
i\partial_{t}w-L_{w}^{\ast}\D_{+}^{(w)}w & =-i\frac{1}{\lambda^{2}}\Big(\frac{\lambda_{s}}{\lambda}+b\Big)\Lambda Q_{b}^{\sharp}-\frac{\gamma_{s}}{\lambda^{2}}Q_{b}^{\sharp}+\frac{1}{\lambda^{2}}(b_{s}+b^{2})[\tfrac{|y|^{2}}{4}Q]_{b}^{\sharp}\\
 & \quad+(V_{Q_{b}^{\sharp}\to z}z-R_{Q_{b}^{\sharp},z}).
\end{align*}
Taking the $\flat$ operation, $w^{\flat}$ solves 
\begin{align*}
i\partial_{s}w^{\flat}-L_{w^{\flat}}^{\ast}\D_{+}^{(w^{\flat})}w^{\flat}+ib\Lambda w^{\flat} & =i\Big(\frac{\lambda_{s}}{\lambda}+b\Big)\Lambda z^{\flat}+\gamma_{s}z^{\flat}+(b_{s}+b^{2})\tfrac{|y|^{2}}{4}Q_{b}\\
 & \quad+(V_{Q_{b}\to z^{\flat}}z^{\flat}-R_{Q_{b},z^{\flat}}).
\end{align*}
Recall that the external potential term $V_{Q_{b}^{\sharp}\to z}$
is introduced to capture strong interactions in $R_{Q_{b}^{\sharp},z}$.
We also note that $V_{Q_{b}^{\sharp}\to z}$ does not depend on modulation
parameters, but only on $z$.

Finally we derive the equation of $\epsilon$. Subtracting the equation
of $u^{\flat}$ by that of $w^{\flat}$, we have 
\begin{align*}
i\partial_{s}\epsilon-\mathcal{L}_{w^{\flat}}\epsilon+ib\Lambda\epsilon & =i\Big(\frac{\lambda_{s}}{\lambda}+b\Big)\Lambda(Q_{b}+\epsilon)+\tilde{\gamma}_{s}Q_{b}+\gamma_{s}\epsilon-(b_{s}+b^{2})\tfrac{|y|^{2}}{4}Q_{b}\\
 & \quad+\tilde R_{Q_{b},z^{\flat}}+R_{u^{\flat}-w^{\flat}}.
\end{align*}
By taking the $\sharp$ operation, we have 
\begin{align*}
i\partial_{t}\epsilon^{\sharp}-\mathcal{L}_{w}\epsilon^{\sharp} & =\frac{1}{\lambda^{2}}i\Big(\frac{\lambda_{s}}{\lambda}+b\Big)\Lambda Q_{b}^{\sharp}+\frac{1}{\lambda^{2}}\tilde{\gamma}_{s}Q_{b}^{\sharp}-\frac{1}{\lambda^{2}}(b_{s}+b^{2})[\tfrac{|y|^{2}}{4}Q]_{b}^{\sharp}\\
 & \quad+\tilde R_{Q_{b}^{\sharp},z}+R_{u-w}.
\end{align*}

\subsubsection*{The case $\eta>0$}

We now consider the case $\eta>0$. Recall \eqref{eq:second-order}
and \eqref{eq:conj-b-phase1} as 
\[
i\partial_{s}Q_{b}^{(\eta)}-L_{Q_{b}^{(\eta)}}^{\ast}\D_{+}^{(Q_{b}^{(\eta)})}Q_{b}^{(\eta)}+ib\Lambda Q_{b}^{(\eta)}-\eta\theta_{\eta}Q_{b}^{(\eta)}=(b_{s}+b^{2}+\eta^{2})\tfrac{|y|^{2}}{4}Q_{b}^{(\eta)}.
\]
Recall that we evolve $z$ by \eqref{eq:fCSS}, which is independent
of $\eta$. If we define 
\[
w(t,x)\coloneqq Q_{b(t)}^{(\eta)\sharp}(x)+z(t,x),
\]
then $w$ solves 
\begin{align*}
 & i\partial_{s}w^{\flat}-L_{w^{\flat}}^{\ast}\D_{+}^{(w^{\flat})}w^{\flat}+ib\Lambda w^{\flat}-\eta\theta_{\eta}w^{\flat}\\
 & =i\Big(\frac{\lambda_{s}}{\lambda}+b\Big)\Lambda z^{\flat}+(\gamma_{s}-\eta\theta_{\eta})z^{\flat}+(b_{s}+b^{2}+\eta^{2})\tfrac{|y|^{2}}{4}Q_{b}^{(\eta)}\\
 & \quad-V_{Q_{b}^{(\eta)}-Q_{b}}z^{\flat}+(V_{Q_{b}^{(\eta)}\to z^{\flat}}z^{\flat}-R_{Q_{b}^{(\eta)},z^{\flat}}),
\end{align*}
where $V_{Q_{b}^{(\eta)}-Q_{b}}$ is an additional error induced from
the external potential.
\begin{equation}
V_{Q_{b}^{(\eta)}-Q_{b}}\coloneqq V_{Q_{b}^{(\eta)}\to z^{\flat}}-V_{Q_{b}\to z^{\flat}}.\label{eq:def-V-tilde}
\end{equation}
Then the evolution of $\epsilon$ is given by 
\begin{align}
 & i\partial_{s}\epsilon-\mathcal{L}_{w^{\flat}}\epsilon+ib\Lambda\epsilon-\eta\theta_{\eta}\epsilon\label{eq:e-prem}\\
 & =i\Big(\frac{\lambda_{s}}{\lambda}+b\Big)\Lambda(Q_{b}^{(\eta)}+\epsilon)+(\tilde{\gamma}_{s}-\eta\theta_{\eta})Q_{b}^{(\eta)}+(\gamma_{s}-\eta\theta_{\eta})\epsilon\nonumber \\
 & \quad-(b_{s}+b^{2}+\eta^{2})\tfrac{|y|^{2}}{4}Q_{b}^{(\eta)}+\tilde R_{Q_{b}^{(\eta)},z^{\flat}}+V_{Q_{b}^{(\eta)}-Q_{b}}z^{\flat}+R_{u^{\flat}-w^{\flat}}.\nonumber 
\end{align}
Applying the $\sharp$ operation, we get 
\begin{align}
i\partial_{t}\epsilon^{\sharp}-\mathcal{L}_{w}\epsilon^{\sharp} & =\frac{1}{\lambda^{2}}i\Big(\frac{\lambda_{s}}{\lambda}+b\Big)\Lambda Q_{b}^{(\eta)\sharp}+\frac{1}{\lambda^{2}}(\tilde{\gamma}_{s}-\eta\theta_{\eta})Q_{b}^{(\eta)\sharp}\label{eq:e-sharp-prem}\\
 & \quad-\frac{1}{\lambda^{2}}(b_{s}+b^{2}+\eta^{2})[\tfrac{|y|^{2}}{4}Q^{(\eta)}]_{b}^{\sharp}+\tilde R_{Q_{b}^{(\eta)\sharp},z}+V_{Q_{b}^{(\eta)\sharp}-Q_{b}^{\sharp}}z+R_{u-w}.\nonumber 
\end{align}

\subsection{\label{subsec:choice-of-geometric}Choice of modulation parameters}

As we have three modulation parameters $b$, $\lambda$, and $\gamma$,
we can impose three conditions to fix dynamical laws and $\epsilon(s,y)$.
In order to guarantee the coercivity of the linearized operator, we
spend two degrees of freedom for orthogonality conditions. Fix any
smooth compactly supported real-valued functions $\mathcal{Z}_{\re},\mathcal{Z}_{\im}:(0,\infty)\to\R$
such that 
\begin{equation}
(\mathcal{Z}_{\re},\Lambda Q)_{r}=(\mathcal{Z}_{\im},Q)_{r}=1\neq0.\label{eq:non-deg}
\end{equation}
We impose the orthogonality conditions as 
\begin{equation}
(\epsilon,[\mathcal{Z}_{\re}]_{b})_{r}=(\epsilon,[i\mathcal{Z}_{\im}]_{b})_{r}=0.\label{eq:orthog}
\end{equation}
In view of Lemma \ref{lem:coercivity}, \eqref{eq:orthog} implies
coercivity of the linearized operator, provided that $b$ is sufficiently
small. A cleverer choice of the orthogonality conditions is not important
in our analysis.

For the remaining one degree of freedom, we impose a dynamical law
between $b$ and $\lambda$ as 
\begin{equation}
2\Big(\frac{\lambda_{s}}{\lambda}+b\Big)b-(b_{s}+b^{2}+\eta^{2})=0.\label{eq:law-fixation}
\end{equation}
This is motivated to cancel out $\frac{|y|^{2}}{4}Q_{b}^{(\eta)}$
in the equation \eqref{eq:e-prem} of $\epsilon$, since it has slow
decay $|\frac{|y|^{2}}{4}Q_{b}|\sim r^{-m}$ as $r\to\infty$. Such
a decay is not sufficient to close our bootstrap argument, particularly
when $m$ is small. Using $\Lambda Q_{b}^{(\eta)}=[\Lambda Q^{(\eta)}]_{b}-ib\frac{|y|^{2}}{2}Q_{b}^{(\eta)}$,
we organize terms containing $\frac{|y|^{2}}{4}Q_{b}^{(\eta)}$ in
\eqref{eq:e-prem} as 
\begin{align*}
 & i\Big(\frac{\lambda_{s}}{\lambda}+b\Big)\Lambda Q_{b}^{(\eta)}-(b_{s}+b^{2}+\eta^{2})\tfrac{|y|^{2}}{4}Q_{b}^{(\eta)}\\
 & =\Big[2\Big(\frac{\lambda_{s}}{\lambda}+b\Big)b-(b_{s}+b^{2}+\eta^{2})\Big]\tfrac{|y|^{2}}{4}Q_{b}^{(\eta)}+i\Big(\frac{\lambda_{s}}{\lambda}+b\Big)[\Lambda Q^{(\eta)}]_{b}.
\end{align*}
By \eqref{eq:law-fixation}, we can delete the $\frac{|y|^{2}}{4}Q_{b}^{(\eta)}$-contribution
in \eqref{eq:e-prem}; we have 
\[
i\Big(\frac{\lambda_{s}}{\lambda}+b\Big)\Lambda Q_{b}^{(\eta)}-(b_{s}+b^{2}+\eta^{2})\tfrac{|y|^{2}}{4}Q_{b}^{(\eta)}=i\Big(\frac{\lambda_{s}}{\lambda}+b\Big)[\Lambda Q^{(\eta)}]_{b}.
\]
This cancellation is heavily used in Sections \ref{subsec:Transfer-H1-to-L2}
and \ref{subsec:virial-lyapunov}. Note that \eqref{eq:law-fixation}
is consistent with 
\[
b(t)=-t\quad\text{and}\quad\lambda(t)=\langle t\rangle.
\]

We now discuss on the existence of $(\lambda,\gamma,b)$ satisfying
\eqref{eq:law-fixation} and orthogonality conditions \eqref{eq:orthog}.
When $\eta=0$, \eqref{eq:law-fixation} is \emph{explicitly solvable};
one sees that
\begin{equation}
\Big(\frac{b}{\lambda^{2}}\Big)_{t}=\frac{1}{\lambda^{2}}\Big(\frac{b}{\lambda^{2}}\Big)_{s}=\frac{b}{\lambda^{4}}\Big(\frac{b_{s}}{b}-2\frac{\lambda_{s}}{\lambda}\Big)=\Big(\frac{b}{\lambda^{2}}\Big)^{2}.\label{eq:solvable}
\end{equation}
This is satisfied by the choice 
\begin{equation}
b(t)=|t|^{-1}\lambda^{2}(t)\label{eq:solve-law}
\end{equation}
for any time $t$. As we expect $\lambda^{2}(t)\sim|t|^{2}$ in our
blow-up regime, \eqref{eq:solve-law} says that $|b(t)|\lesssim|t|$,
which is small provided $t$ is small. From this, one is able to use
the implicit function theorem at each fixed time $t$.

When $\eta>0$, we are not able to solve \eqref{eq:law-fixation}
to obtain a fixed time relation between $b$ and $\lambda$, such
as \eqref{eq:solve-law}. Instead, we solve a system of ODEs for $(b,\lambda,\gamma)$
given by two orthogonality conditions and \eqref{eq:law-fixation}.
The precise setup is as follows. We consider the initial data 
\begin{equation}
u(0,x)=\frac{e^{i\gamma(0)}}{\lambda(0)}Q_{b(0)}^{(\eta)}\Big(\frac{x}{\lambda(0)}\Big)+z^{\ast}(x).\label{eq:init-data}
\end{equation}
We evolve $u$ by \eqref{eq:CSS} and $z$ by \eqref{eq:fCSS}. As
we assume the decomposition
\[
u(t,x)=\frac{e^{i\gamma(t)}}{\lambda(t)}[Q_{b(t)}^{(\eta)}+\epsilon](t,\frac{x}{\lambda(t)})+z(t,x),
\]
 $\epsilon$ is a function of $u,z,\lambda,b,\gamma$. Note that $u$
and $z$ are by now given, so $\epsilon$ is viewed as a function
of $\lambda,b,\gamma$. The orthogonality conditions are viewed as
first-order differential equations of $(b,\lambda,\gamma)$, obtained
by differentiating them in the $t$-variable and substituting \eqref{eq:e-sharp-prem}.
We also interpret \eqref{eq:law-fixation} in the $t$-variable. Thus
we have a system of ODEs 
\begin{equation}
\left\{ \begin{aligned}\partial_{t}(\epsilon,[\mathcal{Z}_{\re}]_{b})_{r} & =0,\\
\partial_{t}(\epsilon,i[\mathcal{Z}_{\im}]_{b})_{r} & =0,\\
2(\lambda\lambda_{t}+b)b-(\lambda^{2}b_{t}+b^{2}+\eta^{2}) & =0.
\end{aligned}
\right.\label{eq:system-ODE}
\end{equation}
From \eqref{eq:non-deg}, the system \eqref{eq:system-ODE} is nondegenerate.
In view of \eqref{eq:testing}, the first and second lines of \eqref{eq:system-ODE}
essentially govern the laws of $\lambda$ and $\gamma$. The third
line governs the law of $b$. We then give the initial data at $t=0$
as $(b(0),\lambda(0),\gamma(0))=(0,\eta,0)$. Local existence and
uniqueness are clear from Picard's theorem. Moreover, we can solve
the system as long as the parameters $b$, $\log\lambda$, and $\gamma$
do not diverge.

In the following sections, we will verify estimates of $b,\lambda,\gamma$,
and $\epsilon$ via a bootstrap argument. The conclusion of the bootstrap
procedure guarantees the existence of the decomposition.\footnote{In the forward problems, when $u(t_{0})=Q+\epsilon(t_{0})$ is given,
one has to do initial decomposition $u(t_{0})=Q_{\mathcal{P}(t_{0})}+\tilde{\epsilon}(t_{0})$
with some modulated profile $Q_{\mathcal{P}(t_{0})}$ to guarantee
that $\tilde{\epsilon}$ satisfies some orthogonality conditions.
For this purpose, one has to use the implicit function theorem at
time $t=t_{0}$. In our case, the orthogonality condition at time
$t=0$ is automatic as $\epsilon(0)=0$. This makes the argument simpler.}

\subsection{\label{subsec:reduction-of-theorem}Reduction of Theorems \ref{thm:BW-sol}
and \ref{thm:instability} to the main bootstrap Lemma \ref{lem:bootstrap}}

We will construct a family of solutions $\{u^{(\eta)}\}_{\eta\in[0,\eta^{\ast}]}$
such that $u^{(0)}$ is a pseudoconformal blow-up solutions. We first
construct a sequence of solutions $\{u^{(\eta)}\}_{\eta\in(0,\eta^{\ast}]}$
on $[t_{0}^{\ast},0]$ that approximately follow \eqref{eq:exact-mod-law-1}.
Then by taking the limit $\eta\to0$, we construct a pseudoconformal
blow-up solution $u^{(0)}$ on $[t_{0}^{\ast},0)$. A similar procedure
was presented in Merle \cite{Merle1990CMP} and Merle-Rapha\"el-Szeftel
\cite{MerleRaphaelSzeftel2013AJM}.

Define 
\begin{equation}
\gamma_{\mathrm{cor}}^{(\eta)}(t)\coloneqq-\int_{0}^{t}\theta_{z\to Q_{b}^{(\eta)\sharp}}(t')dt',\label{eq:def-gamma-eta-cor}
\end{equation}
where $\theta_{z\to Q_{b}^{(\eta)\sharp}}$ is as in \eqref{eq:def-theta-cor}
and \eqref{eq:def-theta-cor-eta}. When $\eta=0$, we abbreviate it
as $\gamma_{\mathrm{cor}}\coloneqq\gamma_{\mathrm{cor}}^{(\eta)}$.
\begin{prop}
\label{prop:prop4.4}There exist $-1<t_{0}^{\ast}<0$ and $0<\eta^{\ast}<1$
such that for all $\eta\in(0,\eta^{\ast}]$ and all sufficiently small
$0<\alpha^{\ast}<1$, we have the following property. Let an $m$-equivariant
profile $z^{\ast}$ satisfy \eqref{eq:H} and $z(t,x)$ solve \eqref{eq:fCSS}
with the initial data $z(0,x)=z^{\ast}(x)$. Then the solution $u$
to \eqref{eq:CSS} with the initial data 
\[
u^{(\eta)}(0,x)=\frac{1}{\eta}Q^{(\eta)}\Big(\frac{x}{\eta}\Big)+z^{\ast}(x)
\]
is global-in-time and scatters. Moreover, $u^{(\eta)}$ admits a decomposition
on time interval $[t_{0}^{\ast},-t_{0}^{\ast}]$\footnote{In fact, $\epsilon^{\sharp}$ and modulation parameters $b$, $\lambda$,
and $\gamma$ depend on $\eta$. We suppress the dependence on $\eta$
to simplify the notations.} 
\[
u^{(\eta)}(t,x)=\frac{e^{i\gamma(t)}}{\lambda(t)}Q_{b(t)}^{(\eta)}\Big(\frac{x}{\lambda(t)}\Big)+z(t,x)+\epsilon^{\sharp}(t,x)
\]
with the estimates 
\begin{multline*}
\langle t\rangle^{-\frac{1}{4}}\|\epsilon^{\sharp}\|_{L^{2}}+\|\epsilon^{\sharp}\|_{\dot{H}_{m}^{1}}+\Big|\frac{\lambda}{\langle t\rangle}-1\Big|+\frac{|b+t|}{\lambda}+|\gamma-\gamma_{\eta}-\gamma_{\mathrm{cor}}^{(\eta)}|\\
\lesssim\alpha^{\ast}(\langle t\rangle^{m+1}+\langle t\rangle^{\frac{1}{4}}\eta^{\frac{3}{4}}),
\end{multline*}
where $\gamma_{\eta}$ is as in \eqref{eq:exact-mod-law-1} and $\langle t\rangle=(t^{2}+\eta^{2})^{\frac{1}{2}}$
is as in \eqref{eq:def-jap-brac}.
\end{prop}

\begin{proof}[Proof of Proposition \ref{prop:prop4.4} assuming the bootstrap lemma
(Lemma \ref{lem:bootstrap}) below]

Note that $z$ is a global small solution (Section \ref{subsec:Evolution-of-z})
with $z(0,x)=z^{\ast}(x)$. We will choose $\eta^{\ast}\in(0,\eta^{\ast\ast}]$
later.

For $\eta\in(0,\eta^{\ast})$, let $u^{(\eta)}:(T_{-},T_{+})\times\R^{2}\to\C$
be the maximal lifespan solution to \eqref{eq:CSS} with the initial
data at time $t=0$ as 
\[
u^{(\eta)}(0,x)=\frac{1}{\eta}Q^{(\eta)}\Big(\frac{x}{\eta}\Big)+z^{\ast}(x).
\]
With the initial data 
\[
(b(0),\lambda(0),\gamma(0))=(0,\eta,0),
\]
let $(b,\lambda,\gamma):(\tilde T_{-},0]\to\R\times\R_{+}\times\R$
be the maximal lifespan solution to the system \eqref{eq:system-ODE}.
Note that $T_{-}\leq\tilde T_{-}<0$. Thus we have a decomposition
of $u^{(\eta)}$ on $(\tilde T_{-},0]$ satisfying 
\begin{equation}
\left\{ \begin{aligned} & u^{(\eta)}(t,x)=\frac{e^{i\gamma(t)}}{\lambda(t)}Q_{b(t)}^{(\eta)}\Big(\frac{x}{\lambda(t)}\Big)+z(t,x)+\epsilon^{\sharp}(t,x),\\
 & (\epsilon,[\mathcal{Z}_{\re}]_{b})_{r}=(\epsilon,i[\mathcal{Z}_{\im}]_{b})_{r}=0,\\
 & 2\Big(\frac{\lambda_{s}}{\lambda}+b\Big)b-(b_{s}+b^{2}+\eta^{2})=0.
\end{aligned}
\right.\label{eq:decomp-u}
\end{equation}
We can now introduce our main bootstrap lemma.
\begin{lem}[Main bootstrap]
\label{lem:bootstrap}There exists $t_{0}^{\ast}<0$ such that for
all sufficiently small $\eta>0$ and $\alpha^{\ast}>0$, we have the
following property. Assume that the decomposition \eqref{eq:decomp-u}
of $u^{(\eta)}$ satisfies the weak bootstrap hypothesis 
\begin{equation}
\sup_{t\in[t_{1},0]}\bigg\{\Big|\frac{\lambda}{\langle t\rangle}-1\Big|+\frac{|b-|t||}{\lambda}\bigg\}\leq\frac{1}{2}\quad\text{and}\quad\sup_{t\in[t_{1},0]}\bigg\{\frac{\|\epsilon\|_{\dot{H}_{m}^{1}}}{\lambda^{\frac{3}{2}}}+\frac{\|\epsilon\|_{L^{2}}}{\lambda^{\frac{1}{2}}}\bigg\}\leq1\label{eq:bootstrap-hyp}
\end{equation}
for some $t_{1}\in[t_{0}^{\ast},0]\cap(\tilde T_{-},0]$. Then, the
following strong conclusion holds for all $t\in[t_{1},0]$: 
\begin{equation}
\left\{ \begin{aligned}\Big|\frac{\lambda}{\langle t\rangle}-1\Big|+\frac{|b-|t||}{\lambda}+|\gamma-\gamma_{\eta}-\gamma_{\mathrm{cor}}^{(\eta)}| & \lesssim\alpha^{\ast}(\lambda^{m+1}+\lambda^{\frac{1}{4}}\eta^{\frac{3}{4}}),\\
\lambda^{\frac{3}{4}}\|\epsilon\|_{L^{2}}+\|\epsilon\|_{\dot{H}_{m}^{1}} & \lesssim\alpha^{\ast}(\lambda^{m+2}+\lambda^{\frac{5}{4}}\eta^{\frac{3}{4}}),
\end{aligned}
\right.\label{eq:bootstrap-conc}
\end{equation}
where $\gamma_{\eta}$ as in \eqref{eq:exact-mod-law-1} and \eqref{eq:def-theta^eta}.
\end{lem}

Let $t_{0}^{\ast}$ be as in Lemma \ref{lem:bootstrap}. Let $\eta$
and $\alpha^{\ast}$ be sufficiently small to satisfy the hypothesis
of Lemma \ref{lem:bootstrap}. As $b(0)=0$, $\lambda(0)=\eta$, and
$\epsilon(0)=0$, the bootstrap hypothesis \eqref{eq:bootstrap-hyp}
is satisfied in a neighborhood of the time $0$. By a standard continuity
argument, the bootstrap conclusion \eqref{eq:bootstrap-conc} is satisfied
on the time interval $[t_{0}^{\ast},0]\cap(\tilde T_{-},0]$. In particular,
the modulation parameters $b$, $\lambda$, and $\gamma$ do not blow
up on $[t_{0}^{\ast},0]\cap(\tilde T_{-},0]$ and hence $T_{-}\leq\tilde T_{-}<t_{0}^{\ast}$.

Finally, to see that $u^{(\eta)}$ scatters backward in time, observe
that 
\begin{align*}
 & \Big\| u^{(\eta)}(t_{0}^{\ast})-\frac{e^{i\gamma(t_{0}^{\ast})}}{|t_{0}^{\ast}|}Q_{|t_{0}^{\ast}|}\Big(\frac{x}{|t_{0}^{\ast}|}\Big)\Big\|_{H^{1}}\\
 & \leq\|\epsilon^{\sharp}(t_{0}^{\ast})\|_{H^{1}}+\|z(t_{0}^{\ast})\|_{H^{1}}+\Big\| e^{i\gamma(t_{0}^{\ast})}\Big\{\frac{1}{\lambda(t_{0}^{\ast})}Q_{b(t_{0}^{\ast})}^{(\eta)}\Big(\frac{x}{\lambda(t_{0}^{\ast})}\Big)-\frac{1}{|t_{0}^{\ast}|}Q_{|t_{0}^{\ast}|}\Big(\frac{x}{|t_{0}^{\ast}|}\Big)\Big\}\Big\|_{H^{1}}.
\end{align*}
The first two terms are bounded by $\alpha^{\ast}$. Since $t_{0}^{\ast}$
is fixed, the last term goes to zero by shrinking $\alpha^{\ast}$
and $\eta$ using \eqref{eq:bootstrap-conc} and \eqref{eq:est-Q-eta-Q}.
This determines the smallness of $\eta^{\ast}$. Hence we can apply
the standard perturbation theory of scattering solutions (Proposition
\ref{prop:L2-Cauchy}). Thus $u$ scatters backward in time.

So far, we have discussed the backward evolution of $u^{(\eta)}$.
To show the forward scattering of $u^{(\eta)}$, we use the time reversal
symmetry $u(t,r)\mapsto\overline{u(-t,r)}$ of \eqref{eq:CSS}.\footnote{This indeed matches the time reversal symmetry mentioned in Introduction.}
Now the initial data has changed into $v^{(\eta)}(0,r)\coloneqq\frac{1}{\eta}Q^{(\eta)}(\frac{r}{\eta})+\overline{z^{\ast}(r)}$.
It is easy to check that $\overline{z^{\ast}(r)}e^{im\theta}$ satisfies
\eqref{eq:H} with the same $\alpha^{\ast}$. Let $v^{(\eta)}(t,r)$
be the backward evolution of \eqref{eq:CSS} with the initial data
$v^{(\eta)}(0,r)$. All the above properties are still satisfied by
$v^{(\eta)}(t,r)$ for $t<0$. Therefore, we can transfer them to
the forward evolution $u^{(\eta)}(-t,r)=\overline{v^{(\eta)}(t,r)}$
for $t<0$.
\end{proof}
For each $\eta\in(0,\eta^{\ast}]$, let $u^{(\eta)}$ be the solution
constructed in Proposition \ref{prop:prop4.4}. In order to use a
limiting argument, we need precompactness of the solutions $\{u^{(\eta)}\}_{\eta\in(0,\eta^{\ast}]}$.
\begin{lem}[$L^{2}$ precompactness]
\label{lem:pre-compact}The set $\{u^{(\eta)}(t_{0}^{\ast})\}_{\eta\in(0,\eta^{\ast}]}\subset L_{m}^{2}$
is precompact in $L_{m}^{2}$.
\end{lem}

\begin{proof}
By the estimates in Proposition \ref{prop:prop4.4}, the set $\{u^{(\eta)}(t_{0}^{\ast})\}_{\eta\in(0,\eta^{\ast}]}$
is bounded in $H_{m}^{1}$. Thus it suffices to show that it is spatially
localized, i.e. 
\[
\lim_{R\to\infty}\sup_{\eta\in(0,\eta^{\ast}]}\|u^{(\eta)}(t_{0}^{\ast})\|_{L^{2}(r\geq R)}=0.
\]
To achieve this, we use local conservation laws \eqref{eq:local-cons}
to obtain 
\[
\partial_{t}\Big(\int(1-\chi_{R})|u^{(\eta)}|^{2}\Big)=-2\int(\partial_{r}\chi_{R})\Im(\overline{u^{(\eta)}}\partial_{r}u^{(\eta)}),
\]
where $\chi_{R}$ is a smooth cutoff to the region $r\lesssim R$,
i.e. $\chi_{R}(r)=\chi(\frac{r}{R})$ for a fixed smooth cutoff $\chi$.
Using the decomposition of $u^{(\eta)}$, we have 
\[
\Big|\int(\partial_{r}\chi_{R})\Im(\overline{u^{(\eta)}(t)}\partial_{r}u^{(\eta)}(t))\Big|\lesssim\frac{1}{R}.
\]
Since 
\begin{align*}
\|u^{(\eta)}(0)\|_{L^{2}(r\gtrsim R)} & \lesssim R^{-(m+1)}+\|z^{\ast}\|_{L^{2}(r\gtrsim R)},
\end{align*}
an application of the fundamental theorem of calculus shows that 
\[
\int_{\R^{2}}(1-\chi_{R})|u^{(\eta)}(t_{0}^{\ast})|^{2}\lesssim\sup_{t\in[t_{0}^{\ast},0]}\|z(t)\|_{L^{2}(r\gtrsim R)}^{2}+\frac{1}{R}.
\]
Using the fact that $z:[t_{0}^{\ast},0]\to L_{m}^{2}$ is a continuous
path (and hence has the compact image in $L_{m}^{2}$), the conclusion
follows.
\end{proof}
We now complete the proof of Theorem \ref{thm:BW-sol}.
\begin{proof}[Proof of Theorem \ref{thm:BW-sol}]

Choose a sequence $\eta_{n}\to0$ and let $u^{(n)}=u^{(\eta_{n})}$
be the associated solution to Proposition \ref{prop:prop4.4}. We
extract a solution from taking limit of solutions $u^{(n)}$. From
$H_{m}^{1}$ boundedness and $L_{m}^{2}$ pre-compactness (Lemma \ref{lem:pre-compact}),
there exists a subsequence (still denoted by $u^{(n)}$) such that
$u^{(n)}(t_{0}^{\ast})\rightharpoonup u(t_{0}^{\ast})$ weakly in
$H_{m}^{1}$ and $u^{(n)}(t_{0}^{\ast})\to u(t_{0}^{\ast})$ strongly
in $L_{m}^{2}$. Interpolating these, $u^{(n)}(t_{0}^{\ast})\to u(t_{0}^{\ast})$
strongly in $H_{m}^{1-}$. Passing to a further subsequence, we may
assume that $b^{(n)}(t_{0}^{\ast})$, $\lambda^{(n)}(t_{0}^{\ast})$,
and $\gamma^{(n)}(t_{0}^{\ast})$ converge to some limit $b(t_{0}^{\ast})$,
$\lambda(t_{0}^{\ast})$, and $\gamma(t_{0}^{\ast})$. Note that $\lambda(t_{0}^{\ast})\approx b(t_{0}^{\ast})\approx|t_{0}^{\ast}|$
and $\gamma(t_{0}^{\ast})\approx-\frac{\pi}{2}(\frac{m+1}{m})$ by
\eqref{eq:bootstrap-conc}. Let $u:(T_{-},T_{+})\times\R^{2}\to\C$
be the maximal lifespan solution with the initial data $u(t_{0}^{\ast})$
at time $t=t_{0}^{\ast}$.

We now show that $(T_{-},T_{+})=(-\infty,0)$ and $u$ scatters backward
in time. As $u(t_{0}^{\ast})$ is close to $\frac{e^{i\gamma(t_{0}^{\ast})}}{|t_{0}^{\ast}|}Q_{|t_{0}^{\ast}|}(\frac{x}{|t_{0}^{\ast}|})$
in $L^{2}$, we have $T_{-}=-\infty$ and $u$ scatters backward in
time. On the other hand, for all $t_{0}^{\ast}<t<\min\{T_{+},0\}$,
we have $u^{(n)}(t)\to u(t)$ strongly in $H_{m}^{1-}$. Since $\|u^{(n)}(t)\|_{H^{1-}}\sim|t^{2}+\eta_{n}^{2}|^{-(\frac{1}{2}-)}$,
we conclude that $T_{+}=0$ by the $H_{m}^{1-}$-subcritical local
well-posedness (Proposition \ref{prop:Hs-Cauchy}).

So far, we only know that $u$ is a $H_{m}^{1}$-solution; we do not
know that $u$ really inherits the decomposition estimates of $u^{(n)}$
due to the weak convergence. Here, we will show that we can transfer
the $H_{m}^{1}$-bound of $(\epsilon^{\sharp})^{(n)}$ to our limit
solution. Let us evolve $b$, $\lambda$, and $\gamma$ under \eqref{eq:system-ODE}
for $\eta=0$ with the initial data $b(t_{0}^{\ast})$, $\lambda(t_{0}^{\ast})$,
and $\gamma(t_{0}^{\ast})$ at time $t=t_{0}$. Since $u^{(n)}\to u$
in $C_{(-\infty,0),\mathrm{loc}}H_{m}^{1-}$ and $(b^{(n)}(t_{0}^{\ast}),\lambda^{(n)}(t_{0}^{\ast}),\gamma(t_{0}^{\ast}))\to(b(t_{0}^{\ast}),\lambda(t_{0}^{\ast}),\gamma(t_{0}^{\ast}))$,
we see that $b(t)$, $\lambda(t)$, and $\gamma(t)$ indeed exist
for all $t\in[t_{0}^{\ast},0)$ and $(b^{(n)}(t),\lambda^{(n)}(t),\gamma^{(n)}(t))\to(b(t),\lambda(t),\gamma(t))$.
For each fixed $t\in[t_{0}^{\ast},0)$, we note that 
\begin{equation}
\gamma(t)=\lim_{n\to\infty}\gamma^{(\eta_{n})}(t)=\lim_{n\to\infty}\gamma_{\eta_{n}}(t)+O(\alpha^{\ast}|t|)=-(m+1)\frac{\pi}{2}+O(\alpha^{\ast}|t|),\label{eq:phase-rot}
\end{equation}
where we estimated $\gamma_{\mathrm{cor}}$ by \eqref{eq:def-gamma-eta-cor}
and \eqref{eq:theta-est}. We write 
\[
u(t,x)=\frac{1}{\lambda(t)}Q_{b(t)}\Big(\frac{x}{\lambda(t)}\Big)e^{i\gamma(t)}+z(t,x)+\epsilon^{\sharp}(t,x).
\]
Thus $(\epsilon^{\sharp})^{(n)}(t)\to\epsilon^{\sharp}(t)$ strongly
in $L_{m}^{2}$. By the Fatou property, the $H_{m}^{1}$-bound of
$\epsilon^{\sharp}$ is now transferred from that of $(\epsilon^{\sharp})^{(n)}$.
Therefore, we have shown that 
\[
|t|^{-\frac{1}{4}}\|\epsilon^{\sharp}\|_{L^{2}}+\|\epsilon^{\sharp}\|_{\dot{H}_{m}^{1}}+\Big|\frac{\lambda}{|t|}-1\Big|+\frac{|b+t|}{\lambda}+|\gamma+(m+1)\frac{\pi}{2}-\gamma_{\mathrm{cor}}|\lesssim\alpha^{\ast}|t|^{m+1}.
\]

We now show \eqref{eq:BW-sol-temp-1}. We first claim \eqref{eq:solve-law}
\[
b(t)=|t|^{-1}\lambda^{2}(t).
\]
To see this, we take a limit $\eta\to0$ of \eqref{eq:law-fixation}
to get 
\[
2\Big(\frac{\lambda_{s}}{\lambda}+b\Big)b-(b_{s}+b^{2})=0.
\]
Because of \eqref{eq:solvable}, the quantity
\[
\frac{\lambda^{2}}{b}+t
\]
is conserved. Since $\frac{\lambda^{2}}{b}+t\to0$ as $t\to0^{-}$,
the claim follows. Therefore, we apply Lemma \ref{lem:est-diff} to
have 
\begin{align*}
 & \|u(t,r)-\frac{1}{|t|}Q_{|t|}\Big(\frac{r}{|t|}\Big)e^{i(-(m+1)\frac{\pi}{2}+\gamma_{\mathrm{cor}})}-z(t,r)\|_{\dot{H}_{m}^{1}}\\
 & \leq\|\epsilon^{\sharp}\|_{\dot{H}_{m}^{1}}+\frac{1}{|t|}\|\frac{|t|}{\lambda}Q_{b}\Big(\frac{|t|}{\lambda}r\Big)e^{i(\gamma+(m+1)\frac{\pi}{2}-\gamma_{\mathrm{cor}})}-Q_{|t|}(r)\|_{\dot{H}_{m}^{1}}\\
 & \lesssim\|\epsilon^{\sharp}\|_{\dot{H}_{m}^{1}}+\frac{1}{|t|}\Big(\Big|\frac{\lambda}{|t|}-1\Big|+|\gamma+(m+1)\frac{\pi}{2}-\gamma_{\mathrm{cor}}|\Big)\\
 & \lesssim\alpha^{\ast}|t|^{m}.
\end{align*}
One can similarly estimate 
\begin{align*}
 & \|u(t,r)-\frac{1}{|t|}Q_{|t|}\Big(\frac{r}{|t|}\Big)e^{i(-(m+1)\frac{\pi}{2}+\gamma_{\mathrm{cor}})}-z(t,r)\|_{L^{2}}\\
 & \lesssim\|\epsilon^{\sharp}\|_{L^{2}}+\Big(\Big|\frac{\lambda}{|t|}-1\Big|+|\gamma+(m+1)\frac{\pi}{2}-\gamma_{\mathrm{cor}}|\Big)\\
 & \lesssim\alpha^{\ast}|t|^{m+1}.
\end{align*}
To complete the proof of Theorem \ref{thm:BW-sol}, we apply the above
procedure for the reversely rotated $e^{-i\frac{(m+1)\pi}{2}}z^{\ast}$
instead of $z^{\ast}$. We then apply the phase invariance $u(t,r)\mapsto e^{i\frac{(m+1)\pi}{2}}u(t,r)$
to get \eqref{eq:BW-sol}.
\end{proof}
\begin{rem}[Alternative proof of Theorem \ref{thm:BW-sol}]
\label{rem:alt-proof}In fact, in order to construct a pseudoconformal
blow-up solution, one can proceed all the above analysis only with
$\eta=0$. When $\eta=0$, we cannot impose the initial data as in
\eqref{eq:init-data}. Instead, one can consider a sequence of solutions
$\{u^{(n)}\}_{n\in\N}$ with the initial data at $t=t_{n}$ as 
\[
u^{(n)}(t_{n},r)=\frac{1}{|t_{n}|}Q_{|t_{n}|}\Big(\frac{r}{|t_{n}|}\Big)+z(t_{n},r),
\]
where $t_{0}^{\ast}<t_{n}<0$ and $t_{n}\to0^{-}$. One then takes
the limit $n\to\infty$.
\end{rem}

We defer the proof of Theorem \ref{thm:cond-uniq} to Section \ref{sec:cond-uniq}.
We now prove Theorem \ref{thm:instability} and Corollary \ref{cor:corollary}.
\begin{proof}[Proof of Theorem \ref{thm:instability}]
The proof is almost immediate by combining Proposition \ref{prop:prop4.4},
and Theorems \ref{thm:BW-sol} and \ref{thm:cond-uniq}. One caveat
is that one should take into account the phase rotation in \eqref{eq:phase-rot}
as in the end of the proof of Theorem \ref{thm:BW-sol}. Indeed, we
apply Proposition \ref{prop:prop4.4} to obtain a family of solutions
$\{u^{(\eta)}\}_{\eta\in(0,\eta^{\ast}]}$ with initial data 
\[
u^{(\eta)}(0,x)=\frac{1}{\eta}Q^{(\eta)}\Big(\frac{x}{\eta}\Big)e^{i(m+1)\frac{\pi}{2}}+z^{\ast}(x).
\]
Let $u$ be the solution constructed in Theorem \ref{thm:BW-sol}.
For any sequence $\eta_{n}\to0$ in $(0,\eta^{\ast}]$, there exists
a further sequence $\{\eta_{n_{k}}\}$ such that $u^{(\eta_{n_{k}})}\to\tilde u$
in the $C_{(-\infty,0),\mathrm{loc}}H_{m}^{1-}$-topology. By the
proof of Theorem \ref{thm:BW-sol}, $\tilde u$ satisfies \eqref{eq:BW-sol-temp-1}.
Applying Theorem \ref{thm:cond-uniq}, we should have $\tilde u=u$.
Therefore, $u^{(\eta)}\to u$ in the $C_{(-\infty,0),\mathrm{loc}}H_{m}^{1-}$-topology.
The rest of the conclusions are contained in Proposition \ref{prop:prop4.4}.
\end{proof}
\begin{proof}[Proof of Corollary \ref{cor:corollary}]
If $\eta>0$, then $u^{(\eta)}$ is well-defined on $(-\infty,0]$
and $\|u^{(\eta)}\|_{L_{(-\infty,0],x}^{4}}<\infty$. Therefore, $\|\mathcal{C}u^{(\eta)}\|_{L_{(0,\infty),x}^{4}}<\infty$
and hence $\mathcal{C}u^{(\eta)}$ scatters forward in time. If $\eta=0$,
we have 
\[
\|\mathcal{C}u-Q-\mathcal{C}z\|_{L^{2}}\lesssim\alpha^{\ast}|t|^{-1}\qquad\text{as }t\to+\infty.
\]
Since $z$ evolves under the $-(m+2)$-equivariant \eqref{eq:CSS},
$\mathcal{C}z$ scatters under the evolution $i\partial_{t}+\Delta_{-m-2}$.
By the remark below, the proof is complete.
\end{proof}
\begin{rem}
\label{rem:A-scattering-solution}A scattering solution under $i\partial_{t}+\Delta_{-m-2}$
also scatters under $i\partial_{t}+\Delta_{m}$. In view of $\Delta_{-m-2}=\Delta_{m}-\frac{4m+4}{r^{2}}$
and $m\neq0$, the argument in \cite{BurqPlanchonStalkerTahvildar-Zadeh2003JFA}
shows $\|\frac{4m+4}{r^{2}}u\|_{L_{t,x}^{4/3}}<\infty$ whenever $u$
is a linear solution to $i\partial_{t}u+\Delta_{-m-2}u=0$ with $L^{2}$-initial
data.
\end{rem}

\emph{Therefore, it only remains to prove Lemma \ref{lem:bootstrap}
and the conditional uniqueness Theorem \ref{thm:cond-uniq}. In Section
\ref{sec:Bootstrap}, we prove Lemma \ref{lem:bootstrap}; we assume
our bootstrap hypothesis \eqref{eq:bootstrap-hyp} and proceed to
prove the bootstrap conclusion \eqref{eq:bootstrap-conc}. In Section
\ref{sec:cond-uniq}, we prove Theorem \ref{thm:cond-uniq}.}

\section{\label{sec:Bootstrap}Proof of bootstrap Lemma \ref{lem:bootstrap}}

The main ingredient of the proof is the Lyapunov/virial functional
method as in Martel \cite{Martel2005AJM}, Rapha\"el-Szeftel \cite{RaphaelSzeftel2011JAMS},
and Merle-Rapha\"el-Szeftel \cite{MerleRaphaelSzeftel2013AJM}. Let
$u$ be a solution on $(t_{0}^{\ast\ast},t_{0}]$ satisfying the bootstrap
hypothesis \eqref{eq:bootstrap-hyp} on $[t_{1},t_{0}]\subset(t_{0}^{\ast\ast},t_{0}]$.
Note that $u$ already has the decomposition \eqref{eq:decomp-u}
satisfying \eqref{eq:orthog} and \eqref{eq:law-fixation}. We will
discuss the details in Section \ref{subsec:virial-lyapunov}.

Recall \eqref{eq:e-prem}. We start with estimates of $\tilde R_{Q_{b}^{(\eta)},z^{\flat}}$
and $R_{u^{\flat}-w^{\flat}}$.

\subsection{\label{subsec:est-RQ,z}Estimates of $\protect\tilde R_{Q_{b}^{(\eta)},z^{\flat}}$}

In Section \ref{subsec:Evolution-of-z}, we have extracted the strong
interactions between $Q_{b}^{(\eta)}$ and $z^{\flat}$, which are
$V_{Q_{b}^{(\eta)}\to z^{\flat}}z^{\flat}$ and $\theta_{z^{\flat}\to Q_{b}^{(\eta)}}Q_{b}^{(\eta)}$.
The remaining part 
\[
\tilde R_{Q_{b}^{(\eta)},z^{\flat}}=R_{Q_{b}^{(\eta)},z^{\flat}}-V_{Q_{b}^{(\eta)}\to z^{\flat}}z^{\flat}-\theta_{z^{\flat}\to Q_{b}^{(\eta)}}Q_{b}^{(\eta)}
\]
will be estimated as an error. Note that this is independent of $\epsilon$
and hence eventually gives an error bound of $\epsilon$.

In estimates of the size of $\tilde R_{Q_{b}^{(\eta)},z^{\flat}}$,
the main strategy is to exploit decoupling of $Q_{b}^{(\eta)}$ and
$z^{\flat}$ in their scales as in \cite{BourgainWang1997}. It is
convenient to introduce functions $F_{n,\ell}$ and $G_{k}$ on $(0,\infty)$
such that 
\begin{equation}
F_{n,\ell}\coloneqq\begin{cases}
r^{n} & \text{if }r\leq1,\\
r^{-\ell} & \text{if }r>1,
\end{cases}\label{eq:def-Fnl}
\end{equation}
and 
\begin{equation}
G_{k}\coloneqq\begin{cases}
|\lambda r|^{k} & \text{if }r\leq\lambda^{-1},\\
1 & \text{if }r>\lambda^{-1}.
\end{cases}\label{eq:def-Gk}
\end{equation}
We note that $|Q_{b}^{(\eta)}|\lesssim F_{m,m+2}$ and $|z^{\flat}|\lesssim\alpha^{\ast}\lambda G_{m+2}$.
As $F$ and $G$ lie in different scales, we have the following decoupling
lemma:
\begin{lem}[Estimates for decoupling at different scales]
\label{lem:good-interaction}Let $1\leq p\leq\infty$; let $k,\ell,n\in\R$
be such that $\ell>\frac{2}{p}$ and $n+k+\frac{2}{p}>0$. Then, 
\[
\|F_{n,\ell}G_{k}\|_{L^{p}}\lesssim_{k,\ell,n,p}\lambda^{\min\{k,\ell-\frac{2}{p}\}}.
\]
\end{lem}

\begin{proof}
The case of $p=\infty$ is immediate. For $1\leq p<\infty$, we observe
\begin{align*}
\int_{0}^{\infty}F_{n,\ell}^{p}G_{k}^{p}rdr & \lesssim\int_{0}^{1}|\lambda r|^{pk}r^{pn}rdr+\int_{1}^{\lambda^{-1}}|\lambda r|^{pk}r^{-p\ell}rdr+\int_{\lambda^{-1}}^{\infty}r^{-p\ell}rdr\\
 & \lesssim_{k,\ell,n,p}\lambda^{pk}+\lambda^{p\ell-2}.
\end{align*}
This completes the proof.
\end{proof}
We will use Lemma \ref{lem:good-interaction} to estimate $\tilde R_{Q_{b}^{(\eta)},z^{\flat}}$.
This is nothing but a rigorous justification of the heuristics in
Section \ref{subsec:Evolution-of-z}. For this purpose, we recall
\[
L_{u}^{\ast}\D_{+}^{(u)}u=-(\partial_{rr}+\frac{1}{r}\partial_{r})u+\Big(\frac{m+A_{\theta}[u]}{r}\Big)^{2}u+A_{0}[u]u-|u|^{2}u.
\]
Following the heuristics in Section \ref{subsec:Evolution-of-z},
we reorganize $\tilde R_{Q_{b}^{(\eta)},z^{\flat}}$ as 
\begin{align*}
\tilde R_{Q_{b}^{(\eta)},z^{\flat}} & =L_{Q_{b}^{(\eta)}+z^{\flat}}^{\ast}\D_{+}^{(Q_{b}^{(\eta)}+z^{\flat})}(Q_{b}^{(\eta)}+z^{\flat})-L_{Q_{b}^{(\eta)}}^{\ast}\D_{+}^{(Q_{b}^{(\eta)})}Q_{b}^{(\eta)}-L_{z^{\flat}}^{\ast}\D_{+}^{(z^{\flat})}z^{\flat}\\
 & \quad-V_{Q_{b}^{(\eta)}\to z^{\flat}}z^{\flat}-\theta_{z^{\flat}\to Q_{b}^{(\eta)}}Q_{b}^{(\eta)}\\
 & =\eqref{eq:R-int-1}+\eqref{eq:R-int-2}+\eqref{eq:R-int-3},
\end{align*}
where 
\begin{align}
 & \left\{ \begin{aligned} & \Big(\frac{m+A_{\theta}[Q^{(\eta)}+z^{\flat}]}{r}\Big)^{2}(Q^{(\eta)}+z^{\flat})-\Big(\frac{m+A_{\theta}[Q^{(\eta)}]}{r}\Big)^{2}Q^{(\eta)}\\
 & \quad-\Big(\frac{m+A_{\theta}[Q^{(\eta)}](+\infty)+A_{\theta}[z^{\flat}]}{r}\Big)^{2}z^{\flat},
\end{aligned}
\right.\label{eq:R-int-1}\\
 & \left\{ \begin{aligned} & A_{0}[Q_{b}^{(\eta)}+z^{\flat}](Q_{b}^{(\eta)}+z^{\flat})-A_{0}[Q_{b}^{(\eta)}]Q_{b}^{(\eta)}\\
 & +\Big(\int_{r}^{\infty}(m+A_{\theta}[Q^{(\eta)}](+\infty)+A_{\theta}[z^{\flat}])|z^{\flat}|^{2}\frac{dr'}{r'}\Big)(Q_{b}^{(\eta)}+z^{\flat})\\
 & +\Big(\int_{0}^{r}(m+A_{\theta}[Q^{(\eta)}](+\infty)+A_{\theta}[z^{\flat}])|z^{\flat}|^{2}\frac{dr'}{r'}\Big)Q_{b}^{(\eta)},
\end{aligned}
\right.\label{eq:R-int-2}\\
 & -|Q_{b}^{(\eta)}+z^{\flat}|^{2}(Q_{b}^{(\eta)}+z^{\flat})+|Q_{b}^{(\eta)}|^{2}Q_{b}^{(\eta)}+|z^{\flat}|^{2}z^{\flat}.\label{eq:R-int-3}
\end{align}
Here, \eqref{eq:R-int-1}, \eqref{eq:R-int-2}, and \eqref{eq:R-int-3}
collect $(\frac{m+A_{\theta}[u]}{r})^{2}u$, $A_{0}[u]$, and $-|u|^{2}u$
parts of $L_{u}^{\ast}\D_{+}^{(u)}u$, respectively. The potential
$V_{Q_{b}^{(\eta)}\to z^{\flat}}$ is distributed into \eqref{eq:R-int-1}
and \eqref{eq:R-int-2}, and $\theta_{z^{\flat}\to Q_{b}^{(\eta)}}$
contributes to \eqref{eq:R-int-2}. As explained in Section \ref{subsec:Evolution-of-z},
in the introduction of $V_{Q_{b}^{(\eta)}\to z^{\flat}}$ and $\theta_{z^{\flat}\to Q_{b}^{(\eta)}}$,
we have approximated $A_{\theta}[Q^{(\eta)}](r)$ by $A_{\theta}[Q^{(\eta)}](+\infty)$,
and in $A_{0}$-term the integral $\int_{r}^{\infty}$ by $\int_{0}^{\infty}$.
It turns out that the difference $A_{\theta}[Q]-A_{\theta}[Q](+\infty)$
exhibits fast decay at infinity and $\int_{0}^{r}$ integral enjoys
degeneracy at the origin. So these differences can be covered by Lemma
\ref{lem:good-interaction}.

The following is the main result of this subsection.
\begin{lem}[Estimates of $\tilde R_{Q_{b}^{(\eta)},z^{\flat}}$]
\label{lem:estimate-RQ,z}We have 
\begin{align}
|\theta_{z^{\flat}\to Q_{b}^{(\eta)}}| & \lesssim(\alpha^{\ast})^{2}\lambda^{2},\label{eq:theta-est}\\
\|\tilde R_{Q_{b}^{(\eta)},z^{\flat}}\|_{H_{m}^{1}} & \lesssim\alpha^{\ast}\lambda^{m+3}|\log\lambda|.\label{eq:est-Rqz}
\end{align}
\end{lem}

\begin{rem}
The log factor of \eqref{eq:est-Rqz} comes from \eqref{eq:lem5.3.8}
below. This factor arises for all $m\geq1$, as opposed to \eqref{eq:Qb-Q}.
\end{rem}

Setting aside the proof, we first collect some estimates of nonlinear
terms in terms of $F$ and $G$.
\begin{lem}[Collection of various estimates]
\label{lem:good-interaction-sub}We have the following.
\begin{equation}
\left\{ \begin{aligned}|Q_{b}^{(\eta)}| & \lesssim F_{m,m+2},\\
|\partial_{r}Q_{b}^{(\eta)}|+|\tfrac{1}{r}Q_{b}^{(\eta)}| & \lesssim F_{m-1,m+3}+\lambda F_{m+1,m+1}.
\end{aligned}
\right.\label{eq:lem5.3.1}
\end{equation}
\begin{equation}
\left\{ \begin{aligned}\big|A_{\theta}[Q^{(\eta)}]|_{r}^{+\infty}| & \eqqcolon|A_{\theta}[Q^{(\eta)}](+\infty)-A_{\theta}[Q^{(\eta)}](r)|\lesssim F_{0,2m+2},\\
|\partial_{r}A_{\theta}[Q^{(\eta)}]| & \lesssim F_{2m+1,2m+3}.
\end{aligned}
\right.\label{eq:lem5.3.2}
\end{equation}
\begin{equation}
\left\{ \begin{aligned}\Big|\int_{r}^{\infty}|Q_{b}^{(\eta)}|^{2}\frac{dr'}{r'}\Big| & \lesssim F_{0,2m+4},\\
\Big|\partial_{r}\int_{r}^{\infty}|Q_{b}^{(\eta)}|^{2}\frac{dr'}{r'}\Big| & \lesssim F_{2m-1,2m+5}.
\end{aligned}
\right.\label{eq:lem5.3.3}
\end{equation}
\begin{equation}
\left\{ \begin{aligned}|z^{\flat}| & \lesssim\alpha^{\ast}\lambda G_{m+2},\\
|\partial_{r}z^{\flat}|+|\tfrac{1}{r}z^{\flat}| & \lesssim\alpha^{\ast}\lambda^{2}G_{m+1}.
\end{aligned}
\right.\label{eq:lem5.3.4}
\end{equation}
\begin{equation}
\left\{ \begin{aligned}|A_{\theta}[z^{\flat}]| & \lesssim(\alpha^{\ast})^{2}G_{2m+6},\\
|\partial_{r}A_{\theta}[z^{\flat}]| & \lesssim(\alpha^{\ast})^{2}\lambda G_{2m+5}.
\end{aligned}
\right.\label{eq:lem5.3.5}
\end{equation}
\begin{equation}
\left\{ \begin{aligned}\int_{0}^{r}|z^{\flat}|^{2}\frac{dr'}{r'} & \lesssim(\alpha^{\ast})^{2}\lambda^{2}G_{2m+4},\\
\Big|\partial_{r}\int_{0}^{r}|z^{\flat}|^{2}\frac{dr'}{r'}\Big| & \lesssim(\alpha^{\ast})^{2}\lambda^{3}G_{2m+3}.
\end{aligned}
\right.\label{eq:lem5.3.6}
\end{equation}
\begin{equation}
\left\{ \begin{aligned}\Big|\int_{0}^{r}\Re(\overline{Q_{b}^{(\eta)}}z^{\flat})r'dr'\Big| & \lesssim\alpha^{\ast}\lambda^{m+1}F_{2m+2,0}G_{2},\\
\Big|\partial_{r}\int_{0}^{r}\Re(\overline{Q_{b}^{(\eta)}}z^{\flat})r'dr'\Big| & \lesssim\alpha^{\ast}\lambda^{m+1}F_{2m+1,1}G_{2}.
\end{aligned}
\right.\label{eq:lem5.3.7}
\end{equation}
\begin{equation}
\left\{ \begin{aligned}\Big|\int_{r}^{\infty}\Re(\overline{Q_{b}^{(\eta)}}z^{\flat})\frac{dr'}{r'}\Big| & \lesssim\alpha^{\ast}\lambda^{m+3}\cdot\begin{cases}
1+|\log\lambda| & \text{if }r\leq1,\\
1+|\log(\lambda r)| & \text{if }1\leq r\leq\lambda^{-1}\\
1 & \text{if }r\geq\lambda^{-1},
\end{cases},\\
\Big|\partial_{r}\int_{r}^{\infty}\Re(\overline{Q_{b}^{(\eta)}}z^{\flat})\frac{dr'}{r'}\Big| & \lesssim\alpha^{\ast}\lambda^{m+3}.
\end{aligned}
\right.\label{eq:lem5.3.8}
\end{equation}
\end{lem}

\begin{proof}
The estimates \eqref{eq:lem5.3.1}-\eqref{eq:lem5.3.3} are immediate
from \eqref{eq:unif-bound} and explicit formula of $Q$. The estimate
\eqref{eq:lem5.3.4} follows from Lemma \ref{lem:gen-hardy-Linfty}.
We turn to \eqref{eq:lem5.3.5}. In the region $r\leq\lambda^{-1}$,
\eqref{eq:lem5.3.5} follows from \eqref{eq:lem5.3.4}. In the region
$r\geq\lambda^{-1}$, we estimate $|A_{\theta}[z^{\flat}]|\lesssim\|z^{\flat}\|_{L^{2}}^{2}\lesssim(\alpha^{\ast})^{2}$
and $|\partial_{r}A_{\theta}[z^{\flat}]|\lesssim|r^{\frac{1}{2}}z^{\flat}|^{2}\lesssim(\alpha^{\ast})^{2}\lambda$
by the Strauss inequality. The estimate \eqref{eq:lem5.3.6} is similar
to \eqref{eq:lem5.3.5}. Finally, the estimates \eqref{eq:lem5.3.7}
and \eqref{eq:lem5.3.8} follow from \eqref{eq:lem5.3.1} and \eqref{eq:lem5.3.4}.
\end{proof}
\begin{rem}
In \eqref{eq:lem5.3.1}-\eqref{eq:lem5.3.8}, the first estimate is
adapted to $L^{2}$-estimate and the second is for $\dot{H}_{m}^{1}$-estimate.
It is worth mentioning that in \eqref{eq:lem5.3.1}, we lose one decay
by taking $\partial_{r}$ but we compensate it by the factor $\lambda$.
In \eqref{eq:lem5.3.4}-\eqref{eq:lem5.3.6}, we lose one $r$-factor
at origin by taking $\partial_{r}$ but we again compensate it by
the factor $\lambda$. For the remaining estimates, $\partial_{r}$-estimate
is even better than the original estimate. This observation allows
us to transfer $L^{2}$-estimate of $\tilde R_{Q_{b}^{(\eta)},z^{\flat}}$
to $H_{m}^{1}$-estimate of $\tilde R_{Q_{b}^{(\eta)},z^{\flat}}$.
\end{rem}

\begin{proof}[Proof of Lemma \ref{lem:estimate-RQ,z}]
Note that \eqref{eq:theta-est} follows from 
\[
|\theta_{z^{\flat}\to Q_{b}^{(\eta)}}|\lesssim\|r^{-1}z^{\flat}\|_{L^{2}}^{2}\lesssim(\alpha^{\ast})^{2}\lambda^{2}.
\]

We turn to \eqref{eq:est-Rqz}. We estimate \eqref{eq:R-int-1}-\eqref{eq:R-int-3}
term by term. We first start with \eqref{eq:R-int-3}. The mixed terms
are estimated as 
\begin{align*}
\|(Q_{b}^{(\eta)})^{2}z^{\flat}\|_{L^{2}} & \lesssim\alpha^{\ast}\lambda\|F_{m,m+2}^{2}G_{m+2}\|_{L^{2}}\lesssim\alpha^{\ast}\lambda^{m+3},\\
\|Q_{b}^{(\eta)}(z^{\flat})^{2}\|_{L^{2}} & \lesssim(\alpha^{\ast}\lambda)^{2}\|F_{m,m+2}G_{m+2}^{2}\|_{L^{2}}\lesssim(\alpha^{\ast})^{2}\lambda^{m+3}.
\end{align*}
To get $\dot{H}_{m}^{1}$-estimate, we take $\partial_{r}$ and $\frac{1}{r}$
to \eqref{eq:R-int-3}. Using the Leibniz rules, \eqref{eq:lem5.3.1},
and \eqref{eq:lem5.3.4}, the same estimate $\alpha^{\ast}\lambda^{m+3}$
follows.

We turn to \eqref{eq:R-int-1}. We expand 
\begin{align*}
\eqref{eq:R-int-1} & =\frac{1}{r^{2}}\Big((m+A_{\theta}[Q^{(\eta)}+z^{\flat}])^{2}-(m+A_{\theta}[Q^{(\eta)}]+A_{\theta}[z^{\flat}])^{2}\Big)(Q_{b}^{(\eta)}+z^{\flat})\\
 & \quad+\frac{1}{r^{2}}\Big((m+A_{\theta}[Q^{(\eta)}]+A_{\theta}[z^{\flat}])^{2}-(m+A_{\theta}[Q^{(\eta)}])^{2}\Big)Q_{b}^{(\eta)}\\
 & \quad+\frac{1}{r^{2}}\Big((m+A_{\theta}[Q^{(\eta)}]+A_{\theta}[z^{\flat}])^{2}-(m+A_{\theta}[Q^{(\eta)}](+\infty)+A_{\theta}[z^{\flat}])^{2}\Big)z^{\flat}.
\end{align*}
We view each line of the form 
\begin{equation}
\frac{1}{r^{2}}(\alpha^{2}-\beta^{2})\varphi=(\alpha+\beta)\cdot\frac{1}{r^{2}}(\alpha-\beta)\varphi.\label{eq:int-temp1}
\end{equation}
We discard $\alpha+\beta$ using $\|\alpha+\beta\|_{L^{\infty}}\lesssim1$.
For the first line, we estimate 
\begin{align*}
\|\frac{1}{r^{2}}(\int_{0}^{r}\Re(\overline{Q_{b}^{(\eta)}}z^{\flat})r'dr')Q_{b}\|_{L^{2}} & \lesssim\alpha^{\ast}\lambda^{m+1}\|F_{3m,m+4}G_{2}\|_{L^{2}}\lesssim\alpha^{\ast}\lambda^{m+3},\\
\|\frac{1}{r^{2}}(\int_{0}^{r}\Re(\overline{Q_{b}^{(\eta)}}z^{\flat})r'dr')z^{\flat}\|_{L^{2}} & \lesssim(\alpha^{\ast})^{2}\lambda^{m+2}\|F_{2m,2}G_{m+4}\|_{L^{2}}\lesssim(\alpha^{\ast})^{2}\lambda^{m+3}.
\end{align*}
For the second line, we estimate 
\[
\|\frac{1}{r^{2}}A_{\theta}[z^{\flat}]Q_{b}^{(\eta)}\|_{L^{2}}\lesssim(\alpha^{\ast})^{2}\|F_{m-2,m+4}G_{2m+6}\|_{L^{2}}\lesssim(\alpha^{\ast})^{2}\lambda^{m+3}.
\]
For the last line, we estimate 
\begin{align*}
\|\frac{1}{r^{2}}(A_{\theta}[Q^{(\eta)}]-A_{\theta}[Q^{(\eta)}](+\infty))z^{\flat}\|_{L^{2}} & =2m\|\frac{1}{r^{2}}(A_{\theta}[Q^{(\eta)}]|_{r}^{+\infty})z^{\flat}\|_{L^{2}}\\
 & \lesssim\alpha^{\ast}\lambda\|F_{-2,2m+4}G_{m+2}\|_{L^{2}}\lesssim\alpha^{\ast}\lambda^{m+3}.
\end{align*}
This concludes the $L^{2}$-estimate of \eqref{eq:R-int-1}. We turn
to the $\dot{H}_{m}^{1}$-estimate. We take $\partial_{r}$ and $\frac{1}{r}$
to \eqref{eq:R-int-1}. If $\partial_{r}$ hits $\alpha+\beta$ part
of \eqref{eq:int-temp1}, we use $\|\partial_{r}(\alpha+\beta)\|_{L^{\infty}}\lesssim1$
followed from Lemma \ref{lem:good-interaction-sub}. If $\partial_{r}$
hits $\frac{1}{r^{2}}$ of \eqref{eq:int-temp1}, then one can move
one $\frac{1}{r}$ to $\varphi$ and apply the derivative estimates
of Lemma \ref{lem:good-interaction-sub}. If $\partial_{r}$ and $\frac{1}{r}$
hits $(\alpha-\beta)\varphi$, then we apply the derivative estimates
of Lemma \ref{lem:good-interaction-sub}. For example, 
\begin{align*}
 & \|\partial_{r}\Big(\frac{1}{r^{2}}A_{\theta}[Q^{(\eta)}]z^{\flat}\Big)\|_{L^{2}}+\|\frac{1}{r}\Big(\frac{1}{r^{2}}A_{\theta}[Q^{(\eta)}]z^{\flat}\Big)\|_{L^{2}}\\
 & \lesssim\|\frac{1}{r^{2}}A_{\theta}[Q^{(\eta)}]\frac{z^{\flat}}{r}\|_{L^{2}}+\|\frac{1}{r^{2}}(\partial_{r}A_{\theta}[Q^{(\eta)}])z^{\flat}\|_{L^{2}}+\|\frac{1}{r^{2}}A_{\theta}[Q^{(\eta)}](\partial_{r}z^{\flat})\|_{L^{2}}\\
 & \lesssim\alpha^{\ast}\lambda^{m+3}.
\end{align*}
The other terms can be treated similarly.

Finally, we consider \eqref{eq:R-int-2}. 
\begin{align*}
 & A_{0}[Q_{b}^{(\eta)}+z^{\flat}](Q_{b}^{(\eta)}+z^{\flat})-A_{0}[Q_{b}^{(\eta)}]Q_{b}^{(\eta)}\\
 & +\Big(\int_{r}^{\infty}(m+A_{\theta}[Q^{(\eta)}](+\infty)+A_{\theta}[z^{\flat}])|z^{\flat}|^{2}\frac{dr'}{r'}\Big)(Q_{b}^{(\eta)}+z^{\flat})\\
 & +\Big(\int_{0}^{r}(m+A_{\theta}[Q^{(\eta)}](+\infty)+A_{\theta}[z^{\flat}])|z^{\flat}|^{2}\frac{dr'}{r'}\Big)Q_{b}^{(\eta)},
\end{align*}
We expand 
\begin{align*}
\eqref{eq:R-int-2} & =-2\Big(\int_{r}^{\infty}(m+A_{\theta}[Q_{b}^{(\eta)}+z^{\flat}])\Re(\overline{Q_{b}^{(\eta)}}z^{\flat})\frac{dr'}{r'}\Big)(Q_{b}^{(\eta)}+z^{\flat})\\
 & \quad+\Big(\int_{r}^{\infty}\big(\int_{0}^{r}\Re(\overline{Q_{b}^{(\eta)}}z^{\flat})r'dr'\big)(|Q_{b}^{(\eta)}|^{2}+|z^{\flat}|^{2})\frac{dr'}{r'}\Big)(Q_{b}^{(\eta)}+z^{\flat})\\
 & \quad+\Big(\int_{r}^{\infty}(A_{\theta}[Q^{(\eta)}](+\infty)-A_{\theta}[Q^{(\eta)}](r))|z^{\flat}|^{2}\frac{dr'}{r'}\Big)(Q_{b}^{(\eta)}+z^{\flat})\\
 & \quad-\Big(\int_{r}^{\infty}(m+A_{\theta}[Q^{(\eta)}]+A_{\theta}[z^{\flat}])|Q_{b}^{(\eta)}|^{2}\frac{dr'}{r'}\Big)(Q_{b}^{(\eta)}+z^{\flat})-A_{0}[Q_{b}^{(\eta)}]Q_{b}^{(\eta)}\\
 & \quad+\Big(\int_{0}^{r}(m+A_{\theta}[Q^{(\eta)}](+\infty)+A_{\theta}[z^{\flat}])|z^{\flat}|^{2}\frac{dr'}{r'}\Big)Q_{b}^{(\eta)}.
\end{align*}
We treat the above line by line. For the first line, we use \eqref{eq:lem5.3.8}
to estimate 
\begin{align}
 & \|(\int_{r}^{\infty}(m+A_{\theta}[Q_{b}^{(\eta)}+z^{\flat}])\Re(\overline{Q_{b}^{(\eta)}}z^{\flat})\tfrac{dr'}{r'})Q_{b}^{(\eta)}\|_{L^{2}}\label{eq:log-div}\\
 & \lesssim\|m+A_{\theta}[Q_{b}^{(\eta)}+z^{\flat}]\|_{L^{\infty}}\|(\int_{r}^{\infty}|Q_{b}^{(\eta)}z^{\flat}|\frac{dr'}{r'})\|_{L^{\infty}}\|Q_{b}^{(\eta)}\|_{L^{2}}\lesssim\alpha^{\ast}\lambda^{m+3}|\log\lambda|\nonumber 
\end{align}
and 
\begin{align*}
 & \|(\int_{r}^{\infty}(m+A_{\theta}[Q_{b}^{(\eta)}+z^{\flat}])\Re(\overline{Q_{b}^{(\eta)}}z^{\flat})\tfrac{dr'}{r'})z^{\flat}\|_{L^{2}}\\
 & \lesssim\|m+A_{\theta}[Q_{b}^{(\eta)}+z^{\flat}]\|_{L^{\infty}}\|(\int_{r}^{\infty}|Q_{b}^{(\eta)}z^{\flat}|\frac{dr'}{r'})z^{\flat}\|_{L^{2}}\lesssim\alpha^{\ast}\lambda^{m+3}.
\end{align*}
Note that \eqref{eq:log-div} is the only term with the logarithmic
factor. For sake of simplicity, we denote by $\psi$ either $Q_{b}^{(\eta)}$
or $z^{\flat}$. For the second line, we estimate 
\begin{align*}
 & \|(\int_{r}^{\infty}(\int_{0}^{r}\Re(\overline{Q_{b}^{(\eta)}}z^{\flat})r'dr')|\psi|^{2}\frac{dr'}{r'})(Q_{b}^{(\eta)}+z^{\flat})\|_{L^{2}}\\
 & \lesssim\alpha^{\ast}\lambda^{m+1}\|(\int_{r}^{\infty}F_{2m+2,0}G_{2}|\psi|^{2}\frac{dr'}{r'})(Q_{b}^{(\eta)}+z^{\flat})\|_{L^{2}}\\
 & \lesssim\alpha^{\ast}\lambda^{m+1}\cdot\begin{cases}
\|F_{3m+2,m+2}G_{2}\|_{L^{2}}\|Q_{b}^{(\eta)}+z^{\flat}\|_{L^{\infty}} & \text{if }\psi=Q_{b}^{(\eta)},\\
\|r^{-1}z^{\flat}\|_{L^{2}}^{2}\|Q_{b}^{(\eta)}+z^{\flat}\|_{L^{2}} & \text{if }\psi=z^{\flat},
\end{cases}\\
 & \lesssim\alpha^{\ast}\lambda^{m+3}.
\end{align*}
For the third line, we estimate
\begin{align*}
 & \|(\int_{r}^{\infty}A_{\theta}[Q^{(\eta)}]|_{r}^{+\infty}|z^{\flat}|^{2}\tfrac{dr'}{r'})(Q_{b}^{(\eta)}+z^{\flat})\|_{L^{2}}\\
 & \lesssim\|A_{\theta}[Q^{(\eta)}]|_{r}^{+\infty}|z^{\flat}|^{2}\|_{L^{2}}\|Q_{b}^{(\eta)}+z^{\flat}\|_{L^{\infty}}\\
 & \lesssim(\alpha^{\ast})^{2}\lambda^{2}\|F_{0,2m+2}G_{2m+4}\|_{L^{2}}\lesssim(\alpha^{\ast})^{2}\lambda^{2m+3}.
\end{align*}
For the fourth line, we estimate 
\[
\|(\int_{r}^{\infty}A_{\theta}[Q^{(\eta)}]|Q_{b}^{(\eta)}|^{2}\tfrac{dr'}{r'})z^{\flat}\|_{L^{2}}\lesssim\alpha^{\ast}\lambda\|F_{0,2m+4}G_{m+2}\|_{L^{2}}\lesssim\alpha^{\ast}\lambda^{m+3}
\]
and 
\begin{align*}
\|(\int_{r}^{\infty}A_{\theta}[z^{\flat}]|Q_{b}^{(\eta)}|^{2}\tfrac{dr'}{r'})(Q_{b}^{(\eta)}+ & z^{\flat})\|_{L^{2}}\lesssim\|A_{\theta}[z^{\flat}]|Q_{b}^{(\eta)}|^{2}\|_{L^{2}}\|Q_{b}^{(\eta)}+z^{\flat}\|_{L^{\infty}}\\
 & \lesssim(\alpha^{\ast})^{2}\|G_{2m+6}F_{2m,2m+4}\|_{L^{2}}\lesssim(\alpha^{\ast})^{2}\lambda^{2m+3}.
\end{align*}
For the last line, we estimate 
\begin{align*}
 & \|(\int_{0}^{r}(m+A_{\theta}[Q^{(\eta)}](+\infty)+A_{\theta}[z^{\flat}])|z^{\flat}|^{2}\tfrac{dr'}{r'})Q_{b}^{(\eta)}\|_{L^{2}}\\
 & \lesssim\|m+A_{\theta}[Q^{(\eta)}](+\infty)+A_{\theta}[z^{\flat}]\|_{L^{\infty}}\|(\int_{0}^{r}|z^{\flat}|^{2}\frac{dr'}{r'})Q_{b}^{(\eta)}\|_{L^{2}}\\
 & \lesssim(\alpha^{\ast})^{2}\lambda^{2}\|F_{m,m+2}G_{2m+4}\|_{L^{2}}\lesssim(\alpha^{\ast})^{2}\lambda^{m+3}.
\end{align*}
In order to get $\dot{H}_{m}^{1}$-estimate, use the algebra 
\[
\partial_{r}\Big(\int_{r}^{\infty}f\frac{dr'}{r'}\Big)=\int_{r}^{\infty}\Big(-(\partial_{r}f)+\frac{f}{r'}\Big)\frac{dr'}{r'}
\]
for the outermost integral $\int_{r}^{\infty}$. One can then apply
the derivative estimates of Lemma \ref{lem:good-interaction-sub}.
This completes the proof.
\end{proof}
Indeed, \eqref{eq:est-Rqz} will be used in various places of the
following sections. In most cases \eqref{eq:est-Rqz} suffices, but
in Section \ref{subsec:virial-lyapunov} we need to improve some estimate
on $\tilde R_{Q_{b}^{(\eta)},z^{\flat}}$ by the $|\log\lambda|$
factor. Inspecting the previous proof, one observes that there is
only one term \eqref{eq:log-div} yielding the factor $|\log\lambda|$.
In the following lemma, we will separate \eqref{eq:log-div} and show
the improved estimate \eqref{eq:improved} below. This is required
to have the uniqueness of the pseudoconformal blow-up solutions constructed
in Theorem \ref{thm:BW-sol}. In fact, without the logarithmic improvement
of the estimate, the error estimate of the energy identity (Lemma
\ref{lem:quad-energy} and \eqref{eq:err}) and hence those of Theorem
\ref{thm:BW-sol} become worse logarithmically. This will not meet
the hypothesis a priori bound of Theorem \ref{thm:cond-uniq} when
$m=1$.
\begin{lem}[Logarithmic improvement]
\label{lem:refined-decomp}Let 
\[
R_{2}\coloneqq-2(\int_{r}^{\infty}(m+A_{\theta}[Q_{b}^{(\eta)}+z^{\flat}])\Re(\overline{Q_{b}^{(\eta)}}z^{\flat})\tfrac{dr'}{r'})Q^{(\eta)}.
\]
Then the decomposition 
\[
\tilde R_{Q_{b}^{(\eta)},z^{\flat}}\eqqcolon R_{1}+[R_{2}]_{b}
\]
satisfies 
\[
\|R_{1}\|_{H_{m}^{1}}\lesssim\alpha^{\ast}\lambda^{m+3}
\]
and 
\begin{align}
\|L_{Q}(iR_{2})\|_{L^{2}}+\|r\partial_{r}L_{Q}(iR_{2})\|_{L^{2}} & \lesssim\alpha^{\ast}\lambda^{m+3},\label{eq:improved}\\
\|rR_{2}\|_{L^{2}}+\|r^{2}\partial_{r}R_{2}\|_{L^{2}} & \lesssim\alpha^{\ast}\lambda^{m+3}|\log\lambda|.\nonumber 
\end{align}
\end{lem}

\begin{proof}
By the proof of \eqref{eq:est-Rqz}, $\|R_{1}\|_{H_{m}^{1}}\lesssim\alpha^{\ast}\lambda^{m+3}$
is done. To prove estimates on $R_{2}$, let
\[
f\coloneqq-2\int_{r}^{\infty}(m+A_{\theta}[Q_{b}^{(\eta)}+z^{\flat}])\Re(\overline{Q_{b}^{(\eta)}}z^{\flat})\frac{dr'}{r'}
\]
so that $R_{2}=fQ^{(\eta)}$. We compute using \eqref{eq:first-order}
that 
\begin{align*}
L_{Q}(f\cdot iQ^{(\eta)}) & =(\partial_{r}f)\cdot iQ^{(\eta)}+f\cdot i\D_{+}^{(Q)}Q^{(\eta)}\\
 & =(\partial_{r}f)\cdot iQ^{(\eta)}+f\cdot\tfrac{1}{r}(A_{\theta}[Q^{(\eta)}]-A_{\theta}[Q])iQ^{(\eta)}-f\cdot i\eta\tfrac{r}{2}Q^{(\eta)}.
\end{align*}
In view of \eqref{eq:lem5.3.8} and $\eta\lesssim\lambda$, we have
\[
\|\partial_{r}f\|_{L^{\infty}}+\|f(A_{\theta}[Q^{(\eta)}]-A_{\theta}[Q])\|_{L^{\infty}}+\|\eta f\|_{L^{\infty}}\lesssim\alpha^{\ast}\lambda^{m+3},
\]
Therefore, 
\[
\|L_{Q}(f\cdot iQ^{(\eta)})\|_{L^{2}}+\|r\partial_{r}L_{Q}(f\cdot iQ^{(\eta)})\|_{L^{2}}\lesssim\alpha^{\ast}\lambda^{m+3}.
\]
The remaining estimates follow from 
\begin{align*}
\|f\|_{L^{\infty}}+\|r\partial_{r}f\|_{L^{\infty}} & \lesssim\alpha^{\ast}\lambda^{m+3}|\log\lambda|,\\
\|rQ^{(\eta)}\|_{L^{2}}+\|r^{2}\partial_{r}Q^{(\eta)}\|_{L^{2}} & \lesssim1.
\end{align*}
\end{proof}
We conclude this subsection with an estimate of $\tilde V_{Q_{b}^{(\eta)}-Q_{b}}$
(recall \eqref{eq:def-V-tilde}).
\begin{lem}[Estimate of $\tilde V_{Q_{b}^{(\eta)}-Q_{b}}$]
We have 
\begin{equation}
\|\tilde V_{Q_{b}^{(\eta)}-Q_{b}}z^{\flat}\|_{H_{m}^{1}}\lesssim\alpha^{\ast}\lambda^{2}\eta.\label{eq:est-V}
\end{equation}
\end{lem}

\begin{proof}
Recall from \eqref{eq:subcritic-mass} that 
\[
|A_{\theta}[Q_{b}^{(\eta)}](+\infty)-A_{\theta}[Q_{b}](+\infty)|\lesssim\eta.
\]
By scaling reasons, we have 
\[
\|\tilde V_{Q_{b}^{(\eta)}-Q_{b}}z^{\flat}\|_{L^{2}}\lesssim\alpha^{\ast}\lambda^{2}\eta.
\]
To show the derivative estimate, we observe 
\begin{align*}
\|\tilde V_{Q_{b}^{(\eta)}-Q_{b}}z^{\flat}\|_{\dot{H}_{m}^{1}} & \lesssim\sum_{T\in\{\partial_{r},r^{-1}\}}\|\tilde V_{Q_{b}^{(\eta)}-Q_{b}}Tz^{\flat}\|_{L^{2}}+\|(|Q^{(\eta)}|^{2}-|Q|^{2})r^{-2}z^{\flat}\|_{L^{2}}\\
 & \lesssim\alpha^{\ast}\lambda^{3}\eta+\alpha^{\ast}\lambda^{2}\||Q^{(\eta)}-Q|\cdot Q\|_{L^{\infty}}\\
 & \lesssim\alpha^{\ast}\lambda^{2}\eta,
\end{align*}
where in the last inequality we used \eqref{eq:differentiability}
and \eqref{eq:unif-bound}. This completes the proof of \eqref{eq:est-V}.
\end{proof}

\subsection{\label{subsec:Ru-w}Estimates of $R_{u^{\flat}-w^{\flat}}$, $\mathcal{L}_{w^{\flat}}-\mathcal{L}_{Q_{b}^{(\eta)}}$,
and $\mathcal{L}_{Q_{b}^{(\eta)}}-\mathcal{L}_{Q}$}

We will estimate each error terms with help of duality. In this section,
we mainly rely on Lemma \ref{lem:duality-1}.
\begin{lem}[Estimates of $R_{u^{\flat}-w^{\flat}}$]
\label{lem:est of Ru-w}We have 
\begin{align}
\|R_{u^{\flat}-w^{\flat}}\|_{L^{2}} & \lesssim\|\epsilon\|_{\dot{H}_{m}^{1}}^{2},\label{eq:Ru-w1}\\
\|(1+r)^{-1}R_{u^{\flat}-w^{\flat}}\|_{L^{1}} & \lesssim(\|\epsilon\|_{L^{2}}+\|\epsilon\|_{\dot{H}_{m}^{1}})\|\epsilon\|_{\dot{H}_{m}^{1}}.\label{eq:Ru-w2}
\end{align}
\end{lem}

\begin{rem}
The second estimate \eqref{eq:Ru-w2} is only used in case of $m=1$,
where we estimate $(R_{u^{\flat}-w^{\flat}},|y|^{2}Q_{b}^{(\eta)})_{r}$
in Section \ref{subsec:virial-lyapunov}. The estimate \eqref{eq:Ru-w2}
is not sharp, but this suffices to close our bootstrap.
\end{rem}

\begin{proof}
We show the first estimate \eqref{eq:Ru-w1}. Note that 
\[
R_{u-w}=\sum_{\substack{\psi_{1},\psi_{2},\psi_{3}\in\{Q_{b}^{(\eta)},z^{\flat},\epsilon\}:\\
\#\{i:\psi_{i}=\epsilon\}\geq2
}
}[\mathcal{N}_{3,0}+\mathcal{N}_{3,1}+\mathcal{N}_{3,2}]+\sum_{\substack{\psi_{1},\dots,\psi_{5}\in\{Q_{b}^{(\eta)},z^{\flat},\epsilon\}:\\
\#\{i:\psi_{i}=\epsilon\}\geq2
}
}[\mathcal{N}_{5,1}+\mathcal{N}_{5,2}].
\]
We will use 
\[
\|\psi\|_{L^{2}}\lesssim1\qquad\forall\psi\in\{Q_{b}^{(\eta)},z^{\flat},\epsilon\}.
\]
Now the estimate \eqref{eq:Ru-w1} follows by distributing two $\dot{H}_{m}^{1}$-norms
to two $\epsilon$'s using \eqref{eq:L2-est-N}.

To show the second estimate \eqref{eq:Ru-w2}, we use 
\[
\|(1+r)^{-1}R_{u^{\flat}-w^{\flat}}\|_{L^{1}}\lesssim\|R_{u^{\flat}-w^{\flat}}\|_{L^{\frac{4}{3}}}.
\]
We then apply Lemma \ref{lem:4/3-est-N}. Here, we put $\dot{H}_{m}^{1}$-norm
for one $\epsilon$. For the remaining arguments, we put $L^{2}$
or $L^{4}$ norms as provided in Lemma \ref{lem:duality-2}. We then
use $\|\epsilon\|_{L^{2}}+\|\epsilon\|_{L^{4}}\lesssim\|\epsilon\|_{L^{2}}+\|\epsilon\|_{\dot{H}_{m}^{1}}$.
\end{proof}
We turn to estimate the difference $\mathcal{L}_{w^{\flat}}-\mathcal{L}_{Q_{b}}$
and $\mathcal{L}_{Q_{b}}-\mathcal{L}_{Q}$. This will allow us to
replace the linearized operator $\mathcal{L}_{w^{\flat}}$ by $\mathcal{L}_{Q_{b}}$
or $\mathcal{L}_{Q}$.
\begin{lem}[Estimates of the linearized operators]
\label{lem:est of Lw-Lq}The following hold:\footnote{In case of $m\geq2$, we can replace $\lambda|\log\lambda|^{\frac{1}{2}}$
by $\lambda$.} 
\begin{align}
\|\mathcal{L}_{w^{\flat}}-\mathcal{L}_{Q_{b}^{(\eta)}}\|_{\dot{H}_{m}^{1}\to L^{2}} & \lesssim\alpha^{\ast}\lambda,\label{eq:Lw-Lqb-1}\\
\|\mathcal{L}_{w^{\flat}}-\mathcal{L}_{Q_{b}^{(\eta)}}\|_{\dot{H}_{m}^{1}\to(\dot{H}_{m}^{1})^{\ast}} & \lesssim\alpha^{\ast},\label{eq:Lw-Lqb-2}\\
\|\mathcal{L}_{Q_{b}^{(\eta)}}-\mathcal{L}_{Q}\|_{\dot{H}_{m}^{1}\to(\dot{H}_{m}^{1})^{\ast}} & \lesssim\lambda|\log\lambda|^{\frac{1}{2}}.\label{eq:Lqb-Lq}
\end{align}
\end{lem}

\begin{proof}
We first consider the estimate for $\mathcal{L}_{w^{\flat}}-\mathcal{L}_{Q_{b}^{(\eta)}}$.
Observe that 
\begin{align*}
(\psi,\mathcal{L}_{w^{\flat}}\varphi-\mathcal{L}_{Q_{b}^{(\eta)}}\varphi)_{r} & =\sum_{\substack{\psi_{1},\psi_{2},\psi_{3}\in\{Q_{b}^{(\eta)},z^{\flat},\varphi\}:\\
\#\{i:\psi_{i}=\varphi\}=1\\
\#\{i:\psi_{i}=z^{\flat}\}\geq1
}
}(\psi,[\mathcal{N}_{3,0}+\mathcal{N}_{3,1}+\mathcal{N}_{3,2}](\psi_{1},\psi_{2},\psi_{3}))_{r}\\
 & \quad+\sum_{\substack{\psi_{1},\dots,\psi_{5}\in\{Q_{b}^{(\eta)},z^{\flat},\varphi\}:\\
\#\{i:\psi_{i}=\varphi\}=1\\
\#\{i:\psi_{i}=z^{\flat}\}\geq1
}
}(\psi,[\mathcal{N}_{5,1}+\mathcal{N}_{5,2}](\psi_{1},\psi_{2},\psi_{3},\psi_{4},\psi_{5}))_{r}.
\end{align*}
We apply Lemma \ref{lem:duality-1}. If we put $\dot{H}_{m}^{1}$-norms
to $\varphi$ and $z^{\flat}$, put $L^{2}$-norms for remaining arguments,
and use $\|z^{\flat}\|_{\dot{H}_{m}^{1}}\lesssim\alpha^{\ast}\lambda$,
then $\dot{H}_{m}^{1}\to L^{2}$ estimate follows. If we put $\dot{H}_{m}^{1}$-norms
to $\varphi$ and $\psi$, put $L^{2}$-norms for remaining arguments,
and use $\|z^{\flat}\|_{L^{2}}\lesssim\alpha^{\ast}$, then $\dot{H}_{m}^{1}\to(\dot{H}_{m}^{1})^{\ast}$
estimate follows.

For $\mathcal{L}_{Q_{b}^{(\eta)}}-\mathcal{L}_{Q}$, the idea is exactly
same with $\dot{H}_{m}^{1}\to(\dot{H}_{m}^{1})^{\ast}$ estimate for
$\mathcal{L}_{w^{\flat}}-\mathcal{L}_{Q_{b}^{(\eta)}}$. Use \eqref{eq:Qb-Q}
and \eqref{eq:est-Q-eta-Q}. Note that $b,\eta\lesssim\lambda$. We
omit the details.
\end{proof}

\subsection{Modulation estimates}

Applying \eqref{eq:law-fixation} to the $\epsilon$-equations \eqref{eq:e-prem}
and \eqref{eq:e-sharp-prem}, we write 
\begin{align}
 & i\partial_{s}\epsilon-\mathcal{L}_{w^{\flat}}\epsilon+ib\Lambda\epsilon-\eta\theta_{\eta}\epsilon\label{eq:e-eq}\\
 & =i\Big(\frac{\lambda_{s}}{\lambda}+b\Big)([\Lambda Q^{(\eta)}]_{b}+\Lambda\epsilon)+(\tilde{\gamma}_{s}-\eta\theta_{\eta})Q_{b}^{(\eta)}+(\gamma_{s}-\eta\theta_{\eta})\epsilon\nonumber \\
 & \quad+\tilde R_{Q_{b}^{(\eta)},z^{\flat}}+V_{Q_{b}^{(\eta)}-Q_{b}}z^{\flat}+R_{u^{\flat}-w^{\flat}}\nonumber 
\end{align}
and 
\begin{align}
i\partial_{t}\epsilon^{\sharp}-\mathcal{L}_{w}\epsilon^{\sharp} & =\frac{1}{\lambda^{2}}i\Big(\frac{\lambda_{s}}{\lambda}+b\Big)[\Lambda Q^{(\eta)}]_{b}^{\sharp}+\frac{1}{\lambda^{2}}(\tilde{\gamma}_{s}-\eta\theta_{\eta})Q_{b}^{(\eta)\sharp}\label{eq:e-sharp-eq}\\
 & \quad+\tilde R_{Q_{b}^{(\eta)\sharp},z}+V_{Q_{b}^{(\eta)\sharp}-Q_{b}^{\sharp}}z+R_{u-w}.\nonumber 
\end{align}
\eqref{eq:e-eq} is the equation of $\epsilon$ we use in what follows.

The main goal of this subsection is to obtain modulation estimates.
Modulation estimates will be obtained from differentiating the orthogonality
conditions.
\begin{lem}[Computation of $\partial_{s}(\epsilon,\psi_{b})_{r}$]
\label{lem:comput-der-inn}We have 
\begin{equation}
\begin{aligned}\partial_{s}(\epsilon,\psi_{b})_{r} & =(\epsilon,[\mathcal{L}_{Q^{(\eta)}}i\psi+\eta\theta_{\eta}i\psi+\eta^{2}\tfrac{i|y|^{2}\psi}{4}]_{b})_{r}\\
 & \quad+\Big(\frac{\lambda_{s}}{\lambda}+b\Big)(\Lambda Q^{(\eta)},\psi)_{r}+(\tilde{\gamma}_{s}-\eta\theta_{\eta})(Q^{(\eta)},i\psi)_{r}\\
 & \quad-\Big(\frac{\lambda_{s}}{\lambda}+b\Big)(\epsilon,[\Lambda\psi]_{b})_{r}+(\tilde{\gamma}_{s}-\eta\theta_{\eta})(\epsilon,i\psi_{b})_{r}\\
 & \quad-\theta_{z^{\flat}\to Q_{b}^{(\eta)}}(\epsilon,i\psi_{b})_{r}+((\mathcal{L}_{w^{\flat}}-\mathcal{L}_{Q_{b}^{(\eta)}})\epsilon,i\psi_{b})_{r}+(\tilde R_{Q_{b}^{(\eta)},z^{\flat}},i\psi_{b})_{r}\\
 & \quad+(V_{Q_{b}^{(\eta)}-Q_{b}}z^{\flat},i\psi_{b})_{r}+(R_{u^{\flat}-w^{\flat}},i\psi_{b})_{r}.
\end{aligned}
\label{eq:testing}
\end{equation}
\end{lem}

\begin{proof}
We compute 
\begin{align*}
\partial_{s}(\epsilon,\psi_{b})_{r} & =(i\partial_{s}\epsilon,i\psi_{b})_{r}+b_{s}(\epsilon,\partial_{b}\psi_{b})_{r}\\
 & =(\mathcal{L}_{Q_{b}^{(\eta)}}\epsilon-ib\Lambda\epsilon+\eta\theta_{\eta}\epsilon,i\psi_{b})_{r}\\
 & \quad+\Big(\frac{\lambda_{s}}{\lambda}+b\Big)(\Lambda Q^{(\eta)},\psi)_{r}+(\tilde{\gamma}_{s}-\eta\theta_{\eta})(Q^{(\eta)},i\psi)_{r}\\
 & \quad+\Big(\frac{\lambda_{s}}{\lambda}+b\Big)(\Lambda\epsilon,\psi_{b})_{r}+(\gamma_{s}-\eta\theta_{\eta})(\epsilon,i\psi_{b})_{r}+b_{s}(\epsilon,\partial_{b}\psi_{b})_{r}\\
 & \quad+((\mathcal{L}_{w^{\flat}}-\mathcal{L}_{Q_{b}^{(\eta)}})\epsilon,i\psi_{b})_{r}+(\tilde R_{Q_{b}^{(\eta)},z^{\flat}},i\psi_{b})_{r}\\
 & \quad+(V_{Q_{b}^{(\eta)}-Q_{b}}z^{\flat},i\psi_{b})_{r}+(R_{u^{\flat}-w^{\flat}},i\psi_{b})_{r}.
\end{align*}
We then use the self-adjointness of $\mathcal{L}_{Q_{b}^{(\eta)}}$,
anti-self-adjointness of $\Lambda$, and \eqref{eq:def-tilde-gamma}
to get 
\begin{align*}
\partial_{s}(\epsilon,\psi_{b})_{r} & =(\epsilon,\mathcal{L}_{Q_{b}^{(\eta)}}i\psi_{b}+b\Lambda\psi_{b}+\eta\theta_{\eta}i\psi_{b})_{r}\\
 & \quad+\Big(\frac{\lambda_{s}}{\lambda}+b\Big)(\Lambda Q^{(\eta)},\psi)_{r}+(\tilde{\gamma}_{s}-\eta\theta_{\eta})(Q^{(\eta)},i\psi)_{r}\\
 & \quad-\Big(\frac{\lambda_{s}}{\lambda}+b\Big)(\epsilon,\Lambda\psi_{b})_{r}+(\tilde{\gamma}_{s}-\eta\theta_{\eta})(\epsilon,i\psi_{b})_{r}+b_{s}(\epsilon,\partial_{b}\psi_{b})_{r}\\
 & \quad-\theta_{z^{\flat}\to Q_{b}^{(\eta)}}(\epsilon,i\psi_{b})_{r}+((\mathcal{L}_{w^{\flat}}-\mathcal{L}_{Q_{b}^{(\eta)}})\epsilon,i\psi_{b})_{r}+(\tilde R_{Q_{b}^{(\eta)},z^{\flat}},i\psi_{b})_{r}\\
 & \quad+(V_{Q_{b}^{(\eta)}-Q_{b}}z^{\flat},i\psi_{b})_{r}+(R_{u^{\flat}-w^{\flat}},i\psi_{b})_{r}.
\end{align*}
Observe by \eqref{eq:law-fixation} 
\begin{align*}
 & -\Big(\frac{\lambda_{s}}{\lambda}+b\Big)(\epsilon,\Lambda\psi_{b})_{r}+b_{s}(\epsilon,\partial_{b}\psi_{b})_{r}\\
 & =-\Big(\frac{\lambda_{s}}{\lambda}+b\Big)(\epsilon,[\Lambda\psi]_{b})_{r}-(b^{2}+\eta^{2})(\epsilon,\partial_{b}\psi_{b})_{r}
\end{align*}
and by \eqref{eq:conj-b-phase2} 
\[
\mathcal{L}_{Q_{b}^{(\eta)}}i\psi_{b}+b\Lambda\psi_{b}-b^{2}\partial_{b}\psi_{b}=[\mathcal{L}_{Q^{(\eta)}}i\psi]_{b}.
\]
Substituting the above two displays, we complete the proof.
\end{proof}
\begin{lem}[Modulation estimates]
\label{lem:mod-est}We have 
\begin{align}
\Big|\frac{\lambda_{s}}{\lambda}+b\Big|+|\tilde{\gamma}_{s}-\eta\theta^{(\eta)}| & \lesssim\alpha^{\ast}\lambda^{2}(\lambda^{m+1}|\log\lambda|+\eta)+\|\epsilon\|_{\dot{H}_{m}^{1}},\label{eq:mod-est-1}\\
|b_{s}+b^{2}+\eta^{2}| & \lesssim\alpha^{\ast}\lambda^{3}(\lambda^{m+1}|\log\lambda|+\eta)+\lambda\|\epsilon\|_{\dot{H}_{m}^{1}}.\label{eq:mod-est-2}
\end{align}
In particular, we have crude estimaes 
\begin{equation}
\Big|\frac{\lambda_{s}}{\lambda}\Big|+|\gamma_{s}|\lesssim\lambda\quad\text{and}\quad|b_{s}|\lesssim\lambda^{2},\label{eq:mod-crude-est}
\end{equation}
and positivity of $b$:
\begin{equation}
b\geq0.\label{eq:pos-b}
\end{equation}
Moreover, we have a degeneracy estimate 
\begin{equation}
(\epsilon,Q_{b}^{(\eta)})_{r}\lesssim\alpha^{\ast}\lambda(\lambda^{m+1}|\log\lambda|+\eta)+(\alpha^{\ast}+\lambda)\mathcal{T}_{\dot{H}_{m}^{1}}^{(\eta,\frac{5}{4})}[\epsilon].\label{eq:mod-degen-est}
\end{equation}
\end{lem}

\begin{rem}[Refined modulation estimates]
\label{rem:refined-mod-est}In the proof of Lemma \ref{lem:mod-est},
we use a crude estimate 
\begin{equation}
|(\epsilon,[\mathcal{L}_{Q^{(\eta)}}i\psi+\eta\theta_{\eta}i\psi+\eta^{2}\tfrac{i|y|^{2}\psi}{4}]_{b})_{r}|\lesssim\|\epsilon\|_{\dot{H}_{m}^{1}}\label{eq:mod-est-temp1}
\end{equation}
for $\psi\in\{\mathcal{Z}_{\re},i\mathcal{Z}_{\im}\}$, which is merely
a consequence of Cauchy-Schwarz. Thus this estimate holds for any
generic $\mathcal{Z}_{\re}$ and $\mathcal{Z}_{\mathrm{im}}$. However,
if one cleverly chooses $\mathcal{Z}_{\re}$ and $\mathcal{Z}_{\im}$
adapted to the generalized nullspace of the linearized operator $\mathcal{L}_{Q}$,
then one can indeed have better estimate. For example, if $\eta=0$,
$\mathcal{Z}_{\re}=\chi_{M}|y|^{2}Q$, and $\mathcal{Z}_{\im}=\chi_{M}\rho$,
then one can obtain a maximal function version of 
\[
\Big|\frac{\lambda_{s}}{\lambda}+b\Big|+|\tilde{\gamma}_{s}|\lesssim o_{M\to\infty}(1)\|L_{Q}\epsilon\|_{L^{2}}+\alpha^{\ast}\lambda^{m+2}|\log\lambda|+(\alpha^{\ast}+C(M)\lambda)\|\epsilon\|_{\dot{H}_{m}^{1}}.
\]
If we look at coefficients of $\|L_{Q}\epsilon\|_{L^{2}}$ and $\|\epsilon\|_{\dot{H}_{m}^{1}}$,
this can be viewed as an improved modulation estimate. However, our
crude modulation estimate \eqref{eq:mod-est-1} still suffices to
close bootstrap.
\end{rem}

\begin{rem}
Notice that \eqref{eq:mod-degen-est} has a small coefficient $\alpha^{\ast}+\lambda$
of $\epsilon$. This is better than a crude estimate $|(\epsilon,Q_{b}^{(\eta)})_{r}|\lesssim\|\epsilon\|_{\dot{H}_{m}^{1}}$.
Such a gain arises because $iQ^{(\eta)}$ lies in the kernel of the
operator $\mathcal{L}_{Q^{(\eta)}}+\eta\theta_{\eta}+\eta^{2}\tfrac{|y|^{2}}{4}$;
see \eqref{eq:phase-inv}. This gain is essential to close our bootstrap
procedure.
\end{rem}

\begin{proof}
Note that the bound \eqref{eq:mod-est-2} follows from \eqref{eq:mod-est-1}
and \eqref{eq:law-fixation}. The crude estimate \eqref{eq:mod-crude-est}
follows from \eqref{eq:mod-est-1} and \eqref{eq:mod-est-2}. The
positivity \eqref{eq:pos-b} follows from $b(t=0)=0$, $b^{2}+\eta^{2}\sim\lambda^{2}$,
and \eqref{eq:mod-est-2} combined with the weak bootstrap hypothesis
$\|\epsilon\|_{\dot{H}_{m}^{1}}\leq\lambda^{\frac{3}{2}}$.

The modulation estimate \eqref{eq:mod-est-1} is obtained by differentiating
the orthogonality conditions in the $s$-variable. Namely, we substitute
$\psi\in\{\mathcal{Z}_{\re},i\mathcal{Z}_{\im}\}$ into \eqref{eq:testing}
and use $(\epsilon,\psi_{b})_{r}=0$.

The fourth and fifth lines of the RHS of \eqref{eq:testing} are treated
as error. Indeed, we use \eqref{eq:theta-est}, \eqref{eq:Lw-Lqb-1},
\eqref{eq:Ru-w1}, \eqref{eq:est-V}, and \eqref{eq:est-Rqz} to get
\begin{equation}
\left\{ \begin{aligned}|\theta_{z^{\flat}\to Q_{b}^{(\eta)}}(\epsilon,i\psi_{b})_{r}| & \lesssim(\alpha^{\ast})^{2}\lambda^{2}\|\epsilon\|_{\dot{H}_{m}^{1}},\\
|(\mathcal{L}_{w^{\flat}}\epsilon-\mathcal{L}_{Q_{b}^{(\eta)}}\epsilon,i\psi_{b})_{r}| & \lesssim\alpha^{\ast}\lambda\|\epsilon\|_{\dot{H}_{m}^{1}},\\
|(R_{u^{\flat}-w^{\flat}},i\psi_{b})_{r}| & \lesssim\|\epsilon\|_{\dot{H}_{m}^{1}}^{2},\\
|(V_{Q_{b}^{(\eta)}-Q_{b}}z^{\flat},i\psi_{b})_{r}| & \lesssim\alpha^{\ast}\lambda^{2}\eta.\\
|(\tilde R_{Q_{b}^{(\eta)},z^{\flat}},i\psi_{b})_{r}| & \lesssim\alpha^{\ast}\lambda^{m+3}|\log\lambda|.
\end{aligned}
\right.\label{eq:mod-est-err}
\end{equation}
This holds for $\psi\in\{\mathcal{Z}_{\re},i\mathcal{Z}_{\im},Q_{b}^{(\eta)}\}$.
Therefore, the fourth and fifth lines of the RHS of \eqref{eq:testing}
can be ignored from now on.

To extract $\frac{\lambda_{s}}{\lambda}+b$ and $\tilde{\gamma}_{s}-\eta\theta_{\eta}$,
we substitute $\psi=\mathcal{Z}_{\re}$ and $\psi=i\mathcal{Z}_{\im}$
to \eqref{eq:testing}, respectively. In case of $\psi=\mathcal{Z}_{\re}$,
we get 
\begin{align*}
 & \Big(\frac{\lambda_{s}}{\lambda}+b\Big)\Big((\Lambda Q^{(\eta)},\mathcal{Z}_{\re})_{r}+O(\|\epsilon\|_{\dot{H}_{m}^{1}})\Big)\\
 & \qquad\approx-(\epsilon,[\mathcal{L}_{Q^{(\eta)}}i\mathcal{Z}_{\re}+\eta\theta_{\eta}i\mathcal{Z}_{\re}+\eta^{2}\tfrac{i|y|^{2}\mathcal{Z}_{\re}}{4}]_{b})_{r}-(\tilde{\gamma}_{s}-\eta\theta_{\eta})(\epsilon,[i\mathcal{Z}_{\re}]_{b})_{r}
\end{align*}
up to error \eqref{eq:mod-est-err}. In case of $\psi=i\mathcal{Z}_{\im}$,
we get 
\begin{align*}
 & (\tilde{\gamma}_{s}-\eta\theta_{\eta})\big((Q^{(\eta)},\mathcal{Z}_{\im})_{r}+O(\|\epsilon\|_{\dot{H}_{m}^{1}})\big)\\
 & \qquad\approx-(\epsilon,[\mathcal{L}_{Q^{(\eta)}}\mathcal{Z}_{\im}+\eta\theta_{\eta}\mathcal{Z}_{\im}+\eta^{2}\tfrac{|y|^{2}\mathcal{Z}_{\im}}{4}]_{b})_{r}-\Big(\frac{\lambda_{s}}{\lambda}+b\Big)(\epsilon,[i\Lambda\mathcal{Z}_{\im}]_{b})_{r}
\end{align*}
up to error \eqref{eq:mod-est-err}. Recall that $(\Lambda Q,\mathcal{Z}_{\re})_{r}=(Q,\mathcal{Z}_{\im})_{r}=1$.
If we \emph{crudely} estimate the inner products containing $\epsilon$
by $\|\epsilon\|_{\dot{H}_{m}^{1}}$ and expand all the errors of
\eqref{eq:mod-est-err}, the modulation estimate \eqref{eq:mod-est-1}
follows.

We turn to the degeneracy estimate \eqref{eq:mod-degen-est}. We substitute
$\psi=Q^{(\eta)}$ into \eqref{eq:testing}. Using the fact that $iQ^{(\eta)}$
lies in the kernel of the operator $\mathcal{L}_{Q^{(\eta)}}+\eta\theta_{\eta}+\eta^{2}\tfrac{|y|^{2}}{4}$
(see \eqref{eq:phase-inv}), i.e. 
\[
\mathcal{L}_{Q^{(\eta)}}iQ^{(\eta)}+\eta\theta_{\eta}iQ^{(\eta)}+\eta^{2}\tfrac{|y|^{2}}{4}iQ^{(\eta)}=0,
\]
 modulation estimate \eqref{eq:mod-est-1}, and error estimate \eqref{eq:mod-est-err},
we get 
\[
|\partial_{s}(\epsilon,Q_{b}^{(\eta)})_{r}|\lesssim\alpha^{\ast}\lambda^{2}(\lambda^{m+1}|\log\lambda|+\eta)+(\alpha^{\ast}\lambda+\|\epsilon\|_{\dot{H}_{m}^{1}})\|\epsilon\|_{\dot{H}_{m}^{1}}.
\]
We then integrate the flow backward in time with $\epsilon(0)=0$
and use \eqref{eq:env-prop4} to get 
\begin{align*}
|(\epsilon,Q_{b}^{(\eta)})_{r}| & \lesssim\int_{t}^{0}\alpha^{\ast}(\lambda^{m+1}|\log\lambda|+\eta)+(\alpha^{\ast}\lambda^{-1}+1)\|\epsilon\|_{\dot{H}_{m}^{1}}dt'\\
 & \lesssim\alpha^{\ast}\lambda(\lambda^{m+1}|\log\lambda|+\eta)+(\alpha^{\ast}+\lambda)\mathcal{T}_{\dot{H}_{m}^{1}}^{(\eta,\frac{5}{4})}[\epsilon].
\end{align*}
This completes the proof of \eqref{eq:mod-degen-est}.
\end{proof}

\subsection{\label{subsec:Transfer-H1-to-L2}$L^{2}$-bound of $\epsilon$}

The coercivity estimate \eqref{eq:coercivity-LQ} controls only $\dot{H}_{m}^{1}$-norm
of $\epsilon$. In order to close our bootstrap argument, we need
to control $L_{m}^{2}$-norm of $\epsilon$. Indeed, we will use Lyapunov/virial
functional method. Since the virial correction is on $\dot{H}^{\frac{1}{2}}$
level, we are required to estimate $L_{m}^{2}$-norm of $\epsilon$.

As we will estimate $\|\epsilon(t)\|_{L^{2}}$ by integrating the
flow backward in time (note that $\epsilon(0)=0$), it is natural
to bound by the time maximal function (recall Section \ref{subsec:Envelopes}).
In this step, we use the Strichartz estimates. 
\begin{lem}[$L^{2}$-estimate of $\epsilon$]
\label{lem:5.16}We have\footnote{Here the time maximal function is applied to $||\epsilon(s(t'),\cdot)\|_{\dot{H}_{m}^{1}}$
for $t'\in[t,0]$.} 
\begin{equation}
\|\epsilon\|_{L^{2}}\lesssim\alpha^{\ast}\lambda^{m+2}|\log\lambda|+\alpha^{\ast}\lambda\eta+\lambda^{-\frac{3}{4}}\mathcal{T}_{\dot{H}_{m}^{1}}^{(\eta,\frac{5}{4})}[\epsilon].\label{eq:L2-est}
\end{equation}
\end{lem}

\begin{rem}
Here we lose $\lambda^{\frac{3}{4}}$ from $\dot{H}_{m}^{1}$ bound.
It is important not to lose $\lambda^{1}$ in order to guarantee the
virial correction term to be relatively small compared to the energy
functional. See Section \ref{subsec:virial-lyapunov}.
\end{rem}

\begin{rem}
In the proof, we use the Strichartz estimates. If one merely tries
to use the energy method on $L^{2}$, then one ends up with $\lambda^{-1}\mathcal{T}[\epsilon]$.
The point of using the Strichartz estimates is that we can control
$L^{p}$ norm of $\epsilon$ for $p<2$. Alternatively, one can apply
the energy method on $\dot{H}^{-\frac{1}{2}}$ with the embedding
$L^{\frac{4}{3}}\hookrightarrow\dot{H}^{-\frac{1}{2}}$ and interpolate
it with $\|\epsilon^{\sharp}\|_{\dot{H}_{m}^{1}}=\lambda^{-1}\|\epsilon\|_{\dot{H}_{m}^{1}}$.
\end{rem}

\begin{rem}
Losing $\lambda^{\frac{3}{4}}$ from $\dot{H}_{m}^{1}$ bound is not
optimal. In fact, one can combine Strichartz estimates and local energy
estimate (see for instance \cite{BurqPlanchonStalkerTahvildar-Zadeh2003JFA,GustafsonNakanishiTsai2010CMP}
and \cite[Theorem 10.1]{GustafsonNakanishiTsai2010CMP})
\[
\||x|^{-1}u\|_{L_{t,x}^{2}}\lesssim\|u_{0}\|_{L^{2}}+\||x|(i\partial_{t}+\Delta_{m})u\|_{L_{t,x}^{2}}
\]
to replace $\lambda^{\frac{3}{4}}$ loss by $\lambda^{\frac{1}{2}}$
loss. Nevertheless, this does not significantly improve the result
of the paper. In particular, $\|\epsilon\|_{L^{2}}\lesssim\lambda^{-\frac{1}{2}}\|\epsilon\|_{\dot{H}_{m}^{1}}$
is still \emph{not enough} to say that the $\eta$-correction term
$\frac{\eta\theta_{\eta}}{2}M[\epsilon]$ is relatively small compared
to the energy functional. See Section \ref{subsec:virial-lyapunov}.
\end{rem}

\begin{rem}
Compare \eqref{eq:e-sharp-prem} and \eqref{eq:e-sharp-eq}. Deleting
$[\frac{|y|^{2}}{4}Q^{(\eta)}]_{b}^{\sharp}$ terms is crucial for
small $m$. By \eqref{eq:law-fixation}, we are able to delete $[\frac{|y|^{2}}{4}Q^{(\eta)}]_{b}^{\sharp}$.
Otherwise, we should have estimated 
\[
\frac{1}{\lambda^{2}}\Big(\Big(\frac{\lambda_{s}}{\lambda}+b\Big)[\Lambda Q_{b}^{(\eta)}]^{\sharp}+(\tilde{\gamma}_{s}-\eta\theta_{\eta})Q_{b}^{(\eta)\sharp}-(b_{s}+b^{2})[\tfrac{|y|^{2}}{4}Q^{(\eta)}]_{b}^{\sharp}\Big).
\]
A slow decay of $\Lambda Q_{b}^{(\eta)}$ and $\frac{|y|^{2}}{4}Q_{b}^{(\eta)}$
is problematic for small $m$. For example, they do not belong to
$L^{p}$ for any $p\leq2$ if $m=1$ and $\eta=0$. This prevents
us to obtain any useful $L^{2}$ bound of $\epsilon$. Here the law
\eqref{eq:law-fixation} is really crucial. Without \eqref{eq:law-fixation},
we guess that one should truncate the profile $Q_{b}^{(\eta)}$ (say
$\tilde Q_{b}^{(\eta)}$) with carefully chosen radius, make an ansatz
$u=\tilde Q_{b}^{(\eta)\sharp}+z+\epsilon^{\sharp}$, and control
the errors arising from $\tilde Q_{b}^{(\eta)}$. This will generate
more error terms and the analysis would become much more complicated.
\end{rem}

\begin{proof}
We use $\epsilon^{\sharp}$-equation \eqref{eq:e-sharp-eq} in the
following form: 
\begin{align}
i\partial_{t}\epsilon^{\sharp}+\Delta_{m}\epsilon^{\sharp} & =\frac{1}{\lambda^{2}}i\Big(\frac{\lambda_{s}}{\lambda}+b\Big)[\Lambda Q^{(\eta)}]_{b}^{\sharp}+\frac{1}{\lambda^{2}}(\tilde{\gamma}_{s}-\eta\theta_{\eta})Q_{b}^{(\eta)\sharp}\label{eq:L2-est-temp}\\
 & \quad+\tilde R_{Q_{b}^{(\eta)\sharp},z}+V_{Q_{b}^{(\eta)\sharp}-Q_{b}^{\sharp}}z\nonumber \\
 & \quad+(\mathcal{L}_{w}+\Delta_{m})\epsilon^{\sharp}+R_{u-w}.\nonumber 
\end{align}
By Strichartz estimates (Lemma \ref{lem:Strichartz}) and $\epsilon(0)=0$,
it suffices to estimate the RHS of \eqref{eq:L2-est-temp} by the
dual Strichartz norm.

We treat the first and second lines of \eqref{eq:L2-est-temp}. Note
that 
\[
\|[\Lambda Q^{(\eta)}]_{b}^{\sharp}\|_{L_{x}^{\frac{4}{3}}}+\|Q_{b}^{(\eta)\sharp}\|_{L_{x}^{\frac{4}{3}}}\lesssim\lambda^{\frac{1}{2}}.
\]
Combining this with the modulation estimate \eqref{eq:mod-est-1}
and the maximal estimate \eqref{eq:env-prop4}, we get 
\begin{align*}
 & \Big\|\frac{1}{\lambda^{2}}\bigg(i\Big(\frac{\lambda_{s}}{\lambda}+b\Big)[\Lambda Q^{(\eta)}]_{b}^{\sharp}+\tilde{\gamma}_{s}Q_{b}^{(\eta)\sharp}\bigg)\Big\|_{L_{[t,0],x}^{\frac{4}{3}}}\\
 & \lesssim\lambda^{-\frac{3}{4}}\Big(\mathcal{T}^{(\eta,\frac{5}{4})}\Big[\frac{\lambda_{s}}{\lambda}+b\Big]+\mathcal{T}^{(\eta,\frac{5}{4})}[\tilde{\gamma}_{s}-\eta\theta_{\eta}]\Big)\\
 & \lesssim\lambda^{-\frac{3}{4}}(\alpha^{\ast}\lambda^{m+3}|\log\lambda|+\alpha^{\ast}\lambda^{2}\eta+\mathcal{T}_{\dot{H}_{m}^{1}}^{(\eta,\frac{5}{4})}[\epsilon]).
\end{align*}
The second line of the RHS of \eqref{eq:L2-est-temp} is estimated
as 
\begin{align*}
\|\tilde R_{Q_{b}^{\sharp},z}+V_{Q_{b}^{(\eta)\sharp}-Q_{b}^{\sharp}}z\|_{L_{[t,0]}^{1}L_{x}^{2}} & \lesssim\alpha^{\ast}\lambda^{m+2}|\log\lambda|+\alpha^{\ast}\lambda\eta.
\end{align*}

To complete the proof, it only remains to show
\[
\|(\mathcal{L}_{w}+\Delta_{m})\epsilon^{\sharp}+R_{u-w}\|_{L_{[t,0],x}^{\frac{4}{3}}}\lesssim\lambda^{-\frac{3}{4}}\mathcal{T}_{\dot{H}_{m}^{1}}^{(\eta,\frac{5}{4})}[\epsilon].
\]
We note that 
\begin{align*}
 & (\mathcal{L}_{w}+\Delta_{m})\epsilon^{\sharp}+R_{u-w}\\
 & =\sum_{\substack{\psi_{1},\psi_{2},\psi_{3}\in\{Q_{b}^{(\eta)\sharp},z,\epsilon^{\sharp}\}:\\
\#\{i:\psi_{i}=\epsilon^{\sharp}\}\geq1
}
}[\mathcal{N}_{3,0}+\mathcal{N}_{3,1}+\mathcal{N}_{3,2}]+\sum_{\substack{\psi_{1},\dots,\psi_{5}\in\{Q_{b}^{(\eta)\sharp},z,\epsilon^{\sharp}\}:\\
\#\{i:\psi_{i}=\epsilon^{\sharp}\}\geq1
}
}[\mathcal{N}_{5,1}+\mathcal{N}_{5,2}].
\end{align*}
In the above expression, the worst terms occur when only one $\psi_{i}$
is $\epsilon^{\sharp}$ and all the other $\psi_{i}$s are $Q_{b}^{(\eta)\sharp}$
because the following estimate saturates when $\psi_{i}=Q_{b}^{(\eta)\sharp}$:
\[
\|\psi\|_{L_{x}^{p}}\lesssim\lambda^{\frac{2}{p}-1};\qquad\forall\psi\in\{Q_{b}^{(\eta)\sharp},z,\epsilon^{\sharp}\},\ p\geq2.
\]
Our main tools are $L^{\frac{4}{3}}$-estimate (Lemma \ref{lem:4/3-est-N})
and maximal function estimate \eqref{eq:env-prop4}. By \eqref{eq:N-est-3},
\begin{align*}
 & \bigg\|\sum_{\substack{\psi_{1},\psi_{2},\psi_{3}\in\{Q_{b}^{(\eta)\sharp},z,\epsilon^{\sharp}\}:\\
\#\{i:\psi_{i}=\epsilon^{\sharp}\}\geq1
}
}[\mathcal{N}_{3,0}+\mathcal{N}_{3,1}+\mathcal{N}_{3,2}]\bigg\|_{L_{x}^{\frac{4}{3}}}\\
 & \lesssim\|\epsilon^{\sharp}\|_{\dot{H}_{m}^{1}}\sum_{\psi_{1},\psi_{2}\in\{Q_{b}^{(\eta)\sharp},z,\epsilon^{\sharp}\}}\|\psi_{1}\|_{L^{2}}\|\psi_{2}\|_{L^{4}}\lesssim\lambda^{-\frac{3}{2}}\|\epsilon\|_{\dot{H}_{m}^{1}}.
\end{align*}
Applying \eqref{eq:env-prop4}, we get 
\[
\bigg\|\sum_{\substack{\psi_{1},\psi_{2},\psi_{3}\in\{Q_{b}^{(\eta)\sharp},z,\epsilon^{\sharp}\}:\\
\#\{i:\psi_{i}=\epsilon^{\sharp}\}\geq1
}
}[\mathcal{N}_{3,0}+\mathcal{N}_{3,1}+\mathcal{N}_{3,2}]\bigg\|_{L_{[t,0],x}^{\frac{4}{3}}}\lesssim\lambda^{-\frac{3}{4}}\mathcal{T}_{\dot{H}_{m}^{1}}^{(\eta,\frac{5}{4})}[\epsilon].
\]
Treating $\mathcal{N}_{5,1}$ and $\mathcal{N}_{5,2}$ terms is similar;
this time one uses \eqref{eq:N-est-4} and scaling relation \eqref{eq:multi-2-scal-rel-1}.
We omit the details. This ends the proof.
\end{proof}

\subsection{\label{subsec:virial-lyapunov}Lyapunov/virial functional}

So far, we have seen that how $\dot{H}_{m}^{1}$ norm on $\epsilon$
controls variations of modulation parameters and also $L^{2}$ norm
of $\epsilon$. In this section, in order to close our bootstrap,
we will control $\dot{H}_{m}^{1}$ norm of $\epsilon$ using the Lyapunov
functional. Our goal is to find a Lyapunov functional $\mathcal{I}$
such that $\mathcal{I}$ is coercive and $\partial_{s}\mathcal{I}$
is almost positive. For instance, if we have 
\begin{equation}
\left\{ \begin{aligned}\tilde{\mathcal{I}} & \sim\|\epsilon\|_{\dot{H}_{m}^{1}}^{2},\\
\partial_{s}\tilde{\mathcal{I}} & \geq-o(\lambda)\|\epsilon\|_{\dot{H}_{m}^{1}}^{2},
\end{aligned}
\right.\label{eq:rough-lyapunov}
\end{equation}
and assume power-type bound $\|\epsilon\|_{\dot{H}_{m}^{1}}\lesssim\lambda^{\ell}\sim|t|^{\ell}$
with $\ell>0$, then by FTC 
\[
\|\epsilon\|_{\dot{H}_{m}^{1}}^{2}\lesssim\tilde{\mathcal{I}}(t)\lesssim o(1)\cdot\int_{t}^{t_{0}}\lambda^{-1}\|\epsilon\|_{\dot{H}_{m}^{1}}^{2}dt'\lesssim o(1)\lambda^{2\ell}.
\]
Note that we have used $\epsilon(t_{0})=0$ and hence $\tilde{\mathcal{I}}(t_{0})=0$.
As we integrate from $0$ to $t$ (backward in time), positive terms
of $\partial_{s}\tilde{\mathcal{I}}$ are safe.

In view of the coercivity, a natural candidate for $\tilde{\mathcal{I}}$
would be the energy functional $E$. In general, for a functional
$\mathcal{A}$, we define its quadratic (and higher) part of $\mathcal{A}[w^{\flat}+\cdot]$
by 
\begin{equation}
\mathcal{A}_{w^{\flat}}^{(\mathrm{qd})}[\epsilon]\coloneqq\mathcal{A}[w^{\flat}+\epsilon]-\mathcal{A}[w^{\flat}]-(\frac{\delta\mathcal{A}}{\delta u}\Big|_{u=w^{\flat}},\epsilon)_{r}.\label{eq:quad-func-def}
\end{equation}
We first start with the energy functional $E_{w^{\flat}}^{\mathrm{(qd)}}[\epsilon]$.
It turns out that, however, $\partial_{s}E_{w^{\flat}}^{\mathrm{(qd)}}[\epsilon]$
does not enjoy satisfactory lower bound. To resolve the difficulty,
we make two modifications from $E_{w^{\flat}}^{\mathrm{(qd)}}[\epsilon]$.
Firstly, we will observe that $\partial_{s}E_{w^{\flat}}^{\mathrm{(qd)}}[\epsilon]$
contains $\frac{2\lambda_{s}}{\lambda}E_{w^{\flat}}^{\mathrm{(qd)}}[\epsilon]$.
Note that $\frac{2\lambda_{s}}{\lambda}\approx-2b$ can only be controlled
by $\lambda$, so $\frac{2\lambda_{s}}{\lambda}E_{w^{\flat}}^{\mathrm{(qd)}}[\epsilon]$
is neither positive nor small. Thus it is natural to rescale it to
hide $\frac{2\lambda_{s}}{\lambda}E_{w^{\flat}}^{\mathrm{(qd)}}[\epsilon]$
using 
\[
\lambda^{2}\partial_{s}(\lambda^{-2}E_{w^{\flat}}^{\mathrm{(qd)}}[\epsilon])=\partial_{s}E_{w^{\flat}}^{\mathrm{(qd)}}[\epsilon]-\frac{2\lambda_{s}}{\lambda}E_{w^{\flat}}^{\mathrm{(qd)}}[\epsilon].
\]
Secondly, there is another non-perturbative effect both from scalings
and rotations. In view of \eqref{eq:e-eq}, $\epsilon$ actually evolves
under 
\[
i\partial_{s}\epsilon-\mathcal{L}_{w^{\flat}}\epsilon+ib\Lambda\epsilon-\eta\theta_{\eta}\epsilon\approx0.
\]
To incorporate this, we add a correction $b\Phi_{A}[\epsilon]+\frac{\eta\theta_{\eta}}{2}M[\epsilon]$
to our energy functional $E_{w^{\flat}}^{\mathrm{(qd)}}[\epsilon]$.
The former term $b\Phi_{A}[\epsilon]$ is a localized virial correction,
and the latter term $\frac{\eta\theta_{\eta}}{2}M[\epsilon]$ is a
mass correction. We will choose 
\[
\mathcal{I}_{A}\coloneqq\lambda^{-2}(E_{w^{\flat}}^{\mathrm{(qd)}}[\epsilon]+b\Phi_{A}[\epsilon]+\tfrac{\eta\theta_{\eta}}{2}M[\epsilon]).
\]
Here, $A\geq1$ is some large parameter to be chosen later. Still,
we need a further correction on $\mathcal{I}_{A}$. This is due to
\[
\frac{2\lambda_{s}}{\lambda}\cdot\frac{1}{8}\int(\Delta^{2}\phi_{A})|\epsilon|^{2}
\]
arising from the virial correction. As the coercivity of $\lambda^{2}\mathcal{I}_{A}$
only controls $\|\epsilon\|_{\dot{H}_{m}^{1}}^{2}$, this term is
on the borderline of sufficient estimates. We can resolve this issue
via averaging the virial correction. Eventually, we will use 
\[
\mathcal{I}\coloneqq\frac{2}{\log A}\int_{A^{1/2}}^{A}\mathcal{I}_{A'}\frac{dA'}{A'}=\lambda^{-2}\Big(E_{w^{\flat}}^{\mathrm{(qd)}}[\epsilon]+\frac{2}{\log A}\int_{A^{1/2}}^{A}b\Phi_{A'}[\epsilon]\frac{dA'}{A'}+\tfrac{\eta\theta_{\eta}}{2}M[\epsilon]\Big).
\]

In the sequel, we denote by $\err$ the collection of small terms
satisfying 
\begin{align}
|\err| & \lesssim\lambda\cdot\Big(A(\alpha^{\ast}\lambda^{m+2}+\alpha^{\ast}\lambda\eta)^{2}\label{eq:err}\\
 & \qquad\qquad+(\alpha^{\ast}+o_{A\to\infty}(1)+A\lambda^{\frac{1}{4}})(\mathcal{T}_{\dot{H}_{m}^{1}}^{(\eta,\frac{5}{4})}[\epsilon])^{2}\Big).\nonumber 
\end{align}
This error may differ line by line. Roughly speaking, $\err$ can
be thought of as $o(\lambda)\|\epsilon\|_{\dot{H}_{m}^{1}}^{2}$ in
\eqref{eq:rough-lyapunov} under rough substitutions $\|\epsilon\|_{L^{2}}\lesssim\lambda^{-\frac{3}{4}}\|\epsilon\|_{\dot{H}_{m}^{1}}$
and \eqref{eq:bootstrap-conc}.
\begin{lem}[Estimates of $E_{w^{\flat}}^{\mathrm{(qd)}}$]
\label{lem:quad-energy}We have coercivity
\begin{equation}
E_{w^{\flat}}^{\mathrm{(qd)}}[\epsilon]+O\big((\alpha^{\ast}+\lambda^{\frac{1}{2}})\|\epsilon\|_{\dot{H}_{m}^{1}}^{2}\big)\sim\|\epsilon\|_{\dot{H}_{m}^{1}}^{2}.\label{eq:quad-energy-coer}
\end{equation}
We have the derivative identity 
\begin{align}
\partial_{s}E_{w^{\flat}}^{\mathrm{(qd)}}[\epsilon] & =\frac{\lambda_{s}}{\lambda}\{2E_{w^{\flat}}^{\text{(qd)}}[\epsilon]-(R_{u^{\flat}-w^{\flat}},\Lambda Q_{b}^{(\eta)})_{r}\}\label{eq:quad-energy-deriv-precise}\\
 & \quad-\gamma_{s}(\mathcal{L}_{w^{\flat}}\epsilon+R_{u^{\flat}-w^{\flat}},i\epsilon)_{r}\nonumber \\
 & \quad+\Big(\mathcal{L}_{Q_{b}^{(\eta)}}\epsilon,\big(\frac{\lambda_{s}}{\lambda}+b\big)[\Lambda Q^{(\eta)}]_{b}-(\tilde{\gamma}_{s}-\eta\theta_{\eta})iQ_{b}^{(\eta)}\Big)_{r}+\err.\nonumber 
\end{align}
We can formally write \eqref{eq:quad-energy-deriv-precise} as\footnote{As we do not assume decay properties on $z^{\flat}$ and $\epsilon$,
the terms $(R_{u^{\flat}-w^{\flat}},\Lambda z^{\flat})_{r}$ and $(\mathcal{L}_{w^{\flat}}\epsilon+R_{u^{\flat}-w^{\flat}},\Lambda\epsilon)_{r}$
do not make sense by themselves.} 
\begin{align}
\partial_{s}E_{w^{\flat}}^{(\mathrm{qd})}[\epsilon] & =\frac{\lambda_{s}}{\lambda}\{(R_{u^{\flat}-w^{\flat}},\Lambda z^{\flat})_{r}+(\mathcal{L}_{w^{\flat}}\epsilon+R_{u^{\flat}-w^{\flat}},\Lambda\epsilon)_{r}\}\label{eq:quad-energy-deriv-formal}\\
 & \quad-\gamma_{s}(\mathcal{L}_{w^{\flat}}\epsilon+R_{u^{\flat}-w^{\flat}},i\epsilon)_{r}\nonumber \\
 & \quad+\Big(\mathcal{L}_{Q_{b}^{(\eta)}}\epsilon,\big(\frac{\lambda_{s}}{\lambda}+b\big)[\Lambda Q^{(\eta)}]_{b}-(\tilde{\gamma}_{s}-\eta\theta_{\eta})iQ_{b}^{(\eta)}\Big)_{r}+\err.\nonumber 
\end{align}
\end{lem}

\begin{proof}
To show \eqref{eq:quad-energy-coer}, we recall the expression \eqref{eq:energy-M-expn}.
By definition \eqref{eq:quad-func-def}, we have 
\begin{align*}
E_{w^{\flat}}^{\mathrm{(qd)}}[\epsilon] & =\frac{1}{2}\int\Big(|\partial_{r}\epsilon|^{2}+\frac{m^{2}}{r^{2}}|\epsilon|^{2}\Big)\\
 & \quad-\sum_{\substack{\psi_{1},\dots,\psi_{4}\in\{w^{\flat},\epsilon\}:\\
\#\{i:\psi_{i}=\epsilon\}\geq2
}
}[\frac{1}{4}\mathcal{M}_{4,0}+\frac{m}{2}\mathcal{M}_{4,1}]+\frac{1}{8}\sum_{\substack{\psi_{1},\dots,\psi_{6}\in\{w^{\flat},\epsilon\}:\\
\#\{i:\psi_{i}=\epsilon\}\geq2
}
}\mathcal{M}_{6}.
\end{align*}
Using the multilinear estimates (Lemma \ref{lem:duality-1}) with
$\|w^{\flat}-Q\|_{L^{2}}\lesssim\alpha^{\ast}+\lambda|\log\lambda|^{\frac{1}{2}}$
(see \eqref{eq:Qb-Q}), we can replace $w^{\flat}$ by $Q$: 
\[
E_{w^{\flat}}^{\mathrm{(qd)}}[\epsilon]=E_{Q}^{\mathrm{(qd)}}[\epsilon]+O\big((\alpha^{\ast}+\lambda|\log\lambda|^{\frac{1}{2}})\|\epsilon\|_{\dot{H}_{m}^{1}}^{2}\big).
\]
Using the expansion of energy \eqref{eq:energy-Q-expn}, bootstrap
hypothesis $\|\epsilon\|_{L^{2}}\leq\lambda^{\frac{1}{2}}$, and multilinear
estimate (Lemma \ref{lem:duality-1}), we have 
\[
E_{w^{\flat}}^{\mathrm{(qd)}}[\epsilon]=\frac{1}{2}\int|L_{Q}\epsilon|^{2}+O\big((\alpha^{\ast}+\lambda^{\frac{1}{2}})\|\epsilon\|_{\dot{H}_{m}^{1}}^{2}\big).
\]
We then apply the coercivity \eqref{eq:coercivity-LQ} to conclude
\eqref{eq:quad-energy-coer}.

We turn to compute $\partial_{s}E_{w^{\flat}}^{(\mathrm{qd})}[\epsilon]$.
If we perform the computation at $(s,y)$-scalings, we encounter unbounded-looking
terms. For example, $\partial_{s}w^{\flat}$ and $\partial_{s}\epsilon$
contain $\frac{\lambda_{s}}{\lambda}\Lambda z^{\flat}$ and $\frac{\lambda_{s}}{\lambda}\Lambda\epsilon$,
respectively. These terms are not in $L^{2}$ as we do not assume
decay on $z$ and $\epsilon$. This issue can be resolved if we perform
the computation at $(t,x)$-scalings instead. But we first start computing
$\partial_{s}E_{w^{\flat}}^{(\mathrm{qd})}[\epsilon]$ at $(s,y)$-scalings
to observe a crucial cancellation and rearrange errors. We estimate
errors in $(s,y)$-scalings rigorously. Later, we redo the computation
of $\partial_{s}E_{w^{\flat}}^{\mathrm{(qd)}}[\epsilon]=\lambda^{2}\partial_{t}(\lambda^{2}E_{w}^{\mathrm{(qd)}}[\epsilon^{\sharp}])$
in $(t,x)$-scalings to justify the formal computation.

We start with the formal computation 
\begin{align*}
\partial_{s}E_{w^{\flat}}^{(\mathrm{qd})}[\epsilon] & =(\frac{\delta E}{\delta u}\Big|_{u=w^{\flat}+\epsilon}-\frac{\delta E}{\delta u}\Big|_{u=w^{\flat}}-\frac{\delta^{2}E}{\delta u^{2}}\Big|_{u=w^{\flat}}\epsilon,\partial_{s}w^{\flat})_{r}\\
 & \quad+(\frac{\delta E}{\delta u}\Big|_{u=w^{\flat}+\epsilon}-\frac{\delta E}{\delta u}\Big|_{u=w^{\flat}},\partial_{s}\epsilon)_{r}.
\end{align*}
Note that 
\begin{align*}
\frac{\delta E}{\delta u}\Big|_{u=w^{\flat}+\epsilon}-\frac{\delta E}{\delta u}\Big|_{u=w^{\flat}} & =\mathcal{L}_{w^{\flat}}\epsilon+R_{u^{\flat}-w^{\flat}},\\
\frac{\delta^{2}E}{\delta u^{2}}\Big|_{u=w^{\flat}}\epsilon & =\mathcal{L}_{w^{\flat}}\epsilon.
\end{align*}
Thus 
\begin{align}
\partial_{s}E_{w^{\flat}}^{(\mathrm{qd})}[\epsilon] & =(R_{u^{\flat}-w^{\flat}},\partial_{s}w^{\flat})_{r}+(\mathcal{L}_{w^{\flat}}\epsilon+R_{u^{\flat}-w^{\flat}},\partial_{s}\epsilon)_{r}.\label{eq:5.37}
\end{align}

For the first term of the RHS of \eqref{eq:5.37}, we claim 
\[
(R_{u^{\flat}-w^{\flat}},\partial_{s}w^{\flat})_{r}=\frac{\lambda_{s}}{\lambda}(R_{u^{\flat}-w^{\flat}},\Lambda z^{\flat})_{r}+\err.
\]
To see this, decompose 
\[
\partial_{s}w^{\flat}=b_{s}(\partial_{b}Q_{b}^{(\eta)})-iL_{z^{\flat}}^{\ast}\D_{+}^{(z^{\flat})}z^{\flat}-iV_{Q_{b}^{(\eta)}\to z^{\flat}}z^{\flat}-i\gamma_{s}z^{\flat}+\frac{\lambda_{s}}{\lambda}\Lambda z^{\flat}.
\]
If we use the estimates \eqref{eq:mod-crude-est}, \eqref{eq:Ru-w1},
\eqref{eq:Ru-w2}, and the energy estimates of $z$ as 
\begin{align*}
\|z^{\flat}\|_{L^{2}} & \lesssim\alpha^{\ast},\\
\|L_{z^{\flat}}^{\ast}\D_{+}^{(z^{\flat})}z^{\flat}\|_{L^{2}}+\|V_{Q_{b}\to z^{\flat}}z^{\flat}\|_{L^{2}} & \lesssim\alpha^{\ast}\lambda^{2},
\end{align*}
we have 
\[
|(R_{u^{\flat}-w^{\flat}},\partial_{s}w^{\flat})_{r}|\lesssim\lambda^{2}(\|\epsilon\|_{L^{2}}\|\epsilon\|_{\dot{H}_{m}^{1}}+\|\epsilon\|_{\dot{H}_{m}^{1}}^{2})+\alpha^{\ast}\lambda\|\epsilon\|_{\dot{H}_{m}^{1}}^{2}.
\]
Substituting the $L^{2}$-bound \eqref{eq:L2-est} and an obvious
bound $\|\epsilon\|_{\dot{H}_{m}^{1}}\leq\mathcal{T}_{\dot{H}_{m}^{1}}^{(\eta,\frac{5}{4})}[\epsilon]$,
the claim follows.

We turn to the second term of the RHS of \eqref{eq:5.37}. We compute
using the $\epsilon$-equation \eqref{eq:e-eq} that 
\begin{equation}
(\mathcal{L}_{w^{\flat}}\epsilon+R_{u^{\flat}-w^{\flat}},\partial_{s}\epsilon)_{r}=\eqref{eq:5.13-temp3}+\eqref{eq:5.13-temp4}+\eqref{eq:5.13-temp6}+\eqref{eq:5.13-temp5},\label{eq:5.13-temp2}
\end{equation}
where 
\begin{align}
 & (\mathcal{L}_{w^{\flat}}\epsilon+R_{u^{\flat}-w^{\flat}},-i\mathcal{L}_{w^{\flat}}\epsilon-iR_{u^{\flat}-w^{\flat}})_{r}\label{eq:5.13-temp3}\\
 & (\mathcal{L}_{w^{\flat}}\epsilon+R_{u^{\flat}-w^{\flat}},\Big(\frac{\lambda_{s}}{\lambda}+b\Big)[\Lambda Q^{(\eta)}]_{b}-(\tilde{\gamma}_{s}-\eta\theta_{\eta})iQ_{b}^{(\eta)})_{r}\label{eq:5.13-temp4}\\
 & (\mathcal{L}_{w^{\flat}}\epsilon+R_{u^{\flat}-w^{\flat}},\frac{\lambda_{s}}{\lambda}\Lambda\epsilon-\gamma_{s}i\epsilon)_{r}\label{eq:5.13-temp6}\\
 & (\mathcal{L}_{w^{\flat}}\epsilon+R_{u^{\flat}-w^{\flat}},-iV_{Q_{b}^{(\eta)}-Q_{b}}z^{\flat}-i\tilde R_{Q_{b}^{(\eta)},z^{\flat}})_{r}.\label{eq:5.13-temp5}
\end{align}
One crucial observation is that \eqref{eq:5.13-temp3} \emph{vanishes}.
This is one of the reason why we work on $(s,y)$-scalings first.

We turn to \eqref{eq:5.13-temp4}. We claim that 
\begin{align*}
 & \Big(\mathcal{L}_{w^{\flat}}\epsilon+R_{u^{\flat}-w^{\flat}},\big(\frac{\lambda_{s}}{\lambda}+b\big)[\Lambda Q^{(\eta)}]_{b}-(\tilde{\gamma}_{s}-\eta\theta_{\eta})iQ_{b}^{(\eta)}\Big)_{r}\\
 & =\Big(\frac{\lambda_{s}}{\lambda}+b\Big)(\mathcal{L}_{Q_{b}^{(\eta)}}\epsilon,[\Lambda Q^{(\eta)}]_{b})_{r}+(\tilde{\gamma}_{s}-\eta\theta_{\eta})(\mathcal{L}_{Q_{b}^{(\eta)}}\epsilon,-iQ_{b}^{(\eta)})_{r}+\err.
\end{align*}
To show this claim, we observe for $\psi\in\{\Lambda Q^{(\eta)},-iQ^{(\eta)}\}$
that 
\begin{align*}
(\mathcal{L}_{w^{\flat}}\epsilon+R_{u^{\flat}-w^{\flat}},\psi_{b})_{r} & =(\mathcal{L}_{Q_{b}^{(\eta)}}\epsilon,\psi_{b})_{r}+((\mathcal{L}_{w^{\flat}}-\mathcal{L}_{Q_{b}^{(\eta)}})\epsilon,\psi_{b})_{r}+(R_{u^{\flat}-w^{\flat}},\psi_{b})_{r}\\
 & =(\mathcal{L}_{Q_{b}^{(\eta)}}\epsilon,\psi_{b})_{r}+O\big((\alpha^{\ast}\lambda+\lambda^{\frac{3}{2}})\|\epsilon\|_{\dot{H}_{m}^{1}}\big),
\end{align*}
where we used $\|\psi\|_{L^{2}}\lesssim1$, \eqref{eq:Lw-Lqb-1},
\eqref{eq:Ru-w1}, and the weak bootstrap hypothesis \eqref{eq:bootstrap-hyp}.
Applying the modulation estimate \eqref{eq:mod-est-1}, the claim
follows. This ends the treatment of \eqref{eq:5.13-temp4}.

We will keep \eqref{eq:5.13-temp6}. This is a non-perturbative term.

Finally, \eqref{eq:5.13-temp5} is considered as an error. We estimate
using \eqref{eq:est-V} to have 
\[
|(\mathcal{L}_{w^{\flat}}\epsilon+R_{u^{\flat}-w^{\flat}},iV_{Q_{b}^{(\eta)}-Q_{b}}z^{\flat})_{r}|\lesssim\|\epsilon\|_{\dot{H}_{m}^{1}}\|V_{Q_{b}^{(\eta)}-Q_{b}}z^{\flat}\|_{H_{m}^{1}}\lesssim\lambda(\alpha^{\ast}\lambda\eta)\|\epsilon\|_{\dot{H}_{m}^{1}}.
\]
We now claim that 
\begin{equation}
|(\mathcal{L}_{w^{\flat}}\epsilon+R_{u^{\flat}-w^{\flat}},i\tilde R_{Q_{b}^{(\eta)},z^{\flat}})_{r}|\lesssim\lambda(\alpha^{\ast}\lambda^{m+2})\|\epsilon\|_{\dot{H}_{m}^{1}}.\label{eq:claim-1}
\end{equation}
Using $\|w^{\flat}-Q_{b}\|_{L^{2}}\lesssim\alpha^{\ast}+\lambda|\log\lambda|^{\frac{1}{2}}$,
\eqref{eq:est-Rqz}, and \eqref{lem:est of Ru-w}, it suffices to
show 
\[
|(\mathcal{L}_{Q_{b}}\epsilon,i\tilde R_{Q_{b}^{(\eta)},z^{\flat}})_{r}|\lesssim\alpha^{\ast}\lambda^{m+3}\|\epsilon\|_{\dot{H}_{m}^{1}}.
\]
By Lemma \ref{lem:refined-decomp}, it suffices to show 
\[
|(\mathcal{L}_{Q_{b}}\epsilon,[iR_{2}]_{b})_{r}|\lesssim\alpha^{\ast}\lambda^{m+3}\|\epsilon\|_{\dot{H}_{m}^{1}}.
\]
We now manipulate using \eqref{eq:conj-b-phase2} 
\begin{align*}
\mathcal{L}_{Q_{b}}(iR_{2})_{b} & =[L_{Q}^{\ast}L_{Q}(iR_{2})]_{b}+ib\Lambda[iR_{2}]_{b}-b^{2}\tfrac{|y|^{2}}{4}[iR_{2}]_{b}\\
 & =[L_{Q_{b}}^{\ast}-ib\tfrac{r}{2}][L_{Q}(iR_{2})]_{b}+ib\Lambda[iR_{2}]_{b}+b^{2}\tfrac{|y|^{2}}{4}[iR_{2}]_{b}.
\end{align*}
We estimate using Lemma \ref{lem:refined-decomp} as 
\[
(\epsilon,L_{Q_{b}}^{\ast}[L_{Q}(iR_{2})]_{b})_{r}\lesssim\alpha^{\ast}\lambda^{m+3}\|L_{Q_{b}}\epsilon\|_{L^{2}}\lesssim\alpha^{\ast}\lambda^{m+3}\|\epsilon\|_{\dot{H}_{m}^{1}}.
\]
All the other terms are treated crudely. We use 
\[
\partial_{r}(e^{-ib\frac{r^{2}}{4}})=-ib\tfrac{r}{2}e^{-ib\frac{r^{2}}{4}}
\]
to integrate by parts 
\begin{align*}
 & (\epsilon,-ib\tfrac{r}{2}[L_{Q}(iR_{2})]_{b})_{r}\\
 & =([-\partial_{r}-\tfrac{1}{r}]\epsilon,[L_{Q}(iR_{2})]_{b})_{r}+(\epsilon,[\partial_{r}L_{Q}(iR_{2})]_{b})_{r}\lesssim\alpha^{\ast}\lambda^{m+3}\|\epsilon\|_{\dot{H}_{m}^{1}}.
\end{align*}
We then estimate 
\[
(\epsilon,ib\Lambda[iR_{2}]_{b})_{r}=(-\Lambda\epsilon,ib[iR_{2}]_{b})_{r}\lesssim\|\epsilon\|_{\dot{H}_{m}^{1}}\|brR_{2}\|_{L^{2}}.
\]
We then integrate by parts 
\begin{align*}
(\epsilon,b^{2}\tfrac{r^{2}}{4}[iR_{2}]_{b})_{r} & =((-\partial_{r}-\tfrac{2}{r})\epsilon,ib\tfrac{r}{2}[iR_{2}]_{b})_{r}+(\epsilon,ib\tfrac{r}{2}[\partial_{r}(iR_{2})]_{b})_{r}\\
 & \lesssim\|\epsilon\|_{\dot{H}_{m}^{1}}(\|brR_{2}\|_{L^{2}}+\|br^{2}\partial_{r}R_{2}\|_{L^{2}}).
\end{align*}
This completes the proof of \eqref{eq:claim-1}. This ends the formal
derivation of \eqref{eq:quad-energy-deriv-formal}.

Now we make all of the above computations rigorous. The terms $(R_{u^{\flat}-w^{\flat}},\Lambda z^{\flat})_{r}$
and $(\mathcal{L}_{w^{\flat}}\epsilon+R_{u^{\flat}-w^{\flat}},\Lambda\epsilon)_{r}$
do not make sense by themselves. However, the sum of two terms $(R_{u^{\flat}-w^{\flat}},\Lambda z^{\flat})_{r}$
and $(\mathcal{L}_{w^{\flat}}\epsilon+R_{u^{\flat}-w^{\flat}},\Lambda\epsilon)_{r}$
does make sense from the following formal computation: 
\begin{align*}
 & (R_{u^{\flat}-w^{\flat}},\Lambda z^{\flat})_{r}+(\mathcal{L}_{w^{\flat}}\epsilon+R_{u^{\flat}-w^{\flat}},\Lambda\epsilon)_{r}\\
 & =\{(R_{u^{\flat}-w^{\flat}},\Lambda w^{\flat})_{r}+(\mathcal{L}_{w^{\flat}}\epsilon+R_{u^{\flat}-w^{\flat}},\Lambda\epsilon)_{r}\}-(R_{u^{\flat}-w^{\flat}},\Lambda Q_{b}^{(\eta)})_{r}\\
 & =\partial_{\lambda}|_{\lambda=1}E_{(w^{\flat})_{\lambda}}^{(\mathrm{qd})}[\epsilon_{\lambda}]-(R_{u^{\flat}-w^{\flat}},\Lambda Q_{b}^{(\eta)})_{r}\\
 & =2E_{w^{\flat}}^{\text{(qd)}}[\epsilon]-(R_{u^{\flat}-w^{\flat}},\Lambda Q_{b}^{(\eta)})_{r},
\end{align*}
where we temporarily denoted $f_{\lambda}(y)\coloneqq\lambda f(\lambda y)$
for $f\in\{w^{\flat},\epsilon\}$. This is how we obtain \eqref{eq:quad-energy-deriv-precise}
from the formal one \eqref{eq:quad-energy-deriv-formal}.

If one wants to avoid unbounded-looking quantities, we can compute
in $(t,x)$-scalings 
\[
\partial_{s}E_{w^{\flat}}^{\mathrm{(qd)}}[\epsilon]=\lambda^{2}\partial_{t}(\lambda^{2}E_{w}^{\mathrm{(qd)}}[\epsilon^{\sharp}])
\]
instead. Here, we use equations
\begin{align*}
i\partial_{t}w & =i\frac{1}{\lambda^{2}}[b_{s}\partial_{b}Q_{b}^{(\eta)}-\frac{\lambda_{s}}{\lambda}\Lambda Q_{b}^{(\eta)}+i\gamma_{s}Q_{b}^{(\eta)}]^{\sharp}+L_{z}^{\ast}\D_{+}^{(z)}z+V_{Q_{b}^{\sharp}\to z}z,\\
i\partial_{t}\epsilon^{\sharp} & =\mathcal{L}_{w}\epsilon^{\sharp}+\frac{1}{\lambda^{2}}i\Big(\frac{\lambda_{s}}{\lambda}+b\Big)[\Lambda Q^{(\eta)}]_{b}^{\sharp}+\frac{1}{\lambda^{2}}(\tilde{\gamma}_{s}-\eta\theta_{\eta})Q_{b}^{(\eta)\sharp}\\
 & \quad+\tilde R_{Q_{b}^{\sharp},z}+V_{Q_{b}^{(\eta)\sharp}-Q_{b}^{\sharp}}z+R_{u-w},
\end{align*}
where we used dynamic rescaling for $i\partial_{t}Q_{b}^{(\eta)\sharp}$,
\eqref{eq:fCSS} for $i\partial_{t}z$, and \eqref{eq:e-sharp-eq}
for $i\partial_{t}\epsilon^{\sharp}$. We then obtain \eqref{eq:quad-energy-deriv-precise}
by lowering to the $(s,y)$-scale and estimating errors as before.
\end{proof}
The RHS of \eqref{eq:quad-energy-deriv-precise} contains non-perturbative
terms. Indeed, we roughly have 
\begin{align*}
\frac{2\lambda_{s}}{\lambda}E_{w^{\flat}}^{\text{(qd)}}[\epsilon] & \approx-2bE_{w^{\flat}}^{\text{(qd)}}[\epsilon]\sim-b\|\epsilon\|_{\dot{H}_{m}^{1}}^{2},\\
\Big|\frac{\lambda_{s}}{\lambda}(R_{u^{\flat}-w^{\flat}},\Lambda Q_{b}^{(\eta)})_{r}\Big| & \lesssim b\|R_{u^{\flat}-w^{\flat}}\|_{L^{2}}\lesssim b\|\epsilon\|_{\dot{H}_{m}^{1}}^{2},
\end{align*}
These are at the borderline of the acceptable error size $\err$.
Next, from $\gamma_{s}\approx\eta\theta_{\eta}$, 
\[
|\gamma_{s}(\mathcal{L}_{w^{\flat}}\epsilon+R_{u^{\flat}-w^{\flat}},i\epsilon)_{r}|\lesssim\eta\|\epsilon\|_{\dot{H}_{m}^{1}}^{2},
\]
is also at the borderline. The remaining terms are also at the borderline.
We know from the modulation estimates that $|\frac{\lambda_{s}}{\lambda}+b|+|\tilde{\gamma}_{s}-\eta\theta_{\eta}|\lesssim\|\epsilon\|_{\dot{H}_{m}^{1}}$.
The problem is that $(\mathcal{L}_{Q_{b}^{(\eta)}}\epsilon,\psi_{b})_{r}$
for $\psi\in\{\Lambda Q^{(\eta)},iQ^{(\eta)}\}$ cannot be estimated
better than $\lambda\|\epsilon\|_{\dot{H}_{m}^{1}}$; from \eqref{eq:conj-b-phase2},
\eqref{eq:phase-inv}, we have 
\begin{equation}
(\epsilon,\mathcal{L}_{Q_{b}^{(\eta)}}iQ_{b}^{(\eta)})_{r}=(\epsilon,-\eta\theta_{\eta}iQ_{b}^{(\eta)}-b\Lambda Q_{b}^{(\eta)})_{r}+(\epsilon,-(b^{2}+\eta^{2})\tfrac{|y|^{2}}{4}iQ_{b}^{(\eta)})_{r},\label{eq:5.43}
\end{equation}
where the term $(\epsilon,-\eta\theta_{\eta}iQ_{b}^{(\eta)}-b\Lambda Q_{b}^{(\eta)})_{r}$
is at best estimated by $\lambda\|\epsilon\|_{\dot{H}_{m}^{1}}$.\footnote{In fact, one may cleverly choose orthogonality conditions to overcome
this issue. For sake of simiplicity, assume $\eta=0$ and $m$ is
large so that $Q$ has rapid decay. If one chooses $\mathcal{Z}_{\re}=\Lambda Q$,
then the first term of \eqref{eq:5.43} vanishes up to error $b^{2}\|\epsilon\|_{\dot{H}_{m}^{1}}$,
which is acceptable. Alternatively, one can choose $(\epsilon,|y|^{2}Q_{b})_{r}=(\epsilon,i\rho_{b})_{r}=0$
to have an improved modulation estimates; see Remark \ref{rem:refined-mod-est}.
Nonetheless, we stick to generic orthogonality conditions $(\epsilon,[\mathcal{Z}_{\re}]_{b})_{r}=(\epsilon,[i\mathcal{Z}_{\im}]_{b})_{r}=0$
to emphasize the algebra \eqref{eq:5.45} below.}

We can resolve this issue by modifying the functional $E_{w^{\flat}}^{\mathrm{(qd)}}[\epsilon]$.
Firstly, we hide 
\[
\frac{2\lambda_{s}}{\lambda}E_{w^{\flat}}^{\text{(qd)}}[\epsilon]
\]
by rescaling the functional as 
\[
\lambda^{2}\partial_{s}(\lambda^{-2}E_{w^{\flat}}^{(\mathrm{qd})}[\epsilon])=-\frac{2\lambda_{s}}{\lambda}E_{w^{\flat}}^{(\mathrm{qd})}[\epsilon]+\partial_{s}E_{w^{\flat}}^{(\mathrm{qd})}[\epsilon].
\]

For the remaining three terms, we can weaken them by adding virial
and mass corrections. In view of \eqref{eq:e-eq}, $\epsilon$ actually
evolves under 
\[
i\partial_{s}\epsilon-\mathcal{L}_{w^{\flat}}\epsilon+ib\Lambda\epsilon-\eta\theta_{\eta}\epsilon\approx0.
\]
The energy functional $E_{w^{\flat}}^{\text{(qd)}}[\epsilon]$ is
adapted to the evolution $(i\partial_{s}-\mathcal{L}_{w^{\flat}})\epsilon\approx0$.
To take the $ib\Lambda\epsilon$ part into account, it is natural
to introduce the virial term 
\[
b\Phi[\epsilon]=\frac{b}{2}\Im\int\overline{\epsilon}\cdot r\partial_{r}\epsilon
\]
so that $\frac{\delta\Phi}{\delta u}=-i\Lambda u$. To take $\eta\theta_{\eta}\epsilon$
into account, we introduce the mass term 
\[
\tfrac{\eta\theta_{\eta}}{2}M[\epsilon]=\tfrac{\eta\theta_{\eta}}{2}\int|\epsilon|^{2}.
\]
We then have 
\begin{align*}
 & \partial_{s}(E_{w^{\flat}}^{\mathrm{(qd)}}[\epsilon]+b\Phi[\epsilon]+\tfrac{\eta\theta_{\eta}}{2}M[\epsilon])\\
 & \approx(R_{u^{\flat}-w^{\flat}},\partial_{s}w^{\flat})_{r}+(\mathcal{L}_{w^{\flat}}\epsilon+R_{u^{\flat}-w^{\flat}}-ib\Lambda\epsilon+\eta\theta_{\eta}\epsilon,\partial_{s}\epsilon)_{r}.
\end{align*}
Thus we observe that \eqref{eq:5.13-temp3} is changed into 
\[
(\mathcal{L}_{w^{\flat}}\epsilon+R_{u^{\flat}-w^{\flat}}-ib\Lambda\epsilon+\eta\theta_{\eta}\epsilon,-i(\mathcal{L}_{w^{\flat}}\epsilon+R_{u^{\flat}-w^{\flat}}-ib\Lambda\epsilon+\eta\theta_{\eta}\epsilon))_{r}=0
\]
and results in further cancellation in \eqref{eq:5.13-temp4} and
\eqref{eq:5.13-temp6}. More precisely, \eqref{eq:5.13-temp4} is
changed into 
\[
\Big(\mathcal{L}_{w^{\flat}}\epsilon+R_{u^{\flat}-w^{\flat}}-ib\Lambda\epsilon+\eta\theta_{\eta}\epsilon,\Big(\frac{\lambda_{s}}{\lambda}+b\Big)[\Lambda Q^{(\eta)}]_{b}-(\tilde{\gamma}_{s}-\eta\theta_{\eta})iQ_{b}^{(\eta)}\Big)_{r},
\]
and \eqref{eq:5.13-temp6} is changed into 
\[
\Big(\mathcal{L}_{w^{\flat}}\epsilon+R_{u^{\flat}-w^{\flat}}-ib\Lambda\epsilon+\eta\theta_{\eta}\epsilon,\big(\frac{\lambda_{s}}{\lambda}+b\big)\Lambda\epsilon-(\gamma_{s}-\eta\theta_{\eta})i\epsilon\Big)_{r}.
\]
In view of the modulation estimates, \eqref{eq:phase-inv}, and \eqref{eq:scale-inv},
the above terms fall into the error.

However, the problem we still have is that $\Phi[\epsilon]$ and $(R_{u^{\flat}-w^{\flat}},\Lambda z^{\flat})_{r}$
do not make sense as we do not assume decay on $\epsilon$ and $z^{\flat}$.
It turns out that we can resolve these difficulties by making a suitable
truncation of $\Phi$. For $A\geq1$, recall the radial weight $\phi_{A}$
in \eqref{eq:def-phiA} and define 
\[
\Phi_{A}[\epsilon]\coloneqq\frac{1}{2}\int\phi_{A}'\Im(\overline{\epsilon}\partial_{r}\epsilon).
\]
Note that $\Phi_{A}$ agrees with the usual virial functional if $y\leq A$
and approximates the Morawetz functional if $y\geq2A$. We write $\Lambda_{A}$
as 
\[
\frac{\delta\Phi_{A}}{\delta u}=-i\Lambda_{A}u\coloneqq-i[\phi_{A}'\partial_{r}+\tfrac{1}{2}(\Delta\phi_{A})]u.
\]
We formally have $\Lambda_{\infty}=\Lambda$. We will frequently use
the estimate 
\[
|\Lambda_{A}u|\leq\min\{A,r\}(|\partial_{r}u|+|r^{-1}u|).
\]

We will view the virial correction $b\Phi_{A}[\epsilon]$ as an error.
The mass correction $\frac{\eta\theta_{\eta}}{2}M[\epsilon]$ is not
perturbative, but it is positive when $\eta\geq0$. This is one of
the places where the sign condition $\eta\geq0$ is necessary.
\begin{lem}[Estimates of the correction]
\label{lem:quad-bPhi}We have 
\begin{align}
|b\Phi_{A}[\epsilon]| & \lesssim A\lambda\|\epsilon\|_{L^{2}}\|\epsilon\|_{\dot{H}_{m}^{1}}\label{eq:quad-bPhi-coer}
\end{align}
and 
\begin{align}
 & \partial_{s}(b\Phi_{A}[\epsilon]+\tfrac{\eta\theta_{\eta}}{2}M[\epsilon])\label{eq:quad-bPhi-deriv}\\
 & =(\mathcal{L}_{w^{\flat}}\epsilon+R_{u^{\flat}-w^{\flat}},-\frac{\lambda_{s}}{\lambda}\Lambda_{A}\epsilon+i\eta\theta_{\eta}\epsilon)_{r}\nonumber \\
 & \quad+(-ib\Lambda_{A}\epsilon+\eta\theta_{\eta}\epsilon,\Big(\frac{\lambda_{s}}{\lambda}+b\Big)[\Lambda Q^{(\eta)}]_{b}-(\tilde{\gamma}_{s}-\eta\theta_{\eta})iQ_{b}^{(\eta)})_{r})_{r}+\mathrm{Err}.\nonumber 
\end{align}
\end{lem}

\begin{proof}
The estimate \eqref{eq:quad-bPhi-coer} is an easy consequence of
$b\lesssim\lambda$ and the Cauchy-Schwarz inequality.

We turn to the estimate \eqref{eq:quad-bPhi-deriv}. Using $|b_{s}|\lesssim\lambda^{2}$
of \eqref{eq:mod-crude-est}, we see that 
\[
|b_{s}\Phi_{A}[\epsilon]|\lesssim\lambda\cdot O(A\lambda\|\epsilon\|_{L^{2}}\|\epsilon\|_{\dot{H}_{m}^{1}})
\]
can be absorbed into $\err$. Thus 
\[
\partial_{s}(b\Phi_{A}[\epsilon]+\tfrac{\eta\theta_{\eta}}{2}M[\epsilon])=(b\Lambda_{A}\epsilon+i\eta\theta_{\eta}\epsilon,i\partial_{s}\epsilon)_{r}+\err.
\]
We compute using the $\epsilon$-equation \eqref{eq:e-eq}: 
\[
(b\Lambda_{A}\epsilon+i\eta\theta_{\eta}\epsilon,i\partial_{s}\epsilon)_{r}=\eqref{eq:5.41-temp1}+\eqref{eq:5.41-temp2}+\eqref{eq:5.41-temp3}+\eqref{eq:5.41-temp4},
\]
where 
\begin{align}
 & (b\Lambda_{A}\epsilon+i\eta\theta_{\eta}\epsilon,\mathcal{L}_{w^{\flat}}\epsilon+R_{u^{\flat}-w^{\flat}})_{r},\label{eq:5.41-temp1}\\
 & (b\Lambda_{A}\epsilon+i\eta\theta_{\eta}\epsilon,i\Big(\frac{\lambda_{s}}{\lambda}+b\Big)[\Lambda Q^{(\eta)}]_{b}+(\tilde{\gamma}_{s}-\eta\theta_{\eta})Q_{b}^{(\eta)})_{r},\label{eq:5.41-temp2}\\
 & (b\Lambda_{A}\epsilon+i\eta\theta_{\eta}\epsilon,i\frac{\lambda_{s}}{\lambda}\Lambda\epsilon+\gamma_{s}\epsilon)_{r},\label{eq:5.41-temp3}\\
 & (b\Lambda_{A}\epsilon+i\eta\theta_{\eta}\epsilon,V_{Q_{b}^{(\eta)}-Q_{b}}z^{\flat}+\tilde R_{Q_{b}^{(\eta)},z^{\flat}})_{r}.\label{eq:5.41-temp4}
\end{align}

We will treat \eqref{eq:5.41-temp1}-\eqref{eq:5.41-temp4} term by
term. Firstly, we show that \eqref{eq:5.41-temp1} is equal to the
first term of the RHS of \eqref{eq:quad-bPhi-deriv} up to the error
$\err$. Indeed, the difference is given by 
\begin{align*}
 & \Big(\frac{\lambda_{s}}{\lambda}+b\Big)(\Lambda_{A}\epsilon,\mathcal{L}_{w^{\flat}}\epsilon+R_{u^{\flat}-w^{\flat}})_{r}\\
 & =\Big(\frac{\lambda_{s}}{\lambda}+b\Big)\cdot\Big((\Lambda_{A}\epsilon,-\Delta_{m}\epsilon)_{r}+(\Lambda_{A}\epsilon,(\mathcal{L}_{w^{\flat}}+\Delta_{m})\epsilon+R_{u^{\flat}-w^{\flat}})_{r}\Big)\\
 & =\Big|\frac{\lambda_{s}}{\lambda}+b\Big|\cdot O(A\|\epsilon\|_{\dot{H}_{m}^{1}}^{2}),
\end{align*}
where we used 
\[
(\Lambda_{A}\epsilon,-\Delta_{m}\epsilon)_{r}=\int\Big(\phi_{A}''|\partial_{r}\epsilon|^{2}-\frac{\Delta^{2}\phi_{A}}{4}|\epsilon|^{2}+\frac{\phi_{A}'}{r}\cdot\frac{m^{2}}{r^{2}}|\epsilon|^{2}\Big)
\]
and the multilinear estimates (Lemma \ref{lem:duality-1}) with $\|\Lambda_{A}\epsilon\|_{L^{2}}\lesssim A\|\epsilon\|_{\dot{H}_{m}^{1}}$.
Thus the above difference can be absorbed into $\err$ by the modulation
estimate.

We will keep the term \eqref{eq:5.41-temp2} as it is not perturbative.

We now show that \eqref{eq:5.41-temp3} can be absorbed into $\err$.
We first observe by anti-symmetricity of $i$, $\Lambda_{A}$, and
$\Lambda$ that 
\[
(b\Lambda_{A}\epsilon+i\eta\theta_{\eta}\epsilon,i\frac{\lambda_{s}}{\lambda}\Lambda\epsilon+\gamma_{s}\epsilon)_{r}=(b\Lambda_{A}\epsilon,i\frac{\lambda_{s}}{\lambda}\Lambda\epsilon)_{r}.
\]
We then observe that 
\begin{align*}
\Big|(b\Lambda_{A}\epsilon,i\frac{\lambda_{s}}{\lambda}\Lambda\epsilon)_{r}\Big| & \lesssim\lambda^{2}|(i[\Lambda,\Lambda_{A}]\epsilon,\epsilon)_{r}|\lesssim A\lambda^{2}\|\epsilon\|_{L^{2}}\|\epsilon\|_{\dot{H}_{m}^{1}},
\end{align*}
where we used \eqref{eq:mod-crude-est} and 
\begin{align*}
[\Lambda_{A},\Lambda] & =-r\phi_{A}''\partial_{r}-\tfrac{1}{2}r(\Delta\phi_{A})'+\phi_{A}'\partial_{r},\\
|(i[\Lambda_{A},\Lambda]\epsilon,\epsilon)_{r}| & =|(i(-r\phi_{A}''+\phi_{A}')\partial_{r}\epsilon,\epsilon)_{r}|\lesssim A\|\epsilon\|_{L^{2}}\|\epsilon\|_{\dot{H}_{m}^{1}}.
\end{align*}
Applying the $L^{2}$-bound \eqref{eq:L2-est}, we conclude that \eqref{eq:5.41-temp3}
is absorbed into $\err$.

Finally, we show that \eqref{eq:5.41-temp4} can be absorbed into
$\err$:
\begin{align*}
|\eqref{eq:5.41-temp4}| & \lesssim\|b\Lambda_{A}\epsilon+i\eta\theta_{\eta}\epsilon\|_{L^{2}}\|V_{Q_{b}^{(\eta)}-Q_{b}}z^{\flat}+\tilde R_{Q_{b}^{(\eta)},z^{\flat}}\|_{L^{2}}\\
 & \lesssim(A\lambda\|\epsilon\|_{\dot{H}_{m}^{1}}+\eta\|\epsilon\|_{L^{2}})(\alpha^{\ast}\lambda^{m+3}|\log\lambda|+\alpha^{\ast}\lambda^{2}\eta).
\end{align*}
Substituting the $L^{2}$-bound \eqref{eq:L2-est} and $\eta\lesssim\lambda$,
\eqref{eq:5.41-temp4} is absorbed into $\err$. This completes the
proof.
\end{proof}
Incorporating the above discussions, we define the \emph{unaveraged
Lyapunov/virial} functional by 
\[
\mathcal{I}_{A}=\lambda^{-2}(E_{w^{\flat}}^{\mathrm{(qd)}}[\epsilon]+b\Phi_{A}[\epsilon]+\tfrac{\eta\theta_{\eta}}{2}M[\epsilon]).
\]

\begin{lem}[Unaveraged Lyapunov/virial functional]
\label{lem:unav-Lyapunov/virial}We have 
\begin{gather}
\lambda^{2}\mathcal{I}_{A}-\tfrac{\eta\theta_{\eta}}{2}M[\epsilon]+O\big((\alpha^{\ast}+\lambda^{\frac{1}{2}})\|\epsilon\|_{\dot{H}_{m}^{1}}^{2}+A\lambda\|\epsilon\|_{L^{2}}\|\epsilon\|_{\dot{H}_{m}^{1}}\big)\sim\|\epsilon\|_{\dot{H}_{m}^{1}}^{2},\label{eq:unav-Lyapunov-coer}\\
\lambda^{2}\partial_{s}\mathcal{I}_{A}=2b\Big(\mathcal{J}_{A}+\tfrac{\eta\theta_{\eta}}{2}M[\epsilon]-\frac{1}{8}\int(\Delta^{2}\phi_{A})|\epsilon|^{2}\Big)+\err,\label{eq:unav-Lyapunov-deriv}
\end{gather}
where 
\begin{align}
\mathcal{J}_{A} & \coloneqq\frac{1}{2}\int\phi_{A}''|\partial_{r}\epsilon|^{2}+\frac{1}{2}\int\frac{\phi_{A}'}{r}\cdot\frac{m^{2}}{r^{2}}|\epsilon|^{2}\label{eq:def-J}\\
 & \quad-\sum_{\substack{\psi_{1},\dots,\psi_{4}\in\{Q,\epsilon\}:\\
\#\{i:\psi_{i}=\epsilon\}=2
}
}[\frac{1}{4}\mathcal{M}_{4,0}^{(A)}+\frac{m}{2}\mathcal{M}_{4,1}^{(A)}]+\frac{1}{8}\sum_{\substack{\psi_{1},\dots,\psi_{6}\in\{Q,\epsilon\}:\\
\#\{i:\psi_{i}=\epsilon\}=2
}
}\mathcal{M}_{6}^{(A)}.\nonumber 
\end{align}
\end{lem}

To prove Lemma \ref{lem:unav-Lyapunov/virial}, we set aside some
computations into a lemma.
\begin{lem}
\label{lem:lem5.5}We have 
\begin{align}
 & (\mathcal{L}_{Q_{b}^{(\eta)}}\epsilon-ib\Lambda_{A}\epsilon+\eta\theta_{\eta}\epsilon,[\Lambda Q^{(\eta)}]_{b})_{r}\label{eq:5.45}\\
 & \qquad=-2\eta\theta_{\eta}(\epsilon,Q_{b}^{(\eta)})_{r}+o_{A\to\infty}(\lambda\|\epsilon\|_{\dot{H}_{m}^{1}})+O(\lambda^{\frac{3}{2}}\|\epsilon\|_{\dot{H}_{m}^{1}})\nonumber 
\end{align}
and
\begin{equation}
(\mathcal{L}_{Q_{b}^{(\eta)}}\epsilon-ib\Lambda_{A}\epsilon+\eta\theta_{\eta}\epsilon,[iQ^{(\eta)}]_{b})_{r}=o_{A\to\infty}(\lambda\|\epsilon\|_{\dot{H}_{m}^{1}})+O(\lambda^{\frac{3}{2}}\|\epsilon\|_{\dot{H}_{m}^{1}}).\label{eq:5.46}
\end{equation}
\end{lem}

\begin{proof}
We will only prove \eqref{eq:5.45}. The proof of \eqref{eq:5.46}
is similar (it is in fact easier), one uses \eqref{eq:phase-inv}
instead of \eqref{eq:scale-inv}.

To prove \eqref{eq:5.45}, we crucially use the explicit computation
\eqref{eq:scale-inv}. Applying the pseudoconformal phase, we obtain
\begin{align*}
 & (\mathcal{L}_{Q_{b}^{(\eta)}}-ib\Lambda_{A}+\eta\theta_{\eta})[\Lambda Q^{(\eta)}]_{b}\\
 & =[(\mathcal{L}_{Q^{(\eta)}}+\eta\theta_{\eta}+\eta^{2}\tfrac{|y|^{2}}{4})(\Lambda Q^{(\eta)})]_{b}+ib(\Lambda-\Lambda_{A})[\Lambda Q^{(\eta)}]_{b}-(b^{2}+\eta^{2})\tfrac{|y|^{2}}{4}[\Lambda Q^{(\eta)}]_{b}\\
 & =-2\eta\theta_{\eta}Q_{b}^{(\eta)}-4\eta^{2}\tfrac{|y|^{2}Q_{b}^{(\eta)}}{4}+ib(\Lambda-\Lambda_{A})[\Lambda Q^{(\eta)}]_{b}-(b^{2}+\eta^{2})\tfrac{|y|^{2}}{4}[\Lambda Q^{(\eta)}]_{b}.
\end{align*}
Thus 
\begin{align*}
 & (\mathcal{L}_{Q_{b}^{(\eta)}}\epsilon-ib\Lambda_{A}\epsilon+\eta\theta_{\eta}\epsilon,[\Lambda Q^{(\eta)}]_{b})_{r}\\
 & =(\epsilon,-2\eta\theta_{\eta}Q_{b}^{(\eta)}-4\eta^{2}\tfrac{|y|^{2}Q_{b}^{(\eta)}}{4}+ib(\Lambda-\Lambda_{A})[\Lambda Q^{(\eta)}]_{b}-(b^{2}+\eta^{2})\tfrac{|y|^{2}}{4}[\Lambda Q^{(\eta)}]_{b})_{r}.
\end{align*}
If $m$ is large, we can use the decay $|Q^{(\eta)}|+|\Lambda Q^{(\eta)}|\lesssim(1+r)^{-(m+2)}$
to conclude the proof.

To treat the worst case $m=1$, however, we have to deal with the
slow decay $r^{2}\Lambda Q^{(\eta)}$ especially when $\eta=0$.\footnote{For each $\eta>0$, $r^{2}\Lambda Q^{(\eta)}$ in fact shows exponential
decay $e^{-\eta\frac{r^{2}}{4}}$. However, we search for uniform
estimates in $\eta$.} By \eqref{eq:ref-unif-bound}, we have 
\[
\eta^{2}(\|r\cdot r^{2}Q^{(\eta)}\|_{L^{2}}+\|r\cdot r^{2}\Lambda Q^{(\eta)}\|_{L^{2}})\lesssim\eta^{\frac{3}{2}}.
\]
Thus 
\begin{align*}
 & (\mathcal{L}_{Q_{b}^{(\eta)}}\epsilon-ib\Lambda_{A}\epsilon+\eta\theta^{(\eta)}\epsilon,[\Lambda Q^{(\eta)}]_{b})_{r}\\
 & =-2\eta\theta^{(\eta)}(\epsilon,Q_{b}^{(\eta)})_{r}+(\epsilon,ib(\Lambda-\Lambda_{A})[\Lambda Q^{(\eta)}]_{b}-b^{2}\tfrac{|y|^{2}}{4}[\Lambda Q^{(\eta)}]_{b})_{r}+O(\eta^{\frac{3}{2}}\|\epsilon\|_{\dot{H}_{m}^{1}}).
\end{align*}
We then use the algebra 
\[
ib(\Lambda-\Lambda_{A})[\Lambda Q^{(\eta)}]_{b}=ib[(\Lambda-\Lambda_{A})(\Lambda Q^{(\eta)})]_{b}+b^{2}(\tfrac{r^{2}-r\phi_{A}'(r)}{2})[\Lambda Q^{(\eta)}]_{b}
\]
and the estimate 
\begin{align*}
\|r\cdot(\Lambda-\Lambda_{A})(\Lambda Q^{(\eta)})\|_{L^{2}} & \lesssim\|\mathbf{1}_{r\geq A}rQ\|_{L^{2}}\lesssim A^{-1}
\end{align*}
to further estimate 
\begin{align*}
 & (\mathcal{L}_{Q_{b}^{(\eta)}}\epsilon-ib\Lambda_{A}\epsilon+\eta\theta^{(\eta)}\epsilon,[\Lambda Q^{(\eta)}]_{b})_{r}\\
 & =-2\eta\theta^{(\eta)}(\epsilon,Q_{b}^{(\eta)})_{r}+b^{2}(\epsilon,\tfrac{(r^{2}-2r\phi_{A}'(r))}{4}[\Lambda Q^{(\eta)}]_{b})_{r}\\
 & \quad+o_{A\to\infty}(\lambda\|\epsilon\|_{\dot{H}_{m}^{1}})+O(\eta^{\frac{3}{2}}\|\epsilon\|_{\dot{H}_{m}^{1}}).
\end{align*}

It now suffices to estimate the inner product $b^{2}(\epsilon,\tfrac{(r^{2}-2r\phi_{A}'(r))}{4}[\Lambda Q^{(\eta)}]_{b})_{r}$.
We decompose it into the parts $\mathbf{1}_{r\leq b^{-1/2}}$ and
$\mathbf{1}_{r\geq b^{-1/2}}$ and estimate the latter using integration
by parts. On one hand, we have 
\begin{align*}
 & b^{2}|(\epsilon,\mathbf{1}_{r\leq b^{-1/2}}\tfrac{(r^{2}-2r\phi_{A}'(r))}{4}[\Lambda Q^{(\eta)}]_{b})_{r}|\\
 & \lesssim b^{2}\|r^{-1}\epsilon\|_{L^{2}}\|\mathbf{1}_{r\leq b^{-1/2}}r^{3}Q\|_{L^{2}}\lesssim b^{\frac{3}{2}}\|\epsilon\|_{\dot{H}_{m}^{1}}.
\end{align*}
On the other hand, we integrate by parts with $\partial_{r}e^{-ib\frac{r^{2}}{4}}=-ib\frac{r}{2}\cdot e^{-ib\frac{r^{2}}{4}}$
to estimate 
\begin{align*}
 & b^{2}|(\epsilon,\mathbf{1}_{r\geq b^{-1/2}}\tfrac{(r^{2}-2r\phi_{A}'(r))}{4}[\Lambda Q^{(\eta)}]_{b})_{r}|\\
 & \lesssim b\|\epsilon\|_{\dot{H}_{m}^{1}}\|\mathbf{1}_{r\geq b^{-1/2}}rQ\|_{L^{2}}+b\|\epsilon\|_{L^{\infty}}([r^{2}Q]\big|_{r=b^{-1/2}})\\
 & \lesssim b^{\frac{3}{2}}\|\epsilon\|_{\dot{H}_{m}^{1}}.
\end{align*}
This completes the proof.
\end{proof}
\begin{proof}[Proof of Lemma \ref{lem:unav-Lyapunov/virial}]
The coercivity \eqref{eq:unav-Lyapunov-coer} follows from \eqref{eq:quad-energy-coer},
\eqref{eq:quad-bPhi-coer}.

We turn to show the derivative estimate \eqref{eq:unav-Lyapunov-deriv}.
Note that 
\begin{align*}
\lambda^{2}\partial_{s}\mathcal{I}_{A} & =(-\frac{2\lambda_{s}}{\lambda}+\partial_{s})(E_{w^{\flat}}^{\mathrm{(qd)}}[\epsilon]+b\Phi_{A}[\epsilon]+\tfrac{\eta\theta_{\eta}}{2}M[\epsilon]).
\end{align*}
On one hand, by the modulation estimate, we have 
\[
-\frac{2\lambda_{s}}{\lambda}(E_{w^{\flat}}^{\mathrm{(qd)}}[\epsilon]+b\Phi_{A}[\epsilon]+\tfrac{\eta\theta_{\eta}}{2}M[\epsilon])=-\frac{2\lambda_{s}}{\lambda}E_{w^{\flat}}^{\mathrm{(qd)}}[\epsilon]+b\eta\theta_{\eta}M[\epsilon]+\err.
\]
Unlike the term $-\frac{2b\lambda_{s}}{\lambda}\Phi_{A}[\epsilon]$,
we \emph{cannot} view $b\eta\theta_{\eta}M[\epsilon]$ as an error
term. Thus we are crucially using the sign condition $b,\eta\geq0$.

On the other hand, we use \eqref{eq:quad-energy-deriv-formal} and
\eqref{eq:quad-bPhi-deriv} to compute 
\begin{align*}
 & \partial_{s}[E_{w^{\flat}}^{\mathrm{(qd)}}[\epsilon]+b\Phi_{A}+\tfrac{\eta\theta_{\eta}}{2}M]\\
 & =\frac{\lambda_{s}}{\lambda}\{(R_{u^{\flat}-w^{\flat}},\Lambda z^{\flat})_{r}+(\mathcal{L}_{w^{\flat}}\epsilon+R_{u^{\flat}-w^{\flat}},(\Lambda-\Lambda_{A})\epsilon)_{r}\}\\
 & \quad+\eqref{eq:deriv-temp1}+\eqref{eq:deriv-temp2}+\err,
\end{align*}
where 
\begin{align}
 & -(\gamma_{s}-\eta\theta_{\eta})(\mathcal{L}_{w^{\flat}}\epsilon+R_{u^{\flat}-w^{\flat}},i\epsilon)_{r},\label{eq:deriv-temp1}\\
 & (\mathcal{L}_{Q_{b}^{(\eta)}}\epsilon-ib\Lambda_{A}\epsilon+\eta\theta_{\eta}\epsilon,\Big(\frac{\lambda_{s}}{\lambda}+b\Big)[\Lambda Q^{(\eta)}]_{b}-(\tilde{\gamma}_{s}-\eta\theta_{\eta})iQ_{b}^{(\eta)})_{r}.\label{eq:deriv-temp2}
\end{align}

We show that \eqref{eq:deriv-temp1} and \eqref{eq:deriv-temp2} can
be absorbed into $\err$. Indeed, \eqref{eq:deriv-temp1} is absorbed
into $\err$ by observing 
\begin{align*}
|\gamma_{s}-\eta\theta_{\eta}| & \lesssim|\theta_{z^{\flat}\to Q_{b}^{(\eta)}}|+|\tilde{\gamma}_{s}-\eta\theta_{\eta}|\lesssim\alpha^{\ast}\lambda^{2}+\|\epsilon\|_{\dot{H}_{m}^{1}},\\
|(\mathcal{L}_{w^{\flat}}\epsilon+R_{u^{\flat}-w^{\flat}},i\epsilon)_{r}| & \lesssim\|\epsilon\|_{\dot{H}_{m}^{1}}^{2}.
\end{align*}
For \eqref{eq:deriv-temp2}, we use Lemma \ref{lem:lem5.5}, the modulation
estimate \eqref{eq:mod-est-1}, and \emph{degeneracy estimate} \eqref{eq:mod-degen-est}
to get 
\begin{align*}
|\eqref{eq:deriv-temp2}| & \lesssim(\alpha^{\ast}\lambda^{m+3}|\log\lambda|+\alpha^{\ast}\lambda^{2}\eta+\|\epsilon\|_{\dot{H}_{m}^{1}})\cdot\Big(\lambda(o_{A\to\infty}(1)+\lambda^{\frac{1}{2}})\|\epsilon\|_{\dot{H}_{m}^{1}}\\
 & \qquad+\eta(\alpha^{\ast}\lambda^{m+2}|\log\lambda|+\alpha^{\ast}\lambda\eta+(\alpha^{\ast}+\lambda)\mathcal{T}_{\dot{H}_{m}^{1}}^{(\frac{5}{4})}[\epsilon])\Big).
\end{align*}
Applying the elementary bound $\eta\lesssim\lambda$, we see that
\eqref{eq:deriv-temp2} is absorbed into $\err$.

Therefore, we are led to 
\begin{align*}
 & \lambda^{2}\partial_{s}\mathcal{I}_{A}-b\eta\theta_{\eta}M[\epsilon]\\
 & =-\frac{\lambda_{s}}{\lambda}\Big(2E_{w^{\flat}}^{\mathrm{(qd)}}[\epsilon]-(\mathcal{L}_{w^{\flat}}\epsilon+R_{u^{\flat}-w^{\flat}},(\Lambda-\Lambda_{A})\epsilon)_{r}-(R_{u^{\flat}-w^{\flat}},\Lambda z^{\flat})_{r}\Big)+\err.
\end{align*}
One then integrate by parts the $(\Lambda-\Lambda_{A})$ portion to
get 
\begin{align}
 & \lambda^{2}\partial_{s}\mathcal{I}_{A}-b\eta\theta_{\eta}M[\epsilon]\label{eq:5.47}\\
 & =-\frac{\lambda_{s}}{\lambda}\Big(2E_{w^{\flat}}^{(A),\text{(qd)}}[\epsilon]-(R_{u^{\flat}-w^{\flat}},\Lambda_{A}z^{\flat})_{r}+(R_{u^{\flat}-w^{\flat}},(\Lambda-\Lambda_{A})Q_{b}^{(\eta)})_{r}\Big)+\mathrm{Err},\nonumber 
\end{align}
 where $E_{w^{\flat}}^{(A),\text{(qd)}}[\epsilon]$ is a deformation
of $E_{w^{\flat}}^{\text{(qd)}}[\epsilon]$, which coincides $E_{w^{\flat}}^{\text{(qd)}}[\epsilon]$
when $A=\infty$ formally. The precise formula of $E_{w^{\flat}}^{(A),\text{(qd)}}[\epsilon]$
is given by the quadratic and higher parts of $E^{(A)}[w^{\flat}+\epsilon]$,
where 
\begin{align*}
E^{(A)}[u^{\flat}] & =\frac{1}{2}\int\phi_{A}''|\partial_{r}u^{\flat}|^{2}-\frac{1}{8}\int(\Delta^{2}\phi_{A})|u^{\flat}|^{2}\\
 & \qquad-\frac{1}{4}\int\frac{\Delta\phi_{A}}{2}|u^{\flat}|^{4}+\frac{1}{2}\int\frac{\phi_{A}'}{r}\Big(\frac{m+A_{\theta}[u^{\flat}]}{r}\Big)^{2}|u^{\flat}|^{2}.
\end{align*}
More precisely, 
\begin{align}
E_{w^{\flat}}^{(A),\mathrm{(qd)}}[\epsilon] & =\frac{1}{2}\int\phi_{A}''|\partial_{r}\epsilon|^{2}-\frac{1}{8}\int(\Delta^{2}\phi_{A})|\epsilon|^{2}+\frac{1}{2}\int\frac{\phi_{A}'}{r}\cdot\frac{m^{2}}{r^{2}}|\epsilon|^{2}\label{eq:EA-quad-def}\\
 & \quad-\sum_{\substack{\psi_{1},\dots,\psi_{4}\in\{w^{\flat},\epsilon\}:\\
\#\{i:\psi_{i}=\epsilon\}\geq2
}
}[\frac{1}{4}\mathcal{M}_{4,0}^{(A)}+\frac{m}{2}\mathcal{M}_{4,1}^{(A)}]+\frac{1}{8}\sum_{\substack{\psi_{1},\dots,\psi_{6}\in\{w^{\flat},\epsilon\}:\\
\#\{i:\psi_{i}=\epsilon\}\geq2
}
}\mathcal{M}_{6}^{(A)}.\nonumber 
\end{align}

We further reduce \eqref{eq:5.47} to 
\begin{equation}
\lambda^{2}\partial_{s}\mathcal{I}_{A}-b\eta\theta_{\eta}M[\epsilon]=2bE_{w^{\flat}}^{(A),\text{(qd)}}[\epsilon]+\mathrm{Err}.\label{eq:I to EA}
\end{equation}
Indeed, we observe 
\begin{align*}
\Big|\frac{\lambda_{s}}{\lambda}(R_{u^{\flat}-w^{\flat}},\Lambda_{A}z^{\flat})_{r}\Big| & \lesssim A\alpha^{\ast}\lambda^{2}\|\epsilon\|_{\dot{H}_{m}^{1}}^{2},\\
\Big|\frac{\lambda_{s}}{\lambda}(R_{u^{\flat}-w^{\flat}},(\Lambda-\Lambda_{A})Q_{b}^{(\eta)})_{r}\Big| & \lesssim\lambda(o_{A\to\infty}(1)+\lambda)\|\epsilon\|_{\dot{H}_{m}^{1}}^{2}+\lambda^{2}\|\epsilon\|_{L^{2}}\|\epsilon\|_{\dot{H}_{m}^{1}},\\
\Big|\Big(\frac{\lambda_{s}}{\lambda}+b\Big)E_{w^{\flat}}^{(A),\mathrm{(qd)}}[\epsilon]\Big| & \lesssim\lambda^{\frac{3}{2}}\|\epsilon\|_{\dot{H}_{m}^{1}}^{2},
\end{align*}
using 
\[
\Big|\frac{\lambda_{s}}{\lambda}\Big|\lesssim\lambda,\quad|\Lambda_{A}z^{\flat}|\lesssim A(|\partial_{r}z^{\flat}|+|r^{-1}z^{\flat}|),\quad|(\Lambda-\Lambda_{A})Q_{b}^{(\eta)}|\lesssim\mathbf{1}_{r\geq A}(r^{-3}+br^{-1}),
\]
the modulation estimate \eqref{eq:mod-est-1}, and weak bootstrap
hypothesis \eqref{eq:bootstrap-hyp}.

We now treat the term $2E_{w^{\flat}}^{(A),\text{(qd)}}[\epsilon]$
to conclude the proof of \eqref{eq:unav-Lyapunov-deriv}. Recall the
expression \eqref{eq:EA-quad-def}. We can replace all the occurences
of $w^{\flat}$ by $Q$. Indeed, we use the multilinear estimates
(Lemma \ref{lem:duality-1}) with $\|w^{\flat}-Q\|_{L^{2}}\lesssim\alpha^{\ast}+\lambda|\log\lambda|^{\frac{1}{2}}$
to see that those terms only contribute up to errors $(\alpha^{\ast}+\lambda|\log\lambda|^{\frac{1}{2}})\|\epsilon\|_{\dot{H}_{m}^{1}}^{2}$.
Moreover, we can eliminate terms containing at least three $\epsilon$'s
up to errors $\|\epsilon\|_{L^{2}}\|\epsilon\|_{\dot{H}_{m}^{1}}^{2}$.
Thus we have (recall the formula of $\mathcal{J}_{A}$ given in \eqref{eq:def-J})
\[
E_{w^{\flat}}^{(A),\mathrm{(qd)}}[\epsilon]=\mathcal{J}_{A}-\frac{1}{8}\int(\Delta^{2}\phi_{A})|\epsilon|^{2}+O\big((\alpha^{\ast}+\|\epsilon\|_{L^{2}}+\lambda|\log\lambda|^{\frac{1}{2}})\|\epsilon\|_{\dot{H}_{m}^{1}}^{2}\big).
\]
Substituting this into \eqref{eq:I to EA}, we conclude \eqref{eq:unav-Lyapunov-deriv}.
\end{proof}
The quantity 
\[
-\frac{1}{8}\int(\Delta^{2}\phi_{A})|\epsilon|^{2}
\]
does not have a good sign. Moreover, a crude bound 
\[
\Big|\int(\Delta^{2}\phi_{A})|\epsilon|^{2}\Big|\lesssim\Big|\int\mathbf{1}_{r\geq A}\frac{A}{r}\cdot\Big|\frac{\epsilon}{r}\Big|^{2}\Big|\lesssim\|r^{-1}\epsilon\|_{L^{2}}^{2}\lesssim\|\epsilon\|_{\dot{H}_{m}^{1}}^{2}
\]
does not suffice. In order to go beyond this borderline, we crucially
use the bound 
\[
|\Delta^{2}\phi_{A}|\lesssim\mathbf{1}_{r\geq A}\frac{A}{r^{3}},
\]
which should be better than $\frac{1}{r^{2}}$\@. We will use an
average argument.

We finally define the \emph{Lyapunov/virial functional }by 
\[
\mathcal{I}\coloneqq\frac{2}{\log A}\int_{A^{1/2}}^{A}\mathcal{I}_{A'}\frac{dA'}{A'}=\lambda^{-2}\Big(E_{w^{\flat}}^{\mathrm{(qd)}}[\epsilon]+\frac{2}{\log A}\int_{A^{1/2}}^{A}b\Phi_{A'}[\epsilon]\frac{dA'}{A'}+\tfrac{\eta\theta_{\eta}}{2}M[\epsilon]\Big).
\]

\begin{prop}[Lyapunov/virial functional]
\label{prop:lyapunov}We have 
\begin{gather}
\lambda^{2}\mathcal{I}-\tfrac{\eta\theta_{\eta}}{2}M[\epsilon]+O\big((\alpha^{\ast}+\lambda^{\frac{1}{2}})\|\epsilon\|_{\dot{H}_{m}^{1}}^{2}+A\lambda\|\epsilon\|_{L^{2}}\|\epsilon\|_{\dot{H}_{m}^{1}}\big)\sim\|\epsilon\|_{\dot{H}_{m}^{1}}^{2},\label{eq:Lyapunov-coer}\\
\lambda^{2}\partial_{s}\mathcal{I}\geq\mathrm{Err}.\label{eq:Lyapunov-deriv}
\end{gather}
\end{prop}

\begin{proof}
The coercivity \eqref{eq:Lyapunov-coer} follows from \eqref{eq:unav-Lyapunov-coer}
and an averaging process.

We turn to show \eqref{eq:Lyapunov-deriv}. Let 
\[
\mathcal{J}\coloneqq\frac{2}{\log A}\int_{A^{1/2}}^{A}\mathcal{J}_{A'}\frac{dA'}{A'}.
\]
From the expression \eqref{eq:unav-Lyapunov-deriv} and $b\lesssim\lambda$,
it suffices to show 
\begin{equation}
\frac{2}{\log A}\int_{A^{1/2}}^{A}\Big(\frac{1}{8}\int(\Delta^{2}\phi_{A})|\epsilon|^{2}\Big)\frac{dA'}{A'}=o_{A\to\infty}(1)\|\epsilon\|_{\dot{H}_{m}^{1}}^{2}\label{eq:avg-error}
\end{equation}
and 
\begin{equation}
\mathcal{J}\geq o_{A\to\infty}(\|\epsilon\|_{\dot{H}_{m}^{1}}^{2}).\label{eq:J-almost-pos}
\end{equation}

We first show \eqref{eq:avg-error}. The crucial observations are
\[
\frac{2}{\log A}=o_{A\to\infty}(1)\quad\text{and}\quad\int_{A^{1/2}}^{A}|\Delta^{2}\phi_{A'}|\frac{dA'}{A'}\lesssim\mathbf{1}_{r\geq A^{1/2}}\frac{1}{r^{2}}.
\]
By Fubini's theorem, we see that 
\[
\frac{2}{\log A}\Big|\int_{A^{1/2}}^{A}\Big(\frac{1}{8}\int(\Delta^{2}\phi_{A'})|\epsilon|^{2}\Big)\frac{dA'}{A'}\Big|\lesssim o_{A\to\infty}(1)\|r^{-1}\epsilon\|_{L^{2}}^{2}.
\]
This shows \eqref{eq:avg-error}.

The rest of the proof is devoted to show the almost positivity \eqref{eq:J-almost-pos}.
As there is a truncation on the region $\{r\leq A'\}$ in the expression
\eqref{eq:def-J} of $\mathcal{J}_{A'}$, it is natural to compare
$\mathcal{J}_{A'}$ with a truncated version of $\frac{1}{2}(\mathcal{L}_{Q}\epsilon,\epsilon)_{r}=\frac{1}{2}\|L_{Q}\epsilon\|_{L^{2}}^{2}$
on the region $\{r\leq A'\}$. To realize this, we decompose the outermost
integral $\int$ of $\mathcal{J}_{A'}$ as $\int\mathbf{1}_{r\leq A'}+\int\mathbf{1}_{r>A'}\eqqcolon\mathcal{J}_{A'}^{(\leq)}+\mathcal{J}_{A'}^{(>)}$.
If we can show that $\mathcal{J}_{A'}^{(\leq)}\approx\frac{1}{2}\|L_{Q}\epsilon\|_{L^{2}(r\leq A')}^{2}$
and $\mathcal{J}_{A'}^{(>)}\gtrsim0$ in an averaged sense, then the
almost positivity of $\mathcal{J}$ would follow.

For $\mathcal{J}_{A'}^{(\leq)}$, we can replace $\phi_{A'}(r)$ by
$\frac{1}{2}r^{2}$ and expect that $\mathcal{J}_{A'}^{(\leq)}$ is
close to $\frac{1}{2}\|L_{Q}\epsilon\|_{L^{2}(r\leq A')}^{2}$. We
claim that 
\begin{equation}
\frac{2}{\log A}\int_{A^{\frac{1}{2}}}^{A}\mathcal{J}_{A'}^{(\leq)}\frac{dA'}{A'}=\frac{1}{\log A}\int_{A^{\frac{1}{2}}}^{A}\Big(\int\mathbf{1}_{r\leq A'}|L_{Q}\epsilon|^{2}\Big)\frac{dA'}{A'}+o_{A\to\infty}(\|\epsilon\|_{\dot{H}_{m}^{1}}^{2}).\label{eq:avg-1}
\end{equation}
To see this, we  observe by integration by parts that 
\[
\frac{1}{2}\Big(\int\mathbf{1}_{r\leq A'}|L_{Q}\epsilon|^{2}\Big)-\mathcal{J}_{A'}^{(\leq)}=2\pi\cdot\Big(\frac{1}{2}(m+A_{\theta}[Q])|\epsilon|^{2}-\epsilon\partial_{r}(Q\int_{0}^{r}\Re(Q\epsilon)r'dr')\Big)\Big|_{r=0}^{A'}.
\]
Note that the boundary value at $r=0$ clearly vanishes. We can estimate
the above as 
\[
\Big|\frac{1}{2}\Big(\int\mathbf{1}_{r\leq A'}|L_{Q}\epsilon|^{2}\Big)-\mathcal{J}_{A'}^{(\leq)}\Big|\lesssim|\epsilon(s,A')|^{2}+o_{A'\to\infty}(\|\epsilon\|_{\dot{H}_{m}^{1}}^{2}).
\]
Notice that we should not use a crude bound $|\epsilon(s,A')|^{2}\lesssim\|\epsilon\|_{\dot{H}_{m}^{1}}^{2}$,
which leads to an unacceptable error. Recalling that $\mathcal{J}$
is an \emph{averaged} version of $\mathcal{J}_{A'}$ (this is also
one of the reason why we need to average $\mathcal{J}_{A'}$), we
obtain 
\begin{align*}
 & \frac{2}{\log A}\Big|\int_{A^{1/2}}^{A}\Big[\mathcal{J}_{A'}^{(\leq)}-\frac{1}{2}\Big(\int_{0}^{A'}\mathbf{1}_{r\leq A'}|L_{Q}\epsilon|^{2}\Big)\Big]\frac{dA'}{A'}\Big|\\
 & \qquad\lesssim\frac{2}{\log A}\int_{A^{1/2}}^{A}\frac{|\epsilon(s,A')|^{2}}{(A')^{2}}A'dA'+o_{A\to\infty}(\|\epsilon\|_{\dot{H}_{m}^{1}}^{2})\lesssim o_{A\to\infty}(\|\epsilon\|_{\dot{H}_{m}^{1}}^{2}).
\end{align*}
This proves \eqref{eq:avg-1}.

For $\mathcal{J}_{A'}^{(>)}$, we claim that 
\begin{equation}
\frac{2}{\log A}\int_{A^{\frac{1}{2}}}^{A}\mathcal{J}_{A'}^{(>)}\frac{dA'}{A'}\geq o_{A\to\infty}(\|\epsilon\|_{\dot{H}_{m}^{1}}^{2}).\label{eq:avg-2}
\end{equation}
To see this, we use smallness of $\|Q\mathbf{1}_{r\geq A'}\|_{L^{2}}\lesssim(A')^{-(m+1)}$
to rearrange 
\[
\mathcal{J}_{A'}^{(>)}=\frac{1}{2}\int\mathbf{1}_{r\geq A'}\phi_{A'}''|\partial_{r}\epsilon|^{2}+\frac{1}{2}\int\mathbf{1}_{r\geq A'}\frac{\phi_{A'}'}{r}\Big(\frac{m+A_{\theta}[Q]}{r}\Big)^{2}|\epsilon|^{2}+o_{A'\to\infty}(\|\epsilon\|_{\dot{H}_{m}^{1}}^{2}).
\]
Discarding nonnegative terms and averaging in $A'$, we obtain \eqref{eq:avg-2}.

Gathering the estimates \eqref{eq:avg-1} and \eqref{eq:avg-2}, we
obtain 
\[
\mathcal{J}\geq o_{A\to\infty}(\|\epsilon\|_{\dot{H}_{m}^{1}}^{2}).
\]
This completes the proof of \eqref{eq:J-almost-pos}, and hence the
proposition.
\end{proof}

\subsection{Closing the bootstrap}

Finally, we can finish the proof of our main bootstrap Lemma \ref{lem:bootstrap}.
The main tools are modulation estimates (Lemma \ref{lem:mod-est}),
$L^{2}$-estimate of $\epsilon$ (Lemma \ref{lem:5.16}), and Lyapunov/virial
estimate (Proposition \ref{prop:lyapunov}).
\begin{proof}[Proof of Lemma \ref{lem:bootstrap} (Finish)]
Applying \eqref{eq:Lyapunov-deriv} and the fundamental theorem of
calculus with $\epsilon(0)=0$, we have 
\begin{align*}
\lambda^{2}\mathcal{I}(t) & \leq\lambda^{2}\int_{t}^{0}\frac{1}{\lambda^{2}}(\partial_{s}\mathcal{I})dt'\lesssim\lambda^{2}\int_{t}^{0}\frac{1}{\lambda^{4}}|\mathrm{Err}|dt'.
\end{align*}
Applying \eqref{eq:Lyapunov-coer}, we obtain 
\[
\|\epsilon\|_{\dot{H}_{m}^{1}}^{2}\lesssim(\alpha^{\ast}+\lambda^{\frac{1}{2}})\|\epsilon\|_{\dot{H}_{m}^{1}}^{2}+A\lambda\|\epsilon\|_{L^{2}}\|\epsilon\|_{\dot{H}_{m}^{1}}+\lambda^{2}\int_{t}^{0}\frac{1}{\lambda^{4}}|\err|dt'.
\]
We then substitute into the above the $L^{2}$-bound \eqref{eq:L2-est},
error bound \eqref{eq:err}, and maximal function estimate \eqref{eq:env-prop4}.
But we need an extra care when we substitute $\lambda\cdot A(\alpha^{\ast}\lambda\eta)^{2}$
in place of $\err$. Indeed, we estimate it as 
\begin{align*}
\lambda^{2}\int_{t}^{0}\frac{1}{\lambda^{4}}\cdot\lambda A(\alpha^{\ast}\lambda\eta)^{2}dt' & =A\lambda^{2}\int_{t}^{0}\frac{(\alpha^{\ast}\eta)^{2}}{\lambda}dt'\\
 & \lesssim A\lambda^{2}\int_{t}^{0}\frac{(\alpha^{\ast}\eta^{\frac{3}{4}})^{2}}{\lambda^{\frac{1}{2}}}dt'\lesssim A(\alpha^{\ast}\lambda^{\frac{5}{4}}\eta^{\frac{3}{4}})^{2}.
\end{align*}
Thus we obtain 
\[
\|\epsilon\|_{\dot{H}_{m}^{1}}^{2}\lesssim A(\alpha^{\ast}\lambda^{m+2}+\alpha^{\ast}\lambda^{\frac{5}{4}}\eta^{\frac{3}{4}})^{2}+(\alpha^{\ast}+o_{A\to\infty}(1)+A\lambda^{\frac{1}{4}})(\mathcal{T}_{\dot{H}_{m}^{1}}^{(\eta,\frac{5}{4})}[\epsilon])^{2}.
\]
Applying the time maximal function with \eqref{eq:env-prop2} and
\eqref{eq:env-prop3}, we get 
\begin{align*}
\mathcal{T}_{\dot{H}_{m}^{1}}^{(\eta,\frac{5}{4})}[\epsilon] & \lesssim A^{\frac{1}{2}}(\alpha^{\ast}\lambda^{m+2}+\alpha^{\ast}\lambda^{\frac{5}{4}}\eta^{\frac{3}{4}})\\
 & \qquad+(\alpha^{\ast}+o_{A\to\infty}(1)+A\lambda^{\frac{1}{4}})^{\frac{1}{2}}\mathcal{T}_{\dot{H}_{m}^{1}}^{(\eta,\frac{5}{4})}[\epsilon].
\end{align*}
Note that $\mathcal{T}_{\dot{H}_{m}^{1}}^{(\eta,\frac{5}{4})}[\epsilon]$
has finite value.\footnote{As we are in the case $\eta>0$, $\mathcal{T}_{\dot{H}_{m}^{1}}^{(\eta,\frac{5}{4})}[\epsilon]$
is by definition finite.} If we choose $A$ large such that $o_{A\to\infty}(1)\ll1$ and $t_{0}^{\ast}$
small such that $A\lambda^{\frac{1}{4}}\lesssim A\langle t_{0}^{\ast}\rangle^{\frac{1}{4}}\ll1$,
then for all sufficiently small $\alpha^{\ast}$ we have 
\[
\|\epsilon(t)\|_{\dot{H}_{m}^{1}}\leq\mathcal{T}_{\dot{H}_{m}^{1}}^{(\eta,\frac{5}{4})}[\epsilon]\lesssim\alpha^{\ast}\lambda^{m+2}+\alpha^{\ast}\lambda^{\frac{5}{4}}\eta^{\frac{3}{4}}.
\]
Applying this strong bound into Lemma \ref{eq:L2-est}, we get 
\[
\|\epsilon(t)\|_{L^{2}}\lesssim\alpha^{\ast}\lambda^{m+\frac{5}{4}}+\alpha^{\ast}\lambda^{\frac{1}{2}}\eta^{\frac{3}{4}}.
\]

We turn to estimate the size $|\lambda(t)-\langle t\rangle|$ and
$|b(t)-|t||$. Applying the modulation estimates (Lemma \ref{lem:mod-est}),
we have 
\begin{align*}
\Big|\Big(\frac{\lambda^{2}}{b^{2}+\eta^{2}}\Big)_{t}\Big| & =\frac{2}{\lambda^{2}}\Big(\frac{\lambda^{2}}{b^{2}+\eta^{2}}\Big)\Big|\Big(\frac{\lambda_{s}}{\lambda}+b\Big)-\frac{(b_{s}+b^{2}+\eta^{2})b}{b^{2}+\eta^{2}}\Big|\\
 & \lesssim\alpha^{\ast}\lambda^{m}+\alpha^{\ast}\lambda^{-\frac{3}{4}}\eta^{\frac{3}{4}}.
\end{align*}
Integrating this from $t=0$, we get 
\begin{equation}
\Big|\frac{\lambda^{2}}{b^{2}+\eta^{2}}-1\Big|\lesssim\alpha^{\ast}\lambda^{m+1}+\alpha^{\ast}\lambda^{\frac{1}{4}}\eta^{\frac{3}{4}}.\label{eq:closing-boot-temp1}
\end{equation}
By \eqref{eq:closing-boot-temp1} and modulation estimates (Lemma
\ref{lem:mod-est}), 
\begin{align*}
|b_{t}+1| & =\frac{1}{\lambda^{2}}\Big|(b_{s}+b^{2}+\eta^{2})+(b^{2}+\eta^{2})\Big(\frac{\lambda^{2}}{b^{2}+\eta^{2}}-1\Big)\Big|\\
 & \lesssim\alpha^{\ast}\lambda^{m+1}+\alpha^{\ast}\lambda^{\frac{1}{4}}\eta^{\frac{3}{4}}.
\end{align*}
Integrating this from $t=0$, we have 
\[
\big|b-|t|\big|\lesssim\alpha^{\ast}\lambda^{m+2}+\alpha^{\ast}\lambda^{\frac{5}{4}}\eta^{\frac{3}{4}}.
\]
Substituting this into \eqref{eq:closing-boot-temp1}, we get 
\[
|\lambda-\langle t\rangle|=\frac{|\lambda^{2}-(b^{2}+\eta^{2})+(b^{2}-|t|^{2})|}{\lambda+\langle t\rangle}\lesssim\alpha^{\ast}\lambda^{m+2}+\alpha^{\ast}\lambda^{\frac{5}{4}}\eta^{\frac{3}{4}}.
\]

Finally, we estimate $\gamma(t)$. By \eqref{eq:mod-est-1}, 
\[
\Big|\gamma_{t}-\frac{\eta\theta^{(\eta)}}{\lambda^{2}}+\theta_{z\to Q_{b}^{(\eta)\sharp}}\Big|=\frac{|\tilde{\gamma}_{s}-\eta\theta^{(\eta)}|}{\lambda^{2}}\lesssim\alpha^{\ast}\lambda^{m}+\alpha^{\ast}\lambda^{-\frac{3}{4}}\eta^{\frac{3}{4}}.
\]
Integrating this, we get 
\[
|\gamma-\gamma_{\eta}-\gamma_{\mathrm{cor}}^{(\eta)}|\lesssim\alpha^{\ast}\lambda^{m+1}+\alpha^{\ast}\lambda^{\frac{1}{4}}\eta^{\frac{3}{4}}.
\]
This completes the proof of Lemma \ref{lem:bootstrap}.
\end{proof}
\begin{rem}
The use of $\mathcal{T}_{\dot{H}_{m}^{1}}^{(\eta,\frac{5}{4})}[\epsilon]$
is not necessary in the proof. One can indeed use $\mathcal{T}_{\dot{H}_{m}^{1}}^{(\eta,s)}[\epsilon]$
for any $s\in(1,\frac{5}{4}]$. The condition $s>1$ is used to have
\[
\lambda^{2}\int_{t}^{0}\frac{1}{\lambda^{4}}\cdot\lambda(\mathcal{T}_{\dot{H}_{m}^{1}}^{(\eta,s)}[\epsilon])^{2}dt'\lesssim(\mathcal{T}_{\dot{H}_{m}^{1}}^{(\eta,s)}[\epsilon])^{2}
\]
in $\lambda^{2}\int_{t}^{0}\frac{1}{\lambda^{4}}|\err|dt'$. The condition
$s\leq\frac{5}{4}$ is used to have 
\[
\mathcal{T}^{(\eta,s)}[\alpha^{\ast}\lambda^{\frac{5}{4}}\eta^{\frac{3}{4}}]\lesssim\alpha^{\ast}\lambda^{\frac{5}{4}}\eta^{\frac{3}{4}}.
\]
\end{rem}

\begin{rem}
One can proceed the proof without using time maximal functions. In
that way, however, one needs to assume stronger bootstrap hypothesis,
for instance \cite{MerleRaphaelSzeftel2013AJM}. Using maximal functions,
we can verify that weak bootstrap hypothesis suffices to propagate
smallness of $\epsilon$ to the past times.
\end{rem}

\section{\label{sec:cond-uniq}Conditional uniqueness}

This section is devoted to the proof of Theorem \ref{thm:cond-uniq}.
As seen in the proof of Theorem \ref{thm:BW-sol}, we have constructed
blow-up solutions by a compactness argument. The (conditional) uniqueness
of the constructed solutions was essential in the proof of Theorem
\ref{thm:instability}, especially when we prove continuity of the
map $\eta\in[0,\eta^{\ast}]\mapsto u^{(\eta)}$. Overall argument
here is similar to as in Sections \ref{sec:modulation} and \ref{sec:Bootstrap}.
Given two solutions $u_{1}$ and $u_{2}$ satisfying the hypothesis
of Theorem \ref{thm:cond-uniq}, we will decompose each $u_{j}$ using
their own modulation parameters and $\epsilon_{j}$. After estimating
differences of the modulation parameters, we show $\epsilon=\epsilon_{1}-\epsilon_{2}=0$,
by the Lyapunov/virial argument. Our approach is inspired by \cite{MartelMerleRaphael2015JEMS}.\footnote{There is an another strategy to prove the conditional uniqueness,
used in \cite{MerleRaphaelSzeftel2013AJM,RaphaelSzeftel2011JAMS}
in the context of \eqref{eq:NLS}. There, they do not impose any orthogonality
conditions on $\epsilon$. Rather, they use a Lyapunov/virial argument
to observe that $\epsilon$ can be controlled by some inner products,
say $(\epsilon,\psi_{b})_{r}$, where $\psi$ is in the generalized
nullspace of the linearized operator. On the other hand, one can control
$(\epsilon,\psi_{b})_{r}$ by a small constant times $\epsilon$,
by differentiating it in the $s$-variable several times. This argument
yields $\epsilon=0$. However, in case of \eqref{eq:CSS}, neither
the quantities $(\epsilon,i\rho_{b})_{r}$ nor $(\epsilon,|y|^{2}Q_{b})_{r}$
make sense if $m\in\{1,2\}$ due to the lack of decay of $Q$. Because
of this, we could not follow their strategy.}

We will decompose each $u_{j}(t)$ using modulation parameters $b_{j}(t)$,
$\lambda_{j}(t)$, and $\gamma_{j}(t)$. We use the notation for modulated
functions by 
\[
f_{b,\lambda,\gamma}(x)\coloneqq\frac{1}{\lambda}f_{b}\Big(\frac{x}{\lambda}\Big)e^{i\gamma}
\]
for $(b,\lambda,\gamma)\in\R\times\R_{+}\times\R$.

Next, we will use $\sharp/\flat$ operations and dynamical rescalings.\emph{
Importantly, in the $\sharp/\flat$ operations, we only use $\lambda_{1}(t)$
and $\gamma_{1}(t)$}. Namely, if $f(s,y)$ and $g(t,x)$ are functions
of $(s,y)$ and $(t,x)$, respectively, then we denote 
\begin{align*}
f^{\sharp}(t,x) & =\frac{1}{\lambda_{1}(t)}f(s(t),\frac{x}{\lambda_{1}(t)})e^{i\gamma_{1}(t)},\\
f^{\flat}(s,y) & =\lambda_{1}(s)g(t(s),\lambda_{1}(s)y)e^{-i\gamma_{1}(s)}.
\end{align*}
The relation between $(s,y)$ and $(t,x)$ is given by 
\begin{equation}
\frac{ds}{dt}\coloneqq\frac{1}{\lambda_{1}^{2}}\quad\text{and}\quad y\coloneqq\frac{x}{\lambda_{1}}.\label{eq:uniq-dynamic-rescale}
\end{equation}
Thus we will \emph{not} use $\lambda_{2}(t)$ and $\gamma_{2}(t)$
when we rescale our solutions.

Finally, since we are in the $\eta=0$ case, we will impose the \emph{special
relation} \eqref{eq:solve-law} between $b_{j}$ and $\lambda_{j}$:
\begin{equation}
b_{j}(t)=|t|^{-1}(\lambda_{j})^{2}(t).\label{eq:uniq-special-rel}
\end{equation}
The reasons for \eqref{eq:uniq-special-rel} are twofold. The first
reason is as same as the dynamical law \eqref{eq:law-fixation} with
$\eta=0$; the relation \eqref{eq:uniq-special-rel} cancels out the
terms with the slowest spatial decay in the equation of $\epsilon_{j}$.
Secondly, \eqref{eq:uniq-special-rel} fixes the ratio $b_{j}/(\lambda_{j})^{2}$
and makes $\|Q_{b_{1},\lambda_{1},0}-Q_{b_{2},\lambda_{2},0}\|_{L^{2}}\lesssim|\log(\frac{\lambda_{1}}{\lambda_{2}})|$
possible. One may compare this with the estimate $\|Q_{b}-Q\|_{L^{2}}\lesssim b|\log b|^{\frac{1}{2}}$,
which is worse than \eqref{eq:prelim-est-diff} by a logarithmic factor.
Having \eqref{eq:prelim-est-diff} is crucial in the proof of Theorem
\ref{thm:cond-uniq}.\footnote{Heuristically speaking, the situation is similar to when we prove
uniqueness assertion of the contraction mapping principle. The uniqueness
is guaranteed when the map is Lipschitz continuous with the Lipschitz
constant less than $1$, not when the map is merely H\"older continuous.}
\begin{lem}[Estimate of difference]
\label{lem:est-diff}Let $(b_{j},\lambda_{j},\gamma_{j})\in\R\times\R_{+}\times\R$
for $j\in\{1,2\}$ be such that 
\begin{equation}
\frac{b_{1}}{(\lambda_{1})^{2}}=\frac{b_{2}}{(\lambda_{2})^{2}}\label{eq:uniq-est-diff-temp1}
\end{equation}
and $|b_{1}|+|\frac{\lambda_{1}}{\lambda_{2}}-1|+|\gamma_{1}-\gamma_{2}|\leq\frac{1}{2}$.
Then, we have 
\begin{equation}
\|\psi_{b_{1}}-\psi_{b_{2},\frac{\lambda_{2}}{\lambda_{1}},\gamma_{2}-\gamma_{1}}\|_{H_{m}^{1}}\lesssim\Big(\Big|\log\Big(\frac{\lambda_{1}}{\lambda_{2}}\Big)\Big|+|\gamma_{1}-\gamma_{2}|\Big),\qquad\forall\psi\in\{Q,\Lambda Q\}.\label{eq:prelim-est-diff}
\end{equation}
\end{lem}

\begin{proof}
Define the path $f:[0,1]\to H_{m}^{1}$ by 
\begin{equation}
f(\tau)\coloneqq\psi_{b_{1}(\lambda_{1}^{-\tau}\lambda_{2}^{\tau})^{2},\lambda_{1}^{-\tau}\lambda_{2}^{\tau},(\gamma_{2}-\gamma_{1})\tau}.\label{eq:def-f(t)}
\end{equation}
Note that $f(0)=\psi_{b_{1}}$ and $f(1)=\psi_{b_{2},\frac{\lambda_{2}}{\lambda_{1}},\gamma_{2}-\gamma_{1}}$.
Because of \eqref{eq:uniq-est-diff-temp1}, we observe 
\[
\partial_{\tau}f=[-\log\Big(\frac{\lambda_{2}}{\lambda_{1}}\Big)\cdot\Lambda\psi+(\gamma_{2}-\gamma_{1})\cdot i\psi]_{b_{1}(\lambda_{1}^{-\tau}\lambda_{2}^{\tau})^{2},\lambda_{1}^{-\tau}\lambda_{2}^{\tau},(\gamma_{2}-\gamma_{1})\tau}.
\]
By the fundamental theorem of calculus and Minkowski's inequality,
the estimate \eqref{eq:prelim-est-diff} follows.
\end{proof}

\subsection{A priori estimates on $\epsilon_{1}$ and $\epsilon_{2}$}

Let 
\[
\theta_{\mathrm{cor}}(t)\coloneqq-\int_{0}^{\infty}(-m-2+A_{\theta}[z(t)])|z(t)|^{2}\frac{dr}{r}.
\]
Note that $\theta_{\mathrm{cor}}$ is equal to $\theta_{z\to Q_{b}^{\sharp}}$
in earlier sections. Recall also that 
\[
\gamma_{\mathrm{cor}}(t)=-\int_{t}^{0}\theta_{\mathrm{cor}}(t')dt'.
\]

\begin{lem}[A priori estimates on $\epsilon_{1}$ and $\epsilon_{2}$]
Assume that two solutions $u_{1}$ and $u_{2}$ satisfy the hypothesis
of Theorem \ref{thm:cond-uniq} for some sufficiently small $c>0$.
For each $j\in\{1,2\}$, there exists unique $(b_{j}(t),\lambda_{j}(t),\gamma_{j}(t))\in\R\times\R_{+}\times\R$
for all $t$ near $0$ satisfying the following properties.
\begin{enumerate}
\item (Decomposition) $u_{j}$ has the decomposition 
\begin{equation}
u_{j}(t,x)=\frac{e^{i\gamma_{j}(t)}}{\lambda_{j}(t)}Q_{b_{j}(t)}\Big(\frac{x}{\lambda_{j}(t)}\Big)+z(t,x)+\frac{e^{i\gamma_{1}(t)}}{\lambda_{1}(t)}\epsilon_{j}(t,\frac{x}{\lambda_{1}(t)})\label{eq:uniq-decomp}
\end{equation}
with 
\begin{gather}
\Big|\frac{\lambda_{j}}{|t|}-1\Big|+|\gamma_{j}-\gamma_{\mathrm{cor}}|+\|\epsilon_{j}\|_{\dot{H}_{m}^{1}}+|t|\|\epsilon_{j}\|_{L_{m}^{2}}\lesssim c|t|^{2},\label{eq:uniq-a-pri-est}\\
(\epsilon_{j},[\mathcal{Z}_{\re}]_{b_{1}})_{r}=(\epsilon_{j},[i\mathcal{Z}_{\im}]_{b_{1}})_{r}=0,\label{eq:uniq-ortho}\\
b_{j}=|t|^{-1}\lambda_{j}^{2}.\label{eq:uniq-dyn-law}
\end{gather}
\item (Equation of $\epsilon_{j}$) We have 
\begin{align}
i\partial_{t}\epsilon_{j}^{\sharp}-\mathcal{L}_{w_{j}}\epsilon_{j}^{\sharp} & =i\Big(\frac{(\lambda_{j})_{t}}{\lambda_{j}}+|t|^{-1}\Big)[\Lambda Q]_{b_{j},\lambda_{j},\gamma_{j}}+((\gamma_{j})_{t}+\theta_{\mathrm{cor}})Q_{b_{j},\lambda_{j},\gamma_{j}}\label{eq:uniq-eq-e-j-sharp}\\
 & \quad+\tilde R_{Q_{b_{j},\lambda_{j},\gamma_{j}},z}+R_{u_{j}-w_{j}}.\nonumber 
\end{align}
\item (A priori modulation estimates) We have 
\begin{equation}
\Big|\frac{(\lambda_{j})_{s}}{\lambda_{j}}+\big(\frac{\lambda_{1}}{\lambda_{j}}\big)^{2}b_{j}\Big|+|(\gamma_{j})_{s}+(\lambda_{1})^{2}\theta_{\mathrm{cor}}|\lesssim(\alpha^{\ast}+c)(\lambda_{1})^{2}.\label{eq:uniq-a-pri-mod-est}
\end{equation}
\end{enumerate}
\end{lem}

\begin{rem}
If we take the $\flat$-operation to \eqref{eq:uniq-decomp} using
$\lambda_{1}$ and $\gamma_{1}$, then we can write 
\[
u_{1}^{\flat}-u_{2}^{\flat}=[Q_{b_{1}}-Q_{b_{2},\frac{\lambda_{2}}{\lambda_{1}},\gamma_{2}-\gamma_{1}}]+\epsilon_{1}-\epsilon_{2}.
\]
Here, we do not see modulation errors from $z$. Notice also that
$\epsilon_{1}$ and $\epsilon_{2}$ are in the same scale, so we can
study $\epsilon=\epsilon_{1}-\epsilon_{2}$ without modulation errors.
\end{rem}

\begin{proof}
The decomposition of $u_{j}$ will follow from the implicit function
theorem. Let 
\[
\tilde u_{j}(t,y)\coloneqq|t|[u-z](t,|t|y)e^{-i\gamma_{\mathrm{cor}}(t)}.
\]
We then have 
\begin{equation}
\|\tilde u_{j}(t)-Q_{|t|}\|_{\dot{H}_{m}^{1}}\leq c|t|^{2}\quad\text{and}\quad\|\tilde u_{j}(t)-Q_{|t|}\|_{L^{2}}\leq c|t|\label{eq:uniq-a-pri-temp1}
\end{equation}
for all $t$ near $0$.

We first consider the case $j=1$. For all $t$ near $0$, define
the function $F_{1}^{(t)}:\dot{H}_{m}^{1}\times\R\times\R_{+}\times\R\to\R^{3}$
by 
\[
F^{(t)}(v,b,\lambda,\gamma)\coloneqq\begin{pmatrix}(v,[\mathcal{Z}_{\re}]_{b,\lambda,\gamma})_{r}-(Q,\mathcal{Z}_{\re})_{r}\\
(v,[i\mathcal{Z}_{\im}]_{b,\lambda,\gamma})_{r}\\
b-|t|\lambda^{2}
\end{pmatrix}.
\]
Because of $\mathcal{Z}_{\re},\mathcal{Z}_{\im}\in C_{c,m}^{\infty}(0,\infty)$,
we see that $F_{1}^{(t)}$ is well-defined and in fact smooth. Note
that 
\[
\frac{\delta F_{1}^{(t)}}{\delta v}=\begin{pmatrix}[\mathcal{Z}_{\re}]_{b,\lambda,\gamma}\\{}
[i\mathcal{Z}_{\im}]_{b,\lambda,\gamma}\\
0
\end{pmatrix}=\begin{pmatrix}\mathcal{Z}_{\re}\\
i\mathcal{Z}_{\im}\\
0
\end{pmatrix}+O_{(\dot{H}_{m}^{1})^{\ast}}(|b|+|\lambda-1|+|\gamma|)
\]
and 
\begin{align*}
\frac{\partial F_{1}^{(t)}}{\partial(b,\lambda,\gamma)} & =\begin{pmatrix}0 & -(\Lambda Q,\mathcal{Z}_{\re})_{r} & 0\\
(Q,\tfrac{|y|^{2}}{4}\mathcal{Z}_{\im})_{r} & 0 & (Q,\mathcal{Z}_{\im})_{r}\\
1 & 0 & 0
\end{pmatrix}\\
 & \quad+O(\|v-Q\|_{\dot{H}_{m}^{1}}+|t|+|b|+|\lambda-1|+|\gamma|).
\end{align*}
Finally, we notice $F_{1}^{(t)}(Q_{|t|},|t|,1,0)=0$ and recall \eqref{eq:uniq-a-pri-temp1}.
By the above estimates, we can apply the implicit function theorem
uniformly in $t$ near $0$ to have unique $(\tilde b_{1}(t),\tilde{\lambda}_{1}(t),\tilde{\gamma}_{1}(t))$
such that 
\[
F_{1}^{(t)}(\tilde u_{1}(t),\tilde b_{1}(t),\tilde{\lambda}_{1}(t),\tilde{\gamma}_{1}(t))=0
\]
and 
\begin{equation}
|\tilde b_{1}(t)-|t||+|\tilde{\lambda}_{1}(t)-1|+|\tilde{\gamma}_{1}(t)|\lesssim c|t|^{2}.\label{eq:uniq-a-pri-temp2}
\end{equation}
We now define 
\[
(b_{1},\lambda_{1},\gamma_{1})\coloneqq(\tilde b_{1},|t|\tilde{\lambda}_{1},\tilde{\gamma}_{1}+\gamma_{\mathrm{cor}})
\]
and recall the definition of $\tilde u_{1}$. The conditions \eqref{eq:uniq-ortho}
and \eqref{eq:uniq-dyn-law} for $j=1$ are clearly satisfied. For
the estimate \eqref{eq:uniq-a-pri-est}, we note that 
\begin{align*}
\|\epsilon\|_{\dot{H}_{m}^{1}} & =\|[\tilde u_{1}]_{\tilde{\lambda}_{1}^{-1}(t),-\tilde{\gamma}_{1}(t)}-Q_{b_{1}(t)}\|_{\dot{H}_{m}^{1}}\\
 & \lesssim\|\tilde u_{1}-Q_{|t|}\|_{\dot{H}_{m}^{1}}+\|Q_{|t|}-Q_{b_{1}(t),\tilde{\lambda}_{1}(t),\tilde{\gamma}_{1}(t)}\|_{\dot{H}_{m}^{1}}\lesssim c|t|^{2}.
\end{align*}
Here the last inequality exploits \eqref{eq:prelim-est-diff}. The
$L^{2}$-estimate can be done similarly. This shows \eqref{eq:uniq-a-pri-est}
for $j=1$.

We now consider the case $j=2$. For all $t$ near $0$, define the
function $F_{2}^{(t)}:\dot{H}_{m}^{1}\times\R\times\R_{+}\times\R\to\R^{3}$
by 
\[
F_{2}^{(t)}(v,b,\lambda,\gamma)\coloneqq\begin{pmatrix}(v-Q_{b,\lambda,\gamma},[\mathcal{Z}_{\re}]_{b_{1}(t),\tilde{\lambda}_{1}(t),\tilde{\gamma}_{1}(t)})_{r}\\
(v-Q_{b,\lambda,\gamma},[i\mathcal{Z}_{\im}]_{b_{1}(t),\tilde{\lambda}_{1}(t),\tilde{\gamma}_{1}(t)})_{r}\\
b-|t|\lambda^{2}
\end{pmatrix}.
\]
Because of $\mathcal{Z}_{\re},\mathcal{Z}_{\im}\in C_{c,m}^{\infty}(0,\infty)$,
we see that $F_{2}^{(t)}$ is smooth. We note that 
\[
\frac{\delta F_{2}^{(t)}}{\delta v}=\begin{pmatrix}[\mathcal{Z}_{\re}]_{b_{1}(t),\tilde{\lambda}_{1}(t),\tilde{\gamma}_{1}(t)}\\{}
[i\mathcal{Z}_{\im}]_{b_{1}(t),\tilde{\lambda}_{1}(t),\tilde{\gamma}_{1}(t)}\\
0
\end{pmatrix}=\begin{pmatrix}\mathcal{Z}_{\re}\\
i\mathcal{Z}_{\im}\\
0
\end{pmatrix}+O_{(\dot{H}_{m}^{1})^{\ast}}(c|t|^{2})
\]
and 
\begin{align*}
\frac{\partial F_{2}^{(t)}}{\partial(b,\lambda,\gamma)} & =\begin{pmatrix}0 & -(\Lambda Q,\mathcal{Z}_{\re})_{r} & 0\\
(Q,\tfrac{|y|^{2}}{4}\mathcal{Z}_{\im})_{r} & 0 & (Q,\mathcal{Z}_{\im})_{r}\\
1 & 0 & 0
\end{pmatrix}\\
 & \quad+O(\|v-Q\|_{\dot{H}_{m}^{1}}+|t|+|b|+|\lambda-1|+|\gamma|).
\end{align*}
Notice that $F_{2}^{(t)}(Q_{b_{1}(t),\tilde{\lambda}_{1}(t),\tilde{\gamma}_{1}(t)},b_{1}(t),\tilde{\lambda}_{1}(t),\tilde{\gamma}_{1}(t))=0$.
Applying the implicit function theorem, there exists unique $(\tilde b_{2}(t),\tilde{\lambda}_{2}(t),\tilde{\gamma}_{2}(t))$
such that 
\[
F_{1}^{(t)}(\tilde u_{2}(t),\tilde b_{2}(t),\tilde{\lambda}_{2}(t),\tilde{\gamma}_{2}(t))=0
\]
and 
\[
|\tilde b_{2}(t)-b_{1}(t)|+|\tilde{\lambda}_{2}(t)-\tilde{\lambda}_{1}(t)|+|\tilde{\gamma}_{2}(t)-\gamma_{2}(t)|\lesssim c|t|^{2}.
\]
If we define 
\[
(b_{2},\lambda_{2},\gamma_{2})\coloneqq(\tilde b_{2},|t|\tilde{\lambda}_{2},\tilde{\gamma}_{\mathrm{cor}}+\tilde{\gamma}_{2}),
\]
apply \eqref{eq:uniq-a-pri-temp2}, recall the definition of $\tilde u_{2}$,
and proceed as in the case $j=1$, then \eqref{eq:uniq-a-pri-est}-\eqref{eq:uniq-dyn-law}
for $j=2$ follows. This completes the proof of the decomposition
part of our lemma.

The equation \eqref{eq:uniq-eq-e-j-sharp} of $\epsilon_{j}^{\sharp}$
is clear from \eqref{eq:e-sharp-eq} with $\eta=0$. We turn to the
modulation estimate \eqref{eq:uniq-a-pri-mod-est}. We take the $\flat$-operation
to \eqref{eq:uniq-eq-e-j-sharp} and get 
\begin{align}
 & i\partial_{s}\epsilon_{j}-\mathcal{L}_{w_{j}^{\flat}}\epsilon_{j}+ib_{1}\Lambda\epsilon_{j}\label{eq:uniq-e-j-eqn}\\
 & =i\Big(\frac{(\lambda_{j})_{s}}{\lambda_{j}}+\frac{(\lambda_{1})^{2}}{(\lambda_{j})^{2}}b_{j}\Big)[\Lambda Q]_{b_{j},\frac{\lambda_{j}}{\lambda_{1}},\gamma_{j}-\gamma_{1}}+((\gamma_{j})_{s}+(\lambda_{1})^{2}\theta_{\mathrm{cor}})Q_{b_{j},\frac{\lambda_{j}}{\lambda_{1}},\gamma_{j}-\gamma_{1}}\nonumber \\
 & \quad+i\Big(\frac{(\lambda_{1})_{s}}{\lambda_{1}}+b_{1}\Big)\Lambda\epsilon_{j}+(\gamma_{1})_{s}\epsilon_{j}+\tilde R_{Q_{b_{j},\frac{\lambda_{j}}{\lambda_{1}},\gamma_{j}-\gamma_{1}},z^{\flat}}+R_{u_{j}^{\flat}-w_{j}^{\flat}}.\nonumber 
\end{align}
We then take the inner product of \eqref{eq:uniq-e-j-eqn} with $i\psi_{b_{1}}$
for $\psi\in\{\mathcal{Z}_{\re},i\mathcal{Z}_{\im}\}$ and using the
orthogonality condition $(\epsilon,\psi_{b_{1}})_{r}=0$. In case
of $j=1$, by the computation as in Lemma \ref{lem:mod-est}, we get
\[
\Big|\frac{(\lambda_{1})_{s}}{\lambda_{1}}+b_{1}\Big|+|(\gamma_{1})_{s}+(\lambda_{1})^{2}\theta_{\mathrm{cor}}|\lesssim\alpha^{\ast}(\lambda_{1})^{m+3}|\log\lambda_{1}|+\|\epsilon_{1}\|_{\dot{H}_{1}^{m}}\lesssim(\alpha^{\ast}+c)(\lambda_{1})^{2}.
\]
Because of \eqref{eq:uniq-special-rel}, we get 
\[
|(b_{1})_{s}+(b_{1})^{2}|\lesssim c(\lambda_{1})^{3}.
\]
Thus we have shown \eqref{eq:uniq-a-pri-mod-est} for $j=1$. In case
of $j=2$, we get 
\begin{align*}
 & \Big(\frac{(\lambda_{2})_{s}}{\lambda_{2}}+\frac{(\lambda_{1})^{2}}{(\lambda_{2})^{2}}b_{2}\Big)([\Lambda Q]_{b_{2},\frac{\lambda_{2}}{\lambda_{1}},\gamma_{2}-\gamma_{1}},\psi_{b_{1}})_{r}+((\gamma_{2})_{s}+(\lambda_{1})^{2}\theta_{\mathrm{cor}})(Q_{b_{2},\frac{\lambda_{2}}{\lambda_{1}},\gamma_{2}-\gamma_{1}},i\psi_{b_{1}})_{r}\\
 & =-(\epsilon_{2},\mathcal{L}_{w_{2}^{\flat}}i\psi_{b_{1}})_{r}+\frac{(\lambda_{1})_{s}}{\lambda_{1}}(\epsilon_{2},\Lambda\psi_{b_{1}})_{r}-(\gamma_{1})_{s}(\epsilon_{2},i\psi_{b_{1}})_{r}+(b_{1})_{s}(\epsilon_{2},i\tfrac{|y|^{2}}{4}\psi_{b_{1}})_{r}\\
 & \quad+(\tilde R_{Q_{b_{2},\frac{\lambda_{2}}{\lambda_{1}},\gamma_{2}-\gamma_{1}},z^{\flat}},i\psi_{b_{1}})_{r}+(R_{u_{2}^{\flat}-w_{2}^{\flat}},i\psi_{b_{1}})_{r}.
\end{align*}
Substituting the modulation estimates for $(b_{1},\lambda_{1},\gamma_{1})$
and estimating all the inner products by $\|\epsilon_{2}\|_{\dot{H}_{m}^{1}}+\alpha^{\ast}(\lambda_{1})^{m+3}|\log\lambda_{1}|$,
we get \eqref{eq:uniq-a-pri-mod-est} for $j=2$.
\end{proof}

\subsection{Estimates of $\epsilon$}

Let 
\[
\epsilon\coloneqq\epsilon_{1}-\epsilon_{2}.
\]
The main goal of this section is to derive the equation of $\epsilon$
and estimate error terms. The strategy is similar to that in Sections
\ref{subsec:est-RQ,z}-\ref{subsec:Transfer-H1-to-L2}. Main issue
here is to estimate the difference of two modulated profiles.

By \eqref{eq:uniq-eq-e-j-sharp}, the difference $\epsilon=\epsilon_{1}-\epsilon_{2}$
satisfies 
\begin{align}
 & i\partial_{t}\epsilon^{\sharp}-\mathcal{L}_{Q_{b_{1},\lambda_{1},\gamma_{1}}}\epsilon^{\sharp}\label{eq:uniq-e-sharp-eq}\\
 & =i\Big[\log\Big(\frac{\lambda_{1}}{\lambda_{2}}\Big)\Big]_{t}[\Lambda Q]_{b_{1},\lambda_{1},\gamma_{1}}+[\gamma_{1}-\gamma_{2}]_{t}Q_{b_{1},\lambda_{1},\gamma_{1}}+(\lambda_{1})^{-2}\Gamma^{\sharp},\nonumber 
\end{align}
where 
\begin{align}
(\lambda_{1})^{-2}\Gamma^{\sharp} & \coloneqq i\Big(\frac{(\lambda_{2})_{t}}{\lambda_{2}}+|t|^{-1}\Big)([\Lambda Q]_{b_{1},\lambda_{1},\gamma_{1}}-[\Lambda Q]_{b_{2},\lambda_{2},\gamma_{2}})\label{eq:uniq-def-Gamma-sharp}\\
 & \quad+((\gamma_{j})_{t}+\theta_{\mathrm{cor}})(Q_{b_{1},\lambda_{1},\gamma_{1}}-Q_{b_{2},\lambda_{2},\gamma_{2}})+(\mathcal{L}_{w_{1}}-\mathcal{L}_{Q_{b_{1},\lambda_{1},\gamma_{1}}})\epsilon^{\sharp}+(\mathcal{L}_{w_{1}}-\mathcal{L}_{w_{2}})\epsilon_{2}^{\sharp}\nonumber \\
 & \quad+(R_{u_{1}-w_{1}}-R_{u_{2}-w_{2}})+(\tilde R_{Q_{b_{1},\lambda_{1},\gamma_{1}},z}-\tilde R_{Q_{b_{2},\lambda_{2},\gamma_{2}},z}).\nonumber 
\end{align}
Applying the $\flat$ operation, we get 
\begin{align}
 & i\partial_{s}\epsilon-\mathcal{L}_{Q_{b_{1}}}\epsilon+ib_{1}\Lambda\epsilon\label{eq:uniq-e-eq}\\
 & =i\Big[\log\Big(\frac{\lambda_{1}}{\lambda_{2}}\Big)\Big]_{s}[\Lambda Q]_{b_{1}}+[\gamma_{1}-\gamma_{2}]_{s}Q_{b_{1}}+i\Big(\frac{(\lambda_{1})_{s}}{\lambda_{1}}+b_{1}\Big)\Lambda\epsilon+(\gamma_{1})_{s}\epsilon+\Gamma,\nonumber 
\end{align}
where 
\begin{align}
\Gamma & \coloneqq i\Big(\frac{(\lambda_{2})_{s}}{\lambda_{2}}+\frac{\lambda_{1}^{2}}{\lambda_{2}^{2}}b_{2}\Big)([\Lambda Q]_{b_{1}}-[\Lambda Q]_{b_{2},\frac{\lambda_{2}}{\lambda_{1}},\gamma_{2}-\gamma_{1}})\label{eq:uniq-def-Gamma}\\
 & \quad+((\gamma_{2})_{s}+(\lambda_{1})^{2}\theta_{\mathrm{cor}})(Q_{b_{1}}-Q_{b_{2},\frac{\lambda_{2}}{\lambda_{1}},\gamma_{2}-\gamma_{1}})+(\mathcal{L}_{w_{1}^{\flat}}-\mathcal{L}_{Q_{b_{1}}})\epsilon+(\mathcal{L}_{w_{1}^{\flat}}-\mathcal{L}_{w_{2}^{\flat}})\epsilon_{2}\nonumber \\
 & \quad+(R_{u_{1}^{\flat}-w_{1}^{\flat}}-R_{u_{2}^{\flat}-w_{2}^{\flat}})+(\tilde R_{Q_{b_{1}},z^{\flat}}-\tilde R_{Q_{b_{2},\frac{\lambda_{2}}{\lambda_{1}},\gamma_{2}-\gamma_{1}},z^{\flat}}).\nonumber 
\end{align}

\begin{lem}[Various estimates]
\label{lem:uniq-various-est}We have the following.
\begin{enumerate}
\item ($H_{m}^{1}$-estimate of the difference) We have
\begin{equation}
\|\psi_{b_{1}}-\psi_{b_{2},\frac{\lambda_{2}}{\lambda_{1}},\gamma_{2}-\gamma_{1}}\|_{H_{m}^{1}}\lesssim(\lambda_{1})^{-1}\mathcal{T}_{\dot{H}_{m}^{1}}^{(\frac{5}{4})}[\epsilon],\qquad\forall\psi\in\{\Lambda Q,iQ\}.\label{eq:uniq-est-diff}
\end{equation}
\item ($H_{m}^{1}$-estimate of $\Gamma$) We have 
\begin{align}
\|\Gamma\|_{H_{m}^{1}} & \lesssim(\alpha^{\ast}+c)\lambda_{1}\mathcal{T}_{\dot{H}_{m}^{1}}^{(\frac{5}{4})}[\epsilon].\label{eq:uniq-est-Gamma}
\end{align}
\item (Modulation estimate) We have 
\begin{equation}
\bigg|\Big[\log\Big(\frac{\lambda_{1}}{\lambda_{2}}\Big)\Big]_{s}\bigg|+|[\gamma_{1}-\gamma_{2}]_{s}|\lesssim\mathcal{T}_{\dot{H}_{m}^{1}}^{(\frac{5}{4})}[\epsilon].\label{eq:uniq-mod-est}
\end{equation}
\end{enumerate}
\end{lem}

\begin{proof}
We first claim 
\begin{equation}
\|\Gamma\|_{H_{m}^{1}}\lesssim(\alpha^{\ast}\lambda_{1}+c(\lambda_{1})^{2})\|\epsilon\|_{\dot{H}_{m}^{1}}+(\alpha^{\ast}+c)(\lambda_{1})^{2}\Big(\Big|\log\Big(\frac{\lambda_{1}}{\lambda_{2}}\Big)\Big|+|\gamma_{1}-\gamma_{2}|\Big).\label{eq:prelim-est-Gamma}
\end{equation}
To see this, we show that each term of \eqref{eq:uniq-def-Gamma}
can be estimated by the RHS of \eqref{eq:prelim-est-Gamma}. The first
two terms of \eqref{eq:uniq-def-Gamma} can be treated by \eqref{eq:uniq-a-pri-mod-est}
and \eqref{eq:prelim-est-diff}. For the third term of \eqref{eq:uniq-def-Gamma},
since 
\[
w_{1}^{\flat}=Q_{b_{1}}+z^{\flat},
\]
we can write 
\begin{align*}
 & (\mathcal{L}_{w_{1}^{\flat}}-\mathcal{L}_{Q_{b_{1}}})\epsilon\\
 & =\sum_{\substack{\psi_{1},\psi_{2},\psi_{3}\in\{Q_{b_{1}},z^{\flat},\epsilon\}.\\
\#\{j:\psi_{j}=\epsilon\}=1,\\
\#\{j:\psi_{j}=z^{\flat}\}\geq1,
}
}[\mathcal{N}_{3,0}+\mathcal{N}_{3,1}+\mathcal{N}_{3,2}]+\sum_{\substack{\psi_{1},\dots,\psi_{5}\in\{Q_{b_{1}},z^{\flat},\epsilon\},\\
\#\{j:\psi_{j}=\epsilon\}=1,\\
\#\{j:\psi_{j}=z^{\flat}\}\geq1,
}
}[\mathcal{N}_{5,1}+\mathcal{N}_{5,2}].
\end{align*}
We then apply \eqref{eq:H1-est-N} by distributing two $\dot{H}_{m}^{1}$
norms to $z^{\flat}$ and $\epsilon$ to get 
\[
\|(\mathcal{L}_{Q_{b}}-\mathcal{L}_{w_{1}^{\flat}})\epsilon\|_{H_{m}^{1}}\lesssim\alpha^{\ast}\lambda_{1}\|\epsilon\|_{\dot{H}_{m}^{1}}.
\]
For the fourth term of \eqref{eq:uniq-def-Gamma}, we write 
\[
w_{1}^{\flat}=(Q_{b_{1}}-Q_{b_{2},\frac{\lambda_{2}}{\lambda_{1}},\gamma_{2}-\gamma_{1}})+w_{2}^{\flat}
\]
and proceed as above by distributing two $\dot{H}_{m}^{1}$ norms
to $Q_{b_{1}}-Q_{b_{2},\frac{\lambda_{2}}{\lambda_{1}},\gamma_{2}-\gamma_{1}}$
and $\epsilon$. By \eqref{eq:H1-est-N}, \eqref{eq:prelim-est-diff},
and \eqref{eq:uniq-a-pri-est}, we get 
\begin{align*}
\|(\mathcal{L}_{w_{1}^{\flat}}-\mathcal{L}_{w_{2}^{\flat}})\epsilon_{2}\|_{H_{m}^{1}} & \lesssim\|Q_{b_{1}}-Q_{b_{2},\frac{\lambda_{2}}{\lambda_{1}},\gamma_{2}-\gamma_{1}}\|_{\dot{H}_{m}^{1}}\|\epsilon_{2}\|_{\dot{H}_{m}^{1}}\\
 & \lesssim c(\lambda_{1})^{2}\Big(\Big|\log\Big(\frac{\lambda_{1}}{\lambda_{2}}\Big)\Big|+|\gamma_{1}-\gamma_{2}|\Big).
\end{align*}
Next, we treat the fifth term of \eqref{eq:uniq-def-Gamma}. We write
\[
w_{1}^{\flat}=(Q_{b_{1}}-Q_{b_{2},\frac{\lambda_{2}}{\lambda_{1}},\gamma_{2}-\gamma_{1}})+w_{2}^{\flat}\quad\text{and}\quad\epsilon_{1}=\epsilon+\epsilon_{2}.
\]
We distribute two $\dot{H}_{m}^{1}$ norms to either the pair $\epsilon_{1}$
and $Q_{b_{1}}-Q_{b_{2},\frac{\lambda_{2}}{\lambda_{1}},\gamma_{2}-\gamma_{1}}$,
or $\epsilon$ and $\epsilon_{2}$. By \eqref{eq:H1-est-N}, \eqref{eq:prelim-est-diff},
and \eqref{eq:uniq-a-pri-est}, we have 
\begin{align*}
 & \|R_{u_{1}^{\flat}-w_{1}^{\flat}}-R_{u_{2}^{\flat}-w_{2}^{\flat}}\|_{H_{m}^{1}}\\
 & \lesssim(\|Q_{b_{1}}-Q_{b_{2},\frac{\lambda_{2}}{\lambda_{1}},\gamma_{2}-\gamma_{1}}\|_{\dot{H}_{m}^{1}}+\|\epsilon\|_{\dot{H}_{m}^{1}})(\|\epsilon_{1}\|_{\dot{H}_{m}^{1}}+\|\epsilon_{2}\|_{\dot{H}_{m}^{1}})\\
 & \lesssim c(\lambda_{1})^{2}\Big(\Big|\log\Big(\frac{\lambda_{1}}{\lambda_{2}}\Big)\Big|+|\gamma_{1}-\gamma_{2}|+\|\epsilon\|_{\dot{H}_{m}^{1}}\Big).
\end{align*}
Finally, we treat the last term of \eqref{eq:uniq-def-Gamma}. We
write 
\[
\tilde R_{Q_{b_{1}},z^{\flat}}-\tilde R_{Q_{b_{2},\frac{\lambda_{2}}{\lambda_{1}},\gamma_{2}-\gamma_{1}},z^{\flat}}=\int_{0}^{1}\partial_{\tau}(\tilde R_{f(\tau),z^{\flat}})d\tau,
\]
where $f(\tau)$ is the path connecting $Q_{b_{1}}$ and $Q_{b_{2},\frac{\lambda_{2}}{\lambda_{1}},\gamma_{2}-\gamma_{1}}$
as in \eqref{eq:def-f(t)}. Observe that 
\begin{align*}
\sup_{\tau\in[0,1]}|\partial_{\tau}f(\tau)| & \lesssim\Big(\Big|\log\Big(\frac{\lambda_{1}}{\lambda_{2}}\Big)\Big|+|\gamma_{1}-\gamma_{2}|\Big)F_{m,m+2},\\
\sup_{\tau\in[0,1]}|\partial_{r}\partial_{\tau}f(\tau)|+|r^{-1}\partial_{\tau}f(\tau)| & \lesssim\Big(\Big|\log\Big(\frac{\lambda_{1}}{\lambda_{2}}\Big)\Big|+|\gamma_{1}-\gamma_{2}|\Big)(F_{m-1,m+3}+\lambda_{1}F_{m+1,m+1}).
\end{align*}
Recalling the proof of Lemma \ref{lem:estimate-RQ,z} and applying
Minkowski's inequality, we get 
\[
\|\tilde R_{Q_{b_{1}},z^{\flat}}-\tilde R_{Q_{b_{2},\frac{\lambda_{2}}{\lambda_{1}},\gamma_{2}-\gamma_{1}},z^{\flat}}\|_{H_{m}^{1}}\lesssim\Big(\Big|\log\Big(\frac{\lambda_{1}}{\lambda_{2}}\Big)\Big|+|\gamma_{1}-\gamma_{2}|\Big)\cdot\alpha^{\ast}\lambda_{1}^{m+3}|\log\lambda_{1}|.
\]
This proves the claim.

We now claim the following preliminary version of the modulation estimates:
\begin{equation}
\bigg|\Big[\log\Big(\frac{\lambda_{1}}{\lambda_{2}}\Big)\Big]_{s}\bigg|+|[\gamma_{1}-\gamma_{2}]_{s}|\lesssim\|\epsilon\|_{\dot{H}_{m}^{1}}+\|\Gamma\|_{L^{2}}.\label{eq:prelim-mod-est}
\end{equation}
Indeed, we proceed as in the proof of Lemma \ref{lem:comput-der-inn}
to compute 
\begin{align*}
0=\partial_{s}(\epsilon,\psi_{b_{1}})_{r} & =(\epsilon,[\mathcal{L}_{Q}i\psi]_{b_{1}})_{r}+\Big[\log\Big(\frac{\lambda_{1}}{\lambda_{2}}\Big)\Big]_{s}(\Lambda Q,\psi)_{r}+[\gamma_{1}-\gamma_{2}]_{s}(Q,i\psi)_{r}\\
 & \quad-\Big(\frac{(\lambda_{1})_{s}}{\lambda_{1}}+b_{1}\Big)(\epsilon,[\Lambda\psi]_{b_{1}})_{r}+(\gamma_{1})_{s}(\epsilon,i\psi_{b_{1}})_{r}+(\Gamma,i\psi_{b_{1}})_{r}
\end{align*}
for all $\psi\in\{\mathcal{Z}_{\re},i\mathcal{Z}_{\im}\}$. Using
$|(\epsilon,[\mathcal{L}_{Q}i\psi]_{b_{1}})_{r}|\lesssim\|\epsilon\|_{\dot{H}_{m}^{1}}$
and \eqref{eq:uniq-a-pri-mod-est}, we conclude \eqref{eq:prelim-mod-est}.

With \eqref{eq:prelim-est-diff}, \eqref{eq:prelim-est-Gamma}, and
\eqref{eq:prelim-mod-est} in hand, we are now in position to conclude
the proof. We use \eqref{eq:prelim-mod-est}, \eqref{eq:prelim-est-Gamma},
and the fundamental theorem of calculus to get 
\begin{align*}
\Big|\log\Big(\frac{\lambda_{1}}{\lambda_{2}}\Big)\Big|+|\gamma_{1}-\gamma_{2}| & \lesssim\int_{t}^{0}\frac{1}{\lambda_{1}^{2}}(\|\epsilon\|_{\dot{H}_{m}^{1}}+\|\Gamma\|_{L^{2}})dt'\\
 & \lesssim\frac{1}{\lambda_{1}}\Big(\mathcal{T}_{\dot{H}_{m}^{1}}^{(\frac{5}{4})}[\epsilon]+c(\lambda_{1})^{2}\mathcal{T}^{(\frac{1}{4})}\Big[\Big|\log\Big(\frac{\lambda_{1}}{\lambda_{2}}\Big)\Big|+|\gamma_{1}-\gamma_{2}|\Big]\Big).
\end{align*}
Taking $\mathcal{T}^{(\frac{1}{4})}$ to the above and using \eqref{eq:env-prop2}
and \eqref{eq:env-prop3}, we get 
\[
\mathcal{T}^{(\frac{1}{4})}\Big[\Big|\log\Big(\frac{\lambda_{1}}{\lambda_{2}}\Big)\Big|+|\gamma_{1}-\gamma_{2}|\Big]\lesssim\lambda_{1}^{-1}\mathcal{T}_{\dot{H}_{m}^{1}}^{(\frac{5}{4})}[\epsilon].
\]
Substituting this into \eqref{eq:prelim-est-diff}, \eqref{eq:prelim-est-Gamma},
and \eqref{eq:prelim-mod-est} again, we obtain \eqref{eq:uniq-est-diff},
\eqref{eq:uniq-est-Gamma}, and \eqref{eq:uniq-mod-est}, respectively.
This completes the proof.
\end{proof}
We also have $L^{2}$ -estimate of $\epsilon$.
\begin{lem}[$L^{2}$-estimate of $\epsilon$]
We have
\begin{equation}
\|\epsilon\|_{L^{2}}\lesssim(\lambda_{1})^{-\frac{3}{4}}\mathcal{T}_{\dot{H}_{m}^{1}}^{(\frac{5}{4})}[\epsilon].\label{eq:uniq-L2-est}
\end{equation}
\end{lem}

\begin{proof}
We recall the equation of $\epsilon^{\sharp}$ in the following form
\begin{align*}
i\partial_{t}\epsilon^{\sharp}+\Delta_{m}\epsilon^{\sharp} & =i\Big[\log\Big(\frac{\lambda_{1}}{\lambda_{2}}\Big)\Big]_{t}[\Lambda Q]_{b_{1},\lambda_{1},\gamma_{1}}+[\gamma_{1}-\gamma_{2}]_{t}Q_{b_{1},\lambda_{1},\gamma_{1}}\\
 & \quad+(\mathcal{L}_{Q_{b_{1},\lambda_{1},\gamma_{1}}}+\Delta_{m})\epsilon^{\sharp}+(\lambda_{1})^{-2}\Gamma^{\sharp}.
\end{align*}
We integrate the flow backward in time. By the Strichartz estimates
(Lemma \ref{lem:Strichartz}),
\[
\|\epsilon^{\sharp}\|_{L_{[t,0)}^{\infty}L_{x}^{2}}\lesssim\eqref{eq:uniq-L2-est-temp1}+\eqref{eq:uniq-L2-est-temp2}+\eqref{eq:uniq-L2-est-temp5},
\]
where 
\begin{align}
 & \Big\| i\Big[\log\Big(\frac{\lambda_{1}}{\lambda_{2}}\Big)\Big]_{t}[\Lambda Q]_{b_{1},\lambda_{1},\gamma_{1}}+[\gamma_{1}-\gamma_{2}]_{t}Q_{b_{1},\lambda_{1},\gamma_{1}}\Big\|_{L_{[t,0),x}^{\frac{4}{3}}},\label{eq:uniq-L2-est-temp1}\\
 & \|(\mathcal{L}_{Q_{b_{1},\lambda_{1},\gamma_{1}}}+\Delta_{m})\epsilon^{\sharp}\|_{L_{[t,0),x}^{\frac{4}{3}}},\label{eq:uniq-L2-est-temp2}\\
 & \|(\lambda_{1})^{-2}\Gamma^{\sharp}\|_{L_{[t,0)}^{1}L_{x}^{2}}.\label{eq:uniq-L2-est-temp5}
\end{align}

For \eqref{eq:uniq-L2-est-temp1}, we recall 
\[
\|[\Lambda Q]_{b_{1},\lambda_{1},\gamma_{1}}\|_{L_{x}^{\frac{4}{3}}}+\|Q_{b_{1},\lambda_{1},\gamma_{1}}\|_{L_{x}^{\frac{4}{3}}}\lesssim(\lambda_{1})^{\frac{1}{2}}.
\]
Thus by \eqref{eq:uniq-mod-est} and \eqref{eq:env-prop4}, we have
\[
\eqref{eq:uniq-L2-est-temp1}\lesssim(\lambda_{1})^{-\frac{3}{4}}\mathcal{T}_{\dot{H}_{m}^{1}}^{(\frac{5}{4})}[\epsilon].
\]
The term \eqref{eq:uniq-L2-est-temp2} is treated as in the proof
of Lemma \ref{eq:L2-est}: 
\[
\eqref{eq:uniq-L2-est-temp2}\lesssim(\lambda_{1})^{-\frac{3}{4}}\mathcal{T}_{\dot{H}_{m}^{1}}^{(\frac{5}{4})}[\epsilon].
\]
The remaining term \eqref{eq:uniq-L2-est-temp5} is treated using
\eqref{eq:uniq-est-Gamma} and \eqref{eq:env-prop4} as 
\[
\eqref{eq:uniq-L2-est-temp5}\lesssim(\alpha^{\ast}+c)\mathcal{T}_{\dot{H}_{m}^{1}}^{(\frac{5}{4})}[\epsilon].
\]
This completes the proof.
\end{proof}

\subsection{Lyapunov/virial functional}

So far, we have estimated modulation parameters and various errors
appearing in the $\epsilon$-equation \eqref{eq:uniq-e-eq} by the
$\dot{H}_{m}^{1}$-norm of $\epsilon$. From now on, we control $\epsilon$
by introducing the Lyapunov/virial functional, as similarly as in
Section \ref{subsec:virial-lyapunov}.

For $A>1$ to be chosen large later, we define 
\[
\mathcal{I}_{A}\coloneqq\lambda^{-2}(E_{Q_{b_{1}}}^{\mathrm{(qd)}}[\epsilon]+b_{1}\Phi_{A}[\epsilon])
\]
and its averaged version 
\[
\mathcal{I}\coloneqq\frac{2}{\log A}\int_{A^{1/2}}^{A}\mathcal{I}_{A'}\frac{dA'}{A'}.
\]
We will collect errors satisfying 
\begin{equation}
|\err|\lesssim\lambda_{1}(\alpha^{\ast}+c+o_{A\to\infty}(1)+A\lambda^{\frac{1}{4}})(\mathcal{T}_{\dot{H}_{m}^{1}}^{(\frac{5}{4})}[\epsilon])^{2}.\label{eq:uniq-def-err}
\end{equation}
The main proposition in this subsection is as follows.
\begin{prop}[Lyapunov/virial estimate]
\label{prop:uniq-Lyapunov/virial}We have 
\begin{gather}
\lambda^{2}\mathcal{I}+O(\lambda|\log\lambda|^{\frac{1}{2}}\|\epsilon\|_{\dot{H}_{m}^{1}}^{2}+A\lambda\|\epsilon\|_{L^{2}}\|\epsilon\|_{\dot{H}_{m}^{1}})\sim\|\epsilon\|_{\dot{H}_{m}^{1}}^{2},\label{eq:uniq-Lyapunov-coer}\\
\lambda^{2}\partial_{s}\mathcal{I}\geq\err.\label{eq:uniq-Lyapunov-deriv}
\end{gather}
\end{prop}

As similarly as in Section \ref{subsec:virial-lyapunov}, we start
by computing $\partial_{s}E_{Q_{b_{1}}}^{\mathrm{(qd)}}[\epsilon]$.
\begin{lem}[Quadratic parts from energy]
\label{lem:uniq-quad-en}We have 
\begin{equation}
E_{Q_{b_{1}}}^{\mathrm{(qd)}}[\epsilon]+O(\lambda_{1}|\log\lambda_{1}|^{\frac{1}{2}}\|\epsilon\|_{\dot{H}_{m}^{1}}^{2})\sim\|\epsilon\|_{\dot{H}_{m}^{1}}^{2}\label{eq:uniq-energy-coer}
\end{equation}
and 
\begin{align}
\partial_{s}E_{Q_{b}}^{\mathrm{(qd)}}[\epsilon] & =(\mathcal{L}_{Q_{b_{1}}}\epsilon,\Big[\log\Big(\frac{\lambda_{1}}{\lambda_{2}}\Big)\Big]_{s}[\Lambda Q]_{b_{1}}-i[\gamma_{1}-\gamma_{2}]_{s}Q_{b_{1}})_{r}\label{eq:uniq-energy-deriv}\\
 & \quad+(\mathcal{L}_{Q_{b_{1}}}\epsilon+R_{(Q_{b_{1}}+\epsilon)-Q_{b_{1}}},\frac{(\lambda_{1})_{s}}{\lambda_{1}}\Lambda\epsilon)_{r}+\err.\nonumber 
\end{align}
\end{lem}

\begin{proof}
The coercivity \eqref{eq:uniq-energy-coer} follows from the proof
of \eqref{eq:quad-energy-coer} using $\|Q_{b_{1}}-Q\|_{L^{2}}\lesssim\lambda_{1}|\log\lambda_{1}|^{\frac{1}{2}}$.

We turn to \eqref{eq:uniq-energy-deriv}. We compute 
\begin{equation}
\partial_{s}E_{Q_{b}}^{\mathrm{(qd)}}[\epsilon]=(R_{(Q_{b_{1}}+\epsilon)-Q_{b_{1}}},\partial_{s}Q_{b_{1}})_{r}+(\mathcal{L}_{Q_{b_{1}}}\epsilon+R_{(Q_{b_{1}}+\epsilon)-Q_{b_{1}}},\partial_{s}\epsilon)_{r}.\label{eq:uniq-energy-deriv-temp1}
\end{equation}
We note that $R_{(Q_{b_{1}}+\epsilon)-Q_{b_{1}}}$ satisfies the analogous
estimates of \eqref{eq:Ru-w1} and \eqref{eq:Ru-w2} by the same proof.
Thus 
\[
|(R_{(Q_{b_{1}}+\epsilon)-Q_{b_{1}}},\partial_{s}Q_{b_{1}})_{r}|\lesssim|(b_{1})_{s}|(\|\epsilon\|_{L^{2}}+\|\epsilon\|_{\dot{H}_{m}^{1}})\|\epsilon\|_{\dot{H}_{m}^{1}}.
\]
Applying $|(b_{1})_{s}|\lesssim(\lambda_{1})^{2}$ and \eqref{eq:uniq-L2-est},
the first term of the RHS of \eqref{eq:uniq-energy-deriv-temp1} is
absorbed into $\err$.

We now treat the second term of the RHS of \eqref{eq:uniq-energy-deriv-temp1}.
We compute 
\[
(\mathcal{L}_{Q_{b_{1}}}\epsilon+R_{(Q_{b_{1}}+\epsilon)-Q_{b_{1}}},\partial_{s}\epsilon)_{r}=\eqref{eq:uniq-energy-deriv-temp2}+\eqref{eq:uniq-energy-deriv-temp3}+\eqref{eq:uniq-energy-deriv-temp4}+\eqref{eq:uniq-energy-deriv-temp5},
\]
where
\begin{align}
 & (\mathcal{L}_{Q_{b_{1}}}\epsilon+R_{(Q_{b_{1}}+\epsilon)-Q_{b_{1}}},-i\mathcal{L}_{Q_{b_{1}}}\epsilon)_{r}\label{eq:uniq-energy-deriv-temp2}\\
 & (\mathcal{L}_{Q_{b_{1}}}\epsilon+R_{(Q_{b_{1}}+\epsilon)-Q_{b_{1}}},\Big[\log\Big(\frac{\lambda_{1}}{\lambda_{2}}\Big)\Big]_{s}[\Lambda Q]_{b_{1}}-i[\gamma_{1}-\gamma_{2}]_{s}Q_{b_{1}})_{r}\label{eq:uniq-energy-deriv-temp3}\\
 & (\mathcal{L}_{Q_{b_{1}}}\epsilon+R_{(Q_{b_{1}}+\epsilon)-Q_{b_{1}}},\frac{(\lambda_{1})_{s}}{\lambda_{1}}\Lambda\epsilon-i(\gamma_{1})_{s}\epsilon)_{r}\label{eq:uniq-energy-deriv-temp4}\\
 & (i\mathcal{L}_{Q_{b_{1}}}\epsilon+R_{(Q_{b_{1}}+\epsilon)-Q_{b_{1}}},-i\Gamma)_{r}.\label{eq:uniq-energy-deriv-temp5}
\end{align}

Let us treat \eqref{eq:uniq-energy-deriv-temp2}-\eqref{eq:uniq-energy-deriv-temp5}
term by term. We arrange the term \eqref{eq:uniq-energy-deriv-temp2}
as 
\[
\eqref{eq:uniq-energy-deriv-temp2}=(\mathcal{L}_{Q_{b_{1}}}\epsilon+R_{(Q_{b_{1}}+\epsilon)-Q_{b_{1}}},iR_{(Q_{b_{1}}+\epsilon)-Q_{b_{1}}})_{r}
\]
and estimate it as 
\[
|\eqref{eq:uniq-energy-deriv-temp2}|\lesssim\|\epsilon\|_{\dot{H}_{m}^{1}}\|R_{(Q_{b_{1}}+\epsilon)-Q_{b_{1}}}\|_{H_{m}^{1}}\lesssim\|\epsilon\|_{\dot{H}_{m}^{1}}^{3}\lesssim c(\lambda_{1})^{2}\|\epsilon\|_{\dot{H}_{m}^{1}}^{2}.
\]
For the term \eqref{eq:uniq-energy-deriv-temp3}, we can discard the
terms having $R_{(Q_{b_{1}}+\epsilon)-Q_{b_{1}}}$: 
\begin{align*}
|(R_{(Q_{b_{1}}+\epsilon)-Q_{b_{1}}},\Big[\log\Big(\frac{\lambda_{1}}{\lambda_{2}}\Big)\Big]_{s}[\Lambda Q]_{b_{1}} & -i[\gamma_{1}-\gamma_{2}]_{s}Q_{b_{1}})_{r}|\\
 & \lesssim\|\epsilon\|_{\dot{H}_{m}^{1}}^{2}\cdot\mathcal{T}_{\dot{H}_{m}^{1}}^{(\frac{5}{4})}[\epsilon]\lesssim c(\lambda_{1})^{2}(\mathcal{T}_{\dot{H}_{m}^{1}}^{(\frac{5}{4})}[\epsilon])^{2}.
\end{align*}
For \eqref{eq:uniq-energy-deriv-temp4}, we keep the scaling term
that is an unbounded looking quantity. The phase term is absorbed
into $\err$ as 
\[
|(\mathcal{L}_{Q_{b_{1}}}\epsilon+R_{(Q_{b_{1}}+\epsilon)-Q_{b_{1}}},-i(\gamma_{1})_{s}\epsilon)_{r}|\lesssim|(\gamma_{1})_{s}|\|\epsilon\|_{\dot{H}_{m}^{1}}^{2}\lesssim(\lambda_{1})^{2}\|\epsilon\|_{\dot{H}_{m}^{1}}^{2}.
\]
Finally, we treat \eqref{eq:uniq-energy-deriv-temp5} as an error
term using \eqref{eq:uniq-est-Gamma}: 
\[
|\eqref{eq:uniq-energy-deriv-temp5}|\lesssim\|\epsilon\|_{\dot{H}_{m}^{1}}\|\Gamma\|_{H_{m}^{1}}\lesssim(\alpha^{\ast}+c)\lambda_{1}\|\epsilon\|_{\dot{H}_{m}^{1}}\mathcal{T}_{\dot{H}_{m}^{1}}^{(\frac{5}{4})}[\epsilon].
\]
This completes the proof.
\end{proof}
\begin{lem}[{Estimate of $b_{1}\Phi_{A}[\epsilon]$}]
\label{lem:uniq-quad-phi}We have 
\begin{align}
|b_{1}\Phi_{A}[\epsilon]| & \lesssim A\lambda_{1}\|\epsilon\|_{L^{2}}\|\epsilon\|_{\dot{H}_{m}^{1}},\label{eq:uniq-Phi-coer}\\
|\partial_{s}(b_{1}\Phi_{A}[\epsilon])| & =-\frac{(\lambda_{1})_{s}}{\lambda_{1}}(\mathcal{L}_{Q_{b_{1}}}\epsilon,\Lambda_{A}\epsilon)_{r}\label{eq:uniq-Phi-deriv}\\
 & \quad+(-ib_{1}\Lambda_{A}\epsilon,\Big[\log\Big(\frac{\lambda_{1}}{\lambda_{2}}\Big)\Big]_{s}[\Lambda Q]_{b_{1}}-i[\gamma_{1}-\gamma_{2}]_{s}Q_{b_{1}})_{r}+\err.\nonumber 
\end{align}
\end{lem}

\begin{proof}
The estimate \eqref{eq:uniq-Phi-coer} is an easy consequence of the
Cauchy-Schwarz. Henceforth, we focus on \eqref{eq:uniq-Phi-deriv}.
We compute 
\[
\partial_{s}(b_{1}\Phi_{A}[\epsilon])=(b_{1})_{s}\Phi_{A}[\epsilon]+b_{1}(\Lambda_{A}\epsilon,i\partial_{s}\epsilon)_{r}=(b_{1}\Lambda_{A}\epsilon,i\partial_{s}\epsilon)_{r}+\err,
\]
where we used $|(b_{1})_{s}|\lesssim(\lambda_{1})^{2}$. We then substitue
the equation \eqref{eq:uniq-e-eq} of $\epsilon$ into the above to
obtain 
\[
\partial_{s}(b_{1}\Phi_{A}[\epsilon])=\eqref{eq:uniq-Phi-deriv-temp1}+\eqref{eq:uniq-Phi-deriv-temp2}+\eqref{eq:uniq-Phi-deriv-temp3}+\eqref{eq:uniq-Phi-deriv-temp4}+\err,
\]
where 
\begin{align}
 & (b_{1}\Lambda_{A}\epsilon,\mathcal{L}_{Q_{b_{1}}}\epsilon)_{r},\label{eq:uniq-Phi-deriv-temp1}\\
 & (b_{1}\Lambda_{A}\epsilon,i\Big[\log\Big(\frac{\lambda_{1}}{\lambda_{2}}\Big)\Big]_{s}[\Lambda Q]_{b_{1}}+[\gamma_{1}-\gamma_{2}]_{s}Q_{b_{1}})_{r},\label{eq:uniq-Phi-deriv-temp2}\\
 & (b_{1}\Lambda_{A}\epsilon,i\frac{(\lambda_{1})_{s}}{\lambda_{1}}\Lambda\epsilon+(\gamma_{1})_{s}\epsilon)_{r},\label{eq:uniq-Phi-deriv-temp3}\\
 & (b_{1}\Lambda_{A}\epsilon,\Gamma)_{r}.\label{eq:uniq-Phi-deriv-temp4}
\end{align}

We treat \eqref{eq:uniq-Phi-deriv-temp1}-\eqref{eq:uniq-Phi-deriv-temp4}
term by term. Firstly, we claim that \eqref{eq:uniq-Phi-deriv-temp1}
and the first term of the RHS of \eqref{eq:uniq-Phi-deriv} agree
up to $\err$. Indeed, the difference is given by 
\[
\Big(\frac{(\lambda_{1})_{s}}{\lambda_{1}}+b_{1}\Big)(\Lambda_{A}\epsilon,\mathcal{L}_{Q_{b_{1}}}\epsilon)_{r}
\]
and is absorbed into $\err$ by mimicking the treatment of \eqref{eq:5.41-temp1}.
Next, we keep the term \eqref{eq:uniq-Phi-deriv-temp2}. We then observe
that \eqref{eq:uniq-Phi-deriv-temp3} is absorbed into $\err$ by
mimicking the treatment of \eqref{eq:5.41-temp3}. Finally, we view
\eqref{eq:uniq-Phi-deriv-temp4} as an error by 
\[
|\eqref{eq:uniq-Phi-deriv-temp4}|\lesssim A(\lambda_{1})^{2}\|\epsilon\|_{\dot{H}_{m}^{1}}\|\Gamma\|_{L^{2}}\lesssim A(\lambda_{1})^{2}(\mathcal{T}_{\dot{H}_{m}^{1}}^{(\frac{5}{4})}[\epsilon])^{2}.
\]
This completes the proof.
\end{proof}
\begin{proof}[Proof of Proposition \ref{prop:uniq-Lyapunov/virial}]
Combining Lemmas \ref{lem:uniq-quad-en} and \ref{lem:uniq-quad-phi},
the rest of the arguments proceed similarly as in the proof of Lemma
\ref{lem:unav-Lyapunov/virial} and Proposition \ref{prop:lyapunov}.
We omit the details.
\end{proof}

\subsection{Proof of conditional uniqueness}

We are now in position to conclude the proof of the conditional uniqueness.
\begin{proof}[Proof of Theorem \ref{thm:cond-uniq} (Finish)]
By \eqref{eq:uniq-Lyapunov-coer}, \eqref{eq:uniq-Lyapunov-deriv},
and the fundamental theorem of calculus, we have 
\[
\|\epsilon\|_{\dot{H}_{m}^{1}}^{2}\lesssim\lambda_{1}|\log\lambda_{1}|^{\frac{1}{2}}\|\epsilon\|_{\dot{H}_{m}^{1}}^{2}+A\lambda_{1}\|\epsilon\|_{L^{2}}\|\epsilon\|_{\dot{H}_{m}^{1}}+\lambda^{2}\int_{t}^{0}\frac{1}{\lambda^{4}}|\err|dt'.
\]
By \eqref{eq:uniq-L2-est}, \eqref{eq:uniq-def-err}, and \eqref{eq:env-prop4},
we get 
\[
\|\epsilon\|_{\dot{H}_{m}^{1}}^{2}\lesssim(\alpha^{\ast}+c+o_{A\to\infty}(1)+A\lambda_{1}^{\frac{1}{4}})(\mathcal{T}_{\dot{H}_{m}^{1}}^{(\frac{5}{4})}[\epsilon])^{2}.
\]
Applying the time maximal function and using \eqref{eq:env-prop2},
we get 
\[
\mathcal{T}_{\dot{H}_{m}^{1}}^{(\frac{5}{4})}[\epsilon]\lesssim(\alpha^{\ast}+c+o_{A\to\infty}(1)+A\lambda_{1}^{\frac{1}{4}})^{\frac{1}{2}}\mathcal{T}_{\dot{H}_{m}^{1}}^{(\frac{5}{4})}[\epsilon].
\]
Recall that a priori decay of $\epsilon$ guarantees finiteness of
$\mathcal{T}$. Thus $\mathcal{T}=0$, saying that $\epsilon=0$.
By the modulation estimates, this also says $\lambda_{1}=\lambda_{2}$
and $\gamma_{1}=\gamma_{2}$. From the relation $b_{j}=|t|^{-1}\lambda_{j}^{2}$,
we have $b_{1}=b_{2}$. Therefore, $u_{1}=u_{2}$.
\end{proof}

\appendix

\section{\label{sec:equiv-sob-sp}Equivariant Sobolev spaces}

In this appendix, we record basic properties of equivariant functions
on $\R^{2}$ and associated Sobolev spaces. More precisely, we want
to state and prove Sobolev and Hardy's inequality for equivariant
functions. We assume equivariant function spaces and associated estimates
are fairly standard. However, we could not find a well-organized reference
discussing basic properties of equivariant functions. So we include,
for self-containedness, a discussion on equivariant functions and
required estimates for this work.

\subsection{Smooth equivariant functions}

\subsubsection*{Smooth radial functions}

Let $f:\R^{d}\to\C$ be a radial function; let $g:\R\to\C$ be defined
by $g(x)\coloneqq f(x,0,\dots,0)$. We then see that $f(x)=g(|x|)$
for all $x\in\R^{d}$ and $g$ is even. Conversely, if we are given
an even function $g:\R\to\C$, then we can form a radial function
$f:\R^{d}\to\C$ by $f(x)\coloneqq g(|x|)$. We shall often view $g$
by its restriction on $\R_{+}$ without mentioning. We note that if
$f$ is determined up to almost everywhere equivalence, so is $g$
and vice versa.

If $f$ is a smooth radial function, it is clear that $g$ is smooth.
We then ask the converse: is $f$ smooth provided that $g$ is smooth?
The answer is affirmative, but we need the following preliminary lemma.
\begin{lem}[Removable singularity at zero]
\label{lem:remov-sing}Let $g:\R\to\C$ be a $C^{k+1}$ function
for some $k\geq0$. Then, 
\[
h(x)\coloneqq\begin{cases}
\frac{g(x)-g(0)}{x} & \text{if }x\neq0,\\
g'(0) & \text{if }x=0,
\end{cases}
\]
is a $C^{k}$ function with 
\[
h^{(j)}(0)=\frac{1}{j+1}g^{(j+1)}(0),\qquad\forall j=0,\dots,k.
\]
In particular, if $g$ is smooth, then so is $h$.
\end{lem}

\begin{proof}
Fix $k\geq0$. Subtracting $g$ by its $(k+1)$-th Taylor polynomial
$\sum_{j=0}^{k+1}\frac{g^{(j)}(0)}{j!}x^{j}$, we may assume that
$g$ satisfies $g^{(j)}(0)=0$ for all $0\leq j\leq k+1$. In particular,
we have 
\begin{equation}
g^{(j)}(x)=o_{x\to0}(|x|^{k+1-j}),\qquad\forall0\leq j\leq k+1.\label{eq:A.1}
\end{equation}

For $0\leq\ell\leq k$, let $P_{\ell}$ be the statement that $h$
is of $C^{\ell}$, $h^{(\ell)}(0)=0$, and $h^{(\ell)}(x)=o_{x\to0}(|x|^{k-\ell})$.
We proceed by induction on $\ell$. The statement $P_{0}$ is obvious
from \eqref{eq:A.1}. We now assume $P_{\ell-1}$ for some $1\leq\ell\leq k$
and show $P_{\ell}$. It is clear that $h$ is of $C^{\ell}$ on $\R\setminus\{0\}$
and by inductive hypothesis, of $C^{\ell-1}$ on $\R$. To show the
remaining assertions, the trick is to use the relation 
\[
xh(x)=g(x),\qquad\forall x\in\R.
\]
We differentiate this relation $\ell$-times to get 
\[
xh^{(\ell)}(x)=g^{(\ell)}(x)-\ell h^{(\ell-1)}(x),\qquad\forall x\in\R\setminus\{0\}.
\]
Applying \eqref{eq:A.1} and the inductive hypothesis, the proof of
the inductive step (and hence the lemma) is complete.
\end{proof}
\begin{lem}[Smooth radial functions]
\label{lem:smooth-radial-criterion}Let $g:\R\to\C$ be an even function;
define a radial function $f:\R^{d}\to\C$ via $f(x)\coloneqq g(|x|)$
for each $x\in\R^{d}$. Then, $f$ is smooth if and only if $g$ is
smooth.
\end{lem}

\begin{proof}
The difficult part is ``if direction,'' which we now show. We use
induction on $k$; let $P_{k}$ be the statement that if $g$ is a
smooth even function, then $f$ is of $C^{k}$. The statement $P_{0}$
is obviously true. Assuming $P_{k}$ for some $k\geq0$, we show $P_{k+1}$.
We observe for each $i=1,\dots,d$ that 
\[
\partial_{i}f(x)=x_{i}h(x),\qquad h(x)\coloneqq\begin{cases}
\frac{g'(|x|)}{|x|} & \text{if }x\neq0,\\
g''(0) & \text{if }x=0.
\end{cases}
\]
The radial part of $h$ is smooth by Lemma \ref{lem:remov-sing}.
Applying $P_{k}$ to $h$, we see that $h$ is of $C^{k}$. Thus $f$
is of $C^{k+1}$, completing the inductive step and hence the proof.
\end{proof}

\subsubsection*{Smooth equivariant functions}

Let $m\in\Z\setminus\{0\}$; let $f$ be $m$-equivariant, i.e. $f(x)=g(r)e^{im\theta}$
for each $x\in\R^{2}\setminus\{0\}$ with the expression $x_{1}+ix_{2}=re^{i\theta}$.
Let $g:\R\to\C$ be a restriction of $f$ defined by $g(x)\coloneqq f(x,0)$.
We note that $g$ is odd (resp., even) if $m$ is odd (resp., even),
which we say as \emph{$g$ has right parity}. Conversely, if we are
given $m\in\Z$ and a function $g:\R\to\C$ having right parity, we
can form an $m$-equivariant function $f$ on $\R^{2}\setminus\{0\}$
via $f(x)\coloneqq g(r)e^{im\theta}$. To extend $f$ on $\R^{2}$,
the value of $f$ at the origin is important. Setting $f(0)\neq0$
yields discontinuity of $f$ at the origin, because of the factor
$e^{im\theta}$. Thus smoothness of $f$ forces $f$ to be degenerate
at the origin. In this section, we want to study how much degeneracy
$f$ should have. The following lemma answers this.
\begin{lem}[Smooth equivariant functions]
\label{lem:smooth-equiv-criterion}Let $m\in\Z\setminus\{0\}$; let
$g:\R\to\C$ be a function having right parity. Define $f:\R^{2}\to\C$
by $f(x)\coloneqq g(r)e^{im\theta}$ for $x\in\R^{2}\setminus\{0\}$
and $f(0)\coloneqq g(0)$. Then, $f$ is smooth if and only if $g$
is smooth and 
\[
g^{(k)}(0)=0
\]
 for each $k$ such that $0\leq k\leq|m|-1$ or $k-m$ is odd. Moreover,
if this is the case, we can express 
\[
f(x)=h(x)\cdot\begin{cases}
(x_{1}+ix_{2})^{m} & \text{if }m>0,\\
(x_{1}-ix_{2})^{m} & \text{if }m<0,
\end{cases}\qquad\forall x\in\R^{2}
\]
for some smooth radial function $h$ on $\R^{2}$.
\end{lem}

\begin{proof}
We first show the ``only if direction.'' Since $f$ must be continuous
at the origin, we should set $f(0)=g(0)=0$. Since $g(x)=f(x,0)$
for all $x\in\R$, we see that $g$ is smooth and $g^{(k)}(0)=0$
when $k-m$ is odd since $g^{(k)}(x)$ has wrong parity. Henceforth,
we focus on showing $g^{(k)}(0)=0$ for each $k\leq|m|-1$.

Given $k\geq0$, let $P_{k}$ be the statement that if $f$ is a smooth
$m$-equivariant function with $|m|\geq k+1$, then $g^{(j)}(0)=0$
for all $j=0,\dots,k$. The case of $k=0$ is obvious by the continuity
issue at the origin. We assume $P_{k}$ for some $k\geq0$ and show
$P_{k+1}$. Namely, we assume $|m|\geq k+2$, know $g^{(0)}(0)=\cdots=g^{(k)}(0)=0$
by inductive hypothesis, and aim to show $g^{(k+1)}(0)=0$. Since
$f$ is smooth at the origin, the following limits must exist and
do not depend on $\theta$: 
\[
\lim_{r\to0}(\partial_{1}+i\partial_{2})^{k+1}[g(r)e^{im\theta}]\quad\text{and}\quad\lim_{r\to0}(\partial_{1}-i\partial_{2})^{k+1}[g(r)e^{im\theta}].
\]
Recall the formulae 
\[
\partial_{1}+i\partial_{2}=e^{i\theta}[\partial_{r}+\tfrac{i}{r}\partial_{\theta}];\qquad\partial_{1}-i\partial_{2}=e^{-i\theta}[\partial_{r}-\tfrac{i}{r}\partial_{\theta}].
\]
In case of $m\geq k+2$, we compute using inductive hypothesis $g^{(0)}(0)=\cdots=g^{(k)}(0)=0$
and Lemma \ref{lem:remov-sing}: 
\begin{align*}
\lim_{r\to0}(\partial_{1}-i\partial_{2})^{k+1}[g(r)e^{im\theta}] & =e^{i(m-k-1)\theta}\lim_{r\to0}(\partial_{r}+\tfrac{m-k}{r})\cdots(\partial_{r}+\tfrac{m}{r})g(r).\\
 & =e^{i(m-k-1)\theta}\Big(\prod_{j=0}^{k}\frac{k+m+1-2j}{j+1}\Big)g^{(k+1)}(0).
\end{align*}
As this expression should not depend on $\theta$, we observe that
$g^{(k+1)}(0)=0$. The case of $m\leq-k-2$ can be treated similarly
by computing $(\partial_{1}+i\partial_{2})^{k+1}f(0)$ instead. This
completes the proof of the inductive step and hence the only if direction.

We turn to show the ``if direction.'' By Lemma \ref{lem:remov-sing},
the function $\tilde h:\R\to\C$ defined by 
\[
\tilde h(r)\coloneqq\begin{cases}
r^{-|m|}g(r) & \text{if }r\neq0\\
\frac{1}{|m|!}g^{(|m|)}(0) & \text{if }r=0
\end{cases}
\]
is smooth and even. By Lemma \ref{lem:smooth-radial-criterion}, the
radial function $h:\R^{2}\to\C$ defined by $h(x)\coloneqq\tilde h(|x|)$
is smooth. Therefore, 
\[
f(x)=h(x)\cdot\begin{cases}
(x_{1}+ix_{2})^{m} & \text{if }m>0\\
(x_{1}-ix_{2})^{m} & \text{if }m<0
\end{cases}
\]
is smooth. This completes the proof.
\end{proof}

\subsection{Equivariant Sobolev spaces}

\subsubsection*{Radial Sobolev spaces}

Let $L_{\rad}^{2}(\R^{d})$ be the set of radial $L^{2}(\R^{d})$
functions. Thus any $f\in L_{\rad}^{2}(\R^{d})$ has a unique (up
to a.e.) expression $f(x)=g(r)$ for some measurable function $g:\R_{+}\to\C$,
with the relation $r=|x|$. We moreover have the unitary equivalence
\begin{align*}
\Phi:L_{\rad}^{2}(\R^{d}) & \to L^{2}(\R_{+},c_{d}r^{d-1}dr)\\
f & \mapsto g,
\end{align*}
where $c_{d}$ is the area of the unit sphere of $\R^{d}$. This in
particular shows that $L_{\rad}^{2}(\R^{d})$ is complete and hence
a closed subspace of $L^{2}(\R^{d})$. The Fourier transform of $L_{\rad}^{2}(\R^{d})$
functions are again radial.

For $s\geq0$, the radial Sobolev space $H_{\rad}^{s}(\R^{d})$ is
defined by 
\[
H_{\rad}^{s}(\R^{d})\coloneqq H^{s}(\R^{d})\cap L_{\rad}^{2}(\R^{d}).
\]
Inner products are inherited from $H^{s}(\R^{d})$. The set $\mathcal{S}_{\rad}(\R^{d})$
of radial Schwartz functions is dense in $H_{\rad}^{s}(\R^{d})$.
For $H^{1}(\R^{d})$ functions, we can define 
\[
\partial_{r}f\coloneqq\frac{1}{r}\sum_{j=1}^{d}x_{j}\partial_{j}f.
\]
At least for smooth radial functions, we have 
\[
\Delta f=\partial_{rr}f+\frac{d-1}{r}\partial_{r}f.
\]
We moreover have 
\[
\|\nabla f\|_{L^{2}(\R^{d})}=\|\partial_{r}f\|_{L^{2}(\R^{d})}=\|\partial_{r}f\|_{L^{2}(\R_{+},c_{d}r^{d-1}dr)},\qquad\forall f\in H_{\rad}^{1}(\R^{d}).
\]
In case of $d=1$, any $H^{1}(\R)$ functions are continuous and bounded.
However, when $d\geq2$, a $H_{\rad}^{1}(\R^{d})$ function need not
be bounded or continuous.

To deal with equivariant functions in the next subsection, we need
the space $C_{c,\rad}^{\infty}(\R^{d}\setminus\{0\})$, which is the
set of smooth radial functions having compact suppport in $\R^{d}\setminus\{0\}$.
There is a well-known Hardy's inequality when $d\geq3$:

\[
\|\partial_{r}f\|_{L^{2}}\geq\frac{d-2}{2}\Big\|\frac{f}{r}\Big\|_{L^{2}},\qquad\forall f\in C_{c,\rad}^{\infty}(\R^{2}\setminus\{0\}).
\]
We conclude this subsection by noting the density theorem by $C_{c,\rad}^{\infty}(\R^{d}\setminus\{0\})$
functions:
\begin{lem}[Density of $C_{c,\rad}^{\infty}(\R^{d}\setminus\{0\})$]
\label{lem:density-radial}\ 
\begin{enumerate}
\item If $0\leq s\leq\frac{d}{2}$, then $C_{c,\rad}^{\infty}(\R^{d}\setminus\{0\})$
is dense in $H_{\rad}^{s}(\R^{d})$.
\item If $s>\frac{d}{2}$, then $C_{c,\rad}^{\infty}(\R^{d}\setminus\{0\})$
is not dense in $H_{\rad}^{s}(\R^{d})$.
\end{enumerate}
\end{lem}

\begin{proof}
The second assertion easily follows from the Sobolev embedding $H^{s}(\R^{d})\hookrightarrow L^{\infty}(\R^{d})$.
When $s\leq\frac{d}{2}$, a crucial observation is the \emph{failure}
of $H^{\frac{d}{2}}\hookrightarrow L^{\infty}$; we can choose a radial
Schwartz function $g$ such that $\|g\|_{L^{\infty}}=g(0)=1$ but
$\|g\|_{H^{\frac{d}{2}}}$ is arbitrarily small. Note that $\mathcal{S}_{\rad}(\R^{d})$
is dense in $H_{\rad}^{s}(\R^{d})$, thanks to the Fourier transform.
Therefore, to show that $C_{c,\rad}^{\infty}(\R^{d}\setminus\{0\})$
is dense in $H_{\rad}^{s}(\R^{d})$, it suffices to approximate radial
Schwartz functions $f$ \emph{with} $f(0)=0$ by $C_{c,\rad}^{\infty}(\R^{d}\setminus\{0\})$
functions. Henceforth, we let $\chi$ to be a smooth radial bump function
on $\R^{d}$, $\chi_{R}\coloneqq\chi(\frac{\cdot}{R})$, and show
that $\|\chi_{R}f\|_{H^{\frac{d}{2}}}\to0$ as $R\to0$.

For any $k\in\N$ and $0<R\ll1$, we observe by Leibniz's rule that
\begin{align*}
\|\chi_{R}f\|_{H^{k}} & \lesssim R^{-k+\frac{d}{2}}\|f\|_{L_{|x|\lesssim R}^{\infty}}+\sum_{j=1}^{k}\|\nabla^{k-j}\chi_{R}\|_{L^{2+}}\|\nabla^{j}f\|_{L_{|x|\lesssim R}^{\infty-}}\\
 & \lesssim R^{-k+\frac{d}{2}}\|f\|_{L_{|x|\lesssim R}^{\infty}}+R^{-k+\frac{d}{2}+1-}\sum_{j=1}^{k}\|\nabla^{j}f\|_{L_{|x|\lesssim R}^{\infty-}}.
\end{align*}
When $d$ is even, we choose $k=\frac{d}{2}$ to get 
\[
\|\chi_{R}f\|_{H^{\frac{d}{2}}}\lesssim\|f\|_{L_{|x|\lesssim R}^{\infty}}+R^{1-}\sum_{j=1}^{k}\|\nabla^{j}f\|_{L_{|x|\lesssim R}^{\infty-}}.
\]
From our assumptions $f\in\mathcal{S}_{\rad}(\R^{d})$ and $f(0)=0$,
we take $R\to0$ to conclude. When $d$ is odd, we observe 
\[
\|\chi_{R}f\|_{H^{\frac{d}{2}}}\lesssim\|\chi_{R}f\|_{H^{\frac{d-1}{2}}}^{\frac{1}{2}}\|\chi_{R}f\|_{H^{\frac{d+1}{2}}}^{\frac{1}{2}}\to0
\]
as $R\to0$. Note that $H^{\frac{d+1}{2}}$ norm may diverge with
a factor $R^{-\frac{1}{2}}$ but this is cancelled with the factor
$R^{\frac{1}{2}}$ from $H^{\frac{d-1}{2}}$ norm. This completes
the proof.
\end{proof}

\subsubsection*{Equivariant Sobolev spaces}

We turn to equivariant Sobolev spaces. Define $L_{m}^{2}$ by the
set of $m$-equivariant $L^{2}(\R^{2})$ functions. Inner product
of $L_{m}^{2}$ is inherited from $L^{2}(\R^{2})$. Any $f\in L_{m}^{2}$
has a unique (up to a.e.) expression $f(x)=g(r)e^{im\theta}$ for
some measurable function $g:\R_{+}\to\C$. We have unitary equivalence
\begin{align*}
\Phi:L_{m}^{2}(\R^{2}) & \to L^{2}(\R_{+},2\pi rdr)\\
f & \mapsto g.
\end{align*}

For $s\geq0$, the $m$-equivariant Sobolev space $H_{m}^{s}$ is
defined as 
\[
H_{m}^{s}\coloneqq H^{s}(\R^{2})\cap L_{m}^{2},
\]
equipped with the inner product of $H^{s}(\R^{2})$. We define $\mathcal{S}_{m}$
by the set of $m$-equivariant Schwartz functions. We define $C_{c,m}^{\infty}(\R^{2}\setminus\{0\})$
by the set of smooth $m$-equivariant functions having compact support
in $\R^{2}\setminus\{0\}$. Note that the Laplacian $\Delta$ has
expression 
\[
\Delta=\partial_{rr}+\frac{1}{r}\partial_{r}-\frac{m^{2}}{r^{2}}
\]
for $m$-equivariant functions.

Note that $m$-equivariant functions can be viewed as radial functions
in higher dimensions. Indeed, if we define 
\begin{align*}
\Psi:L_{m}^{2} & \to L_{\rad}^{2}(\R^{2|m|+2})\\
f & \mapsto c_{m}r^{-m}f
\end{align*}
for some universal constant $c_{m}$, then $\Psi$ is unitary and
the following diagram commutes: 
\[
\xymatrix{\mathcal{S}_{m}\subset L_{m}^{2}\ar[r]\sp(0.3)\Psi\ar[d]^{\Delta} & \mathcal{S}_{\rad}(\R^{2|m|+2})\subset L_{\rad}^{2}(\R^{2|m|+2})\ar[d]^{\Delta}\\
\mathcal{S}_{m}\subset L_{m}^{2}\ar[r]\sp(0.3)\Psi & \mathcal{S}_{\rad}(\R^{2|m|+2})\subset L_{\rad}^{2}(\R^{2|m|+2})
}
\]
We note that $\Psi(\mathcal{S}_{m})=\mathcal{S}_{\rad}(\R^{2|m|+2})$,
thanks to Lemma \ref{lem:smooth-equiv-criterion}. This allows us
to transfer the density theorem for radial Sobolev spaces to equivariant
Sobolev spaces.
\begin{lem}[Density of $C_{c,m}^{\infty}(\R^{2}\setminus\{0\})$]
\label{lem:density-equiv}\ 
\begin{enumerate}
\item If $0\leq s\leq|m|+1$, then $C_{c,m}^{\infty}(\R^{2}\setminus\{0\})$
is dense in $H_{m}^{s}$.
\item If $s>|m|+1$, then $C_{c,m}^{\infty}(\R^{2}\setminus\{0\})$ is not
dense in $H_{m}^{s}$.
\end{enumerate}
\end{lem}

\begin{proof}
The idea is to transfer the radial results using the above commutative
diagram. Indeed, for $f,g\in\mathcal{S}_{m}$, we have 
\begin{align*}
\|f-g\|_{H_{m}^{|m|+1}(\R^{2})}^{2} & =(f-g,(1-\Delta)^{|m|+1}(f-g))\\
 & =(\Psi(f-g),(1-\Delta)^{|m|+1}\Psi(f-g))\\
 & =\|\Psi f-\Psi g\|_{H_{\rad}^{|m|+1}(\R^{2|m|+2})}.
\end{align*}

1. It suffices to approximate $m$-equivariant Schwartz functions
by $C_{c,m}^{\infty}(\R^{2}\setminus\{0\})$ in $H_{m}^{|m|+1}$ topology.
Because of the above diagram and Lemma \ref{lem:density-radial},
we know that $\Psi f\in\mathcal{S}_{\rad}(\R^{2|m|+2})$ can be approximated
by a $C_{c,\rad}^{\infty}(\R^{2|m|+2}\setminus\{0\})$ function, say
$h$. We then set $g\coloneqq\Psi^{-1}h$.

2. This also follows from the transference. We omit the proof.
\end{proof}
For $m$-equivariant functions with $m\neq0$, we have a nontrivial
endpoint estimate of Hardy or Sobolev type. In general, it is well-known
that the Hardy inequality fails for $d=2$ and Sobolev inequality
for $d=2$ fails, i.e. $\dot{H}^{1}(\R^{2})\not\hookrightarrow L^{\infty}(\R^{2})$.
\begin{lem}[Hardy-Sobolev inequality]
\label{lem:hardy-equiv}Let $m\in\Z\setminus\{0\}$. For any $f\in H_{m}^{1}$,
we have 
\[
\|r^{-1}f\|_{L^{2}}+\|f\|_{L^{\infty}}\lesssim\|\nabla f\|_{L^{2}}.
\]
\end{lem}

\begin{proof}
By density, we may assume $f\in C_{c,m}^{\infty}(\R^{2}\setminus\{0\})$.
From the formula of $\Delta$, we have 
\[
\|\nabla f\|_{L^{2}}^{2}=(f,-\Delta f)=\|\partial_{r}f\|_{L^{2}}^{2}+m^{2}\|r^{-1}f\|_{L^{2}}^{2}\geq m^{2}\|r^{-1}f\|_{L^{2}}^{2}.
\]
Another nice feature of $m$-equivariant functions is that we have
embedding $\dot{H}_{m}^{1}\hookrightarrow L^{\infty}$. Indeed, 
\[
|f(r)|^{2}\leq\int_{0}^{\infty}|f||\partial_{r}f|dr'\leq\|r^{-1}f\|_{L^{2}}\|\partial_{r}f\|_{L^{2}}\lesssim\|\nabla f\|_{L^{2}}^{2}.
\]
This completes the proof.
\end{proof}
We can generalize Hardy's inequality as follows.
\begin{lem}[Generalized Hardy's inequality]
\label{lem:gen-hardy-equiv}For any $f\in C_{c,m}^{\infty}(\R^{2}\setminus\{0\})$
and $0\leq k\leq|m|$, we have 
\[
\|r^{-k}f\|_{L^{2}}+\|r^{-(k-1)}\partial_{r}f\|_{L^{2}}+\cdots+\|\partial_{r}^{k}f\|_{L^{2}}\sim_{k,m}\|f\|_{\dot{H}_{m}^{k}}.
\]
\end{lem}

\begin{proof}
For simplicity, we only consider the case $m>0$. The case $m<0$
can be treated in a similar way, which we omit the proof. We claim
that the following set of operators 
\[
\{e^{i(k-2\ell)\theta}(\partial_{1}-i\partial_{2})^{k-\ell}(\partial_{1}+i\partial_{2})^{\ell}\}_{0\leq\ell\leq k}
\]
acting on $C_{c,m}^{\infty}(\R^{2}\setminus\{0\})$ is linearly independent.
Indeed, we observe for $0\leq j,\ell\leq k$ that 
\begin{align*}
(\partial_{1}+i\partial_{2})^{\ell}[r^{m+2j}e^{im\theta}] & =0,\qquad\text{if }\ell\geq j+1,\\
e^{i(k-2\ell)\theta}(\partial_{1}-i\partial_{2})^{k-\ell}(\partial_{1}+i\partial_{2})^{\ell}[r^{m+2\ell}e^{im\theta}] & =a_{k,\ell}r^{m+2\ell-k}e^{im\theta},
\end{align*}
where 
\[
a_{k,\ell}\coloneqq2^{k}(\ell!)\prod_{n=0}^{k-\ell-1}(m+\ell-n).
\]
As $0\leq k\leq m$, we have $a_{k,\ell}\neq0$ for each $0\leq\ell\leq k$.
Namely, acting $r^{m+2j}e^{im\theta}$ on our set of operators and
taking $r=1$, we obtain a nonsingular matrix 
\[
\begin{bmatrix}a_{k,0} & \ast & \ast & \cdots & \ast\\
0 & a_{k,1} & \ast & \cdots & \ast\\
0 & 0 & a_{k,2} & \cdots & \ast\\
\vdots & \vdots & \vdots & \ddots & \vdots\\
0 & 0 & 0 & \cdots & a_{k,k}
\end{bmatrix}.
\]
This proves the claim.

The claim immediately implies 
\[
\dim_{\C}\mathrm{span}_{\C}\{e^{i(k-2\ell)\theta}(\partial_{1}-i\partial_{2})^{k-\ell}(\partial_{1}+i\partial_{2})^{\ell}\}_{0\leq\ell\leq k}=k+1.
\]
In view of 
\[
\mathrm{span}_{\C}\{e^{i(k-2\ell)\theta}(\partial_{1}-i\partial_{2})^{k-\ell}(\partial_{1}+i\partial_{2})^{\ell}\}_{0\leq\ell\leq k}\subseteq\mathrm{span}_{\C}\{r^{-(k-\ell)}\partial_{r}^{\ell}\}_{0\leq\ell\leq k},
\]
we conclude the equality. Thus 
\[
\sum_{0\leq\ell\leq k}\|r^{-(k-\ell)}\partial_{r}^{\ell}f\|_{L^{2}}\sim_{k,m}\sum_{0\leq\ell\leq k}\|(\partial_{1}+i\partial_{2})^{\ell}(\partial_{1}-i\partial_{2})^{k-\ell}f\|_{L^{2}}\sim_{k,m}\|f\|_{\dot{H}_{m}^{k}}.
\]
This completes the proof.
\end{proof}
\begin{lem}[Degeneracy at the origin]
\label{lem:gen-hardy-Linfty}Assume $s>|m|+1$. Then, 
\[
\sup_{r>0}\sum_{\ell=0}^{|m|}|r^{-(|m|-\ell)}\partial_{r}^{\ell}f|\lesssim_{m,s}\|f\|_{H_{m}^{s}}.
\]
\end{lem}

\begin{proof}
Let $f$ be $m$-equivariant. From the proof of Lemma \ref{lem:gen-hardy-equiv},
we know that 
\[
\mathrm{span}_{\C}\{e^{i(|m|-2\ell)\theta}(\partial_{1}-i\partial_{2})^{|m|-\ell}(\partial_{1}+i\partial_{2})^{\ell}\}_{0\leq\ell\leq|m|}=\mathrm{span}_{\C}\{r^{-(|m|-\ell)}\partial_{r}^{\ell}\}_{0\leq\ell\leq|m|}.
\]
From the usual Sobolev embedding $H^{1+}\hookrightarrow L^{\infty}$,
\begin{align*}
\sup_{r>0}\sum_{\ell=0}^{|m|}|r^{-(|m|-\ell)}\partial_{r}^{\ell}f| & \lesssim_{m}\sum_{\ell=0}^{|m|}\|(\partial_{1}-i\partial_{2})^{|m|-\ell}(\partial_{1}+i\partial_{2})^{\ell}f\|_{L^{\infty}}\\
 & \lesssim_{m,s}\sum_{\ell=0}^{|m|}\|(\partial_{1}-i\partial_{2})^{|m|-\ell}(\partial_{1}+i\partial_{2})^{\ell}f\|_{H^{s-m}}\lesssim_{m,s}\|f\|_{H^{s}}.
\end{align*}
This completes the proof.
\end{proof}

\section{\label{sec:Local-Theory}Equivariant local theory}

In this section, we sketch the proof of Proposition \ref{prop:Hs-Cauchy}.
Arguments here are close to Liu-Smith \cite[Section 2]{LiuSmith2016}.
By a standard contraction argument, the proof reduces to show the
nonlinear estimates (Proposition \ref{prop:nonline-est-Hs}). Recall
the nonlinearity of \eqref{eq:CSS-coulomb-phi} 
\[
\mathcal{N}(\phi)=\frac{2m}{r^{2}}A_{\theta}\phi+\frac{1}{r^{2}}A_{\theta}^{2}\phi+A_{0}\phi-g|\phi|^{2}\phi.
\]

Decompose $A_{0}=A_{0}^{(1)}+A_{0}^{(2)}$ such that 
\begin{align*}
A_{0}^{(1)}[\phi] & \coloneqq-\int_{r}^{\infty}A_{\theta}[\phi]|\phi|^{2}\frac{dr'}{r'},\\
A_{0}^{(2)}[\phi] & \coloneqq-m\int_{r}^{\infty}|\phi|^{2}\frac{dr'}{r'}.
\end{align*}
We denote the bilinear form associated to $A_{\theta}$ by 
\[
A_{\theta}[\psi_{1},\psi_{2}]\coloneqq-\frac{1}{2}\int_{0}^{r}\Re(\psi_{1}\overline{\psi_{2}})r'dr'.
\]
We will use the following estimates: 
\begin{align}
\Big\|\int_{0}^{r}f(r')r'dr'\Big\|_{L^{\infty}} & \lesssim\|f\|_{L^{1}},\label{eq:B.1}\\
\Big\|\frac{1}{r^{2}}\int_{0}^{r}f(r')r'dr'\Big\|_{L^{2}} & \lesssim\|f\|_{L^{2}},\label{eq:B.2}\\
\Big\|\int_{r}^{\infty}f(r')\frac{dr'}{r'}\Big\|_{L^{2}} & \lesssim\|f\|_{L^{2}}.\label{eq:B.3}
\end{align}
Note that \eqref{eq:B.1} follow from the observation $2\pi rdr=dx$
on $\R^{2}$. The estimates \eqref{eq:B.2} and \eqref{eq:B.3} follows
by a change of variables and an application of Minkowski's inequality.

We first review the nonlinear estimate in Liu-Smith \cite{LiuSmith2016}.
\begin{lem}[$L^{2}$-critical nonlinear estimates]
\label{lem:L2-nonlinearity-est}We have 
\begin{align*}
\|\mathcal{N}(\phi)\|_{L_{t,x}^{\frac{4}{3}}} & \lesssim(|m|+|g|+\|\phi\|_{L_{t}^{\infty}L_{x}^{2}}^{2})\|\phi\|_{L_{t,x}^{4}}^{3}.
\end{align*}
\end{lem}

\begin{proof}
Observe 
\begin{align*}
\Big\|\frac{2m}{r^{2}}A_{\theta}\phi+\frac{1}{r^{2}}A_{\theta}^{2}\phi\Big\|_{L_{t,x}^{\frac{4}{3}}} & \lesssim(|m|+\|A_{\theta}\|_{L_{t,x}^{\infty}})\Big\|\frac{1}{r^{2}}A_{\theta}\Big\|_{L_{t,x}^{2}}\|\phi\|_{L_{t,x}^{4}},\\
\|A_{0}\phi\|_{L_{t,x}^{\frac{4}{3}}} & \leq\|A_{0}\|_{L_{t,x}^{2}}\|\phi\|_{L_{t,x}^{4}},\\
\|g|\phi|^{2}\phi\|_{L_{t,x}^{\frac{4}{3}}} & \leq|g|\|\phi\|_{L_{t,x}^{4}}^{3}.
\end{align*}
Thus, it suffices to show the estimates 
\begin{align}
\|A_{\theta}\|_{L_{t,x}^{\infty}} & \lesssim\|\phi\|_{L_{t}^{\infty}L_{x}^{2}}^{2},\label{eq:B.4}\\
\Big\|\frac{1}{r^{2}}A_{\theta}\Big\|_{L_{t,x}^{2}} & \lesssim\|\phi\|_{L_{t,x}^{4}}^{2},\label{eq:B.5}\\
\|A_{0}\|_{L_{t,x}^{2}} & \lesssim(|m|+\|\phi\|_{L_{t}^{\infty}L_{x}^{2}}^{2})\|\phi\|_{L_{t,x}^{4}}^{2}.\label{eq:B.6}
\end{align}
The estimate \eqref{eq:B.4} follows from \eqref{eq:B.1}. The estimate
\eqref{eq:B.5} follows from \eqref{eq:B.2}. The estimate \eqref{eq:B.6}
follows from \eqref{eq:B.3} and \eqref{eq:B.4}:
\[
\|A_{0}\|_{L_{t,x}^{2}}\lesssim(|m|+\|A_{\theta}\|_{L_{t,x}^{\infty}})\Big\|\int_{r}^{\infty}|\phi|^{2}(r')\frac{dr'}{r'}\Big\|_{L_{t,x}^{2}}\lesssim(|m|+\|\phi\|_{L_{t}^{\infty}L_{x}^{2}}^{2})\|\phi\|_{L_{t,x}^{4}}^{2}.
\]
This completes the proof.
\end{proof}
One can obtain analogous estimate for $\mathcal{N}(\phi_{1})-\mathcal{N}(\phi_{2})$.
This allows us to get equivariant $L^{2}$-Cauchy theory as in Proposition
\ref{prop:L2-Cauchy}.

We turn to $H_{m}^{s}$-subcritical Cauchy theory for $s>0$. Nonlinearities
involved in $A_{\theta}$ and $A_{0}$ are not of product type. In
view of the Biot-Savart law \eqref{eq:biot-savart}, there is no $\text{(low)}\times\text{(low)}\to\text{(high)}$
interaction in $A_{x}$, and hence in $A_{\theta}$. However, it is
not obvious that there is no $\text{(low)}\times\text{(low)}\to\text{(high)}$
interaction in $\frac{1}{r^{2}}A_{\theta}$. A key observation in
Liu-Smith \cite[Section 4]{LiuSmith2016} is that $\text{(low)}\times\text{(low)}\to\text{(high)}$
interaction is forbidden. This is formulated in the following two
frequency localization lemmas.
\begin{lem}[{Fourier transform of $A_{\theta}$ and $r^{-2}A_{\theta}$ \cite[Lemma 4.1]{LiuSmith2016}}]
\label{lem:FT-of-A_theta}We have 
\begin{align*}
\mathcal{F}[A_{\theta}](\rho) & =-\frac{1}{2\rho}\partial_{\rho}\mathcal{F}[|\phi|^{2}],\qquad\forall\rho\neq0,\\
\mathcal{F}\Big[\frac{1}{r^{2}}A_{\theta}\Big](\rho) & =-\frac{1}{2}\int_{\rho}^{\infty}\frac{1}{\rho'}\mathcal{F}[|\phi|^{2}]d\rho'.
\end{align*}
More generally, 
\begin{align*}
\mathcal{F}[A_{\theta}[\psi_{1},\psi_{2}]](\rho) & =-\frac{1}{2\rho}\partial_{\rho}\mathcal{F}[\Re(\overline{\psi_{1}}\psi_{2})],\qquad\forall\rho\neq0,\\
\mathcal{F}\Big[\frac{1}{r^{2}}A_{\theta}[\psi_{1},\psi_{2}]\Big](\rho) & =-\frac{1}{2}\int_{\rho}^{\infty}\frac{1}{\rho'}\mathcal{F}[\Re(\overline{\psi_{1}}\psi_{2})]d\rho'.
\end{align*}
\end{lem}

\begin{lem}[{Fourier transform of $A_{0}$ \cite[Lemma 4.2 and proof of Lemma 4.3]{LiuSmith2016}}]
\label{lem:FT-of-A_0}Define 
\[
Q_{12}(f,g)\coloneqq\partial_{1}f\partial_{2}g-\partial_{2}f\partial_{1}g.
\]
We have 
\begin{align*}
\mathcal{F}[A_{0}^{(1)}] & =\frac{1}{\rho}\partial_{\rho}\mathcal{F}\Big[\frac{1}{r^{2}}A_{\theta}|\phi|^{2}\Big],\\
A_{0}^{(2)} & =\Delta^{-1}\Im Q_{12}(\overline{\phi},\phi).
\end{align*}
More generally, 
\begin{align*}
 & \mathcal{F}\Big[\int_{r}^{\infty}\Big(\int_{0}^{r'}\Re(\overline{\psi_{1}}\psi_{2})r''dr''\Big)\Re(\overline{\psi_{3}}\psi_{4})\frac{dr'}{r'}\Big]\\
 & \qquad\qquad\qquad\qquad=-\frac{1}{\rho}\partial_{\rho}\mathcal{F}\Big[\frac{1}{r^{2}}\Big(\int_{0}^{r}\Re(\overline{\psi_{1}}\psi_{2})r'dr'\Big)\Re(\overline{\psi_{3}}\psi_{4})\Big],\\
 & \int_{r}^{\infty}m\cdot\Re(\overline{\psi_{1}}\psi_{2})\frac{dr'}{r'}=\Delta^{-1}\Im Q_{12}(\overline{\psi_{1}},\psi_{2}).
\end{align*}
\end{lem}

We use Lemmas \ref{lem:FT-of-A_theta} and \ref{lem:FT-of-A_0} to
show the nonlinear estimate required for Proposition \ref{prop:Hs-Cauchy}.
We use the Besov norm 
\[
\|f\|_{B_{r,2}^{s}}\coloneqq\|2^{js}P_{j}f\|_{\ell_{j\geq0}^{2}L_{x}^{r}},
\]
where $\{P_{j}\}_{j\geq0}$ are the Littlewood-Paley projectors on
$|\xi|\lesssim1$ if $j=0$ and $|\xi|\sim2^{j}$ if $j\geq1$.
\begin{prop}[Nonlinear estimates; higher regularity]
\label{prop:nonline-est-Hs}Let $s>0$. Then, 
\begin{align*}
\|\mathcal{N}(\phi)\|_{L_{t}^{\frac{4}{3}}B_{4/3,2}^{s}} & \lesssim_{s}(|m|+|g|+\|\phi\|_{L_{t}^{\infty}L_{x}^{2}}^{2})\|\phi\|_{L_{t,x}^{4}}^{2}\|\phi\|_{L_{t}^{4}B_{4,2}^{s}}\\
 & \qquad+\|\phi\|_{L_{t}^{\infty}L_{x}^{2}}\|\phi\|_{L_{t,x}^{4}}^{3}\|\phi\|_{L_{t}^{\infty}H_{x}^{s}}.
\end{align*}
\end{prop}

\begin{proof}
In what follows, implicit constants are allowed to depend on $s$.
Let $\frac{1}{p}=\frac{1}{p_{1}}+\frac{1}{p_{2}}$ with $1\leq p,p_{1},p_{2}\leq\infty$.
We then have a standard paraproduct estimate 
\begin{align*}
\Big\|2^{ks}\sum_{j\geq k}\|(P_{\leq j}f)(P_{j}g)\|_{L_{x}^{p}}\Big\|_{\ell_{k}^{2}} & \lesssim\|P_{\leq j}f\|_{\ell_{j}^{\infty}L_{x}^{p_{1}}}\Big\|\sum_{j\geq k}2^{(k-j)s}\cdot2^{js}\|P_{j}g\|_{L_{x}^{p}}\Big\|_{\ell_{k}^{2}}\\
 & \lesssim\|f\|_{L_{x}^{p_{1}}}\|g\|_{B_{p_{2},2}^{s}},
\end{align*}
where we used Schur's test in the second inequality. In particular,
we have 
\begin{align*}
\|2^{js}(P_{\leq j}f)(P_{j}g)\|_{\ell_{j}^{2}L_{x}^{p}} & \lesssim\|f\|_{L_{x}^{p_{1}}}\|g\|_{B_{p_{2},2}^{s}},\\
\|fg\|_{B_{p,2}^{s}} & \lesssim\|f\|_{L_{x}^{p_{1}}}\|g\|_{B_{p_{2},2}^{s}}+\|f\|_{B_{p_{1},2}^{s}}\|g\|_{L_{x}^{p_{2}}}.
\end{align*}
Thus we can estimate each term of the nonlinearity as 
\begin{align*}
\|g|\phi|^{2}\phi\|_{L_{t}^{\frac{4}{3}}B_{4/3,2}^{s}} & \lesssim|g|\|\phi\|_{L_{t,x}^{4}}^{2}\|\phi\|_{L_{t}^{4}B_{4,2}^{s}},\\
\Big\|\frac{2m}{r^{2}}A_{\theta}\phi\Big\|_{L_{t}^{\frac{4}{3}}B_{4/3,2}^{s}} & \lesssim|m|\Big(\Big\|\frac{1}{r^{2}}A_{\theta}\Big\|_{L_{t,x}^{2}}\|\phi\|_{L_{t}^{4}B_{4,2}^{s}}+\Big\|\frac{1}{r^{2}}A_{\theta}\Big\|_{L_{t}^{2}B_{2,2}^{s}}\|\phi\|_{L_{t,x}^{4}}\Big),\\
\Big\|\frac{1}{r^{2}}A_{\theta}^{2}\phi\Big\|_{L_{t}^{\frac{4}{3}}B_{4/3,2}^{s}} & \lesssim\Big(\Big\|\frac{1}{r^{2}}A_{\theta}\Big\|_{L_{t,x}^{2}}\|A_{\theta}\phi\|_{L_{t}^{4}B_{4,2}^{s}}+\Big\|\frac{1}{r^{2}}A_{\theta}\Big\|_{L_{t}^{2}B_{2,2}^{s}}\|A_{\theta}\phi\|_{L_{t,x}^{4}}\Big),\\
\|A_{0}\phi\|_{L_{t}^{\frac{4}{3}}B_{4/3,2}^{s}} & \lesssim(\|A_{0}\|_{L_{t,x}^{2}}\|\phi\|_{L_{t}^{4}B_{4,2}^{s}}+\|A_{0}\|_{L_{t}^{2}B_{2,2}^{s}}\|\phi\|_{L_{t,x}^{4}}).
\end{align*}
For the terms with $L^{2}$-critical Strichartz pairs, we can proceed
as in Lemma \ref{lem:L2-nonlinearity-est}. Thus it suffices to show
the estimates 
\begin{align}
\|A_{\theta}\phi\|_{L_{t}^{4}B_{4,2}^{s}} & \lesssim\|\phi\|_{L_{t}^{\infty}L_{x}^{2}}(\|\phi\|_{L_{t,x}^{4}}\|\phi\|_{L_{t}^{\infty}H_{x}^{s}}+\|\phi\|_{L_{t}^{\infty}L_{x}^{2}}\|\phi\|_{L_{t}^{4}B_{4,2}^{s}}),\label{eq:B.7}\\
\Big\|\frac{1}{r^{2}}A_{\theta}\Big\|_{L_{t}^{2}B_{2,2}^{s}} & \lesssim\|\phi\|_{L_{t,x}^{4}}\|\phi\|_{L_{t}^{4}B_{4,2}^{s}},\label{eq:B.8}\\
\|A_{0}^{(1)}\|_{L_{t}^{2}B_{2,2}^{s}} & \lesssim\|\phi\|_{L_{t}^{\infty}L_{x}^{2}}\|\phi\|_{L_{t,x}^{4}}(\|\phi\|_{L_{t,x}^{4}}\|\phi\|_{L_{t}^{\infty}H_{x}^{s}}+\|\phi\|_{L_{t}^{\infty}L_{x}^{2}}\|\phi\|_{L_{t}^{4}B_{4,2}^{s}}),\label{eq:B.9}\\
\|A_{0}^{(2)}\|_{L_{t}^{2}B_{2,2}^{s}} & \lesssim|m|\|\phi\|_{L_{t,x}^{4}}\|\phi\|_{L_{t}^{4}B_{4,2}^{s}}.\label{eq:B.10}
\end{align}
Our main tools are \eqref{eq:B.1}-\eqref{eq:B.3} and Lemmas \eqref{lem:FT-of-A_theta}-\eqref{lem:FT-of-A_0}.

To show \eqref{eq:B.7}, we use Lemma \ref{lem:FT-of-A_theta} to
observe\footnote{The display is not exactly correct, but it captures essential features
of the proof.} 
\begin{align*}
P_{k}(A_{\theta}\phi) & =P_{k}\bigg(\Big(2\sum_{\max\{j_{1},j_{2},j_{3}\}=j_{1}\geq k}+\sum_{\max\{j_{1},j_{2},j_{3}\}=j_{3}\geq k}\Big)A_{\theta}[P_{j_{1}}\phi,P_{j_{2}}\phi]P_{j_{3}}\phi\bigg)\\
 & =P_{k}\Big(2\sum_{j\geq k}A_{\theta}[P_{j}\phi,P_{\leq j}\phi]P_{\leq j}\phi+\sum_{j\geq k}A_{\theta}[P_{\leq j}\phi]P_{j}\phi\Big).
\end{align*}
We then estimate 
\begin{align*}
\|A_{\theta}\phi\|_{L_{t}^{4}B_{4,2}^{s}} & \lesssim\Big\|2^{sk}\sum_{j\geq k}\|A_{\theta}[P_{j}\phi,P_{\leq j}\phi]\|_{L_{x}^{\infty}}\|P_{\leq j}\phi\|_{L_{x}^{4}}\Big\|_{L_{t}^{4}\ell_{k}^{2}}\\
 & \qquad+\Big\|2^{sk}\sum_{j\geq k}\|A_{\theta}[P_{\leq j}\phi]\|_{L_{x}^{\infty}}\|P_{j}\phi\|_{L_{x}^{4}}\Big\|_{L_{t}^{4}\ell_{k}^{2}}\\
 & \lesssim\|P_{\leq j}\phi\|_{L_{t}^{\infty}\ell_{j}^{\infty}L_{x}^{2}}\|P_{\leq j}\phi\|_{L_{t}^{4}\ell_{j}^{\infty}L_{x}^{4}}\Big\|2^{sk}\sum_{j\geq k}P_{j}\phi\Big\|_{L_{t}^{\infty}\ell_{k}^{2}L_{x}^{2}}\\
 & \qquad+\|P_{\leq j}\phi\|_{L_{t}^{\infty}\ell_{j}^{\infty}L_{x}^{2}}^{2}\Big\|2^{sk}\sum_{j\geq k}P_{j}\phi\Big\|_{L_{t}^{4}\ell_{k}^{2}L_{x}^{4}}\\
 & \lesssim\|\phi\|_{L_{t}^{\infty}L_{x}^{2}}(\|\phi\|_{L_{t,x}^{4}}\|\phi\|_{L_{t}^{\infty}H_{x}^{s}}+\|\phi\|_{L_{t}^{\infty}L_{x}^{2}}\|\phi\|_{L_{t}^{4}B_{4,2}^{s}}).
\end{align*}

To show \eqref{eq:B.8}, we use Lemma \ref{lem:FT-of-A_theta} to
observe
\[
P_{k}\Big(\frac{1}{r^{2}}A_{\theta}\Big)=P_{k}\Big(\frac{1}{r^{2}}\sum_{\max\{j,\ell\}\geq k}A_{\theta}[P_{j}\phi,P_{\ell}\phi]\Big)=2P_{k}\Big(\frac{1}{r^{2}}\sum_{j\geq k}A_{\theta}[P_{j}\phi,P_{\leq j}\phi]\Big).
\]
We then use \eqref{eq:B.2} to estimate 
\begin{align*}
\Big\|\frac{1}{r^{2}}A_{\theta}\Big\|_{L_{t}^{2}B_{2}^{s,2}} & \lesssim\Big\|2^{sk}P_{k}(\frac{1}{r^{2}}\sum_{j\geq k}A_{\theta}[P_{j}\phi,P_{\leq j}\phi])\Big\|_{L_{t}^{2}\ell_{k}^{2}L_{x}^{2}}\\
 & \lesssim\Big\|2^{sk}\sum_{j\geq k}\|P_{j}\phi P_{\leq j}\phi\|_{L_{x}^{2}}\Big\|_{L_{t}^{2}\ell_{k}^{2}}\lesssim\|\phi\|_{L_{t,x}^{4}}\|\phi\|_{L_{t}^{4}B_{4,2}^{s}}.
\end{align*}

To show \eqref{eq:B.9}, we use Lemma \ref{lem:FT-of-A_0} to observe
\begin{align*}
P_{k}A_{0}^{(1)} & =2P_{k}\bigg(\Big(\sum_{\max\{j_{1},\dots j_{4}\}=j_{1}\geq k}+\sum_{\max\{j_{1},\dots,j_{4}\}=j_{3}\geq k}\Big)\\
 & \qquad\qquad\qquad\times\int_{r}^{\infty}A_{\theta}[P_{j_{1}}\phi,P_{j_{2}}\phi]\Re(\overline{P_{j_{3}}\phi}P_{j_{4}}\phi)\frac{dr'}{r'}\bigg)\\
 & =2P_{k}\Big(\sum_{j\geq k}\int_{r}^{\infty}\frac{1}{r'}A_{\theta}[P_{j}\phi,P_{\leq j}\phi]|P_{\leq j}\phi|^{2}dr'\\
 & \qquad\qquad\qquad+\sum_{j\geq k}\int_{r}^{\infty}A_{\theta}[P_{\leq j}\phi]\Re(\overline{P_{j}\phi}P_{\leq j}\phi)\frac{dr'}{r'}\Big).
\end{align*}
We then use \eqref{eq:B.3} and \eqref{eq:B.1} to estimate 
\begin{align*}
\|A_{0}^{(1)}\|_{L_{t}^{2}B_{2}^{s,2}} & \lesssim\Big\|2^{sk}\sum_{j\geq k}\|A_{\theta}[P_{j}\phi,P_{\leq j}\phi]\|_{L_{x}^{\infty}}\|P_{\leq j}\phi\|_{L_{x}^{4}}^{2}\Big\|_{L_{t}^{2}\ell_{k}^{2}}\\
 & \qquad+\Big\|2^{sk}\sum_{j\geq k}\|A_{\theta}[P_{\leq j}\phi]\|_{L_{x}^{\infty}}\|P_{j}\phi\|_{L_{x}^{4}}\|P_{\leq j}\phi\|_{L_{x}^{4}}\Big\|_{L_{t}^{2}\ell_{k}^{2}}\\
 & \lesssim\|P_{\leq j}\phi\|_{L_{t}^{\infty}\ell_{j}^{\infty}L_{x}^{2}}\|P_{\leq j}\phi\|_{L_{t}^{4}\ell_{j}^{\infty}L_{x}^{4}}^{2}\cdot\Big\|2^{sk}\sum_{j\geq k}\|P_{j}\phi\|_{L_{x}^{2}}\Big\|_{L_{t}^{\infty}\ell_{k}^{2}}\\
 & \qquad+\|P_{\leq j}\phi\|_{L_{t}^{\infty}\ell_{j}^{\infty}L_{x}^{2}}^{2}\|P_{\leq j}\phi\|_{L_{t}^{4}\ell_{j}^{\infty}L_{x}^{4}}\Big\|2^{sk}\sum_{j\geq k}\|P_{j}\phi\|_{L_{x}^{4}}\Big\|_{L_{t}^{4}\ell_{k}^{2}}\\
 & \lesssim\|\phi\|_{L_{t}^{\infty}L_{x}^{2}}\|\phi\|_{L_{t,x}^{4}}^{2}\|\phi\|_{L_{t}^{\infty}H_{x}^{s}}+\|\phi\|_{L_{t}^{\infty}L_{x}^{2}}^{2}\|\phi\|_{L_{t,x}^{4}}\|\phi\|_{L_{t}^{4}B_{4,2}^{s}}.
\end{align*}

To show \eqref{eq:B.10}, we use Lemma \ref{lem:FT-of-A_0} to observe
\[
P_{k}A_{0}^{(2)}=2P_{k}\sum_{j\geq k}\int_{r}^{\infty}\frac{m}{r'}\Re(\overline{P_{j}\phi}P_{\leq j}\phi)dr'.
\]
We then use \eqref{eq:B.3} to estimate 
\[
\|A_{0}^{(2)}\|_{L_{t}^{2}B_{2}^{s,2}}\lesssim|m|\Big\|2^{sk}\sum_{j\geq k}\|(P_{j}\phi)(P_{\leq j}\phi)\|_{L_{x}^{4}}\Big\|_{L_{t}^{2}\ell_{k}^{2}}\lesssim|m|\|\phi\|_{L_{t,x}^{4}}\|\phi\|_{L_{t}^{4}B_{4,2}^{s}}.
\]
This completes the proof.
\end{proof}
One can get an analogous estimate for the difference $\mathcal{N}(\phi_{1})-\mathcal{N}(\phi_{2})$.
One then has local existence and uniqueness for $H_{m}^{s}$ data
by the standard contraction principle. However, the local existence
in \emph{subcritical} sense is not clear. To get this, we use H\"older's
inequality in time and the embedding $B_{4-,2}^{s}\hookrightarrow L^{4}$
to observe 
\[
\|\phi\|_{L_{I,x}^{4}}\lesssim_{s}|I|^{0+}\|\phi\|_{L_{|I|}^{4+}B_{4-,2}^{s}}
\]
such that above pair $(4+,4-)$ is admissible. Observe that in Proposition
\ref{prop:nonline-est-Hs}, we have $L_{t,x}^{4}$ factor in the upper
bound. This guarantees local existence in subcritical sense. Thus
$H_{m}^{s}$-subcritical local theory (Proposition \ref{prop:Hs-Cauchy})
follows by standard arguments.

\bibliographystyle{plain}
\bibliography{References}

\end{document}